%% file: eulerOnRn.tex
\documentclass{article}
\usepackage{amsmath}
\usepackage{amsthm}
\usepackage{amssymb}
\usepackage{hyperref}
\usepackage{verbatim}
\usepackage{graphicx}
\usepackage[usenames,dvipsnames,svgnames,table]{xcolor}

\usepackage[hmargin=3cm,vmargin=3.5cm]{geometry}
\def\a{\alpha}
\def\b{\beta}
\def\ga{\gamma}
\def\Ga{\Gamma}
\def\de{\delta}
\def\De{\Delta}
\def\ep{\epsilon}

\def\la{\lambda}

\def\si{\sigma}

\def\Om{\Omega}

\def\th{\theta}

\def\nab{\nabla}
\def\varep{\varepsilon}

\def\CC{{\cal C}}

\def\II{{\cal I}}

\def\PP{{\cal P}}

\newcommand{\N}[0]{\mathbb{N}}

\newcommand{\R}[0]{\mathbb{R}}
\newcommand{\Z}[0]{\mathbb{Z}}

\newcommand{\T}[0]{\mathbb{T}}

\newcommand{\supp}{\mathrm{supp} \,}

\newcommand{\fr}[2]{\frac{#1}{#2}}

\newcommand{\ALI}[1]{\begin{align*} #1 \end{align*}}
\newcommand{\tx}[1]{\mbox{#1}}
\newcommand{\leqc}[0]{\lesssim}

\newcommand{\pr}[0]{\partial}

\newcommand{\co}[1]{\| #1 \|_{C^0}}

\newcommand{\Ddt}[0]{\fr{\bar D}{\partial t}}
\newcommand{\Ddtof}[1]{\fr{\bar D #1}{\partial t}}
\newcommand{\DDdt}[0]{\fr{\bar D^2}{\partial t^2}}

\newcommand{\ali}[1]{ \begin{align} #1 \end{align} }

\def\XXint#1#2#3{{\setbox0=\hbox{$#1{#2#3}{\int}$}
     \vcenter{\hbox{$#2#3$}}\kern-.5\wd0}}

\newcommand{\bMinus}{ } 

\newcommand{\nrm}[1]{\Vert#1\Vert}
\newcommand{\abs}[1]{\vert#1\vert}
\newcommand{\brk}[1]{\langle#1\rangle}
\newcommand{\set}[1]{\{#1\}}

\newcommand{\aeq}{\sim}

\newcommand{\ud}{d}
\newcommand{\rd}{\partial}
\newcommand{\nb}{\nabla}

\newcommand{\bb}{\Big}

\newcommand{\alp}{\alpha}
\newcommand{\bt}{\beta}
\newcommand{\gmm}{\gamma}

\newcommand{\dlt}{\delta}
\newcommand{\Dlt}{\Delta}
\newcommand{\eps}{\epsilon}

\newcommand{\Lmb}{\Lambda}
\newcommand{\sgm}{\sigma}

\newcommand{\tht}{\theta}

\newcommand{\Omg}{\Omega}

\newcommand{\zt}{\zeta}

\newcommand{\bfb}{{\bf b}}

\newcommand{\bfd}{{\bf d}}

\newcommand{\bbR}{\mathbb R}

\newcommand{\bbT}{\mathbb T}

\newcommand{\calF}{\mathcal F}

\newcommand{\calO}{\mathcal O}

\newcommand{\calU}{\mathcal U}

\newcommand{\LCyl}{\hat{\Gamma}}
\newcommand{\ECyl}{\hat{C}}

\newcommand{\approxMd}{\frac{\overline{\overline{D}}}{\rd t}}

\newcommand{\moll}[2]{{}^{(#1)} \widetilde{#2}}			
\newcommand{\spMP}{\bar{\rho}}				
\newcommand{\tMP}{\bar{\tau}}					
\renewcommand{\ell}{l}

\newtheorem{thm}{Theorem}[section]
\newtheorem{lem}{Lemma}[section]

\newtheorem{prop}{Proposition}[section]

\newtheorem{claim}{Claim}
\newtheorem{conj}{Conjecture}

\theoremstyle{definition}
\newtheorem{defn}{Definition}[section]

\theoremstyle{remark}
\newtheorem*{rem}{Remark}

\title{ On Nonperiodic Euler Flows with H\"{o}lder Regularity }
\author{ Philip Isett\thanks{Department of Mathematics, MIT, Cambridge, MA  (\href{mailto:isett@math.mit.edu}{isett@math.mit.edu}).  The work of P. Isett is supported supported by the National Science Foundation under Award No. DMS-1402370.} \\ Sung-Jin Oh\thanks{Department of Mathematics, UC Berkeley, Berkeley, CA (\href{mailto:sjoh@math.berkeley.edu}{sjoh@math.berkeley.edu}). S.-J. Oh is a Miller Research Fellow, and would like to thank the Miller Institute at UC Berkeley for support.}
}

\begin{document}
\maketitle

\begin{abstract}
In \cite{isett2}, the first author proposed a strengthening of Onsager's conjecture on the failure of energy conservation for incompressible Euler flows with H\"{o}lder regularity not exceeding $1/3$.  This stronger form of the conjecture implies that anomalous dissipation will fail for a generic Euler flow with regularity below the Onsager critical space $L_t^\infty B_{3,\infty}^{1/3}$ due to low regularity of the energy profile.  In this paper, we establish two theorems that may be viewed as first steps towards establishing the conjectured failure of energy regularity for generic solutions with H\"{o}lder exponent less than $1/5$.  

Our first result shows that any non-negative function with compact support and H\"older regularity $1/2$ can be prescribed as the energy profile of an Euler flow in the class $C^{1/5-\epsilon}_{t,x}$.  The exponent $1/2$ is sharp in view of a regularity result of Isett \cite{isett2}.  The proof employs an improved greedy algorithm scheme that builds upon that in Buckmaster--De~Lellis--Sz\'ekelyhidi \cite{deLSzeBuck}.  Our second result shows that any given smooth Euler flow can be perturbed in $C^{1/5-\epsilon}_{t,x}$ on any pre-compact subset of $\R\times \R^3$ to violate energy conservation.  In particular, there exist nonzero $C^{1/5-\epsilon}_{t,x}$ solutions to Euler with compact space-time support, generalizing previous work of the first author \cite{isett} to the nonperiodic setting.  

The construction of nonperiodic solutions involves new issues that are closely related to the conservation of angular momentum, most notably the need to construct symmetric tensors with a prescribed divergence $\pr_j R^{jl} = U^l$ and good decay.  A key idea to address this difficulty is the construction of operators with good transport properties that yield compactly supported solutions to the symmetric divergence equation for compatible data.  Through these operators and the use of spatially localized waves, we achieve simplifications in the iteration that are desirable from a physical point of view, including exponential growth of frequencies in the construction, and estimates that are consistent with the scaling and Galilean symmetries of the equations.
\end{abstract}

\tableofcontents

\section{Introduction}
\input{introEulerRn3}

\section{Organization of the Paper}
\input{organization}

\section{The Importance of Angular Momentum Conservation}\label{sec:theMainIssue}
\input{angMomentum}

\section{The Main Lemma}\label{sec:theMainLemma}
\input{mainLemEulerRn2}

\section{Preliminaries on Eulerian and Lagrangian Cylinders} \label{sec:prelimCyls}
\input{cylinders}

\section{Basic Technical Outline}\label{sec:techOutline}
\input{techOutline}

\section{The Shape of the Corrections}\label{sec:correctionShape}
\input{theCorrections}



\section{Choosing the Parameters} \label{sec:choosingParameters}
\input{choosingTheParams}

\section{Estimates for the Corrections} \label{sec:correctionEstimates}
\input{correctionEstimates}

\section{Estimates for the new Stress} \label{sec:stressEstimates}
\input{newStressEstimates}


\subsection{Applying the parametrix} \label{sec:applyParametrix}
\input{applyParametrix}

\section{Solving the Symmetric Divergence Equation}\label{sec:solvingSymmDiv}
\input{solveSymmDiv4}

\section{Perturbations of Smooth Euler Flows}\label{sec:mainLemImpliesMainThm}
\input{iterateMainLemImpThm2}

\section{Prescribing the Energy Profile} \label{sec:prescribeEnergy}
\input{prescribingEnergy3}

\appendix
\section{$h$-Principle for incompressible Euler on Euclidean space } \label{sec:hPrinc}
\input{hPrincipleRemark}

\bibliographystyle{alpha}
\bibliography{eulerOnRn}

\end{document}

%% file: introEulerRn3.tex
The present work concerns the construction of H\"{o}lder continuous solutions to the incompressible Euler equations on $\R \times \R^3$
\ali{
\label{eq:theEulerEqns}
\tag{E}
\begin{split}
\pr_t v^l + \pr_j(v^j v^l) + \pr^l p &= 0 \\
\pr_j v^j &= 0
\end{split}
}
that fail to conserve energy.  As we consider solutions with fractional regularity, what we mean by a solution to \eqref{eq:theEulerEqns} is a continuous velocity field $v : \R \times \R^3 \to \R^3$ and pressure $p : \R \times \R^3 \to \R$ that together satisfy \eqref{eq:theEulerEqns} in the sense of distributions.  For continuous solutions, this notion of solution may be formulated equivalently in terms of the integral laws of momentum balance and balance of mass, which are commonly used to derive \eqref{eq:theEulerEqns} in continuum mechanics; see \cite{deLSzeCtsSurv}.

A central question concerning weak solutions to \eqref{eq:theEulerEqns} is the possibility of dissipation or creation of energy for solutions to Euler in H\"{o}lder or Besov type spaces where the known results on energy conservation do not apply.  The interest in this question originates from a 1949 note of L. Onsager on statistical turbulence \cite{onsag}, wherein Onsager proposed a mechanism for turbulent energy dissipation driven by frequency cascades that he postulated may exist even among appropriately defined weak solutions to the inviscid equation \eqref{eq:theEulerEqns}.  There Onsager stated that energy is conserved by periodic solutions in the class $L_t^\infty C_x^\a$ if $\a > 1/3$, and conjectured that energy conservation may fail for such solutions if $\a < 1/3$ (see \cite{deLSzeCtsSurv, eyinkSreen} for detailed expositions).  The conservation of energy stated by Onsager was proven in \cite{eyink, CET}, and this result was refined in \cite{ches} to show that energy conservation holds for energy class solutions in the space $L_t^3 B_{3,c_0(\N)}^{1/3}$ on either $I \times \T^n$ or $I \times \R^n$ (see also \cite{duchonRobert, isettOh} for further proofs).  On the other hand, the proof of energy conservation fails for the space $L_t^3 B_{3,\infty}^{1/3}$, and an example in \cite{ches} suggests that anomalous dissipation of energy may be possible in this class.  The Besov regularity $\dot{B}_{p,\infty}^{1/3}$ carries a special significance in turbulence theory as it agrees with the $p = 3$ case of the scaling $\langle |v(x+ \De x) - v(x)|^p \rangle^{1/p} \sim \varep^{\fr{1}{3}} |\De x|^{\fr{1}{3}}$ predicted by Kolmogorov's theory \cite{K41}.  See \cite{eyinkDissip, shvOns} for further discussion.  Recently there has also been a series of advances towards the negative direction of Onsager's conjecture that we will discuss further below \cite{deLSzeCts, deLSzeHoldCts, isett, buckDeLIsettSze}.

Following the works above, the first author proposed in \cite{isett2} a stronger form of Onsager's conjecture that will be a main concern of the present work.  The conjecture of \cite{isett2} states that a generic solution to incompressible Euler with regularity at most $1/3$ will not only fail to conserve energy, but also will possess an energy profile of minimal regularity.  For periodic solutions in the class $C_t C_x^\a$ with $\a < 1/3$, the conjecture may be formulated precisely as follows:

\newpage

\begin{conj}[Generic Failure of Energy Regularity] \label{conj:energy:reg}  For any $\a < 1/3$, there exists a solution to \eqref{eq:theEulerEqns} in the class $v \in C_t C_x^\a(\R \times \T^n)$ whose energy profile $e(t) = \int_{\T^n} |v|^2(t,x) dx$ fails to have any regularity above the exponent $2\a/(1-\a)$, in the sense that\footnote{Here we use $W^{s,p}$ to denote the Sobolev space with ``$s$'' derivatives measured in $L^p$.} $e(t) \notin W^{2\a/(1-\a) + \ep, p}(I)$ for every $\ep > 0$, $p \geq 1$ and every open time interval $I \subseteq \R$.  

Furthermore, the set of all such solutions $v$ with the above property is residual (in the sense of category) within the space of all weak solutions to \eqref{eq:theEulerEqns} in the class $v \in C_t C_x^\a(\R \times \T^n)$ when the latter space is endowed with the topology from the $C_t C_x^\a$ norm.
\end{conj}

Conjecture~\ref{conj:energy:reg} conveys a sense in which anomalous dissipation should be unstable and nongeneric for weak solutions to Euler with regularity strictly below $1/3$.  
Assuming Conjecture~\ref{conj:energy:reg}, anomalous dissipation fails to hold for generic solutions to Euler in the class $C_t C_x^\a$ when $\a < 1/3$, as the energy profile of a typical solution in such a space will fail to be of bounded variation, and hence fail to be monotonic.  Instead, the only regularity one can expect for the energy profile of a solution in this class would be provided by the following estimate, proven in \cite{isett2}:
\ali{
\sup_t \sup_{\De t \neq 0} \fr{| e(t + \De t) - e(t) |}{|\De t|^{\fr{2 \a}{1 - \a}}} \leq C_\a \| v \|_{C_t \dot{B}_{3,\infty}^{\a}}^3. \label{eq:holderEnergy}
}
One expects that the $C^{\fr{2 \a}{1-\a}}$ bound above should be sharp, since the proof of \eqref{eq:holderEnergy} can be viewed as a generalization of the argument used by \cite{CET} to prove the positive direction of Onsager's conjecture.  (The proof of \eqref{eq:holderEnergy} gives more precise information, showing that the fluctuations in the energy profile at time scales of the order $\tau$ are governed by contributions from wavenumbers of the order $\tau^{-\fr{1}{1 - \a}}$.)


The formulation of Conjecture~\ref{conj:energy:reg} captures part of the intuition that even slight perturbations in a space of solutions with regularity below $1/3$ will typically produce small, rapid oscillations in time for the energy profile of the solution, and the regularity of these oscillations will be governed by the regularity of the perturbation in accordance with the proof of inequality \eqref{eq:holderEnergy}.  The same intuition offers a picture of what may be expected for solutions in the Onsager critical spaces $L_t^p \dot{B}_{3,\infty}^{1/3}$ for $p \leq \infty$, namely that anomalous dissipation (if possible) would be similarly nongeneric for solutions in $L_t^p \dot{B}_{3,\infty}^{1/3}$ for $p < \infty$, but in contrast would be stable under perturbation for Euler flows in the  $L_t^\infty \dot{B}_{3,\infty}^{1/3}$, where having a strictly positive rate of energy dissipation $- \fr{d}{dt} \int \fr{|v|^2}{2}(t,x) dx \geq \varep > 0$ is an open condition\footnote{This stability result derives from the following estimate for the difference of the energy profiles $e_1$, $e_2$ of two weak solutions to Euler in the class $v_1, v_2 \in L_t^\infty \dot{B}_{3,\infty}^{1/3}$ with domain $I \times \T^n$ or $I \times \R^n$, which was observed in \cite[Section 3]{isett2} by extending the argument of \cite{CET, ches}:
\ALI{
\left\| \fr{d}{dt} ( e_2 - e_1 ) \right\|_{L_t^\infty(I)} \leq C \| v_1 - v_2 \|_{L_t^\infty \dot{B}_{3,\infty}^{1/3}} \max \left\{ \| v_1 \|_{L_t^\infty \dot{B}_{3,\infty}^{1/3}} ,  \| v_2 \|_{L_t^\infty \dot{B}_{3,\infty}^{1/3}} \right\}^2.
}
Here it is important to consider solutions with uniform in time bounds rather than $L_t^p$ integrability, since the analogous estimate in the class $L_t^p \dot{B}_{3,\infty}^{1/3}$ for $3 < p < \infty$ controls only the $L_t^{p/3}$ norm of $\fr{de}{dt}$ (or the total variation norm of $\fr{de}{dt}$ in the case $p = 3$).  For $p < \infty$, one should expect instead that the set of all solutions with nonincreasing energy profiles would be a closed set with empty interior (and hence be nowhere dense) in the space of all $L_t^p B_{3,\infty}^{1/3}$ solutions.
}.  One goal of our work is to give rigorous support to the above intuition in the range of exponents $\a < 1/5$ for dimension $n = 3$.

Our first main result, Theorem~\ref{thm:eulerOnRn:presEnergy} below, proves the existence of solutions to Euler with energy profiles approaching the minimal regularity $2 \a / (1-\a)$ for $0 < \a < 1/5$, thus confirming that the $2\a/(1-\a)$-H\"{o}lder estimate \eqref{eq:holderEnergy} is sharp in this range.  This Theorem supports the intuition underlying Conjecture~\ref{conj:energy:reg}, as we show moreover that irregularity of the energy profile may arise from a compactly supported perturbation of the $0$ solution in $C_{t,x}^\a$.  As in previous constructions, the range $\a \geq 1/5$ is out of reach of our method.  The construction of $(1/5-\varep)$-H\"{o}lder solutions that fail to conserve energy was first achieved in \cite{isett} improving on initial constructions of $(1/10-\varep)$-H\"{o}lder solutions in \cite{deLSzeCts, deLSzeHoldCts} (see also \cite{deLSzeBuck, buckDeLIsettSze} for a shorter proof closer to the scheme of \cite{deLSzeCts, deLSzeHoldCts}).  We also note the construction of compactly supported solutions in the class $C_{t,x}^0 \cap L_t^1 C_x^{1/3 - \varep}$ by \cite{Buckmaster, buckDeLSzeOnsCrit}, although a result on the control of the energy profile does not appear to be available in this setting (c.f. Section~\ref{sec:localizing}).

\begin{thm}[Euler flows with prescribed energy profile] \label{thm:eulerOnRn:presEnergy}
Let $\a < 1/5$, let $I \subseteq \R$ be a bounded open interval, and let $\calU$ be an open subset of $\R^3$.  
Let $\bar{e}(t) \geq 0$ be any non-negative function with compact support in $I$ which belongs to the class $\bar{e}(t) \in C_t^{\ga}$ for some $\ga > \fr{2 \a}{1-\a}$.  Then there exists a weak solution $(v, p)$ to the incompressible Euler equations in the class $v \in C_{t,x}^\a(\R\times\R^3)$ with support contained in 
 \[ \supp v \cup \supp p \subseteq I \times \calU \] 
such that the energy profile of $v$ is equal to $\int_{\R^3} |v|^2(t,x) dx = \bar{e}(t)$ for all $t \in \R$.  Moreover, one may choose a one parameter family of solutions $(v_A, p_A)$, $0 \leq A \leq 1$, with the above properties such that the energy profile of $v_A$ is equal to $\int_{\R^3} |v_A|^2(t,x) dx = A \bar{e}(t)$ and such that $\| v_A \|_{C_{t,x}^\a} \to 0$ as $A \to 0$.
\end{thm}
Theorem~\ref{thm:eulerOnRn:presEnergy} builds upon work of \cite{deLSzeCts,deLSzeHoldCts,deLSzeBuck} for prescribing smooth energy profiles the periodic setting and on the organizational framework developed in \cite{isett}.  
We discuss the improvements in Theorem~\ref{thm:eulerOnRn:presEnergy} in Section~\ref{sec:newIdeas:energy} below.  We remark that our arguments also allow one to achieve an energy profile that does not have compact support provided the norm $\| e \|_{C_t^\ga} = \sup_t |e(t)| + \sup_t \sup_{|\De t| \neq 0} \fr{|e(t+\De t) - e(t)|}{|\De t|^\ga}$ is finite.



Our second main result, Theorem~\ref{thm:eulerOnRn:cptPert} below, shows that the failure of energy conservation and higher regularity may arise from perturbations of an arbitrary smooth, Euler flow.
\begin{thm}[Perturbation of smooth Euler flows] \label{thm:eulerOnRn:cptPert}
Let $(v_{(0)}, p_{(0)})$ be any smooth solution to the incompressible Euler equations on $\bbR \times \bbR^{3}$. Then for any $\eps, \dlt > 0$ and pre-compact open sets $\Omg_{(0)}$, $\calU$ such that $\Omg_{(0)} \neq \emptyset$ and $\overline{\Omg_{(0)}} \subseteq \calU$, there exists a weak solution $(v, p) \in C^{1/5-\eps}_{t,x} \times C^{2(1/5-\eps)}_{t,x}$ to the incompressible Euler equations on $\bbR \times \bbR^{3}$ such that the following statements hold:
\begin{enumerate}
\item The solutions $(v, p)$ and $(v_{(0)}, p_{(0)})$ coincide outside $\calU$, i.e.,
	\begin{equation} 
		(v, p) = (v_{(0)}, p_{(0)}) \hbox{ on } (\bbR \times \bbR^{3}) \setminus \calU.
	\end{equation}		

\item The solutions $(v, p)$ and $(v_{(0)}, p_{(0)})$ differ at most by $\dlt$ in the $C^{1/5-\eps}_{t,x} \times C^{2(1/5-\eps)}_{t,x}$ topology, i.e.,
	\begin{equation} 
		\nrm{v-v_{(0)}}_{C^{1/5-\eps}_{t,x}} + \nrm{p - p_{(0)}}_{C^{2(1/5-\eps)}_{t,x}} < \dlt.
	\end{equation}	 
\item For every $t \in \R$ and open set $\Om' \subseteq \R^3$ such that $\{ t \} \times \Om' \subseteq \Om_{(0)}$, the solution $v(t,x)$ fails to be in the class $v(t, \cdot) \notin C^{1/5}(\Om')$, and furthermore fails to belong to the Sobolev space $v(t, \cdot) \notin W^{1/5, 1}(\Om')$.  As a consequence, $v$ does not coincide with $v_{(0)}$ on any open subset of $\Om_{(0)}$.

\item There exists $t_\star \in \R$ and a smooth, non-negative function $\psi = \psi(x) \geq 0$ with compact support such that
\[ \{ x ~|~ (t_\star, x) \in \calU \} \subseteq \{ x ~|~ \psi(x) = 1 \} \]
and we have 
\begin{equation} \label{}
		\int_{\bbR^{3}} \psi(x) \frac{\abs{v(t_{\star},x)}^{2}}{2} \, \ud x > 
		\int_{\bbR^{3}} \psi(x) \frac{\abs{v_{(0)}(t_{\star},x)}^{2}}{2} \, \ud x.
	\end{equation}
In particular, the solution $v$ fails to conserve energy if its energy is finite. 
	
\end{enumerate}
\end{thm}
We view our proofs of Theorems~\ref{thm:eulerOnRn:presEnergy} and \ref{thm:eulerOnRn:cptPert} as first steps towards establishing Conjecture~\ref{conj:energy:reg} in the range of exponents $\a < 1/5$.  Namely, in the greater scheme of proving Conjecture~\ref{conj:energy:reg}, one could proceed by showing that the set of exceptions to the Conjecture is contained in a countable union of closed subsets of $C_tC_x^\a$ having empty interior.  Verifying the empty interior condition amounts to proving a perturbation result, which would roughly amount to showing that an arbitrary solution with $v \in C_t C_x^\a$ can be perturbed in $C_t C_x^\a$ to obtain a solution $\tilde{v} \in C_tC_x^\a$ whose energy profile fails to belong to $W^{2\a/(1-\a)+\ep, 1}(I)$ on every open interval $I \subseteq \R$.  Theorem~\ref{thm:eulerOnRn:presEnergy} shows that the trivial solution $v = 0$ can be perturbed in $C_tC_x^\a$ to achieve any given energy profile $\bar{e}(t)$ which is small in $C^{2\a/(1-\a) + \ep/2}$ (see Section~\ref{sec:remarksOnThm} below).  Our proof of Theorem~\ref{thm:eulerOnRn:cptPert} suggests that a similar perturbation should be possible with the $0$ solution replaced by an arbitrary smooth background flow. 

We make some remarks regarding Theorem \ref{thm:eulerOnRn:cptPert}.
\begin{itemize}
\item  Theorem \ref{thm:eulerOnRn:cptPert} holds as well for background solutions $(v_{(0)}, p_{(0)})$ which are defined only on some open set $\calO$ which contains $\overline{\calU}$. Indeed, all our arguments go through essentially verbatim, as all of our techniques are localized.
\item In terms of the Cauchy problem, Theorem \ref{thm:eulerOnRn:cptPert} demonstrates that, within the class of weak solutions constructed in the Theorem, uniqueness and conservation of energy fail for all smooth initial data in the energy class.
\end{itemize}

Theorems~\ref{thm:eulerOnRn:presEnergy} and \ref{thm:eulerOnRn:cptPert} provide the first constructions of finite energy, continuous solutions failing conserve energy that take place outside the setting of periodic tori.  In particular, we obtain failure of energy conservation for $C_{t,x}^{1/5-\ep}$ solutions on any bounded domain, and the existence of compactly supported solutions on $\R \times \R^3$ by taking $(v_{(0)}, p_{(0)}) \equiv 0$ and $\Omg_{(0)}$ to be a non-empty pre-compact open subset of a suitable domain $\calU$.  


An important goal of our work is to emphasize the perspective that Onsager's conjecture is inherently a local problem, where the main issue at hand concerns high frequency oscillations in the velocity field at small spatial scales.  Other results that help draw attention to this point of view are the works of \cite{duchonRobert, deLSzeAdmiss, isettOh}.  This local perspective on the problem is emphasized by the local character of Theorems~\ref{thm:eulerOnRn:presEnergy} and \ref{thm:eulerOnRn:cptPert}, and by the improvements in our construction that allow us to achieve this localization.


In considering the problem of constructing nonperiodic solutions, we are confronted with new issues that are closely connected to the conservation of angular momentum and did not arise in the previous work in the periodic setting
.  That is, in the setting of the whole space every weak solution to the Euler equations with finite energy and appropriate integrability conserves both linear and angular momentum, and these conservation laws pose further restrictions on the construction of weak solutions that were not present in the periodic setting.  Thus, even if one is only interested in constructing solutions with finite energy without requiring the additional property of compact support, there is essentially no way to avoid considerations regarding the conservation of angular momentum.  The main difficulty we face in this regard involves the construction of symmetric tensors with a prescribed divergence $\pr_j R^{jl} = U^l$ and good decay.  See Sections~\ref{sec:angMom:symmDiv}, \ref{sec:theMainIssue} and \ref{sec:solvingSymmDiv} below for further discussion.  At the same time, our method of constructing compactly supported solutions by localizing the construction also appears to be the most straightforward approach to obtaining finite energy, continuous solutions on the whole space or on a bounded domain.  

In connection with the conservation of angular momentum, we observe that our methods yield a result of $h$-principle type that is of independent interest.  The result we obtain (Theorem~\ref{thm:hPrinciple} below) states that any smooth incompressible velocity field with compact support that satisfies the conservation of linear and angular momentum can be realized as a limit in $L_{t,x}^\infty$ weak-$*$ of some sequence of compactly supported $C_{t,x}^{1/5-\ep}$ Euler flows.  This theorem contributes to the growing literature on $h$-principle type results in fluid equations \cite{deLSzeHFluid, choff, choffSzeStationary, isettVicol}.  
See Appendix~\ref{sec:hPrinc} below for further discussion.

The proofs of Theorems~\ref{thm:eulerOnRn:presEnergy} and \ref{thm:eulerOnRn:cptPert} are simplified substantially by the fact that we are able to obtain an exponential growth of frequencies in the iteration, and to truncate a parametrix expansion in the argument after a bounded number of steps.  These simplifications are achieved through the use of spatially localized waves, through a family of operators designed to solve the symmetric divergence equation (see Sections~\ref{sec:angMom:symmDiv} and \ref{subsec:intro:turbulence} below), and through the use of sharp estimates for the regularized velocity field, phase functions and stress that were developed in the work of \cite{isett} using an accelerated mollification technique.  The same novelties in the proof lead to other features in the construction that are desirable from a physical point of view, including a compatibility with the scaling and Galilean symmetries of the equations, and a self-similarity of the construction.  We discuss these further in Sections~\ref{sec:localizing} and \ref{subsec:intro:turbulence} below.  Our proof also features a simple proof of a key property of the mollification along the flow technique introduced in \cite{isett}, which is included in Section~\ref{subsec:MLMT:moll}.

Our overall construction is based on the method of convex integration that has been used to construct H\"{o}lder continuous Euler flows in the periodic setting \cite{deLSzeCts, deLSzeHoldCts, isett, buckDeLIsettSze}.  In particular, we follow rather closely the notation and framework developed in the first author's earlier paper \cite{isett}.  However, the present construction also involves several modifications compared to \cite{isett} that specifically address the issue of angular momentum conservation, and which are used to localize the construction.  We have therefore made an effort to give a summary of the new construction that is mostly self-contained, referring to \cite{isett} only for some basic results and estimates. 


\paragraph{Acknowledgements}

The authors are grateful to Peter Constantin for conversations related to Theorem~\ref{thm:eulerOnRn:cptPert}.  We thank Emil Wiedemann and Camillo De Lellis for encouraging our pursuit of Theorem~\ref{thm:eulerOnRn:presEnergy}.  
We also thank the Institut Henri Poincar\'e for its hospitality, where part of this work was done.

\subsection{New Ideas in the Proof of Theorem~\ref{thm:eulerOnRn:presEnergy}} \label{sec:newIdeas:energy}
Theorem~\ref{thm:eulerOnRn:presEnergy} on solutions with prescribed rough energy profiles builds upon work of \cite{deLSzeCts,deLSzeHoldCts,deLSzeBuck} in the periodic setting, which exhibit solutions whose energy profiles can be any given smooth, strictly positive function on a closed interval $[0,T]$.  These results show in particular that for any $\a < 1/5$ it is possible to construct solutions with $C_{t,x}^{\a}$ regularity whose energy profiles are strictly increasing or strictly decreasing (which we expect to be nongeneric solutions, as in Conjecture~\ref{conj:energy:reg}).  Theorem~\ref{thm:eulerOnRn:presEnergy} improves on these results by obtaining sharp regularity for the energy profile, and by removing the restriction of having a strictly positive lower bound on the desired energy profile.  To achieve these improvements, we develop a more delicate greedy algorithm for choosing the energy increments at each stage of the iteration, and develop a sharper form of the Main Lemma in the iteration that allows us to execute this algorithm.  A quadratic commutator estimate akin to the one used in the proof of energy conservation in \cite{CET, ches} (as well as the proof of the higher regularity \eqref{eq:holderEnergy} of the energy profile in \cite{isett2}) plays a key role in the proof.

Theorem~\ref{thm:eulerOnRn:presEnergy} appears to be inherently a more technical result when compared to other results that can be deduced from the construction, such as Theorem~\ref{thm:eulerOnRn:cptPert} or the prescription of smooth energy profiles.  In order to aid the readability of our manuscript, we devote the main body of our paper to the proof of Theorem~\ref{thm:eulerOnRn:cptPert} (which proceeds through a simpler Main Lemma), and provide the more technical application of Theorem~\ref{thm:eulerOnRn:presEnergy} in a subsequent Section~\ref{sec:prescribeEnergy}.

\subsection{New Ideas in the Construction}

The main new ideas in our construction revolve around the issue of angular momentum conservation and the related problem of localizing the construction to obtain compactly supported solutions.  The new ideas we employ result in some new features for the construction that are desirable from a physical point of view, including compatibility with the symmetries of the equations and the exponential growth of frequencies.

\subsubsection{Conservation of angular momentum and the symmetric divergence equation} \label{sec:angMom:symmDiv}

The point at which the conservation of momentum plays an important role during the construction occurs when solving the following underdetermined, elliptic equation, which we call the {\bf symmetric~divergence~equation:}
\ali{
 \pr_j R^{jl} &= U^l \label{eq:theDivEqn00}
}
Here $R^{jl} = R^{lj}$ is an unknown, symmetric tensor contributing to the stress in the next stage of the iteration, and $U^l = e^{i \la \xi} u^l$ is a high frequency vector field representing a force.  The main innovation in our construction is a new method for solving equation \eqref{eq:theDivEqn00} which enables us to control the support of the solution $R^{jl}$ and to obtain $C^0$ estimates for $R^{jl}$ and its derivatives which are compatible with dimensional analysis and with the transport structure of the problem.  Our solutions are given by explicit linear operators applied to the data $U^l$, which retain the property of compact support when the data $U^l$ is compactly supported, and produce solutions to \eqref{eq:theDivEqn00} whenever $U^l$ is $L^{2}$-orthogonal to constant and rotational vector fields, i.e.
\begin{equation} \label{eq:NC4divEqn}
	\int U^{l} \, \ud x = 0, \quad
	\int x^{k} U^{l} - x^{l} U^{k} \, \ud x =0, \hskip4em
	k, \ell = 1, \ldots, n.
\end{equation}
Note that the conditions in \eqref{eq:NC4divEqn} are necessary for a compactly supported solution to \eqref{eq:theDivEqn00} to exist (see Section~\ref{sec:theMainIssue} below for a more general statement).  



The starting point behind the construction of solution operators to \eqref{eq:theDivEqn00} can be illustrated in the context of the simpler problem of finding compactly supported solutions to the divergence equation
\ali{
\pr_\ell R^\ell &= U \label{eq:scalarDivEqn}
}
where $U$ is a scalar function and $R^\ell$ is an unknown vector field on $\R^n$.  This equation arises often in hydrodynamics, as well as in the foundations of the differential forms approach to degree theory \cite[Section 1.19]{taylorBook}.  Before we explain our method in the context of Equation~\ref{eq:scalarDivEqn}, it is important to observe that the problem of solving \eqref{eq:scalarDivEqn} is strictly less difficult than solving the symmetric divergence equation \eqref{eq:theDivEqn00}.  Namely, if one solves equation \eqref{eq:scalarDivEqn} for each component $U^l$ in the equation~\eqref{eq:theDivEqn00}, the tensor one obtains will solve \eqref{eq:theDivEqn00} but will not be symmetric.  Furthermore, in order to solve \eqref{eq:theDivEqn00}, the data must satisfy orthogonality conditions in \eqref{eq:NC4divEqn} which interrelate the components of the vector field $U^l$, and the operator we construct must be sensitive to these conditions.





Assume now that the scalar field $U$ in \eqref{eq:scalarDivEqn} has compact support, and satisfies $\int_{\R^n} U dx = 0$.  These conditions are clearly necessary for a compactly supported solution to \eqref{eq:scalarDivEqn} to exist.  Our starting point for solving \eqref{eq:scalarDivEqn} is that, when these necessary conditions are satisfied, one can obtain a solution to \eqref{eq:scalarDivEqn} by Taylor expanding in frequency space
\ali{
\widehat{U}(\xi) &= \widehat{U}(0) + \sum_{i=1}^n \xi_\ell \int_0^1 \pr^\ell \hat{U}(\si \xi) d\si = \sum_{i=1}^n i \xi_\ell  \widehat{R}^\ell(\xi) \label{eq:canSeeSolveDiv} \\
\widehat{R}^\ell(\xi) &= \fr{1}{i} \int_0^1 \pr^\ell \hat{U}(\si \xi) d\si \label{eq:taylorAbout0div}
}
From the physical space expression of \eqref{eq:taylorAbout0div} one can see that the vector field $\widehat{R}^\ell$ defined in \eqref{eq:taylorAbout0div} actually has compact support in a ball of radius $\rho$ about $0$ whenever $U$ is supported in the ball of radius $\rho$.  One can see also see from \eqref{eq:canSeeSolveDiv} that the vector field $R^\ell$ defined by \eqref{eq:taylorAbout0div} solves \eqref{eq:scalarDivEqn} whenever $\int U dx = \widehat{U}(0) = 0$.

We now view Formula~\eqref{eq:taylorAbout0div} as the frequency space representation of a linear operator applied to the scalar function $U$.  The problem with this operator is that the resulting solution $R^\ell$ apparently has a singularity at the origin in physical space.
Our cure for this problem is to ``spread out'' the singularity by taking advantage of the translation invariance of Equation \eqref{eq:scalarDivEqn}.  Namely, one can construct new solutions to \eqref{eq:scalarDivEqn} by conjugating the operator defined by \eqref{eq:taylorAbout0div} with a translation operator, thereby translating the singularity.  By taking a smooth average of such conjugates we obtain an operator which is explicit and does not have a singularity while also maintaining control over the support of the solution.  The operator obtained in this way turns out to coincide with a known formula introduced by Bogovskii\footnote{The authors thank Hao Jia and Peter Constantin for bringing our attention to Bogovskii's formula.} in \cite{bogovskii1980solution} for solving Equation~\eqref{eq:scalarDivEqn}.  The novelty here is that we obtain a conceptual derivation of this formula that generalizes to solving the symmetric divergence equation.


For the applications of our paper, the above approach must be generalized to obtain solutions to Equation \eqref{eq:scalarDivEqn} which are symmetric tensors.  As we have discussed, such solutions will not arise from solving the scalar divergence equation componentwise.  However, we are able to obtain operators which provide symmetric solutions by generalizing the above idea and taking a second order Taylor expansion in frequency space.  The key computation in the derivation of these operators is closely related to the proof that constant and rotational vector fields span the entire space of Killing vector fields on $\R^n$.  That is, the computation classifies all local obstructions to solving Equation~\ref{eq:theDivEqn00}, and hence characterizes all solutions $K^l$ to the Killing equation $\pr_j K_l + \pr_l K_j = 0$, which is dual to \eqref{eq:scalarDivEqn} (see Section~\ref{sec:theMainIssue}).


In our context, it is also important that the solution $R^{jl}(t,x)$ moves with time along the ambient coarse scale flow of the construction when the data $U^l(t,x)$ travels in the same way.  In other words, our operators should commute well with the advective derivative along the coarse scale flow.  We achieve good transport properties for our solution operators by taking advantage of the freedom to conjugate with any smooth family of translations we desire when defining the solution operator at each time slice.  By averaging with respect to a family of translations which moves along the ambient coarse scale flow, we are able to achieve solution operators with good commutator properties with respect to the advective derivative.  We refer to Section \ref{sec:solvingSymmDiv} for the full details of the solution to \eqref{eq:theDivEqn00}.


\subsubsection{Localization of the construction } \label{sec:localizing}

Our construction relies on the use of localized waves that are supported on small length scales which vary inversely with the frequency of the iteration.  In contrast, the constructions in the periodic setting use waves supported on length scales of order $\approx 1$ independent of the ambient frequency.  Thus, in our construction, the number of waves occupying each time slice is very large at high frequencies.  The corrections in the construction are modified so that they maintain the balance of angular momentum as well as the divergence-free property.  Our technique to achieve the angular momentum balance is to use corrections of a double curl form, although we consider other approaches which may be more natural from a geometric point of view.  

Due to the use of localized waves and a rearrangement of the error terms in the construction, we always solve \eqref{eq:theDivEqn00} with data which satisfies the necessary orthogonality conditions while simultaneously remaining localized to a small length scale $\rho$.  This smallness of support leads to a gain of a factor $\rho$ for the solution to \eqref{eq:theDivEqn00}, which is an estimate one expects from dimensional analysis
\[ \co{ R } \leqc \rho \co{ U }. \]
The gain of this smallness parameter $\rho$ allows us to achieve for the first time exponential (rather than double-exponential) growth of frequencies during the iteration.  Eliminating the need for double-exponential growth of frequencies in the iteration leads to some technical simplifications in the proof, and also leads to solutions that appear more natural from a physical point of view.

In contrast to the periodic case, where the increment to the energy in each stage of the iteration is a prescribed function of time $e(t)$, we prescribe a local energy increment $e(t,x)$ that is a function of both space and time, allowing for the possibility of compact support in time and space.  In order to ensure that our increments satisfy the required bounds on both spatial and advective derivatives, we apply the machinery of mollifying along the flow introduced in \cite{isett}.  For the Main Lemma of the iteration, we also prove estimates on the local energy increment (c.f. the estimates \eqref{eq:energyPrescribed} and \eqref{eq:DtenergyPrescribed} below).  These estimates are applied in our paper to prove the nontriviality of our solutions.  We hope that they may also be useful in future applications, such as the study of admissibility criteria for the Euler equations as initiated in \cite{deLSzeAdmiss}.

Using our bounds on the local energy increments and the other natural estimates of the construction, we prove that the solutions obtained from the iteration fail to belong to $v(t, \cdot) \notin C^{1/5 + \de}$ on essentially every open ball contained in their support.   While this lack of regularity is a new result concerning the solutions produced by the iteration, we emphasize that this property actually follows from the construction without any modifications.  For instance, the same argument shows that the solutions of \cite{isett, deLSzeBuck}, which belong to $C_t C_x^{\a^* - \ep}$ for some $\a^* < 1/5$ and all $\ep > 0$, actually fail to belong to $v(t, \cdot) \notin C^{\a^* + \de}$ for any $\de > 0$ on basically their whole support.  Our solutions, whose frequencies grow exponentially, necessarily fail to belong to $v(t, \cdot) \notin C^{1/5 + \de}$, and we show that they may also fail to be in $v(t, \cdot) \notin C^{1/5}$ by taking an appropriate choice of frequencies in the iteration.

Our result on the failure of higher regularity confirms that the estimates applied in \cite{isett} are sharp, and that any improvement in regularity for the solutions requires modifications in the construction.  For example, we see that the solutions of \cite{Buckmaster}, which possess regularity $v(t, \cdot) \in C^{1/3 - \ep}$ for a.e. $t$, must be obtained through a nontrivial modification of the construction of \cite{deLSzeBuck}.  The proof of this lack of regularity also suggests that such results may be more difficult to obtain without losing control over the energy profile, as we show that failure of spatial regularity follows from the same family of estimates that are used to control the energy profile.

Our techniques for localizing the construction and addressing the issue of angular momentum conservation lead to a framework that accords well with the symmetries of the Euler equations.  Thanks to the combination of our use of localized waves and our new method for solving the divergence equation \eqref{eq:theDivEqn00}, we have obtained an iteration framework whose bounds are all dimensionally correct with constants that are universal, and we have eliminated the need for the super-exponential growth of frequencies.  We have also maintained the property that the estimates of the iteration depend only on estimates for relative velocities as opposed to absolute velocities (for example, $\co{ \nab v }$ as opposed to $\co{v}$).  This property is natural in view of the Galilean symmetries of the Euler equations, and the fact that Onsager's conjecture itself is Galilean invariant.  



\subsubsection{Comparison with ideas in turbulence} \label{subsec:intro:turbulence}
		
Previous constructions of Euler flows by convex integration have led to many features that are regarded as unphysical and sharply contrast the well known description of turbulence in the physics literature.		One of the most glaringly unphysical features of previous solutions obtained through convex integration has been the requirement of double-exponentially growing frequencies in the iteration, which result in large gaps in the energy spectrum of the solutions.  In contrast, turbulent flows are well-known to exhibit a power law in their energy spectrum, which was first predicted by the foundational theory of Kolmogorov.  Another strange feature common to previous constructions of H\"{o}lder continuous solutions is the use of waves occupying length scales of size $\approx 1$ independent of their frequency.  Turbulent flows, on the other hand, have since the seminal ideas of Richardson been described to first approximation as being composed of a self-similar hierarchy of eddies occupying smaller and smaller length scales\footnote{It is important to note, however, that self-similarity in turbulence is known to fail. 
  The phenomenon of intermittency which describes this failure of self-similarity is an active area of research; see, for instance \cite{chesShv} for a recent mathematical approach expanding the model of \cite{frisch1978simple} and for further references.}.

The solutions constructed in the present work turn out to have a closer resemblance to the above physical descriptions of turbulent flows.  We use waves that are supported on small length scales which allow us to achieve an exponential growth of frequencies for the iteration while obtaining a purely local framework for the construction. This use of localized waves seems to be essentially forced on us by the problem.
The resulting solutions exhibit a self-similar structure similar to what was once imagined to be characteristic of turbulence.  From the point of view of comparing to turbulence, the most significant, unphysical characteristic that remains for our solutions is the failure to reach the regularity $1/3$ conjectured by Onsager, which is predicted by Kolmogorov's theory and agrees with certain experimental measurements of turbulent flows.  In this regard, it is important to remark that the new methods introduced here do not introduce any error terms which would prevent improvements in the H\"{o}lder regularity (see Remark \ref{rem:spaceScaleRemark}).  A closely related, unphysical property of our solutions is the presence of anomalous time scales in the construction that are inconsistent with the H\"{o}lder regularity $1/3$ and also differ from the natural time scales of turbulent flows.  We clarify this point in Section \ref{sec:anoalousTimeScales} below.

\subsubsection{A remark about anomalous time scales} \label{sec:anoalousTimeScales}

An important reason we fail to reach the conjectured $1/3$ regularity is the presence of anomalous time scales in the construction.  The construction involves partitions of unity in time, which impose a rapid time scale on the motion of the high frequency components of the solutions.  The idea of using time cutoffs in the construction of solutions by itself appears to be natural.  The rise and vanishing of time cutoffs is reminiscent of the physical picture of the formation and turnover of eddies in turbulent motion.  Also, as remarked in \cite{deLSzeBuck}, time cutoffs play a role in the construction similar to the role of CFL conditions in the numerical analysis of evolution equations.  What is unnatural about the time cutoffs we employ is that they are too sharp to be compatible with increased H\"{o}lder regularity of solutions.  We can see this discrepancy by comparing the advective derivative bounds which hold for our solutions, to those which must hold for arbitrary solutions with better spatial regularity.  

In \cite{isett2}, it was shown that any solution to Euler in the class $v \in C_t C_x^\a(I \times \R^n)$ for $\a > 0$ have Littlewood-Paley projections obeying the estimates
\ali{
\label{eq:coarseScaleLPestimates}
\begin{split}
\| P_{q+1} v \|_{C^0(I \times \R^n)} &\leq C 2^{- q \a } \| v \|_{C_t \dot{C}_x^\a(I \times \R^n)} \\
\| (\pr_t + P_{\leq q} v \cdot \nab )P_{q+1} v \|_{C^0(I \times \R^n)} &\leq C 2^{(1-2\a) q} \| v \|_{C_t \dot{C}_x^\a(I \times \R^n)}^2 
\end{split}
}
The cost of taking the advective derivative (namely, the factor $2^{(1 - \a)q} \| v \|_{C_t \dot{C}_x^\a}$) represents the inverse time scale at which the Littlewood-Paley piece of frequency $2^q$ moves along the coarse scale flow of $P_{\leq q} v$.  The estimates \eqref{eq:coarseScaleLPestimates} therefore require that the time scale associated with frequency $\la$ should be no smaller than $\tau \sim \la^{-(1-\a)}\| v \|_{C_t \dot{C}_x^\a}^{-1}$ for a solution in the class $C_t C_x^\a$ that saturates the first inequality in \eqref{eq:coarseScaleLPestimates}.  This time scale has a physical significance, as it agrees when $\a = 1/3$ with the the time scale $\varep^{-2/3} \la^{-2/3}$ obtained from turbulence theory for the turnover time of eddies at scale $\la^{-1}$.  As with the other constructions of $C_{t,x}^{1/5 - \ep}$ solutions, the time scale we choose for the motion of frequency $\la$ components is of the order $\tau \sim \la^{-4/5}$ (see Section \ref{sec:mainLemImpliesMainThm} below), which consequently limits the regularity of the construction to $C_t C_x^{1/5}$ since our construction may saturate both inequalities in \eqref{eq:coarseScaleLPestimates}.







%% file: organization.tex
We organize the remainder of the paper as follows.  Section \ref{sec:theMainIssue} is an introductory section which explains the underlying role played by the conservation of angular momentum in the construction, and why the consideration of this conservation law is essential for constructing finite energy solutions.  The results in this section are given for motivation, but are not used in the proof of the Theorem \ref{thm:eulerOnRn:cptPert}.  In Section \ref{sec:theMainLemma}, we state the Main Lemma of the paper, which provides an essentially complete statement of the result of a single iteration of the construction. 

After some preliminaries on the geometry of flow maps in Section \ref{sec:prelimCyls}, the construction itself begins in Section \ref{sec:techOutline}, which provides a high level summary of the scheme and derives a list of the error terms in the construction.  Section \ref{sec:correctionShape} finishes the description of the construction up to the choice of some length and time scale parameters that are determined in Section \ref{sec:choosingParameters}.  Sections \ref{sec:correctionEstimates} - \ref{sec:stressEstimates} are devoted to estimating the elements of the construction and the resulting error terms.  In Section \ref{sec:solvingSymmDiv}, we derive and prove estimates for the new operators which are applied in Section \ref{sec:applyParametrix} of the construction to solve the symmetric divergence equation \eqref{eq:theDivEqn00}.  In Section \ref{sec:mainLemImpliesMainThm}, we show how that the Main Lemma of the paper implies Theorem \ref{thm:eulerOnRn:cptPert} on the perturbation of smooth Euler flows.  In Section,~\ref{sec:prescribeEnergy}, we introduce a sharper version of the Main Lemma of the paper, and show how this sharper Main Lemma implies Theorem~\ref{thm:eulerOnRn:presEnergy} on the construction of Euler flows with a prescribed energy profile.

The paper concludes in Appendix~\ref{sec:hPrinc}, where we indicate how the Main Lemma of the paper can be used in combination with the operators of Section~\ref{sec:solvingSymmDiv} to yield a sharp $h$-Principle type result for incompressible Euler on $\R \times \R^3$.



%% file: angMomentum.tex
The purpose of the following introductory section is to explain in some detail why the conservation of angular momentum is the basic issue at hand in order to execute convex integration for the Euler equations on the full space $\R\times\R^n$ as compared to the torus $\R \times \T^n$.  The issues we discuss here will play an important underlying role in the analysis.  However, the results stated in this section are all very elementary and will not be used during the construction, so that the reader interested in the proofs of Theorems \ref{thm:eulerOnRn:presEnergy}-\ref{thm:eulerOnRn:cptPert} can pass to Section \ref{sec:theMainLemma} below.  It is interesting to observe, however, that the proofs of the results below are connected to the methods that are used in the construction.


The basic reason why we are forced to be concerned with the issue of angular momentum conservation to construct Euler flows on $\R \times \R^n$ is the observation that every weak solution to Euler with finite energy conserves both linear and angular momentum provided the linear and angular momentum are well-defined as classical integrals.  More precisely, the following Proposition holds
\begin{prop}\label{prop:conserveLaw}  Let $I$ be an open interval and let $(v, p)$ be a weak solution to the incompressible Euler equations such that
\ali{
 (1 + |x|) v \in L^1_{t,x}(I \times \R^n), \qquad v \in L^2_{t,x}(I \times \R^n) \label{eq:energyClass}
}
then for any vector field $K^l$ on $\R^n$ which solves the Killing equation
\ali{
 \pr_j K_l + \pr_l K_j &= 0 \label{eq:killingEqn}
}
we have (as a distribution in the variable $t$) the conservation law
\ali{
\fr{d}{dt} \int K_l v^l(t,x) ~\ud x &= 0 \label{eq:momentumConserveLaws}
}
\end{prop}
The geometric interpretation of equation \eqref{eq:killingEqn} is that the left hand side is the deformation tensor of the vector field $K$ (i.e. the Lie derivative of the Euclidean metric $\de_{jl}$ under the flow of $K$).  Thus, the solutions to \eqref{eq:killingEqn} are exactly those vector fields whose flows are isometries of the Euclidean space, and one can view the fact that the conservation laws \eqref{eq:momentumConserveLaws} arise from these symmetries $K^l$ as an application of Noether's theorem.  It is a classical fact that the space of solutions to \eqref{eq:killingEqn} on $\R^n$ is spanned by translations (i.e. constant vector fields) $e_i$, $i = 1, \ldots, n$ and rotational vector fields of the form
\ali{
 \Om_{(ij)} &= x^i e_j - x^j e_i, \qquad 1 \leq i < j \leq n. \label{eq:rotationVFs}
}
Hence, equation \eqref{eq:momentumConserveLaws} summarizes the conservation of both linear and angular momentum for Euler flows on $\R \times \R^n$.  In contrast, the torus $\T^n$ admits only translations but not rotations as symmetries, so in the periodic case there is no globally defined concept of angular momentum.

The fact that the linear conditions \eqref{eq:momentumConserveLaws} are satisfied for the solutions to Euler that we are trying to construct implies restrictions on the approximate solutions which can be used in the method of convex integration.  Namely, it is a basic principle of convex integration that the approximate solutions one uses in the construction turn out to be weak limits of solutions to the partial differential equation (or inclusion) that is being solved (see, for example, \cite{deLSzeHFluid,choff} and the introduction to \cite{isett} for more discussion).  As a consequence, we are forced to work with approximate solutions which likewise satisfy the conservation laws \eqref{eq:momentumConserveLaws}.

Beginning with the work of \cite{deLSzeCts}, the space of approximate solutions used to build continuous solutions to the incompressible Euler equations consists of the solutions to the following underdetermined system known as the ``Euler-Reynolds equations''
%
\ali{
\label{eq:euReynolds}
\begin{split}
\pr_t v^l + \pr_j(v^j v^l) + \pr^l p &= \pr_j R^{jl} \\
\pr_j v^j &= 0
\end{split}
}
Here, $R^{jl}$ is a symmetric tensor called the Reynolds stress whose trace-free part\footnote{The convention initiated in \cite{deLSzeCts} is slightly different in that the Reynolds stress is represented as ${\mathring R}^{jl}$ and there is an additional requirement that ${\mathring R}^{jl}$ has vanishing trace.  Although we will not use this convention, one obtains an equivalent definition of Euler-Reynolds flows since the trace part can be absorbed into the pressure gradient $\pr^l p = \pr_j(p \de^{jl})$.} measures the error by which $(v,p)$ fail to solve the Euler equations.  Solutions to \eqref{eq:euReynolds} are called Euler-Reynolds flows.  A well-known and important property of the equation \eqref{eq:euReynolds} is that it contains weak limits of solutions to the Euler equations.  Namely, the divergence free property of $v$ remains true after taking weak limits, and a weak limit of tensors $v^j v^l$ must be symmetric, even though it may fail to be rank $1$.  

Under appropriate decay assumptions, the space of Euler-Reynolds flows on $\R \times \R^n$ can also be viewed as the space of incompressible velocity fields which conserve both linear and angular momentum.  Namely, the conservation law \eqref{eq:momentumConserveLaws} can also be proven for Euler-Reynolds flows under the assumption that $R^{jl} \in L_{t,x}^1$ (which is exactly the integrability one obtains if $R^{jl}$ is obtained from weak limits of Euler flows with uniform bounds on $\|v\|_{L^2_{t,x}}$).  Thus, Euler-Reynolds flows in the class \eqref{eq:energyClass} with $R \in L^1$ conserve both linear and angular momentum.  Conversely, if $v^l$ is divergence free and conserves both linear and angular momentum, then (formally) one can represent $v$ as an Euler-Reynolds flow with $p = 0$ by solving the underdetermined, elliptic equation
\ali{
\pr_j R^{jl} &= U^l \label{eq:theDivEqn}
}
for $U^l = \pr_t v^l + \pr_j(v^j v^l)$.  That is, the conservation of linear and angular momentum for $v^l$ ensures that the implied force $U^l = \pr_t v^l + \pr_j(v^jv^l)$ is orthogonal to every solution $K^l$ to equation \eqref{eq:killingEqn}.  We can interpret the space of solutions to \eqref{eq:killingEqn} as the kernel of the operator $\pr_j K_l + \pr_l K_j$, which is (up to a sign) the adjoint to the divergence operator $\pr_j R^{jl}$ operating on symmetric tensors.  According to the Fredholm Alternative, we should therefore (at least formally) be able to solve \eqref{eq:theDivEqn} when the laws of conservation of linear and angular momentum are satisfied by $v$.


As previously noted, constructing solutions to the equation \eqref{eq:theDivEqn} is among the main technical steps in the construction of continuous Euler flows by convex integration, and we must execute this step more carefully on $\R^n$ than was done in the periodic setting.  In general, finding solutions to \eqref{eq:theDivEqn} can be done in a straightforward manner using the Fourier transform, which reduces the problem to solving an underdetermined linear equation 
\ali{
i \xi_j {\hat R}^{jl}(\xi) &= {\hat U}^l(\xi) \label{eq:undetLinearEqn}
}
for each Fourier coefficient.  One can then write down an explicit solution to \eqref{eq:undetLinearEqn} as a Fourier multiplier ${\hat R}^{jl}(\xi) = q^{jl}(\xi)[{\hat U}(\xi)]$ which is homogeneous of degree $-1$ in $\xi$ for $\xi \neq 0$.  On the torus, this formula gives a solution to the PDE \eqref{eq:theDivEqn} for any smooth data with integral $0$, and it turns out that this type of operator is sufficient for the constructions of solutions to Euler by convex integration (see \cite{deLSzeBuck}).  On $\R^n$, the same procedure gives a solution to \eqref{eq:theDivEqn} for any smooth, compactly supported data (regardless of whether $U$ satisfies any orthogonality conditions), but this solution $R^{jl}$ in general fails to be in $L^1$.  
Since the corrections $V$ to the velocity field in the construction have size $|V| \sim |R|^{1/2}$, one sees that solutions to Euler with finite energy $v \in L_{t,x}^2$ cannot be constructed in the same way and a new strategy which takes the angular momentum conservation into account is necessary. 

Previously, weak solutions to Euler with compact support have been constructed using schemes based on differential inclusions, including the original construction of compactly supported solutions on $L^2_{t,x}(\R \times \R^2)$ by \cite{scheff} and the constructions of compactly supported solutions on $L^\infty(\R \times \R^n)$ and $L^\infty \cap C_t L_x^2(\R \times \R^n)$ for $n \geq 2$ by \cite{deLSzeIncl, deLSzeAdmiss}.  In these constructions, the authors consider approximate solutions which solve certain underdetermined systems of linear PDE, and the basic building blocks of the constructions are explicit, oscillatory solutions to these PDE.  The approximate solutions considered in these papers can be viewed as restricted classes of solutions to the Euler-Reynolds system\footnote{The connection between the scheme of \cite{deLSzeIncl} and the Euler-Reynolds system is explained in \cite{deLSzeCtsSurv}.  The scheme in \cite{scheff} corresponds to a stress matrix of the form $R = \fr{-1}{2} \begin{pmatrix} - \xi_{21} & \xi_{22} \\ \xi_{22} & \xi_{21} \end{pmatrix}$ which is integrated against test functions of the form $(- \pr_2 \psi, \pr_1 \psi)$ in Definition 1.3 of \cite{scheff}.}, although the methods employed in these papers do not require solving the underdetermined equation \eqref{eq:theDivEqn}, so that the requirement of angular momentum conservation does not explicitly appear as an issue in these schemes.  Given the greater generality of the Euler-Reynolds equations and their pivotal role in constructing continuous solutions, one might expect that using the Euler-Reynolds system should be a more robust framework for executing convex integration for the Euler equations in general.  Our main theorems give some support to this expectation, but do not completely confirm this point, as there have been many other applications of the method in \cite{deLSzeIncl} which have not been obtained through the Euler-Reynolds framework (we note in particular the work of \cite{choffSzeStationary} on stationary solutions to Euler).

It is interesting to observe that some of the techniques which apply to the construction in the present paper can also be used to prove Proposition \ref{prop:conserveLaw} on conservation laws.  Thus, for introductory purposes, we sketch a proof of Proposition \ref{prop:conserveLaw} and the more general claim that Euler-Reynolds flows in the class \eqref{eq:energyClass} with $R^{jl} \in L^1_{t,x}$ conserve linear and angular momentum.
\begin{proof}[Proof of Proposition \ref{prop:conserveLaw}]
Observe that every solution to the Euler in the class \eqref{eq:energyClass} can be written as a solution to 
\ali{
\pr_t v^l + \pr^l p &= \pr_j R^{jl} \label{eq:partialReynolds}
}
with $R^{jl} \in L^1_{t,x}$ a symmetric tensor field.  The proof of \eqref{eq:momentumConserveLaws} proceeds by multiplying \eqref{eq:partialReynolds} by a smooth, compactly supported approximation to the vector field $\chi(t) K_l$ where $K_l$ is a solution to \eqref{eq:killingEqn}, and $\chi(t)$ is a smooth function of $t$ compactly supported in $I$.  Since we have no control over the integrability of the pressure $p = \De^{-1} \pr_l \pr_j R^{jl}$, we must restrict to integrating against approximations that are also divergence free vector fields.

First, suppose that $K = e_i$ is a coordinate vector field, and observe that $e_i$ can be written as the divergence of an antisymmetric tensor 
\[ e_i = \tx{div} [ x^j (e_j \otimes e_i - e_i \otimes e_j ) ], \qquad j \neq i \]
Using this expression we can approximate $e_i$ as a limit of divergence free vector fields with compact support $e_i = \lim_{B \to \infty} K^{(B)}$.  Namely, let $\phi \in C_c^\infty(\R^n)$ be a smooth function of compact support equal to $1$ on a neighborhood of $0$ and integrate equation \eqref{eq:partialReynolds} against the smooth, compactly supported, divergence free vector field
\[ \chi(t) K^{(B)} = \chi(t) \tx{div} \left[ \phi\left( \fr{x}{B} \right) x^j (e_j \otimes e_i - e_i \otimes e_j ) \right] \]
to obtain an approximate conservation law
\ali{
- \int \chi'(t) v^l K^{(B)}_l \ud t \ud x &= - \int \chi(t) \pr_j K^{(B)}_l R^{jl} ~\ud t \ud x\\
&= - \fr{1}{2} \int \chi(t) [ \pr_j K^{(B)}_l + \pr_l K^{(B)}_j ] R^{jl} ~\ud t \ud x
}
Now we can apply the dominated convergence theorem to conclude \eqref{eq:momentumConserveLaws}, since we have the pointwise bounds $\| K^{(B)} \|_{L^\infty(\R^n)} + \| (1 + |x|) (\pr_j K^{(B)}_l + \pr_l K^{(B)}_j) \|_{L^\infty(\R^n)}  \leq C$ independent of $B$ and we see that $K^{(B)} \to e_i$ and $\pr_j K^{(B)}_l + \pr_l K^{(B)}_j \to 0$ pointwise as $B \to \infty$.

For the rotation vector fields, \eqref{eq:rotationVFs}, we take approximations $\Om_{(ij)} = \lim_{B \to \infty} K^{(B)}$ of the form
\ali{  
K^{(B)} &= \phi\Big(\fr{|x|}{B}\Big) \Om_{(ij)} \label{eq:approxAngularFields}
}
where the cutoff function $\phi(\fr{|x|}{B})$ is a rescaling of a radial function that is equal to $1$ in a neighborhood of $0$.   The approximate rotation fields \eqref{eq:approxAngularFields} defined in this way are divergence free, since 
\ali{
\tx{div } [ \phi\Big(\fr{|x|}{B}\Big) \Om_{(ij)} ] &= \Om_{(ij)} \cdot \nab \left[ \phi\Big(\fr{|x|}{B}\Big) \right] = 0.
}
We also have the estimates $\| (1 + |x|)^{-1} K^{(B)} \|_{L^\infty(\R^n)} + \| (\pr_j K^{(B)}_l + \pr_l K^{(B)}_j) \|_{L^\infty(\R^n)}  \leq C$ independent of $B$ and we have the convergence $K^{(B)} \to \Om_{(ij)}$ and $\pr_j K^{(B)}_l + \pr_l K^{(B)}_j \to 0$ pointwise, so that Proposition \ref{prop:conserveLaw} follows by the same dominated convergence argument.
\end{proof}
As we will explain below (see the remark right before Section \ref{sec:applyParametrix}), the vector fields $\phi(\fr{|x|}{B}) \Om_{(ij)}$ which appear in the above proof can also be used as natural ``angular momentum correctors'' for the construction in the present paper, although we have used an alternative technique in the current construction.

Our construction of operators solving equation \eqref{eq:theDivEqn} relies on a calculation closely tied to the identity
\ali{
 \pr_i \pr_i K_j &= \pr_i ( \pr_i K_j + \pr_j K_i) - \pr_j(\pr_i K_i), \qquad i, j = 1, \ldots, n \label{eq:usingSymmDerivFormula}
}
which is the same formula that is at the heart of the characterization of solutions to \eqref{eq:killingEqn}.  Namely, a generalization of formula \eqref{eq:usingSymmDerivFormula} shows that the symmetric part of the derivative of $K$ determines the second derivative of $K$ in the direction of any line.  From a polarization argument, one can then deduce that every solution to \eqref{eq:killingEqn} on $\R^n$ has a vanishing second derivative, which quickly leads to the characterization of solutions to \eqref{eq:killingEqn} on $\R^n$.

%% file: mainLemEulerRn2.tex
In this section, we present the Main Lemma which is responsible for the proof of Theorem \ref{thm:eulerOnRn:cptPert}.  The purpose of this lemma is to describe precisely the result of one step of the convex integration procedure.  Theorem \ref{thm:eulerOnRn:cptPert} follows from iteration of this Lemma as we will explain in Section \ref{sec:mainLemImpliesMainThm}.

To state the Main Lemma, we recall the notion of frequency and energy levels for Euler-Reynolds flows introduced in Sections 9 and 10 of \cite{isett}.
\begin{defn} \label{def:subSolDef}
Let $L \geq 1$ be a fixed integer.  Let $\Xi \geq 2$, and let $e_v$ and $e_R$ be positive numbers with $e_R \leq e_v$.  Let $(v, p, R)$ be a solution to the Euler-Reynolds system.  We say that the frequency and energy levels of $(v,p,R)$ are below $(\Xi, e_v, e_R)$ (to order $L$ in $C^0 = C^0_{t,x}(\R \times \R^{3})$) if the following estimates hold.
\begin{align}
|| \nab^k v ||_{C^0} &\leq \Xi^k e_v^{1/2} &k = 1, \ldots, L \label{bound:nabkv} \\
|| \nab^k p ||_{C^0} &\leq \Xi^k e_v &k = 1, \ldots, L \label{bound:nabkp} \\
|| \nab^k R ||_{C^0} &\leq \Xi^k e_R &k = 0, \ldots, L \label{bound:nabkR} \\
|| \nab^k (\pr_t + v \cdot \nab) R ||_{C^0} &\leq \Xi^{k+1} e_v^{1/2} e_R  &k = 0, \ldots, L - 1 \label{bound:dtnabkR}
\end{align}
Here $\nab$ refers only to derivatives in the spatial variables.
\end{defn}
It is important to note that the bounds in Definition \ref{def:subSolDef} are consistent with the dimensional analysis of the Euler equations: namely, the frequency level $\Xi \sim [L]^{-1}$ is an inverse length and the energy levels $e_v$ and $e_R$ have the dimensions of an energy density $e_v, e_R \sim \fr{[L]^2}{[T]^2}$.  We refer to Sections 9 and 10 of \cite{isett} for the motivation for this definition and further discussion.

Our Main Lemma is based on the Main Lemma in Section 10 of \cite{isett} but also keeps track of how the support of the approximate solution enlarges after the addition of a correction.  The way the support enlarges is governed by the geometry of the flow map of the velocity field $v$, so the following definition will be useful for keeping track of this support.
\begin{defn}[$v$-adapted Eulerian cylinder]\label{def:vCylinder}  Let $\Phi_s = \Phi_s(t,x)$ be the flow map associated to a vector field $v$.  Given $\tau, \rho > 0$ and a point $(t_{0}, x_{0})$ of the space-time $\R \times \R^{3}$, we define the \emph{$v$-adapted Eulerian cylinder} ${\hat C}_{v}(\tau, \rho; t_{0}, x_{0})$ centered at $(t_{0}, x_{0})$ with duration $2 \tau$ and base radius $\rho > 0$ to be
\begin{equation} \label{eq:cmptCyl}
	{\hat C}_{v}(\tau, \rho; t_{0}, x_{0}) := \big\{ \Phi_{s}(t_0, x_0) + (0, h) ~:~ |s| \leq \tau, |h| \leq \rho \big\} 
\end{equation}
In other words, ${\hat C}_{v} (\tau, \rho; t_{0}, x_{0})$ is the union of spatial balls of radius $\rho$ about the trajectory of $(t_{0}, x_{0})$ along the flow of $v$ for $t \in [t_{0} - \tau, t_{0} + \tau]$.

Similarly, if $S \subseteq \R \times \R^3$ is a set, we define
\begin{equation} \label{eq:cmptCylSet}
	{\hat C}_{v}(\tau, \rho; S) := \bigcup_{(t_0, x_0) \in S} {\hat C}_{v}(\tau, \rho; t_{0}, x_{0})
\end{equation}
\end{defn}
With these definitions in hand, we can state the Main Lemma.
\begin{lem}[The Main Lemma] \label{lem:mainLemma}
Suppose that $L \geq 2$.  Let $K$ be the constant in Section 7.3 of \cite{isett}, and let $M \geq 1$ be a constant.  
There exist constants $C_{0}, C > 1$, which depend only on $M$ and $L$, such that that following holds:

Let $(v,p,R)$ be any solution of the Euler-Reynolds system whose frequency and energy levels are below $(\Xi, e_v, e_R)$
 to order $L$ in $C^0$.

Define the time-scale $\th = \Xi^{-1} e_v^{-1/2}$, and let
\[ e(t, x) : \R\times \R^3 \to \R_{\geq 0} \]
be any non-negative function which satisfies the lower bound
\begin{align}
 e(t,x) \geq K e_R \quad \quad \mbox{ for all } (t,x) \in {\hat C}_{v}(\th, \Xi^{-1}; \supp R) \label{eq:lowBoundEoftx}
\end{align}
(using the notation of Definition \ref{def:vCylinder}) and whose square root satisfies the estimates
\begin{align}
|| \nab^k (\pr_t + v \cdot \nab)^r e^{1/2} ||_{C^0} &\leq M \Xi^k (\Xi e_v^{1/2})^r e_R^{1/2} & 0 \leq r \leq 1, 0 \leq k + r \leq L \label{ineq:goodEnergy}
\end{align}

Now let $N$ be any positive number obeying the bound
\begin{align} 
 N &\geq \left(\fr{e_v}{e_R} \right)^{3/2}\label{eq:conditionsOnN2}
\end{align}
and define the dimensionless parameter ${\bf b} = \left(\fr{e_v^{1/2}}{e_R^{1/2}N} \right)^{1/2}$.

Then there exists a solution $(v_1, p_1, R_1)$ of the Euler-Reynolds system of the form $v_1 = v + V$, $p_1 = p + P$ whose frequency and energy levels are below
\begin{align}
 (\Xi', e_{v}', e_{R}') 
 &= (C_{0} N \Xi, e_R, \left(\fr{e_v^{1/2}}{e_R^{1/2}N} \right)^{1/2} e_R   ) \\
&= (C_{0} N \Xi, e_R, {\bf b}^{-1} \fr{e_v^{1/2} e_R^{1/2}}{N}) \label{eq:theNewEnergyLevel}
\end{align}
to order $L$ in $C^0$, and whose stress $R_1$ is supported in
\begin{align}
 \supp R_1 \subseteq \ECyl_{v}(\th, \Xi^{-1}; \supp e) \label{eq:goalForR1supp}
\end{align}
The correction $V = v_1 - v$ is of the form $V = \nab \times W$ and can be guaranteed to obey the bounds
\begin{align}
||V||_{C^0} &\leq C e_R^{1/2} \label{eq:Vco} \\
||\nab V ||_{C^0} &\leq C N \Xi e_R^{1/2} \label{eq:nabVco} \\
||(\pr_t + v^j \pr_j) V ||_{C^0} &\leq C {\bf b}^{-1} \Xi e_v^{1/2} e_R^{1/2} \label{eq:matVco} \\
||W||_{C^0} &\leq C \Xi^{-1} N^{-1} e_R^{1/2} \label{eq:Wco}\\
\co{ \nab W } &\leq C e_R^{1/2} \label{eq:nabWco} \\
||(\pr_t + v^j \pr_j) W ||_{C^0} &\leq C {\bf b}^{-1} N^{-1} e_v^{1/2} e_R^{1/2} \label{eq:matWco}
\end{align}
The energy of the correction can be prescribed locally up to errors bounded uniformly in $t$ by
\begin{align}
\left| \int_{\R^3} |V|^2(t,x) \psi(x) dx - \int_{\R^3} e(t,x) \psi(x) dx \right| &\leq C \fr{e_v^{1/2} e_R^{1/2}}{N} \left( \| \psi \|_{L^1} + \Xi^{-1} \| \nab \psi \|_{L^1} \right) \label{eq:energyPrescribed}
\end{align}
for any smooth test function $\psi(x) \in C_c^\infty(\R^3)$.  Furthermore, the local variations of the energy increment are bounded by
\ali{
\left|  \fr{d}{dt} \int_{\R^3} |V|^2 \psi(t,x) dx \right| &\leq C {\bf b}^{-1} \Xi e_v^{1/2} e_R (\| \psi \|_{L^1} + \Xi^{-1} \| \nab \psi \|_{L^1} + \Xi^{-2} \| \nab^2 \psi \|_{L^1}) \label{eq:DtenergyPrescribed} \\
&+ e_R (\| (\pr_t + v \cdot \nab) \psi \|_{L^1} + \Xi^{-1} \| \nab (\pr_t + v \cdot \nab ) \psi \|_{L^1} ) \notag
}
at any fixed time $t \in \R$ and for any smooth test function $\psi(t,x) \in C_c^\infty(\R \times \R^3)$.  In \eqref{eq:energyPrescribed} - \eqref{eq:DtenergyPrescribed}, we always mean $L^1 = L^1(\R^3)$.

The correction to the pressure $P = p_1 - p_0$ satisfies the estimates
\begin{align}
|| P ||_{C^0} &\leq C e_R \label{eq:Pco} \\
|| \nab P ||_{C^0} &\leq C N \Xi e_R \label{eq:nabPco} \\
|| (\pr_t + v\cdot \nab) P ||_{C^0} &\leq C {\bf b}^{-1} \Xi e_v^{1/2} e_R \label{eq:matPco}
\end{align}
Finally, the space-time supports of $V$ and $P$ are also contained in 
\ali{
\supp V \cup \supp P &\subseteq \ECyl_{v}(\th, \Xi^{-1};~ \supp e ) \label{eq:goalForVPsupp}
}
\end{lem}

Lemma \ref{lem:mainLemma} is very similar to the Main Lemma in \cite{isett}, but there are a few differences which are important to observe.  

Unlike the Main Lemma in \cite{isett}, Lemma \ref{lem:mainLemma} is entirely consistent with dimensional analysis and does not impose a restriction on $N$ that would force super-exponential growth of frequencies.  Namely, the Main Lemma of \cite{isett} imposes an additional condition $N \geq \Xi^\eta$ for some $\eta > 0$, and this condition on the frequency growth parameter forces a double-exponential growth of frequencies when the lemma is iterated to construct solutions to Euler.  The condition $N \geq \Xi^\eta$ is also unfavorable for being inconsistent with dimensional analysis, as the parameter $\Xi$ has the dimensions of an inverse length, whereas the parameter $N$ is supposed to be dimensionless.  By excluding the requirement $N \geq \Xi^\eta$, Lemma \ref{lem:mainLemma} is now completely consistent with dimensional analysis, and hence agrees with the scaling symmetries of the Euler-Reynolds equations.  Furthermore, Lemma \ref{lem:mainLemma} allows for an exponential (rather than double-exponential) growth of frequencies in the iteration, which gives our solutions a closer resemblance to the classical picture of turbulent flows.

Another feature of Lemma \ref{lem:mainLemma} contrasting the Main Lemma of \cite{isett} is that Lemma \ref{lem:mainLemma} keeps track of the enlargement of support of $R$ in terms of the $v$-compatible Eulerian cylinders in Definition \ref{def:vCylinder}.  Also, the function $e(t,x)$ which determines the increment of energy to the system is a function of both time and space rather than simply a function of time $e(t)$ as in \cite{isett}.  Thus, the required estimates \eqref{ineq:goodEnergy} for $e(t,x)$ are stated in terms of both advective and spatial derivatives, and the lower bound \eqref{eq:lowBoundEoftx} is stated in terms of the $v$-compatible cylinders.  In order to apply Lemma \ref{lem:mainLemma} to construct solutions, we will have to show that there exist energy profiles which satisfy the necessary conditions \eqref{eq:lowBoundEoftx} and \eqref{ineq:goodEnergy} for any given values of $\Xi$ and $e_v$.  We have included an additional parameter $M$ in \eqref{ineq:goodEnergy} to ensure that such functions can be constructed.  

The estimates \eqref{eq:energyPrescribed} and \eqref{eq:DtenergyPrescribed} for the energy increment also differ from those of \cite{isett} for the increment to the total energy.  Here our energy increment estimates are localized as they are stated in terms of a test function $\psi$.  This type of estimate allows us to establish local properties of the resulting solutions, including the failure of local $C^{1/5}$ regularity stated in Theorem \ref{thm:eulerOnRn:cptPert}.  The estimates \eqref{eq:energyPrescribed} and \eqref{eq:DtenergyPrescribed} are also more natural in terms of dimensional analysis.  From this point of view, the factor of $( \| \psi \|_{L^1} + \Xi^{-1} \| \nab \psi \|_{L^1} )$ has the dimensions of volume when we regard $\psi$ as being dimensionless.  The units of volume have been normalized to agree with units of mass $[M]$ in the physical derivation of the Euler equations.  Therefore, both sides of \eqref{eq:energyPrescribed} have the dimensions of energy $\fr{[M][L]^2}{[T]^2}$, and likewise both sides of \eqref{eq:DtenergyPrescribed} have the units of energy per unit time.

Finally, we point out that, as in \cite{isett}, the enlargement of support is expressed in terms of the support of $R$, rather than the support of $v$ and $p$.  Thus, if the Euler-Reynolds flow $(v,p,R)$ satisfies the Euler equations except on a compact subset $\overline{\Om} \subseteq \R \times \R^3$ on which $R$ is supported, the corrections to the pressure and velocity and the resulting error $R_1$ can be made to have compact support in a neighborhood of $\overline{\Om}$ by the appropriate choice of $e(t,x)$, even if the ambient velocity and pressure $(v,p)$ do not have compact support.  


Now we begin the proof of Lemma \ref{lem:mainLemma}.  We start in Section \ref{sec:prelimCyls} with some preliminary lemmas concerning the geometry of the Eulerian cylinders of Definition \ref{def:vCylinder} that will play an important role in the proof.  We then give a technical outline of the scheme in Section \ref{sec:techOutline} wherein we organize a list of the error terms in the construction.  We continue the proof of Lemma~\ref{lem:mainLemma} through Section~\ref{sec:solvingSymmDiv}. 

%% file: cylinders.tex
Here we collect some basic facts about the geometry of Eulerian cylinders which will be useful during the construction.  We will assume throughout this section that we are working with time-dependent vector fields $v(t,x) = (v^1, v^2, v^3)$ defined on $\R \times \R^3$ which are continuous in $(t,x)$ and $C^1$ in the spatial variables with uniform bounds on $\| \nab v \|_{C^0}$.  We denote by $\Phi_s$ the flow map associated to $v$, which is the one-parameter group of mappings $\Phi_s : \R \times \R \times \R^3 \to \R \times \R^3$ generated by the space-time vector field $\pr_t + v \cdot \nab$.  If $v$ is defined only on an open subset of $\R \times \R^3$, then likewise $\Phi_s(t,x)$ is defined only on an open subset of $\R \times \R \times \R^3$.

In addition to Eulerian cylinders, we will also be interested in the concept of a Lagrangian cylinder adapted to a vector field $v$, which we define as follows.
\begin{defn}[$v$-adapted Lagrangian cylinder]\label{def:vLagCylinder}  Let $v = (v^1, v^2, v^3)$ be as above.  Let $\Phi_s = \Phi_s(t,x)$ be the flow map associated to a vector field $v$.  Given $\tau, \rho > 0$ and a point $(t_{0}, x_{0})$ of the space-time $\R \times \R^{3}$, we define the \emph{$v$-adapted Lagrangian cylinder} ${\hat \Ga}_{v}(\tau, \rho; t_{0}, x_{0})$ centered at $(t_{0}, x_{0})$ with duration $2 \tau$ and base radius $\rho > 0$ to be
\begin{equation} \label{eq:cmptLagCyl}
	\LCyl_{v}(\tau, \rho; t_{0}, x_{0}) := \big\{ \Phi_{s}(t_0, x_0 + h) ~:~ |s| \leq \tau, |h| \leq \rho \big\} 
\end{equation}
In other words, ${\hat \Ga}_{v} (\tau, \rho; t_{0}, x_{0})$ is the union of trajectories for times $t \in [t_{0} - \tau, t_{0} + \tau]$ emanating from a spatial ball of radius $\rho$ about $x_{0}$.

Similarly, if $S \subseteq \R \times \R^3$ is a set, we define
\begin{equation} \label{eq:cmptLagCylSet}
	\LCyl_{v}(\tau, \rho; S) := \bigcup_{(t_0, x_0) \in S} {\hat \Ga}_{v}(\tau, \rho; t_{0}, x_{0})
\end{equation}
\end{defn}
Throughout the proof, we will often make use of the following duality between Eulerian and Lagrangian cylinders:
\ali{
 (t', x') \in \ECyl_v(\tau, \rho; t,x) &\iff (t,x) \in \LCyl_v(\tau, \rho; t', x')  \label{iff:EuLagDuality}
}

Our first Lemma provides the most basic estimate on the geometry of the flow of $v$.
\begin{lem} \label{lem:MLMT:basicEst4FlowFirst}
Let $v = (v^{1}, v^{2}, v^{3})$ be as above 
and let $\Phi_s(t,x) = (t+s, \Phi_s'(t,x))$ be the flow map associated to $\pr_t + v \cdot \nab$.  Then for every $(t_{0}, x_{0}), (s, h) \in \bbR \times \bbR^{3}$, we have
\begin{equation}
	\abs{h} e^{- s \nrm{\nb v}_{C^{0}}} 
	\leq \abs{ \Phi_{s}'(t_{0}, x_{0}) -  \Phi_{s}'(t_{0}, x_{0} + h)}
	\leq \abs{h} e^{s \nrm{\nb v}_{C^{0}}}.
\end{equation}
\end{lem}
\begin{proof} 
Define $\bfd^{2}(s) := \abs{ \Phi_{s}'(t_{0}, x_{0}) -  \Phi_{s}'(t_{0}, x_{0} + h)}^{2}$. Recalling the definition of $ \Phi'_{s}$, we easily compute
\begin{align*}
	\frac{\ud}{\ud s} \bfd^{2}(s)  
	 = 2 \bb( \Phi_{s}'(t_{0}, x_{0}) -  \Phi_{s}'(t_{0}, x_{0} + h) \bb) \cdot \bb( v(t_{0} + s,  \Phi_{s}'(t_{0}, x_{0})) - v(t_{0} + s,  \Phi_{s}'(t_{0}, x_{0})) \bb).
\end{align*}

By the mean value theorem, we see that
\begin{equation*}
	- 2 \nrm{\nb v}_{C^{0}} \, \bfd^{2}(s) \leq \frac{\ud}{\ud s} \bfd^{2}(s) \leq 2 \nrm{\nb v}_{C^{0}} \, \bfd^{2}(s).
\end{equation*}


Note that $\bfd^{2}(0) = \abs{h}^{2}$. Thus, applying Gronwall on each side, we obtain
\begin{equation*}
	\abs{h}^{2} e^{- 2 s \nrm{\nb v}_{C^{0}}} \leq \bfd^{2}(s) \leq \abs{h}^{2} e^{2 s \nrm{\nb v}_{C^{0}}}.
\end{equation*}

Taking the square root, we then arrive at the desired set of inequalities. \qedhere
\end{proof}

A simple consequence of Lemma \ref{lem:MLMT:basicEst4FlowFirst} is the equivalence of Eulerian and Lagrangian cylinders.
\begin{lem}[Equivalence of Eulerian and Lagrangian cylinders]\label{lem:cylinderEquiv} Let $v$ be as above, let $(t_{0}, x_{0}) \in \R \times \R^3$ and let $\tau, \rho > 0$.  Then
\ali{
\LCyl_v(\tau, e^{-\tau \| \nab v \|_0} \rho; t_0, x_0) \subseteq \ECyl_v(\tau, \rho; t_0, x_0) \subseteq \LCyl_v(\tau, e^{ \tau \| \nab v \|_0}\rho; t_0, x_0) \label{subset:LagEuCylEquiv}
}
\end{lem}
\begin{proof}  Let $(t,x) \in \LCyl_v(\tau, e^{-\tau \| \nab v \|_0} \rho; t_0, x_0)$.  Then there exist $s \in \R$, $|s| \leq \tau$ and $h \in \R^3$ with $|h| \leq e^{-\tau \| \nab v \|_0} \rho$ such that
\ali{
(t,x) &= \Phi_s(t_0, x_0 + h) \\
&= \Phi_s(t_0, x_0) + \tilde{h} \\
\tilde{h} &= \Phi_s(t_0, x_0 + h) - \Phi_s(t_0, x_0)
}
From Lemma \ref{lem:MLMT:basicEst4FlowFirst} we have $|\tilde{h}| \leq e^{s \| \nab v \|_0} |h| \leq \rho$.  This bound establishes the first containment in \eqref{subset:LagEuCylEquiv}.

For the second containment, let $(t,x) \in \ECyl_v(\tau, \rho; t_0, x_0)$.  Then there exist $s \in \R$, $|s| \leq \tau$ and $h \in \R^3, |h| \leq \rho$ such that
\ali{
(t,x) &= \Phi_s(t_0, x_0) + (0,h) \\
&= \Phi_{s}(t_0, x_0 + \tilde{h} )   \\
(0, \tilde{h}) &= \Phi_{-s}( \Phi_s(t_0, x_0) + (0, h) ) - \Phi_{-s} ( \Phi_s(t_0, x_0) )
}
From Lemma \ref{lem:MLMT:basicEst4FlowFirst} we have $|\tilde{h}| \leq e^{|s| \| \nab v \|_0} |h| \leq e^{\tau \| \nab v \|_0 }\rho$, which concludes the proof.
\end{proof}

From Lemma \ref{lem:MLMT:basicEst4FlowFirst} we can also quickly prove the following containment properties of cylinders
\begin{lem}\label{lem:cylContainmentProps} Let $v$ be as above, let $(t_{0}, x_{0}) \in \bbR \times \bbR^{3}$ and let $\tau_0, \tau_1, \rho_0, \rho_1$ be positive numbers.  Then
\ali{
\ECyl_v(\tau_2, \rho_2; \ECyl_v(\tau_1, \rho_1; t_0, x_0) ) &\subseteq \ECyl_v(\tau_1 + \tau_2, \rho_2 + e^{ \| \nab v \|_0 \tau_2 } \rho_1; t_0, x_0 ) \label{eq:ECylECylinECyl} \\
\ECyl_v(\tau_2, \rho_2; \LCyl_v(\tau_1, \rho_1; t_0, x_0) ) &\subseteq \ECyl_v( \tau_1 + \tau_2, \rho_2 + e^{ \| \nab v \|_0 (\tau_1 + \tau_2) } \rho_1; t_0, x_0) \label{eq:ECylLCylinECyl} \\
\LCyl_v(\tau_2, \rho_2; \ECyl_v(\tau_1, \rho_1; t_0, x_0) ) &\subseteq \ECyl_v(\tau_1 + \tau_2, e^{ \| \nab v \|_0 \tau_2 } (\rho_1 + \rho_2); t_0, x_0 ) \label{eq:LCylECylinECyl}
}
\end{lem}
\begin{proof}
To see the containment \eqref{eq:ECylECylinECyl}, let $(t,x) \in \ECyl_v(\tau_1, \rho_1; \ECyl_v(\tau_0, \rho_0; t_0, x_0) )$.  Then we can write
\ali{
(t,x) &= \Phi_{s_2}( \Phi_{s_1}(t_0, x_0) + (0, h_1) ) + (0, h_2) \label{eq:txinEcylecyl}
}
with $|s_i| \leq \tau_i$ and $|h_i| \leq \rho_i$, $i = 1,2$.  We rewrite \eqref{eq:txinEcylecyl} as 
\ali{
(t,x) &= \Phi_{s_2 + s_1}(t_0, x_0) ) + (0, \tilde{h}_1 + h_2) \notag \\
(0, \tilde{h}_1) &= \Phi_{s_2}( \Phi_{s_1}(t_0, x_0) + (0, h_1) ) - \Phi_{s_2} ( \Phi_{s_1}(t_0, x_0) ) \notag
}
Then $|\tilde{h}_1| \leq e^{\| \nab v \|_0 s_2 } |h_1| \leq e^{ \| \nab v \|_0 \tau_2 } \rho_1$ by Lemma \ref{lem:MLMT:basicEst4FlowFirst}, and the containment \eqref{eq:ECylECylinECyl} follows by the triangle inequality.  The containments \eqref{eq:ECylLCylinECyl} and \eqref{eq:LCylECylinECyl} are proven similarly.
\end{proof}

We will sometimes have the need to compare the cylinders of two related velocity fields.  To prepare for such a comparison, we start with the following preliminary estimate.
\begin{lem}\label{lem:basicDifferentVEstimate}  Suppose that $u(t,x)$ and $v(t,x)$ be vector fields on $\R \times \R^3$ as above.  Denote by ${}^{(v)} \Phi_s(t,x) = ( t + s, {}^{(v)} \Phi_s' )$ and ${}^{(u)} \Phi_s(t,x) = ( t + s, {}^{(u)} \Phi_s' )$ their associated flow maps.  Let $(t_0, x_0) \in \R \times \R^3$.  Then we have a comparison estimate
\ali{
| {}^{(v)} \Phi_s'(t_0, x_0) - {}^{(u)} \Phi_s'(t_0, x_0) | &\leq \co{ v - u } |s| e^{ |s| \| \nab u \|_0 } \label{ineq:basicChangeVEstimate}
}
\end{lem}
\begin{proof}
Define 
\[ \bfd^2(s) = | {}^{(v)} \Phi_s'(t_0, x_0) - {}^{(u)} \Phi_s'(t_0, x_0) |^2 \]
Then we have
\ali{
\fr{d}{ds} (\bfd^2 ) &= 2 \bb( {}^{(v)} \Phi_s'(t_0, x_0) - {}^{(u)} \Phi_s'(t_0, x_0) \bb) \cdot \bb( v( {}^{(v)} \Phi_s(t_0, x_0) ) - u({}^{(u)} \Phi_s(t_0, x_0)) \bb)
}
Writing $v = (v - u) + u $ and applying the mean value theorem, we have
\ali{
| \fr{d}{ds} (\bfd^2 ) | &\leq 2 \co{ v - u } \bfd + 2 \co{ \nab u } \bfd^2 \label{ineq:almostRelVuGronwall}
}
Inequality \eqref{ineq:basicChangeVEstimate} now follows from \eqref{ineq:almostRelVuGronwall} and the fact that $\bfd^2(0) = 0$.
\end{proof}
The estimate \eqref{ineq:basicChangeVEstimate} is most useful when the vector field $u$ is the smoother of the two vector fields.  Note that the inequality \eqref{ineq:basicChangeVEstimate} reduces to the trivial bound $| \Phi_s(t_0, x_0) - (t_0 + s, x_0 + s \bar{v}) | \leq \co{ v - \bar{v} } |s| $ when we take $u = \bar{v}$ to be a constant vector field.  We also remark that Lemma \ref{lem:basicDifferentVEstimate} holds even if $v$ is only continuous, in which case the trajectory ${}^{(v)} \Phi_s(t,x)$ through any given point may fail to be unique.

From Lemma \ref{lem:basicDifferentVEstimate}, we have the following Cylinder Comparison Lemma
\begin{lem}[Cylinder Comparison Lemma] \label{lem:cylCompare}  Let $v$ and $u$ be as in Lemma \ref{lem:basicDifferentVEstimate}, $\tau > 0, \rho > 0$ and $(t_0, x_0) \in \R \times \R^3$.  Then,
\ali{
\ECyl_v(\tau, \rho; t_0, x_0) &\subseteq \ECyl_u(\tau, \rho + \tau \co{ v - u } e^{ \tau \min \{ \| \nab u \|_0, \| \nab v \|_0 \} }; t_0, x_0) \label{incl:cylComparison}
}
\end{lem}
\begin{proof}
For $(t,x) \in \ECyl_v(\tau, \rho; t_0, x_0)$, There exists $s$ and $h$ with $|s| \leq \tau$ and $|h| \leq \rho$ such that
\ALI{
(t,x) &= {}^{(v)} \Phi_s( t_0, x_0) + (0, h) \\
&= {}^{(u)} \Phi_s( t_0, x_0) + ({}^{(v)} \Phi_s( t_0, x_0) - {}^{(u)} \Phi_s( t_0, x_0) ) + (0,h)
}
The containment \eqref{incl:cylComparison} now follows from Lemma \ref{lem:basicDifferentVEstimate} and the triangle inequality.
\end{proof}

The following Lemma will be our basic tool in Section \ref{sec:mainLemImpliesMainThm} for keeping track of the enlargement of support of the approximation solutions during the iteration.
\begin{lem} \label{lem:MLMT:fsp4Flow}
Let $v$, $u$ be $C^{1}$ vector fields on $\bbR \times \bbR^{3}$ such that $v = u$ on $(\bbR \times \bbR^{3}) \setminus Z$, where $Z$ is a closed set. Then for any open set $\Omg \subset \bbR \times \bbR^{3}$ containing $Z$ (i.e., $Z \subseteq \Omg$) and $\tau, \rho > 0$, we have
\begin{equation*}
\LCyl_{v}(\tau, \rho; \, \Omg) = \LCyl_{u}(\tau, \rho; \, \Omg).
\end{equation*}
\end{lem}
\begin{proof} 
By symmetry, it suffices to show that $\LCyl_{u} (\tau, \rho; \, \Omg) \subseteq \LCyl_{v} (\tau, \rho; \, \Omg)$, or equivalently, 
\begin{equation} \label{eq:MLMT:fsp4Flow:pf:1}
	(\bbR \times \bbR^{3}) \setminus \LCyl_{v} (\tau, \rho; \, \Omg) \subset (\bbR \times \bbR^{3}) \setminus \LCyl_{u} (\tau, \rho; \, \Omg).
\end{equation}

Let $(t,x) \in (\bbR \times \bbR^{3}) \setminus \LCyl_{v} (\tau, \rho; \, \Omg)$. By definition, this is equivalent to the statement 
\begin{equation*}
{}^{(v)} \Phi_{s}(t,x) \not \in B(\rho; t_{0}, x_{0}) \hbox{ for any } (t_{0}, x_{0}) \in \Omg \hbox{ and } \abs{s} \leq \tau.
\end{equation*}

In particular, ${}^{(v)} \Phi_{s}(t,x) \in (\bbR \times \bbR^{3}) \setminus \Omg$ for $\abs{s} \leq \tau$. Notice, however, that $v = u$ in the region $(\bbR \times \bbR^{3}) \setminus \Omg$. Therefore, $(t,x) \in (\bbR \times \bbR^{3}) \setminus \LCyl_{u} (\tau, \rho; \, \Omg)$, which proves \eqref{eq:MLMT:fsp4Flow:pf:1}. \qedhere
\end{proof}

%% file: techOutline.tex
In this section, we recall the basic technical outline of the scheme and give a list of the error terms which arise.  

Let $(v,p,R)$ be a velocity field, pressure and stress tensor which satisfy the Euler-Reynolds equations \eqref{eq:euReynolds} with frequency energy levels below $(\Xi, e_v, e_R)$.  To perform the construction, we add corrections to the velocity and the pressure $v_1 = v + V$, $p_1 = p + P$ where the correction to the velocity is a sum of high frequency, divergence-free waves $V = \sum_I V_I$ which have the form
\ali{
 V_I &= e^{i \la \xi_I}( v_I + \de v_I) \label{eq:theCorrectionShape} \\
&= e^{i \la \xi_I} {\tilde v}_I \label{eq:actualAmplitude}
}
The phase function $\xi_I(t,x)$ and amplitude $v_I(t,x)$ are at disposal, but vary slowly in space relative to the large frequency parameter $\la$.  The small corrections $\de v_I$ are present to ensure that \eqref{eq:theCorrectionShape} is divergence-free, and also to make sure that each correction has vanishing linear and angular momentum.  
Each individual wave has a conjugate wave $V_{\bar{I}} = \bar{V}_I$ which oscillates in the opposite direction $\xi_{\bar I} = - \xi_I$ and has amplitude $v_{\bar{I}} = \bar{v}_I$, so that the overall correction is real-valued.

The corrected velocity and pressure now satisfy the system
\ali{
\pr_t v_1^l + \pr_j(v_1^j v_1^l) + \pr^l p_1 =&~ \pr_t V^l + \pr_j(v^j V^l) + \pr_j(V^j v^l) \label{eq:expandingTermsI} \\
&+ \sum_{J \neq {\bar I}} \pr_j(V_I^j V_J^l) + \pr_j \left[ \sum_I V_I^j {\bar V}_I^l + P \de^{jl} + R^{jl} \right] \label{eq:expandingTermsII} \\
\pr_j v_1^j =&~ 0 \notag
}
Our goal is to represent the terms on the right hand side of \eqref{eq:expandingTermsI}-\eqref{eq:expandingTermsII} as the divergence $\pr_j R_1^{jl}$ of a symmetric tensor $R_1^{jl}$ which is small and which satisfies appropriate bounds on its spatial and advective derivatives.  First it is necessary to define appropriate mollifications $v_\ep$ and $R_\ep$ of the given $v$ and $R$ so that the building blocks of the construction will be influenced only by the low frequency part of the given $(v,p,R)$.  These mollifications give rise to the following error term:
\ali{
 Q_M^{jl} &= (v^j - v_\ep^j) V^l + V^j (v^l - v_\ep^l) + (R^{jl} - R_\ep^{jl}).  \label{eq:mollifyTerm}
}
We now gather the remaining terms in \eqref{eq:expandingTermsI}-\eqref{eq:expandingTermsII}.  Expanding the first term in \eqref{eq:expandingTermsI} using the Ansatz \eqref{eq:theCorrectionShape} leads us to impose the transport equation\footnote{It is not clear that one should be forced to choose $0$ for the right hand side of \eqref{eq:transportForPhase}, but it is unclear how one could obtain better regularity through a different choice.  One can interpret equation \eqref{eq:transportForPhase} as an assumption that the high frequency features are carried by the coarse scale flow.  As the paper \cite{isett2} demonstrates, this behavior is forced by the Euler equations in a quantitative sense, so that some approximation to \eqref{eq:transportForPhase} is necessary.}
\ali{
\pr_t \xi_I + v_\ep^j \pr_j \xi_I &= 0 \label{eq:transportForPhase}
}
for the phase functions $\xi_I$.

Assuming \eqref{eq:transportForPhase} and using $\pr_j V_I^j = 0$, the remaining error terms in Equation \eqref{eq:expandingTermsI} then have the form
\ali{
\pr_j Q_T^{jl} =& \pr_t V^l + \pr_j(v_\ep^j V^l) + \pr_j(V^j v_\ep^l) \label{eq:solveTransportTerm} \\
=& \sum_I e^{i \la \xi_I} (\pr_t + v_\ep^j \pr_j) {\tilde v}_I^l \label{eq:transportTermI}  \\
&+ \sum_I e^{i \la \xi_I} {\tilde v}_I^j \pr_j v_\ep^l \label{eq:highLowTerm}
}
The term \eqref{eq:transportTermI} is referred to as the {\bf transport term} since it involves the advective derivative.  In contrast to the work of \cite{isett} in the periodic setting, it is necessary to keep the terms \eqref{eq:transportTermI}-\eqref{eq:highLowTerm} together for working in the whole space.  The reason is that $\pr_j(V^j v_\ep^l)$  by itself may fail to be orthogonal to rotational vector fields, even though it is guaranteed to have integral $0$ and is therefore orthogonal to constants (i.e. translations).  The combination $\pr_j(v_\ep^j V^l) + \pr_j(V^j v_\ep^l)$ on the other hand is already the divergence of a symmetric tensor, and therefore satisfies the necessary orthogonality condition discussed in Section \ref{sec:theMainIssue}.  Therefore, as long as we ensure that the term $\pr_t V$ also satisfies the necessary orthogonality conditions (that is, $V$ conserves both linear and angular momentum), one can hope to solve \eqref{eq:solveTransportTerm}.

We also isolate the {\bf high frequency interference terms} from \eqref{eq:expandingTermsII}, which we can gather in symmetric pairs so that once again the necessary orthogonality conditions are clearly satisfied
\ali{
\sum_{J \neq {\bar I}} \pr_j(V_I^j V_J^l) 
&= \fr{1}{2} \sum_{J \neq {\bar I}} \pr_j( V_J^j V_I^l + V_I^j V_J^l) \label{eq:theHighHighTermsAreSymm} \\
&= \fr{1}{2} \sum_{J \neq {\bar I}}  (V_J^j \pr_j V_I^l + V_I^j \pr_j V_J^l) 
}

To treat these terms, we draw on the idea introduced in \cite{deLSzeCts} of using Beltrami flows.  Following the treatment in \cite{isett}, this approach involves adding additional correction terms to the pressure
\ali{
 P &= P_0 + \sum_{J \neq {\bar I}} P_{I,J} \label{eq:pressureCorrection}
}
where $P_{I,J} = - \fr{1}{2} V_J \cdot V_I$, and imposing the ``microlocal Beltrami flow'' condition
\[ (i \nab \xi_I) \times v_I = |\nab \xi_I| v_I \]
so that the waves $V_I$ in \eqref{eq:theCorrectionShape} serve as curl eigenfunctions to leading order.  

After we apply the identity
\[ V_I \cdot \nab V_J + V_J \cdot \nab V_I = - V_I \times (\nab \times V_J ) - V_J \times ( \nab \times V_I ) + \nab V_I \cdot V_J \]
and add the gradients of the pressure terms $P_{I,J}$, the remainder of the high frequency interference terms \eqref{eq:theHighHighTermsAreSymm} can then be written as a main term which is made small after choosing sharp time cutoffs
\ali{
\pr_j Q_H^{jl} &= \fr{-1}{2}\sum_{J \neq {\bar I}}  \la e^{i \la(\xi_I + \xi_J) } \big[ v_I \times ( |\nab \xi_J| - 1 ) v_J + v_J \times (|\nab \xi_I| - 1) v_I \big] \label{eq:eqnForHighHigh}
}
plus lower order error terms involving the small corections $\de v_I$, which we express using \eqref{eq:actualAmplitude}
\ali{
\label{eq:eqnForSmallerHighHigh}
\begin{aligned}
\pr_j Q_{H'}^{jl} &= \fr{-1}{2} \sum_{J \neq {\bar I}} \la e^{i \la (\xi_I + \xi_J) } \big[ \de v_I \times [ (i \nab \xi_J) \times \tilde{v}_J] + \de v_J \times [ (i \nab \xi_I) \times \tilde{v}_I] \big] \\
&- \fr{1}{2} \sum_{J \neq {\bar I}} \la e^{i \la (\xi_I + \xi_J) } \big[ v_I \times[(i \nab \xi_J) \times  \de v_J] + v_J \times[(i \nab \xi_I) \times  \de v_I] \big]\\
&- \fr{1}{2} \sum_{J \neq {\bar I}} e^{i \la (\xi_I + \xi_J) } \big[ \tilde{v}_I \times ( \nab \times \tilde{v}_J) + \tilde{v}_J \times ( \nab \times \tilde{v}_I) \big]
\end{aligned}
}
We remark that our estimates for the terms \eqref{eq:eqnForHighHigh} and \eqref{eq:eqnForSmallerHighHigh} rely on a nonstationary phase argument, so it is important to check that we have uniform bounds on $\co{~|\nab(\xi_I + \xi_J)|^{-1} ~}$ for all pairs of indices $I, J$, $J \neq \bar{I}$ which interact in the construction.

The final term in \eqref{eq:expandingTermsII} is called the {\bf stress term} and takes the form
\ali{
Q_S^{jl} &= \sum_I ( V_I^j {\bar V}_I^l ) + P_0 \de^{jl} + R_\ep^{jl} \label{eq:stressTermI}
}
where $P_0$ is the low frequency part of the correction to the pressure \eqref{eq:pressureCorrection}.  The term $Q_S^{jl}$ is the only error term (including \eqref{eq:mollifyTerm}) which is of low frequency.  We expand \eqref{eq:stressTermI} using the Ansatz \eqref{eq:theCorrectionShape}, and to ensure that \eqref{eq:stressTermI} is small, we choose the amplitudes $P_0$ and $v_I$ so that the leading order term in \eqref{eq:stressTermI} cancels.  This choice leads to the {\bf stress equation} for the amplitudes:
\ali{
\sum_I v_I^j {\bar v}_I^l &= - P_0 \de^{jl} - R_\ep^{jl} \label{eq:stressEqn}
}
The role of the term $P_0$ in \eqref{eq:stressEqn} is essentially to ensure that the right hand side of \eqref{eq:stressEqn} is positive definite, and also to help prescribe the leading order term in the energy increment of the correction as in the estimate \eqref{eq:energyPrescribed}.  Note that equation \eqref{eq:stressEqn} leads to the estimates $v_I \sim |R|^{1/2}$ and $|P_0| \sim |R|$ for the amplitudes of the corrections indicated in inequalities \eqref{eq:Vco}, \eqref{eq:Pco}.

The remaining stress term is then given by
\ali{
Q_S^{jl} &= \sum_I ( \de v_I^j {\bar v}_I^l + v_I^j \overline{\de v}_I^l + \de v_I^j \overline{\de v}_I^l ) \label{eq:stressTermII}
}
Thus, the new stress takes the form
\ali{
R_1^{jl} &= Q_M^{jl} + Q_S^{jl} + Q_T^{jl} + Q_H^{jl} + Q_{H'}^{jl} \label{eq:theNewStressList}
}
where $Q_M^{jl}$ and $Q_S^{jl}$ are represented by equations \eqref{eq:mollifyTerm} and \eqref{eq:stressTermII} and where $Q_T$, $Q_H$ and $Q_{H'}$ are obtained by solving the elliptic equations \eqref{eq:solveTransportTerm}, \eqref{eq:eqnForHighHigh}, \eqref{eq:eqnForSmallerHighHigh}.  

We now proceed in Section \ref{sec:correctionShape} below to describe the correction in more detail.  In Section \ref{sec:controlStressSupport}, we will complete the outline of the scheme by indicating how the error terms in \eqref{eq:theNewStressList} are organized, and how the support of the error terms remains under control during the iteration.

%% file: theCorrections.tex
Our correction has the form of a sum of individual waves
\ali{
V^\ell &= \sum_I V_I^\ell
}
The individual waves are complex-valued and take the form
\ali{
V_I^\ell &= e^{i \la \xi_I} ( v_I^\ell + \de v_I^\ell) \label{eq:VIisPhaseAmplitude}
}
where the phase function $\xi_I(t,x)$ is allowed to be nonlinear, and the amplitude $v_I$ is complex-valued and required to satisfy $
v_I \in \langle \nab \xi_I \rangle^\perp$ so that the wave \eqref{eq:VIisPhaseAmplitude} is divergence free to leading order.  The nonlinear phase functions $\xi_I$ and amplitudes $v_I$ vary slowly in comparison to the large frequency parameter $\la$.  The correction $\de v_I$ in \eqref{eq:VIisPhaseAmplitude} is a lower order term defined in Equation \eqref{eq:lowerOrderAmplitude} below which is present so that each wave $V_I$ is exactly divergence free.  

In previous approaches, the divergence-free property was ensured by taking the wave $V_I$ to be the curl of a vector field
\ali{
 V_I &= \nab \times W_I \label{eq:curlForm}
}
as in \cite{isett} or by solving a divergence equation to correct the main term as in \cite{deLSzeCts}.  Here, we use waves of the form
\ali{
V_I &= \nab \times \nab \times Y_I \label{eq:doubleCurlForm}
}
where the potential $Y_I$ is given by
\ali{
Y_I = \fr{1}{\la^2} e^{i \la \xi_I} y_I, \qquad y_I = \fr{1}{|\nab \xi_I|^2} v_I \label{YIDef}
}
We impose the double-curl form \eqref{eq:doubleCurlForm} because our waves are required to be divergence free and also to ensure that the corrections have $0$ angular momentum.  Thus, the curl form \eqref{eq:curlForm} is also achieved, and it will be easy to see that the associated $W_I = \nab \times Y_I$ obeys all of the same estimates stated in inequalities \eqref{eq:Wco}-\eqref{eq:matWco} as in \cite{isett}.  With the Ansatz \eqref{eq:doubleCurlForm}, we have
\ali{
v_I &= [ (i \nab \xi_I) \times ]^2 ~y_I \\
\de v_I &= \fr{1}{\la} \nab \times \left[ (i \nab \xi_I) \times y_I + \fr{\nab \times y_I}{\la} \right] \label{eq:lowerOrderAmplitude}
}  
Our amplitudes are required to satisfy the ``microlocal Beltrami flow'' condition
\ali{
(i \nab \xi_I) \times v_I &= |\nab \xi_I| v_I \label{eq:microlocBeltrami}
}
so that \eqref{eq:VIisPhaseAmplitude} behaves to leading order like an eigenfunction of the curl operator with eigenvalue $\la |\nab \xi_I|$.  Condition \eqref{eq:microlocBeltrami} allows us to control interference terms between high frequency waves provided we include sharp time cutoffs which keep the phase gradients very close to $1$ in absolute value $|\nab \xi_I| \approx 1$.

To specify the amplitudes $v_I$ more precisely, we must first specify the index set $\II$ for the indices $I \in \II$.  The index $I \in \II$ has two parts $I = (k, f)$.  The discrete coordinate $k = (k_0, k_1, k_2, k_3) \in \Z \times \Z^3$ indicates the location of the wave $V_I$ in space time.  Namely, a wave with location index $k = k(I)$ will be located in a neighborhood of the point $(t(I), x(I)) = (k_0 \tau, k_1 \rho, k_2 \rho, k_3 \rho)$ and more specifically its support will be contained in a Lagrangian cylinder adapted to $v_\ep$
\ali{
\supp V_I &\subseteq \LCyl_{v_\ep}(\fr{2\tau}{3}, \rho; t(I), x(I)) \label{eq:controlOfVIsupp}
}
The waves $V_I$ will be arranged so that every given point and time $(t,x)$ has at most $2^4$ location indices $k$ for which the wave $V_{(k,f)}(t,x)$ may be nonzero.  The time scale $\tau > 0$ and the space scale $\rho > 0$ are small parameters which will be specified during the construction.

The other part of the index $I = (k,f)$ is the direction coordinate $f$, which specifies the direction of oscillation of the wave $V_{(k,f)}$.  This coordinate $f \in F$ belongs to a finite set $F$ of cardinality $|F| = 12$, which we take as in \cite{isett} to be the set of faces of a regular dodecahedron
 \[ F = \left\{ \pm\fr{(0, 1, \pm \varphi)}{\sqrt{1 + \varphi^2}}, \pm\fr{(1, \pm \varphi, 0)}{\sqrt{1 + \varphi^2}}, \pm\fr{(\pm \varphi, 0, 1)}{\sqrt{1 + \varphi^2}} \right\} \] 
with $\varphi = (1 + \sqrt{5})/2$ being the golden ratio.  Thus, each location index $k$ supports $12$ waves indexed by $(k, f)$ and the number of nonzero waves at a given point $(t,x)$ is bounded by $2^4 \cdot 12$.  The reason for the cardinality $|F| = 12$ is that $6$ independent directions are necessary in order to span the space of symmetric tensors in equation \eqref{eq:stressEqn}, and each direction $f$ must come with a conjugate direction $-f$ corresponding to the conjugate index ${\bar I} = (k, -f)$.

To explain the amplitude $v_I$ more precisely, we decompose $v_I = a_I + i b_I$ into its real and imaginary parts, which both take values in $a_I, b_I \in \langle \nab \xi_I \rangle^\perp$ pointwise.  The condition \eqref{eq:microlocBeltrami} is equivalent to the relationship
\ali{
 a_I &= - \fr{(\nab \xi_I) }{|\nab \xi_I|} \times b_I \label{eq:aIandbI}
}
The imaginary part is then represented in the form
\ali{
b_I^l &= \tilde{e}^{1/2}(t,x) \eta\left(\fr{t - t(I)}{\tau}\right) \psi_k(t,x) \ga_I \mathcal{P}_I^\perp (\nab \xi_{\si I})^l\label{eq:formOfbI}
}
Let us explain the terms appearing in equation \eqref{eq:formOfbI}.  The factor 
\ali{
 \eta\left(\fr{t - t(I)}{\tau}\right) &= \eta \left( \fr{t - k_0 \tau}{\tau} \right) \label{eq:timeCutoffDef}
}
is an element of a rescaled, quadratic partition of unity 
\ali{
 \sum_{y \in \Z} \eta^2(t - y) = 1 
}
that is used to glue local solutions of the homogeneous, quadratic equation \eqref{eq:stressEqn}.  Hence, the wave $V_{(k,f)}$ is supported in the time interval of size $\fr{2 \tau}{3}$ about $t(I) = k_0 \tau$ as desired.

Similarly, the factor $\psi_k(t,x)$ is an element of a partition of unity in space
\ali{
 \sum_{k_1, k_2, k_3} \psi_{(k_0, k_1, k_2, k_3)}^2(t,x) &= 1 \label{eq:spacePartition} 
}
which localizes $V_{(k,f)}$ to the cylinder $\LCyl_{v_\ep}( \fr{2}{3} \tau, \rho; t(I), x(I) )$.  More specifically, $\psi_k$ solves a transport equation
\ali{
(\pr_t + v_\ep^j \pr_j) \psi_k &= 0 \label{eq:spacePartitionTransport} \\
\psi_k(t(I), x) = \psi_k( k_0 \tau, x) &= {\bar \psi}_k(x)
}
whose initial conditions 
\ali{
 {\bar \psi}_k(x) = \eta \left( \fr{x_1 - k_1 \rho}{\rho} \right) \eta \left( \fr{x_2 - k_2 \rho}{\rho} \right) \eta \left( \fr{x_3 - k_3 \rho}{\rho} \right)  \label{eq:initialCondSpaceCutoff}
}
form a rescaled, quadratic partition of unity in space as in \eqref{eq:spacePartition}.  A partition of unity in space as in \eqref{eq:spacePartition},\eqref{eq:spacePartitionTransport} was first introduced in \cite{isett} to localize the construction.  There the purpose of $\psi_k$ is to allow for phase functions $\xi_I$ that exist on a coordinate chart of the torus without having critical points; this approach allows the proof of \cite{isett} to generalize to an arbitrary torus $\R^3 / \Ga$ without relying on the high multiplicity of the spectrum of $\nab \times$ on the particular torus $\R^3 / \Z^3$.  

Here we introduce the new element of including a small length scale $\rho$ on which the waves are localized.  Having such sharp cutoffs in space is natural in view of the goal of obtaining solutions with compact support.  We will also find in Section \ref{sec:solvingSymmDiv} that these cutoffs play a role in ensuring that our new method of solving the symmetric divergence equation obeys the correct bounds which eliminate the need for super-exponential growth of frequencies.  We will see that $\rho$ is chosen to be of size $\sim \Xi^{-1}$, the same length scale on which the building blocks $v_I$ and $\nab \xi_I$ vary.  

The factor 
\ali{
\PP_I^\perp(\nab \xi_{\si I})^l &= \pr^l \xi_{\si I} - \fr{(\nab \xi_{\si I} \cdot \nab \xi_I)}{|\nab \xi_I|^2} \pr^l \xi_I \label{eq:perpProjection}
}
is a vector field of size $\approx 1$ which takes values in the plane $\langle \nab \xi_I \rangle^\perp$.  This vector field was constructed by taking an orthogonal projection of one of the other phase gradients $\nab \xi_{\si I}$, $\si I = (k, \si f)$ which occupies the same location indexed by $k$, but ocillates in a different direction $\si f$.  The vector field \eqref{eq:perpProjection} is essentially the smoothest vector field of size $\approx 1$ taking values in $\langle \nab \xi_I \rangle^\perp$ that one can hope to construct.  Placing $\PP_I^\perp(\nab \xi_{\si I})^l$ in \eqref{eq:formOfbI} ensures that $b_I$ (and hence $a_I$ defined in \eqref{eq:aIandbI}) takes values in $\langle \nab \xi_I \rangle^\perp$.  The index $\si I \neq I$ is chosen to satisfy $\si {\bar I} = \overline{\si I}$, which ensures that $b_{\bar{I}} = - b_I$ and hence $a_{\bar{I}} = a_I$, so that $V_{{\bar I}} = {\bar V}_I$ is indeed a conjugate wave.

The factor $\tilde{e}^{1/2}(t,x)$ is a regularized version of the function $e^{1/2}(t,x)$ described in the statement of the Main Lemma (Lemma \ref{lem:mainLemma}).  Thus, we will show that
\ali{
 \tilde{e}^{1/2}(t,x) &\geq K e_R^{1/2} \label{eq:lowerBoundOnEtilde}
}
on a neighborhood of the support of $R^{jl}$ which will contain the support of $R_\ep^{jl}$.  The function $\tilde{e}^{1/2}(t,x)$ satisfies all the bounds stated in \eqref{ineq:goodEnergy}, and can also be differentiated in space an arbitrary number of times with good bounds.  From \eqref{eq:formOfbI}, the amplitude $v_I$ can be written in the form
\ali{
v_I &= {\tilde e}^{1/2} {\mathring v}_I \label{eq:renormalization}
}
The factor ${\tilde e}^{1/2}$ accounts for the size of the amplitudes $|v_I| \leq C e_R^{1/2}$ with $C$ depending on the constant $M$ in \eqref{ineq:goodEnergy}, while ${\mathring v}_I$ has size of the order $|{\mathring v}_I| \approx 1$.  The renormalization \eqref{eq:renormalization} leads to a renormalization of the stress equation \eqref{eq:stressEqn} for the renormalized amplitudes ${\mathring v}_I$.  We choose $P_0$ in \eqref{eq:stressEqn} to be 
\[ P_0 = - \fr{1}{3} {\tilde e}  + \fr{1}{3} R_\ep^{jl} \de_{jl} = - \fr{1}{3} {\tilde e}  + \fr{1}{3} \tx{tr } R_\ep \]
With this choice, the right hand side of the Stress Equation has a prescribed trace $\tilde{e}(t,x)$
\ali{
\sum_I v_I^j \bar{v}_I^l &= \tilde{e}(t,x) \fr{\de^{jl}}{3} - {\mathring R}_\ep^{jl} \label{eq:prescribedTrace}
}
Here ${\mathring R}_\ep^{jl}$ denotes the trace free part of $R^{jl}$.  The function $\tilde{e}$ turns out to be the main term in the increment to the energy (see Section \ref{sec:energyIncControl} below).  In terms of the renormalized amplitudes ${\mathring v}_I$, Equation \eqref{eq:prescribedTrace} becomes
\ali{
\sum_I {\mathring v}^j_I {\bar {\mathring v}}_I^l &= \fr{\de^{jl}}{3} + \varep^{jl} \label{eq:renormalizedStress} \\
\varep^{jl} &= - \fr{ {\mathring R}_\ep^{jl} }{{\tilde e}} \label{eqref:varepEqn}
}
The tensor $\varep^{jl}$ in \eqref{eqref:varepEqn} is bounded by $\co{ \varep^{jl} } = O(1/K)$ due to the lower bound $e(t,x) \geq K e_R$ assumed for $e(t,x)$ in \eqref{eq:lowerBoundOnEtilde}.  In Section \ref{sec:checkLowerBound} below, we verify that, on the support of $R_\ep$, the regularized function $\tilde{e}$ maintains the same lower bound satisfied by $e$.  As long as $K$ is larger than some absolute constant, this bound ensures that the term $\varep^{jk}$ in \eqref{eqref:varepEqn} is smaller than the term $\fr{\de^{jl}}{3}$ in \eqref{eq:renormalizedStress}, so that the right hand side of \eqref{eq:renormalizedStress} is positive definite and solutions ${\mathring v}_I$ to \eqref{eq:renormalizedStress} exist.  

We can rewrite Equation \eqref{eq:renormalizedStress} as a quadratic equation for the unknown coefficients $\ga_I$ appearing in \eqref{eq:formOfbI}, which all have size on the order of $\approx 1$.  It turns out that the coefficients $\ga_I$ can be written as
\ali{
 \ga_I &= \ga_f(\nab \xi_k, \varep^{jl})  \label{eq:gaIis}
}
for some smooth, real-valued functions $\ga_f$ depending only on the gradients of the phase functions occupying the same location $\nab \xi_I, I \in k \times F$ and the tensor $\varep^{jl}$ appearing in \eqref{eqref:varepEqn}.  In fact, only six different functions $\ga_f$ are used for the formula \eqref{eq:gaIis}, so that one is not worried about seeing an infinite multitude of constants in the construction.

The phase functions themselves are chosen to satisfy the transport equation
\ali{
(\pr_t + v_\ep^j \pr_j) \xi_I &= 0 \\
\xi_I(t(I), x) &= {\hat \xi}_I(x)
}
where the initial data ${\hat \xi}_I$ is a linear function whose gradient has absolute value $|\nab {\hat \xi}_I| = 1$.  The direction of the initial data $\hat{\xi}_{(k,f)}$ is obtained by taking the faces $f \in F$ of the dodecahedron, and applying different rotation matrices $O_{[k]}$ to these faces
\ali{
{\hat \xi}_{(k,f)}(x) &= f \cdot O_{[k]} ( x - x(I) ) \\
x(I) &= ( k_1 \rho, k_2 \rho, k_3 \rho )
}
Here we use a family of $2^4$ rotations $O_{[k]}$ depending on the equivalence class of $[k] \in (\Z/(2 \Z))^4$.  These rotations ensure that no two phase functions occupying adjacent location indices $k$ will oscillate in the same direction.  More precisely, they satisfy the following Proposition taken from \cite[Lemma 7.1]{isett}.
\begin{prop}\label{prop:bunchOfRotations} There exists a collection of $2^4$ rotations $O_{[k]}$ indexed by $[k] \in (\Z/(2 \Z))^4$ and a positive number $c > 0$ with the property that
\ali{
|f \circ O_{[k]} + f' \circ O_{[k']}| \geq c \qquad f, f' \in F \quad [k], [k'] \in (\Z/(2 \Z))^4
}
holds unless $f' = - f$ and $[k] = [k']$.
\end{prop}
This arrangement will allow us to have uniform bounds on $|\nab(\xi_I + \xi_J)|^{-1} \leq A$, so that the phase functions $\xi_I + \xi_J$ appearing in \eqref{eq:eqnForHighHigh} remain uniformly nonstationary (see Proposition \ref{prop:statPhaseProp} below).

We refer to Section 7 of \cite{isett} for a full derivation of the construction.

\subsection{ A preliminary bound on the support of the new stress } \label{sec:controlStressSupport}

Having specified the construction in more detail, we can now briefly indicate how the support of the stress $R_1^{jl}$ calculated in \eqref{eq:theNewStressList} will remain under control during the iteration.  Here we explain the rationale for including sharp cutoffs in space $\psi_k$ in our definition of the amplitudes $v_I$.

The support of the terms $Q_M$ and $Q_S$ will be relatively easy to control, and one can see from equations \eqref{eq:mollifyTerm}, \eqref{eq:stressEqn} and \eqref{eq:stressTermII} that these terms will be supported in a neighborhood of the support of $R$ containing the union of the supports of the waves $V_I$ composing $V$.  We therefore focus on the term $Q_O^{jl} = Q_T^{jl} + Q_H^{jl} + Q_{H'}^{jl}$, which is obtained by solving the elliptic equation
\ali{
\pr_j Q_O^{jl} &= U^l \\
U^l &= \pr_t V^l + \pr_j (v_\ep^j V^l + V^j v_\ep^l) + \sum_{J \neq \bar{I}} \pr_j(V_I^j V_J^l + P_{IJ} \de^{jl}) \label{eq:exressOscForce}
}
We will construct $Q_O$ as a sum of individual parts $Q_O^{jl} = \sum_I Q_{O,I}^{jl}$.  Each individual part $Q_{O, I}^{jl}$ accounts for the wave $V_I^l$ and the interaction terms involving $V_I^l$ by solving the equation
\ali{
\pr_j Q_{O,I}^{jl} &= U_I^l \label{eq:solveIndividError} \\
U_I^l &= \pr_t V_I^l + \pr_j (v_\ep^j V_I^l + V_I^j v_\ep^l)  \\
&+ \fr{1}{4} \sum_{J:J \neq \bar{I}} \pr_j(V_I^j V_J^l + V_J^j V_I^l - V_I \cdot V_J \de^{jl} ) \label{eq:highFreqIndividInteractions}
}
where we recall the choice of $P_{I,J}$ in \eqref{eq:pressureCorrection}.

Note that the force term $U_I^l$ satisfies the orthogonality conditions necessary for \eqref{eq:solveIndividError} to admit a solution, as the individual waves $V_I$ are required to have $0$ linear and angular momentum at all times, and because we keep the interactions of line \eqref{eq:highFreqIndividInteractions} in a symmetric form.  Furthermore, observe that the support of $U_I$ is contained in the support of $V_I$.

Our method of solving the Equation \eqref{eq:solveIndividError} has the property that if the data $U_I^l$ is supported in an Eulerian cylinder $\ECyl_{v_\ep}(\bar{\tau}, \bar{\rho}; t_0, x_0)$ adapted to $v_\ep$ and furthermore $U_I^l$ satisfies the orthogonality conditions necessary for the existence of a solution, then the solution $Q_{O, I}^{jl}$ we construct is also supported in the same Eulerian cylinder $\ECyl_{v_\ep}(\bar{\tau}, \bar{\rho}; t_0, x_0)$.  From Lemma \ref{lem:cylinderEquiv} on the equivalence of Eulerian and Lagrangian cylinders, it follows from \eqref{eq:controlOfVIsupp} that
\ali{
\supp V_I \cup \supp Q_{O, I} &\subseteq \ECyl_{v_\ep}(\fr{2\tau}{3}, e^{\fr{2}{3} \| \nab v_\ep \|_0 \tau} \rho; t(I), x(I)) \label{eq:limitedWaves}
}
The containment \eqref{eq:limitedWaves} will play an important role in controlling the support of the overall stress $R_1$, which is achieved in Section \ref{sec:chooseLengthScaleOfConstruction}.

The containment \eqref{eq:limitedWaves} will also be essential for proving that only a limited number of waves $V_I$ and stress terms can be nonzero at any given point.  We summarize this basic property of the construction as a Proposition, which we prove in Section \ref{sec:boundOnInteractionTerms} after the parameters of the construction have been chosen.
\begin{prop}[Limited Interactions]\label{prop:limitedInteractions}
Let $\#(I)$ denote the number of indices $I'$ such that the support of $V_{I'}$ intersects the support of $V_I$, plus the number of stress terms $Q_{O,I'}$ whose supports intersect the support of $Q_{O, I}$.  Then $\#(I)$ is bounded by an absolute constant. 
\end{prop}

%% file: choosingTheParams.tex
We now assume that we are given a solution $(v,p,R)$ to the Euler-Reynolds equation with frequency-energy levels below $(\Xi, e_v, e_R)$ to order $L$ in $C^0$ in the sense of Definition \ref{def:subSolDef}.  In Section \ref{sec:correctionShape}, we defined a correction of the form
\[ V_I = e^{i \la \xi_I} ( v_I + \de v_I) \]
up to the choice of several parameters in the construction.  These parameters include: 
the frequency parameter $\lambda$; the mollification parameter $\eps_{v}$ for $v_{\eps}$;
the mollification parameters $\eps_{t}, \eps_{x}$ for $R_{\eps}$ and $\tilde{e}^{1/2}$; the
time scale parameter $\tau$; and the length scale parameter $\rho$.

The purpose of this section is to specify our choices of these parameters. Moreover, we show that the support bounds \eqref{eq:goalForR1supp}, \eqref{eq:goalForVPsupp} and Proposition \ref{prop:limitedInteractions} hold under our choices, provided that \eqref{eq:limitedWaves} holds. We remark that \eqref{eq:limitedWaves} ultimately follows from our procedure of finding a compactly support solution to the symmetric divergence equation, which will be presented in Sections \ref{sec:stressEstimates}-\ref{sec:solvingSymmDiv}. 

The large frequency parameter $\la$ has the form
\ali{
\la &= B_\la N \Xi
}
where $N$ is the frequency growth parameter satisfying the conditions of Lemma \ref{lem:mainLemma}, and $B_\la$ is a large constant which will be chosen at the very end of the argument. 

\subsection{Defining the coarse scale velocity field}
To begin the construction, it is necessary to define a suitable regularization $v_\ep$ of the velocity field $v$.  We define 
\ali{
v_\ep &= \eta_{\ep_v} \ast \eta_{\ep_v} \ast v \label{eq:doubleMollifyv}
}
to be a double mollification of $v$ in the spatial variables at a length scale $\ep_v$.  Regularity in time for $v_\ep$ is established from the Euler-Reynolds equations, and having a double mollification is useful for proving the commutator estimate for $(\pr_t + v_\ep \cdot \nab) v_\ep$.  The most important requirement on the length scale $\ep_v$ is that $\ep_v^{-1}$ is smaller than $\la$, which ensures that the effective frequency of $v_\ep$ (or the cost of taking a spatial derivative) is small compared to $\la$.

Associated to the coarse scale velocity field $v_\ep$, we also define the coarse scale advective derivatives
\ali{
\Ddt := (\pr_t + v_\ep \cdot \nab) , \qquad \DDdt := (\pr_t + v_\ep \cdot \nab)^2
}

The regularization in Equation \eqref{eq:doubleMollifyv} gives rise to an error term of the form
\[ (v^j - v_\ep^j) V^l + V^j( v^l - v_\ep^l) \]
described in Equation \eqref{eq:mollifyTerm}.  The parameter $\ep_v$ is chosen in order to achieve a good estimate on the leading order part of this error term, which is given by
\ali{
Q_{M,1}^{jl} &= \sum_I e^{i \la \xi_I} [ (v^j - v_\ep^j) v_I^l + v_I^j( v^l - v_\ep^l) ]  \label{eq:leadingOrderMollifyTerm}
}
Strictly speaking, the amplitudes $v_I$ in \eqref{eq:leadingOrderMollifyTerm} depend on the choice of $v_\ep$.  However, the construction of Section \ref{sec:correctionShape}, in particular Equation \eqref{eq:formOfbI}, guarantees that the amplitudes obey an estimate
\ali{
\co{ \sum_I |v_I| } &\leq A \co{ \tilde{e}^{1/2} } \label{eq:aPrioriForAmplitudes}
}
as long as the lower bound $\tilde{e} \geq K e_R$ is satisfied on the support of $R_\ep$, and provided the phase gradients $\nab \xi_I$ remain within a certain distance of their initial values.  See Section 7 of \cite{isett}.

We construct the function $\tilde{e}^{1/2}(t,x)$ in Section \ref{sec:regularizeStressEnergy} by regularizing the function $e^{1/2}(t,x)$ given in Lemma \ref{lem:mainLemma}, so we expect to prove a bound of the type
\ali{
\co{ \tilde{e}^{1/2} } &\unlhd \co{ e^{1/2} } \label{ineq:weWillRegularizeE}
}
Here we recall the notation that the symbol $\unlhd$ denotes an inequality which has not been proven, but will be established later in the construction.  (In particular, there is no implied constant.)

Assuming \eqref{ineq:weWillRegularizeE}, the bound \eqref{eq:aPrioriForAmplitudes} implies an estimate
\ali{
\co{ \sum_I |v_I| } &\leq A M e_R^{1/2} \label{ineq:secondAmplitudeBound}
}
where $M$ is the constant in Lemma \ref{lem:mainLemma}.  Inequality \eqref{ineq:secondAmplitudeBound} implies that
\ali{
\co{Q_{M,1}^{jl}} &\leq A M e_R^{1/2} \co{ v - v_\ep }
}
We now choose the parameter $\ep_v$ in \eqref{eq:doubleMollifyv} to ensure that $Q_{M,1}$ obeys a bound which is consistent with a scheme aimed at the regularity $1/3$ (see Section 13 of \cite{isett})
\ali{
\co{Q_{M,1}^{jl}} &\unlhd \fr{e_v^{1/2} e_R^{1/2}}{200 N} \label{ineq:ourGoalForFirstMollTerm}
}
Using well-known estimates for mollifications (see Sections 14 and 15 of \cite{isett}), one has that
\ali{
\co{ v - v_\ep } &\leq A \ep_v^L \co{\nab^L v} \label{eq:standardEstimateMoll}
}
provided that the mollifying kernel $\eta_{\ep_v}$ satisfies vanishing moment conditions $\int h^a \eta_{\ep_v}(h) dh = 0$ for all multi-indices $1 \leq |a| < L$.

We achieve the estimate \eqref{ineq:ourGoalForFirstMollTerm} by taking $\ep_v$ of the form
\ali{
\ep_v &= a N^{-1/L} \Xi^{-1} \label{eq:theChoiceOfEpx}
}
where $a$ is a small constant depending on the $A$ and $M$ in inequalities \eqref{eq:aPrioriForAmplitudes}-\eqref{eq:standardEstimateMoll}.  Observe that $\ep_v^{-1} = N^{1/L} \Xi$ is smaller than $\la \approx N \Xi$ since we assume control over at least $L \geq 2$ derivatives in Lemma \ref{lem:mainLemma}.  We also note that the choice of $\ep_v$ here coincides up to a constant with the choice of parameter in Section 15 of \cite{isett}, which will allow us to quote the estimates from \cite{isett}.

\subsection{Defining the regularized stress and energy increment} \label{sec:regularizeStressEnergy}

In addition to defining the coarse scale velocity field $v_\ep$, we also require suitable regularizations of the energy increment $e(t,x)$ and the stress $R^{jl}(t,x)$.  These regularizations $\tilde{e}(t,x)$ and $R_\ep^{jl}(t,x)$ are used to define the amplitudes in Equations \eqref{eq:formOfbI} and \eqref{eq:gaIis} of Section \ref{sec:correctionShape}.

Our definition of $R_\ep$ follows the construction in Section 18 of \cite{isett}.  We first regularize $R$ in space using a double convolution $R_{\ep_x} = \eta_{\ep_x} \ast \eta_{\ep_x} \ast R$, and then regularize in time by averaging along the trajectories of the vector field $(\pr_t + v_\ep \cdot \nab)$ to form
\ali{
R_\ep^{jl}(t,x) &:= \int R_{\ep_x}^{jl}(\Phi_s(t,x)) \eta_{\ep_t}(s) ds \label{eq:defineRep}
}
The map $\Phi_s$ appearing in \eqref{eq:defineRep}, which we call the {\bf coarse scale flow}, is the one-parameter family of diffeomorphisms of $\R \times \R^3$ generated by the space-time vector field $(\pr_t + v_\ep \cdot \nab)$.  Namely, $\Phi_s(t,x) : \R \times \R \times \R^3 \to \R \times \R^3$ is the unique solution to the initial value problem
\ALI{
\fr{d}{ds} \Phi_s^0(t,x) = 1, \quad \fr{d}{ds} \Phi_s^i(t,x) = v_\ep^i(\Phi_s(t,x)), \quad  \Phi_0(t,x) &= (t,x)
}
The motivation for averaging along the coarse scale flow comes from the need to estimate the first advective derivative $\Ddt Q_T$ of the transport term $Q_T$ obtained from solving equation \eqref{eq:transportTermI}.  In particular, estimating $\Ddt Q_T$ requires estimates on the second advective derivatives of the amplitudes $v_I$, and therefore requires estimates on $\DDdt R_\ep$ and $\DDdt \tilde{e}$ by virtue of the construction of Section \ref{sec:correctionShape}.  The key fact which allows for estimates on the second advective derivative is the fact that $\Ddt$ commutes with pullback along the flow, and hence commutes with the averaging in \eqref{eq:defineRep}
\ali{
(\pr_t + v_\ep^a \pr_a) R_\ep^{jl}(t,x) &= \int \Ddtof{R_{\ep_x}^{jl}}(\Phi_s(t,x)) \eta_{\ep_t}(s) ds \notag \\
&= \int \fr{d}{ds} R_{\ep_x}^{jl}(\Phi_s(t,x)) \eta_{\ep_t}(s) ds \label{eq:Ddtisdds}
}
Integrating by parts in \eqref{eq:Ddtisdds} allows one to estimate $\DDdt R_\ep$, whereas estimating spatial derivatives requires preliminary estimates on the coarse scale flow $\Phi_s$.  These estimates are established in Section 18 of \cite{isett}.  There, the double-mollification in space plays a role in the commutator estimates for spatial derivatives $\nab^k \Ddt R_\ep$.  We will also give an alternative proof of the identity \eqref{eq:Ddtisdds} in Section \ref{sec:mainLemImpliesMainThm} below.

The parameters $\ep_x$ and $\ep_t$ have the form
\ali{
\ep_x = c N^{-1/L} \Xi^{-1}, \quad \ep_t = c N^{-1} \Xi^{-1} e_R^{-1/2}  \label{eq:RmollParamChoice}
}
where $c$ is a small constant which is chosen to ensure that the error term generated by the mollification satisfies the bound
\ali{
\co{ R - R_\ep } &\leq \fr{e_v^{1/2} e_R^{1/2}}{100 N} \label{eq:mollTermStressBound}
}
The important point about the parameter $\ep_t$ is that $\ep_t$ is smaller than the natural time scale $\Xi^{-1} e_v^{-1/2}$ within which the flow of $v_\ep$ remains under control.  This upper bound follows from the condition $N \geq \left(\fr{e_v^{1/2}}{e_R^{1/2}}\right)$.

As for the energy increment $e(t,x)$ we define the regularized energy increment $\tilde{e}$ by regularizing the square root of $e$ in essentially the same way.  Namely, we define
\ali{
\tilde{e}^{1/2}(t,x) &= \int (e^{1/2})_{\ep_x}(\Phi_s(t,x)) \eta_{\ep_t}(s) ds \label{eq:honestTildeE}
}
where $(e^{1/2})_{\ep_x} = \eta_{\ep_x} \ast \eta_{\ep_x} \ast e^{1/2}$ is a spatial mollification of $e^{1/2}$.  With this definition, the inequality \eqref{ineq:weWillRegularizeE} follows immediately.

Note that bounds we assume for $e^{1/2}(t,x)$ in \eqref{ineq:goodEnergy} are identical to those assumed for $R^{jl}$ in Definition \ref{def:subSolDef} up to a factor of $M e_R^{-1/2}$.  Therefore, all of the estimates for $\tilde{e}$ follow with the exact same proofs as the estimates for $R_\ep$.  In particular, we can again choose parameters $\ep_x$ and $\ep_t$ of the form \eqref{eq:RmollParamChoice} depending on the constant $M$ in \eqref{ineq:goodEnergy} in such a way that the estimate
\ali{
\co{e^{1/2} - \tilde{e}^{1/2} } &\leq \fr{e_v^{1/2}}{100 N} \label{eq:enIntError}
}
is satisfied.  

To ensure that $\tilde{e}$ is suitable for the construction, we now must check that the lower bound 
\ali{
\tilde{e}(t_0, x_0) &\geq K e_R \label{ineq:lowerBoundWeNeed}
}
is satisfied for $(t_0, x_0)$ on the support of $R_\ep$, where $\tilde{e} = (\tilde{e}^{1/2})^2$.  Inequality \eqref{ineq:lowerBoundWeNeed} is verified in Section \ref{sec:checkLowerBound} below, where additional constraints are imposed on the kernels $\eta_{\ep_x}$ and $\eta_{\ep_t}$.

\subsection{Checking the lower bound on the energy increment}\label{sec:checkLowerBound}

Here we verify that the square root of the regularized energy increment, which takes the form
\ali{
\tilde{e}^{1/2}(t,x) &= \int e^{1/2}(\Phi_s(t,x) + (0,h)) \eta_{\ep_x}(h) \eta_{\ep_t}(s) dh ds, \label{eq:explicitAverageTildee}
}
satisfies the lower bound $\tilde{e}^{1/2}(t,x) \geq K^{1/2} e_R^{1/2} $ for $(t,x)$ in the support of $R_\ep$.  Here we abuse notation by writing $\eta_{\ep_x}(h)$ to abbreviate the expression $\eta_{\ep_x} \ast \eta_{\ep_x}(h)$ coming from \eqref{eq:honestTildeE}.

  What we are given in the Main Lemma is that the function $e^{1/2}(t,x)$ being averaged in \eqref{eq:explicitAverageTildee} already satisfies the lower bound $e^{1/2} \geq K^{1/2} e_R^{1/2}$ on any $v$-adapted Eulerian cylinder $C_v(\Xi^{-1} e_v^{-1/2}, \Xi^{-1}; t_0, x_0)$ centered at a point $(t_0,x_0)$ in the support of $R$.

To ensure that the function $\tilde{e}^{1/2}$ inherits the necessary lower bound from $e^{1/2}$, we impose an additional assumption that both kernels in \eqref{eq:explicitAverageTildee} are non-negative
\ali{
 \eta_{\ep_t}, \eta_{\ep_x} \geq 0  \label{eqn:posConditKernels}
}
This assumption prohibits us from imposing the vanishing moment condition $\int h^a \eta_{\ep_x}(h) dh = 0$ for moments of second order $|a| = 2$, as it will be necessary for $\int |h|^2 \eta_{\ep_x}(h) dh > 0$.  As a consequence, we are forced to take $L = 2$ for our choice of $\ep_x$ in the choice of parameters \eqref{eq:RmollParamChoice} for $\tilde{e}^{1/2}$.  This choice of parameter results in slightly worse bounds for derivatives of $\tilde{e}^{1/2}$ compared to what would be achieved by a larger value of $L$, but these slightly weaker estimates do not affect the proof.  The key properties we maintain are the fact that $\ep_x^{-1} \approx N^{1/2} \Xi$ is smaller than the frequency $\la \approx N \Xi$ by a factor of $N^{1/2}$, and the factors of $\ep_x^{-1}$ do not appear in the estimates until more than two derivatives of $\tilde{e}^{1/2}$ are taken.

Assuming the conditions \eqref{eqn:posConditKernels}, we can now check that the lower bound $\tilde{e}^{1/2}(t_0, x_0) \geq K^{1/2} e_R^{1/2}$ holds for $(t_0, x_0)$ in the support of $R_\ep$ provided the constants in $\ep_t$ and $\ep_x$ are chosen appropriately small.  First we make a simple observation that the support of $R_\ep$ is contained in a Lagrangian cylindrical neighborhood of the support of $R$
\ali{
\supp R_\ep \subseteq \hat{\Ga}_{v_\ep}( \ep_t, \ep_x; \supp R ) \label{eq:supportRep}
}
The containment \eqref{eq:supportRep} follows immediately from the Definition \eqref{eq:defineRep} of $R_\ep$ and the Definition \ref{def:vLagCylinder} of a Lagrangian cylinder.

From the Definition \eqref{eq:explicitAverageTildee} and the condition that $\eta_{\ep_t}$ and $\eta_{\ep_x}$ are non-negative with $\int \eta_{\ep_t}(s) ds = \int \eta_{\ep_x}(h) dh = 1$, we know that the lower bound \eqref{ineq:lowerBoundWeNeed} is satisfied at a point $(t_1, x_1)$ provided that $e^{1/2}(t,x) \geq K^{1/2} e_R^{1/2}$ on the Eulerian cylinder $(t,x) \in \ECyl_{v_\ep}(\ep_t, \ep_x; t_1, x_1)$.   Combining this observation with \eqref{eq:supportRep} and the assumed lower bound \eqref{eq:lowBoundEoftx} on $e^{1/2}$, we obtain the desired lower bound \eqref{ineq:lowerBoundWeNeed} as a corollary of the following Lemma.
\begin{lem}\label{lem:cylindersIn}  If the constant $c$ in \eqref{eq:RmollParamChoice} is chosen sufficiently small, then
\ali{
\hat{C}_{v_\ep}( \ep_t, \ep_x; \hat{\Ga}_{v_\ep}(\ep_t, \ep_x; \supp R ) ) &\subseteq \hat{C}_v(\Xi^{-1} e_v^{-1/2}, \Xi^{-1}; \supp R)
}
\end{lem}
\begin{proof}
According to Lemma \ref{lem:cylContainmentProps}, we have
\ALI{
\hat{C}_{v_\ep}( \ep_t, \ep_x; \hat{\Ga}_{v_\ep}(\ep_t, \ep_x; \supp R ) ) &\subseteq \ECyl_{v_\ep} ( 2 \ep_t, \ep_x ( 1 + e^{ \| \nab v_\ep \|_0 \ep_t}); \supp R ) \\
&\subseteq \ECyl_{v_\ep} ( 2 \ep_t, 3 \ep_x; \supp R )
}
for the appropriate choice of $c$ in \eqref{eq:RmollParamChoice}.  According to the Cylinder Comparison Lemma \ref{lem:cylCompare}, we have
\ALI{
\ECyl_{v_\ep} ( 2 \ep_t, 3 \ep_x; \supp R ) &\subseteq \ECyl_v( 2 \ep_t, 3 \ep_x + 2 \co{ v - v_\ep } \ep_t e^{ \| \nab v_\ep \|_0 \ep_t }; \supp R )
} 
Substituting the choice \eqref{eq:RmollParamChoice} and applying the estimates $\co{ v - v_\ep } \leq e_v^{1/2}$ and $\left( \fr{e_v^{1/2}}{e_R^{1/2} N} \right) \leq 1$, we have
\ALI{
\ECyl_{v_\ep} ( 2 \ep_t, 3 \ep_x; \supp R ) & \subseteq \ECyl_v( 2 c \Xi^{-1} e_v^{-1/2}, 6 c \Xi^{-1}; \supp R)
}
This establishes Lemma \ref{lem:cylindersIn}, and consequently the lower bound \eqref{ineq:lowerBoundWeNeed}, when $c$ is chosen to be a sufficiently small constant.
\end{proof}

\subsection{Choosing the time scale of the construction} \label{sec:timeCutoffs}

Having chosen the parameters for mollifying the velocity, energy increment and stress, we have now completely specified the building blocks in the construction up to the choice of three parameters.  The three parameters which remain are: the time scale $\tau$, which determines the lifespan of the time cutoffs $\eta\left( \fr{t - t(I)}{\tau}\right)$ of Equation \eqref{eq:timeCutoffDef} which enter into the amplitudes, the space scale $\rho$, which determines the size of the support in space for the initial data $\bar{\psi}_k(t(I), x)$ of the spatial cutoffs $\psi_k(t,x)$ in equation \eqref{eq:initialCondSpaceCutoff}, and the constant $B_\la$ in the definition of the frequency parameter $\la = B_\la N \Xi$.  Among these three, the first parameter we specify is the lifespan parameter $\tau$.

The parameter $\tau$ takes the form
\ali{
\tau &= b \Xi^{-1} e_v^{-1/2} \label{eq:tauForm}
}
where $b \leq 1$ is a small, dimensionless parameter which we will now specify.

The choice of the lifespan parameter $\tau$ is restricted by several aspects of the construction.  First of all, $\tau$ cannot be larger than a multiple of $\Xi e_v^{1/2}$ as the elements of the construction which are transported by $v_\ep$ cannot be controlled with good bounds for times larger than $\co{ \nab v_\ep }^{-1}$.  Secondly, it is necessary for the gradients of the phase functions to remain within a certain, finite distance $c_0$ of their initial values in order to ensure the construction is well-defined
\ali{
\co{ \nab \xi_I - \nab \hat{\xi}_I } &\unlhd c_0 \label{eq:cantLetPhaseGradsMove}
}
When the requirement \eqref{eq:cantLetPhaseGradsMove} is satisfied for a sufficiently small constant $c_0$, we may guarantee that the phase functions in the construction remain nonstationary, which is necessary for gaining cancellations while solving the equation $\pr_j Q^{jl} = e^{i \la \xi_I} u_I^l$ with oscillatory data.  Namely, we have the following Proposition
\begin{prop}[Nonstationary Phase]\label{prop:statPhaseProp}  There exists $b_0 >0$ and an absolute constant $A > 0$ such that for $\tau$ of the form \eqref{eq:tauForm} with $b < b_0$ we have
\ali{
\co{ ~|\nab \xi_I|^{-1} } + \co{~|\nab(\xi_I + \xi_J)|^{-1} } &\leq A \label{eq:nonstatPhaseProp}
}
for all indices $I$ and all pairs of indices $I, J$ with $J \neq \bar{I}$ whose supports intersect.
\end{prop} 
By Proposition \ref{prop:bunchOfRotations}, the construction is arranged so that \eqref{eq:nonstatPhaseProp} is satisfied for by the initial data for the phase gradients $\nab \hat{\xi}_I$ and $\nab(\hat{\xi}_I + \hat{\xi}_J)$.
The bound \eqref{eq:nonstatPhaseProp} remains satisfied (with a larger constant) provided \eqref{eq:cantLetPhaseGradsMove} holds.  It is also necessary to impose \eqref{eq:cantLetPhaseGradsMove} with a possibly smaller constant $c_0$ to ensure that equation \eqref{eq:renormalizedStress} admits solutions in $\mathring{v}_I \in \langle \nab \xi_I \rangle^\perp$ with uniform bounds.  See Lemma 7.5 and Proposition 7.2 of \cite{isett}.

Assuming $\tau \leq \Xi^{-1} e_v^{-1/2}$, the estimate we obtain from the transport equation for $\nab \xi_I$ is
\ali{
\co{ \nab \xi_I - \nab \hat{\xi}_I } &\leq A \Xi e_v^{1/2} \tau = A b \label{eq:perturbPhase}
}
Therefore, all the aforementioned requirements on the phase gradients $\nab \xi_I$ of the construction can be guaranteed by choosing $\tau$ of the form \eqref{eq:tauForm}, where $b \leq b_0$ is an appropriately small constant such that the desired bound \eqref{eq:cantLetPhaseGradsMove} holds.

Choosing $b = b_0$ to be a small constant (or something close) would in principle be necessary to obtain the conjectured $1/3$ regularity of solutions for the type of convex integration scheme we consider.  However, the smallness of the parameter $\tau$ plays a crucial role in controlling the High Frequency Interference Terms, and for this reason we are forced to choose $b$ much smaller than a constant, ultimately leading to solutions with lesser regularity $1/5$.  This obstruction to higher regularity was studied in \cite{isett}, where it was observed that the Transport Term of Equation \eqref{eq:transportTermI}, which obeys the bound (see Section 19 of \cite{isett})
\ali{
 \co{ Q_T } &\leq C \la^{-1} \tau^{-1} e_R^{1/2} + \tx{ Lower order terms }, \\
&\leq C b^{-1} \fr{e_v^{1/2} e_R^{1/2}}{B_\la N } + \tx{ Lower order terms } \label{eq:boundForTransportMain}
}
can only be guaranteed to have the size $\fr{e_v^{1/2} e_R^{1/2}}{N}$ desired for the $1/3$ regularity if the $b$ chosen in \eqref{eq:tauForm} is taken to be a constant.  On the other hand, the High Frequency Interference Terms, which obey the bound
\ali{
\co{ Q_H } &\leq C ~ b e_R + \tx{ Lower order terms } \label{eq:boundForHighMain}
}
require $b$ to be significantly smaller than a constant in order for an improvement in the error to be observed.  The estimate \eqref{eq:boundForHighMain} arises from Equation \eqref{eq:eqnForHighHigh}, which shows that $Q_H$ will only be small provided the terms $\co{ |\nab \xi_I| - 1}$ are small.  Optimizing between \eqref{eq:boundForTransportMain} and \eqref{eq:boundForHighMain} leads to the choice
\ali{
 b &= b_0 B_\la^{-1/2} \left( \fr{e_v^{1/2}}{e_R^{1/2} N} \right)^{1/2}  \label{eq:choiceOfb}
}
Now the only parameters which remain to be chosen are the length scale $\rho$ and the large parameter $B_\la$.

\subsection{Choosing the length scale and controlling the support of $R_1$} \label{sec:chooseLengthScaleOfConstruction}

A new feature of our construction is the presence of a small length scale parameter $\rho$ which determines the size of the region on which the spatial cutoffs $\psi_k(t,x)$ are supported.  The purpose of these sharp cutoffs is to control the supports of the corrections $V^{l}$, $P$ and the new stress $R_1^{jl}$ obtained at the end of each stage of the iteration, which are required to stay within a neighborhood of the support of the energy increment $e^{1/2}(t,x)$ according to Lemma \ref{lem:mainLemma}, i.e.,
\ali{
\supp V & \subseteq \ECyl_{v}(\Xi^{-1} e_{v}^{-1/2}, \Xi^{-1}; \supp e), \label{eq:goalForSuppV} \\
\supp P & \subseteq \ECyl_{v}(\Xi^{-1} e_{v}^{-1/2}, \Xi^{-1}; \supp e), \label{eq:goalForSuppP} \\
 \supp R_1 &\subseteq  \ECyl_{v}( \Xi^{-1} e_v^{-1/2}, \Xi^{-1}; \supp e). \label{eq:goalForSuppR1}
}

We will take the parameter $\rho$ to have the form
\ali{
\rho &= c_\rho \Xi^{-1}  \label{eq:rhoDef}
}
where $c_\rho$ is a small constant associated to $\rho$ which we choose here so that the containments \eqref{eq:goalForSuppV}-\eqref{eq:goalForSuppR1} can be guaranteed. Note that these containments are identical to \eqref{eq:goalForR1supp} and \eqref{eq:goalForVPsupp} in the Main Lemma.

\begin{rem}\label{rem:spaceScaleRemark}
Before we proceed to choose $\rho$, it is important to point out that length scales significantly smaller than \eqref{eq:rhoDef} would be forbidden for a construction aimed at proving the conjectured $1/3$ regularity.  Namely, the presence of sharp space cutoffs at scale $\rho$ gives rise to a term of size
\ali{
\co{ Q_S } &\leq \rho^{-1} (N \Xi)^{-1} e_R + \ldots \label{eq:affectOnStress}
}
within the stress term $Q_S$ defined in \eqref{eq:stressTermII}.  That is, $Q_S$ is schematically of size \[ |Q_S| \sim \sum_I \fr{|\nab v_I| \cdot |v_I|}{\la}, \] and terms of size \eqref{eq:affectOnStress} appear when the derivative hits the spatial cutoff\footnote{The same estimate also arises from the term \eqref{eq:eqnForSmallerHighHigh}.}.  Ideally, the bound \eqref{eq:affectOnStress} should be of size $\fr{e_v^{1/2} e_R^{1/2}}{N}$ to obtain $1/3$-H\"{o}lder solutions (see Section 13 of \cite{isett}), and this requirement gives restrictions on the use of length scales smaller than \eqref{eq:rhoDef}. \qedsymbol
\end{rem}

We now proceed to estimate the support of $R_1$ in terms of the parameter $\rho$.  As discussed in Section \ref{sec:controlStressSupport}, the term composing $R_1$ with the largest support is the term $Q_O^{jl} = Q_T^{jl} + Q_H^{jl} + Q_{H'}^{jl}$, which is obtained as a solution to the elliptic equation $\pr_j Q_O^{jl} = U^l$.  According to the containment \eqref{eq:limitedWaves}, the term $Q_O^{jl} = \sum_I Q_{O, I}^{jl}$ is obtained as a sum of localized pieces, with
\ali{
 \supp Q_{O, I} &\subseteq \ECyl_{v_\ep}(\tau, e^{\| \nab v_\ep \|_0 \tau} \rho; t(I), x(I)) \label{eq:limitedWaves2}
}
The term $Q_{O, I}$ is nonzero only when the wave $V_I$ is nonzero, so we now study the conditions under which $V_I$ is nonzero.

By construction, the support of each wave $V_I$ is contained in the support of its spatial cutoff $\psi_k$ and its corresponding time cutoff $\eta_{k_0}(t)$, which together are supported on some Lagrangian cylinder 
\ali{
 \supp V_I &\subseteq \LCyl_{v_\ep}(\fr{2\tau}{3}, \rho; t(I), x(I)) \label{sub:firstBoundSuppVI}
} 
A wave $V_I$ can only be nonzero if the cylinder supporting $V_I$ intersects the support of $\tilde{e}^{1/2}$, implying that the terms $V_I$ and $Q_{O,I}$ are nontrivial only when
\[ (t(I), x(I)) \in \ECyl_{v_\ep}(\tau, \rho; \supp \tilde{e} ) \]
by the duality \eqref{iff:EuLagDuality} between Eulerian and Lagrangian cylinders.
Thus from \eqref{eq:limitedWaves2}, we have 
\ali{
\supp Q_{O,I} &\subseteq \ECyl_{v_\ep}(\tau, A \rho; \ECyl_{v_\ep}(\tau, \rho; \supp \tilde{e} ) )
}
Here the constant $A$ is an absolute constant which changes from line to line, and we have used the fact that $e^{\| \nab v\|_0 \tau} \leq e^1 \leq A$ is bounded. By Lemma \ref{lem:cylContainmentProps}, the right-hand side is bounded by
\ali{
\supp Q_{O,I} &\subseteq \ECyl_{v_\ep}(2 \tau, A \rho; \supp \tilde{e} ) 
}

From the definition of $\tilde{e}$, we have 
\[ \supp \tilde{e} \subseteq \LCyl_{v_\ep}(\ep_t, \ep_x; \supp e), \] 
and it follows from Lemma \ref{lem:cylContainmentProps} that
\ali{
\supp Q_{O,I} &\subseteq \ECyl_{v_\ep}( 2 \tau + \ep_t, A \rho + A \ep_x; \supp e)
}
where $A$ is an absolute constant coming from the bound $\| \nab v\|_0 \ep_t \leq \Xi e_v^{1/2} \ep_t \leq 1$.  From the cylinder comparison Lemma \ref{lem:cylCompare}, we obtain
\ali{
\supp Q_{O,I} &\subseteq \ECyl_{v}( 2 \tau + \ep_t, A \rho + A \ep_x + A \co{ v - v_\ep } (\tau + \ep_t); \supp e) \label{eq:bigger:cyl:supp}
}
Using the estimate $\co{v - v_\ep} \leq e_v^{1/2}$ guaranteed in line \eqref{eq:theChoiceOfEpx}, we can therefore guarantee the bound
\ali{
\supp Q_{O,I} &\subseteq \ECyl_{v}( \Xi^{-1} e_v^{-1/2}, \Xi^{-1}; \supp e)
}
after possibly choosing smaller constants $c$, $c_\rho$ and $b_0$ in the definitions 
\eqref{eq:RmollParamChoice}, \eqref{eq:rhoDef}, \eqref{eq:tauForm} and \eqref{eq:choiceOfb} for the parameters $\ep_x$, $\ep_t$, $\rho$ and $\tau$.  We also see that the sum $Q_O = \sum_I Q_{O,I}$ has the same bound on its support from $\supp Q_O \subseteq \bigcup_I Q_{O, I}$.  Finally, it is clear that the other terms $Q_M$ and $Q_S$ contributing to $R_1$ in \eqref{eq:mollifyTerm} and \eqref{eq:stressEqn} have even smaller support.  Therefore the containment \eqref{eq:goalForSuppR1} has been guaranteed. 

By construction, these choices also guarantee that 
\begin{equation*}
	\supp \widetilde{e} \subseteq \ECyl_{v}(\Xi^{-1} e_{v}^{-1/2}, \Xi^{-1}; \supp e),
\end{equation*}
which implies the desired containments \eqref{eq:goalForSuppV}-\eqref{eq:goalForSuppP} for $V$ and $P$.

\subsection{Bounding the number of interaction terms}\label{sec:boundOnInteractionTerms}

Having chosen the time and length scales of the construction, we can now verify Proposition \ref{prop:limitedInteractions}, which states that each wave $V_I$ and stress term $Q_{O,I}$ shares support with a bounded number of distinct indices.

First, for a given index $I$, let $\#(I)$ denote the number of indices $I'$ such that the support of $V_{I'}$ intersects the support of $V_I$.  Recall from \eqref{sub:firstBoundSuppVI} that each wave is contained in a cylinder $\supp V_I \subseteq \LCyl_{v_\ep}(\fr{2\tau}{3}, \rho; t(I), x(I))$.  Therefore, if $V_J$ is a wave whose support intersects the support of $V_I$, the cylinders corresponding to the two waves intersect, and by \eqref{iff:EuLagDuality} we have
\ali{
(t(J), x(J)) &\subseteq \ECyl_{v_\ep}\left( \fr{2\tau}{3}, \rho; \LCyl_{v_\ep}\left(\fr{2\tau}{3}, \rho; t(I), x(I)\right) \right) \notag
}
By Lemma \ref{lem:cylContainmentProps} and the bound $\| \nab v_\ep \|_0 \tau \leq 1$, we have
\ali{
(t(J), x(J)) &\subseteq \ECyl_{v_\ep}\left(\fr{4\tau}{3}, 10 \rho; t(I), x(I)\right) \label{incl:latticeIsInCylinder}
}
The number of lattice points $(t(J), x(J)) = (k_0 \tau, k_1 \rho, k_2 \rho, k_3 \rho)$ with $k_i \in \Z$ which can belong to a cylinder \eqref{incl:latticeIsInCylinder} is clearly bounded, and so is the number of indices $J = (k_0, k_1, k_2, k_3, f) \in \Z^4 \times F$ which occupy such locations, since at most a finite number $|F|$ indices $J$ share a given location index $k$.  Thus, the number of waves $\#(I)$ which interact with $V_I$ is bounded by an absolute constant.

To finish the proof of Proposition \ref{prop:limitedInteractions}, it suffices to bound the number of stress terms $Q_{O,I}$ occupying a given point.  This number is bounded by following the same line of reasoning, but considering the Eulerian cylinders in \eqref{eq:limitedWaves2} containing the support of $Q_{O,I}$, and applying the corresponding bound \eqref{eq:LCylECylinECyl} in Lemma \ref{lem:cylContainmentProps}.

%% file: correctionEstimates.tex
In this section, we verify the estimates stated in the Main Lemma (Lemma \ref{lem:mainLemma}) concerning the corrections $V$ and $P$.
More precisely, we establish the estimates
\begin{align} 
	\nrm{\nb^{k} v_{1}}_{C^{0}} \unlhd & (\Xi')^{k} (e'_{v})^{1/2} & k=1, \ldots, L \label{eq:CE:goal4v1}\\
	\nrm{\nb^{k} p_{1}}_{C^{0}} \unlhd & (\Xi')^{k} e'_{v} & k=1, \ldots, L \label{eq:CE:goal4p1}
\end{align}
concerning the frequency and energy levels of $v_{1} = v + V$ and $p_{1} = p +P$, with $(\Xi', e'_{v}) = (C_{0} N \Xi, e_{R})$. We also prove the bounds \eqref{eq:Vco}-\eqref{eq:matWco}, \eqref{eq:Pco}-\eqref{eq:matPco} for the corrections $V$ and $P$, respectively, and the local energy increment bounds \eqref{eq:energyPrescribed}, \eqref{eq:DtenergyPrescribed}.
The estimates considered in this Section will also prepare us for estimating the resulting stress $R_{1}$ in the next two Sections.


First we state the bounds satisfied by the elements of the construction obtained from solving a transport equation.  We recall the following estimates were established for the phase gradients $\nab \xi_I$ in the construction of \cite{isett}.  To state the estimates, it will be convenient to use the notation $y_+ := \max \{ y, 0 \}$.
\begin{prop}[Transport Estimates]\label{prop:transportEstimates}  Let $L \geq 2$ be as in Lemma \ref{lem:mainLemma}. There exist constants $C_a$ such that for all $a \geq 1$ and $0 \leq r \leq 2$, the bounds
\ali{
\Xi^{-1} \co{ \nab^a \left( \Ddt \right)^r \nab \psi_k } + \co{ \nab^a \left( \Ddt \right)^r \nab \xi_I } &\leq C_a \Xi^a ( \Xi e_v^{1/2} )^r N^{(a + (r - 1)_+ + 1 - L)_+/L} \label{eq:boundForTransport}
}
are satisfied.  Here $\nab^a$ denotes any spatial derivative of order $a$.  Moreover, if $D^{(a,r)}$ denotes any derivative of the form
\ali{
D^{(a,r)} &= \nab^{a_1} (\pr_t + v_\ep \cdot \nab)^{r_1} \nab^{a_2} (\pr_t + v_\ep \cdot \nab)^{r_1} \nab^{a_3} \label{eq:notationForAdvSpaceDeriv}
}
with $a = a_1 + a_2 + a_3, r = r_1 + r_2$, $a_i, r_i \geq 0$ and $r \leq 2$, we also have the bound
\ali{
\Xi^{-1} \co{ D^{(a,r)} \nab \psi_k } + \co{ D^{(a,r)} \nab \xi_I } &\leq C_a \Xi^a ( \Xi e_v^{1/2} )^r N^{(a + (r - 1)_+ + 1 - L)_+/L} \label{eq:genBoundForTransport}
}
\end{prop}
According to Proposition \ref{prop:transportEstimates}, every spatial derivative costs at most $|\nab| \leq N^{1/L} \Xi$ in the estimate, and each coarse scale advective derivative costs at most $| \Ddt  | \leq \Xi e_v^{1/2}$.  In particular, as $L \geq 2$, the cost of a derivative $|\nab|$ is smaller than the frequency parameter $\la \approx N \Xi$ by a factor of $N^{-(1 - 1/L)} \leq N^{-1/2}$, which means that the terms $\psi_k$ and $\xi_I$ can be regarded as having frequency less than $\la$.  Also, since we have imposed that $L \geq 2$, it is important to note that the factors $N^{1/L}$ do not appear in the estimate until at least two derivatives have been taken.

Proposition \ref{prop:transportEstimates} was established for the phase gradients $\nab \xi_I$ in Section 17 of \cite{isett}.  The estimates for $0 \leq r \leq 1$ are fairly straightforward consequences of the transport equation
\ali{
(\pr_t + v_\ep^j \pr_j) \pr^l \xi_I &= - \pr^l v_\ep^j \pr_j \xi_I \label{eq:phaseGradTrans}
}
satisfied by the phase gradients.  For example, the estimate for spatial derivatives $\co{ \nab^\ga  \nab \xi_I } \leq C_\ga N^{(|\ga| + 1 - L)_+ / L} \Xi$ for times $\tau \leq \Xi^{-1} e_v^{-1/2}$ follows from \eqref{eq:phaseGradTrans} and the estimate 
\[ \co{\nab^\ga v_\ep} \leq C_\a N^{(\ga - L)_+/L} \Xi^\ga e_v^{1/2}, \qquad \ga \geq 1 \] 
by a Gronwall inequality argument.  The same argument, presented in Section 17.2 of \cite{isett}, also applies to establish the estimate \eqref{eq:boundForTransport} for the cutoff gradients $\nab \psi_k$.  

The estimates for second advective derivatives $\left( \Ddt \right)^2$ of $\nab \xi_I$ and $\nab \psi_k$ are more subtle, and require estimates for $(\pr_t + v_\ep \cdot \nab) v_\ep$ and its spatial derivatives.  In this case, the estimates for $v_\ep = \eta_{\ep_v} \ast \eta_{\ep_v} \ast v$ are obtained by commuting the mollifier $\eta_{\ep+\ep} = \eta_{\ep_v} \ast \eta_{\ep_v}$ with the Euler Reynolds equations, to obtain
\ali{
(\pr_t + v_\ep^k \pr_k) v_\ep^j &= \eta_{\ep + \ep} \ast ( \pr^j p + \pr_k R^{kj} ) + Z^j \label{eq:commutedEREqns} \\
Z^j &= v_\ep^k \pr_k v_\ep^j - \eta_{\ep+\ep} \ast(v^k \pr_k v^j) \label{eq:commutatorTerm}
}
The estimates for the right hand side now follow from the Definition \ref{def:subSolDef} of frequency energy levels, and a commutator estimate for the term \eqref{eq:commutatorTerm}.  These estimates are in fact the only part of the argument where the estimates \eqref{bound:nabkp} for the pressure gradient play an important role.  The estimates obtained in this way for the velocity field are
\begin{prop}[Coarse Scale Velocity Estimates]\label{prop:coarseScaleVEstimates}  Let $L \geq 2$ be as in Lemma \ref{lem:mainLemma}.  The vector field $v_\ep$ defined in \eqref{eq:doubleMollifyv} satisfies the bounds
\ali{
 \co{ \nab^a v_\ep} &\leq C_a \Xi^{a} e_v^{1/2}N^{(a-L)_+/L}, \quad a \geq 1 \\
 \co{ \nab^a(\pr_t + v_\ep \cdot \nab) v_\ep } &\leq C_a \Xi^{1 + a} e_v N^{(1 + a - L)_+/L}, \qquad a \geq 0
}
\end{prop}
We refer to Section 16 of \cite{isett} for the proof.  

We also state the bounds satisfied for the terms $\tilde{e}$ and $R_\ep$ which were defined in Section \ref{sec:regularizeStressEnergy} using a mollification along the flow of $v_\ep$.  For compactness, we use the notation of line \eqref{eq:notationForAdvSpaceDeriv} and also
\[ (r \geq b) = \begin{cases} 1 &\tx{if } r \geq b \\ 0 &\tx{if } r < b \end{cases} \]
\begin{prop}[Stress and Energy Increment estimates]\label{prop:eRepEstimates}  Let $L = 2$.  Then for every $a \geq 0$ and $0 \leq r \leq 2$, there is a constant $C_a$ such that
\ali{
e_R^{1/2} \co{ D^{(a,r)} \tilde{e}^{1/2} } + \co{ D^{(a,r)} R_\ep } &\leq C_a \Xi^a e_R (\Xi e_v^{1/2})^{(r \geq 1)} ( N \Xi e_R^{1/2} )^{(r \geq 2)} N^{(a + 1 - L)_+/L}  \label{eq:stressEnergyBound}
}
\end{prop}
Proposition \ref{prop:eRepEstimates} was established in Section 18 of \cite{isett} for the term $R_\ep$.  A large part of the work goes into estimating the coarse scale flow $\Phi_s$ associated to $v_\ep$, and into establishing basic properties of mollification along the flow.  Since the function $e^{1/2}$ that was regularized to form $\tilde{e}^{1/2}$ obeys the same estimates as those assumed for $e_R^{-1/2} R$, the same estimates follow for $\tilde{e}^{1/2}$.  The restriction to $L = 2$ (which was not present in \cite{isett}) arises from the considerations in Section \ref{sec:checkLowerBound}.

From Propositions \ref{prop:transportEstimates} and \ref{prop:eRepEstimates}, we obtain estimates for the basic building blocks of the construction


\begin{prop}[Amplitude estimates]\label{prop:ampEstimates}  For $L = 2$, the amplitudes $v_I$ satisfy the bounds
\ali{
 \co{D^{(a,r)} v_I } &\leq C_a \Xi^a e_R^{1/2} \tau^{-r} N^{(a + 1 - L)_+/L}  \label{eq:amplitudeEstimates} \\
\co{D^{(a,r)} \de v_I } &\leq C_a B_\la^{-1} N^{-1} \Xi^a e_R^{1/2} \tau^{-r} N^{(a + 2 - L)_+/L} \label{eq:tinyCorrectionEstimates}
}
for $a \geq 0$ and $0 \leq r \leq 2$.
\end{prop}
The estimates for $v_I$ follow from Propositions \ref{prop:transportEstimates} and \ref{prop:eRepEstimates} after repeated applications of the chain and product rule using the expressions \eqref{eq:formOfbI} and \eqref{eq:aIandbI} for the real and imaginary parts of $v_I$.  The estimates \eqref{eq:tinyCorrectionEstimates} for the small correction terms $\de v_I$ then follow from the estimates \eqref{eq:amplitudeEstimates} for $v_I$ and the estimates for $\nab \xi_I$ of Proposition \ref{prop:transportEstimates} using the expression \eqref{eq:lowerOrderAmplitude} for $\de v_I$.  The details are carried out in Sections 20 and 21 of \cite{isett}, although there the correction $\de v_I$ has a slightly different form.  The main point is that, schematically, $\de v_I$ has the form 
\[ \de v_I \sim \fr{1}{\la} \nab v_I + \fr{1}{\la^2} \nab^2 v_I\]
up to some factors involving phase gradients.  The first derivative $\nab$ hitting $v_I$ costs a factor of $|\nab| \leq \Xi$ compared to the bound $\co{ v_I } \leq e_R^{1/2}$, whereas the factor $\fr{1}{\la}$ gains a factor of $(B_\la N \Xi)^{-1}$ in the estimate, and the additional term involving $\fr{\nab^2}{\la^2}$ is lower order.  The additional restriction to $L = 2$ in the estimates arises from the considerations in Section \ref{sec:checkLowerBound} as in Proposition \ref{prop:eRepEstimates}.  This restriction does not affect the final conclusion of the Main Lemma.

The bounds \eqref{eq:Vco}-\eqref{eq:matVco} and \eqref{eq:Pco}-\eqref{eq:matPco} stated in Lemma \ref{lem:mainLemma} for the corrections $V$ and $P$ to the velocity and pressure, are straightforward applications of Propositions \ref{prop:transportEstimates}-\ref{prop:ampEstimates}.  Furthermore, the frequency and energy level bounds \eqref{eq:CE:goal4v1} and \eqref{eq:CE:goal4p1} for $v_{1} = v + V$ and $p_{1} = p + P$ (with $C_{0} > 1$ sufficiently large) follow in a similar manner.
A key point in this implication is Proposition \ref{prop:limitedInteractions} which states that only a bounded number of waves can interact at any point.  The relevant arguments are carried out in Section 22 of \cite{isett}.  The estimate \eqref{eq:goalForVPsupp} on the support of $V$ and $P$ has been established during the proof of the containment \eqref{eq:goalForR1supp} in Section \ref{sec:chooseLengthScaleOfConstruction}, as we have that $\supp V \cup \supp P \subseteq \supp \tilde{e}$.  The estimates \eqref{eq:Wco}-\eqref{eq:matWco} for the potential $W = \sum_I \nab \times Y_I$ defined in line \eqref{eq:doubleCurlForm} are also straightforward applications of the same estimates, even though our terms $W_I = \nab \times Y_I$ have a slightly different form than the corresponding terms in \cite{isett}.

Regarding the corrections, the only parts of Lemma \ref{lem:mainLemma} which do not follow from the proof of the Main Lemma of \cite{isett} are the estimates \eqref{eq:energyPrescribed}-\eqref{eq:DtenergyPrescribed} concerning the local energy increments.  We now turn to the proof of these estimate in the following Section \ref{sec:energyIncControl} below.

\subsection{Local estimates on the energy increment} \label{sec:energyIncControl}

Here we verify the estimate \eqref{eq:energyPrescribed} on the energy increment of the solution.

Let $\psi(x)$ be a smooth test function on $\R^3$ with compact support and let $t \in \R$.  We wish to estimate the error in prescribing the energy estimate.  The main point is that, if we expand $V = \sum_I V_I$ into individual waves, the main interactions come from conjugate waves $I, \bar{I}$
\ali{
\int |V|^2(t,x) \psi(x) dx &= \sum_{I,J} \int V_I \cdot V_J \psi(x) dx \\
&= E_{1} + E_2 + E_3 \\
E_1 &= \sum_I \int |v_I|^2(t,x) \psi(x) dx \label{eq:mainEnergyIncrement} \\
E_2 &= \sum_I \int ( v_I^j \overline{\de v}_I^l + \de v_I^j \bar{v}_I^l + \de v_I^j \overline{\de v}_I^l ) \de_{jl} \psi(x) dx \\ 
E_3 &= \sum_{J \neq \bar{I}} \int e^{i \la(\xi_I + \xi_J)} \tilde{v}_I \cdot \tilde{v}_J \psi(x) dx
}
Taking the trace of \eqref{eq:prescribedTrace}, we see that the main term \eqref{eq:mainEnergyIncrement} is equal to
\ali{
E_1 &= \int \tilde{e}(t,x) \psi(x) dx
}
The term $E_1$ gives rise to the main term in \eqref{eq:energyPrescribed}, with an error bounded by
\ali{
| \int e(t,x) \psi(x) dx - \int \tilde{e}(t,x) \psi(x) dx | &\leq \int |e^{1/2}(t,x) - \tilde{e}^{1/2}(t,x)| e^{1/2}(t,x) |\psi(x)| dx \notag \\
&+ \int |e^{1/2}(t,x) - \tilde{e}^{1/2}(t,x)| \tilde{e}^{1/2}(t,x) |\psi(x)| dx \notag \\
&\leq C \fr{e_R^{1/2} e_v^{1/2}}{N} \| \psi \|_{L^1} \label{eq:prescribeApproxEnInc}
}
from \eqref{eq:enIntError}.  The term $E_2$ is bounded by
\ALI{
 |E_2(t)| &\leq C \fr{e_R}{B_\la N} \| \psi \|_{L^1} 
}
from Proposition \ref{prop:ampEstimates}, and finally $E_3$ is estimated by integration by parts
\ali{
E_3 &= \fr{1}{i \la} \sum_{J \neq \bar{I}} \int \fr{\pr^a(\xi_I + \xi_J)}{|\nab(\xi_I + \xi_J)|^2} \pr_a e^{i \la(\xi_I + + \xi_J)} \tilde{v}_I \cdot \tilde{v}_J \psi(x) dx \\
&= \fr{-1}{i \la} \sum_{J \neq \bar{I}} \int e^{i \la(\xi_I + \xi_J)} \pr_a\left[ \fr{\pr^a(\xi_I + \xi_J)}{|\nab(\xi_I + \xi_J)|^2} \tilde{v}_I \cdot \tilde{v}_J \psi(x) \right] dx \\
| E_3 | &\leq C \fr{1}{B_\la N \Xi} ( \Xi e_R \| \psi \|_{L^1} + e_R \| \nab \psi \|_{L^1}  ) \label{eq:intByPartsEnergyEst}
}
Here we use Proposition \ref{prop:limitedInteractions} to bound the number of interacting waves, and also take advantage of the uniform bounds on $\co{~|\nab(\xi_I + \xi_J)|^{-1} }$ for nonconjugate interacting waves $I,J$ in Proposition \ref{prop:statPhaseProp}.

Estimate \eqref{eq:intByPartsEnergyEst} concludes the proof of \eqref{eq:energyPrescribed}.  We omit the proof of \eqref{eq:DtenergyPrescribed}, which is similar, since the estimate \eqref{eq:DtenergyPrescribed} is not needed in any of the applications of Lemma \ref{lem:mainLemma} considered here.

%% file: newStressEstimates.tex
To complete the proof of the Main Lemma (Lemma \ref{lem:mainLemma}), we must calculate the new stress $R_1$ and establish the following estimates
\ali{
\co{ \nab^k R_1 } &\unlhd (\Xi')^k e_R', \quad &k= 0, \ldots, L \label{eq:goalNewStress} \\
\co{ \nab^k (\pr_t + v_1 \cdot \nab) R_1 } &\unlhd (\Xi')^k (\Xi' (e_v')^{1/2}) e_R', \quad &k = 0, \ldots, L - 1 \\
(\Xi', e_v', e_R') &= \left(C_{0} N \Xi, e_R, \fr{e_v^{1/4} e_R^{3/4}}{N^{1/2}} \right) & \label{eq:newFreqEnLevels1}
}
Recall from Section \ref{sec:techOutline} that the new stress is composed of several terms
\ali{
R_1^{jl} &= Q_M^{jl} + Q_S^{jl} + Q_T^{jl} + Q_H^{jl} + Q_{H'}^{jl} \label{eq:theNewStressListRepeated}
}
For the terms $Q_M$ and $Q_S$, we can appeal to \cite{isett} Section 25, where the estimates  \eqref{eq:goalNewStress}-\eqref{eq:newFreqEnLevels1} are verified for essentially identical terms.  The only difference in our case is the presence of sharper cutoffs $\psi_k$ and a regularized energy increment $\tilde{e}^{1/2}$ which do not affect the estimates.  We are therefore left with the terms $Q_T$, $Q_H$ and $Q_{H'}$ calculated in \eqref{eq:solveTransportTerm}, \eqref{eq:eqnForHighHigh}, \eqref{eq:eqnForSmallerHighHigh}.  

As outlined in Section \ref{sec:controlStressSupport}, these terms are calculated by solving the symmetric divergence equation with high frequency data
\ali{
\pr_j Q_{O,I}^{jl} &= U_I^l \label{eq:solveIndividError2} \\
U_I^l &= U_{T,I}^l + \sum_{J: J \neq \bar{I}} U_{H,IJ}^l  \label{eq:individErrorTermsAgain} \\
U_{T,I}^l &= \pr_t V_I^l + \pr_j (v_\ep^j V_I^l + V_I^j v_\ep^l) \label{eq:individTransportTerm} \\
U_{H,IJ}^l &= \fr{1}{4} \pr_j(V_I^j V_J^l + V_J^j V_I^l - 2 V_I \cdot V_J \de^{jl} ) 
}
The terms $U_{T,I}^l$, $U_{H,IJ}^l$ consist of the individual terms in the summations \eqref{eq:solveTransportTerm}, \eqref{eq:eqnForHighHigh}-\eqref{eq:eqnForSmallerHighHigh}.  

A key point in solving the equation \eqref{eq:solveIndividError2} is that we expect to gain a factor $\la^{-1}$ in the estimate $\co{ Q_{O,I} } \leq \la^{-1} \co{U_I}$ up to lower order terms, because the data on the right hand side has high frequency $\la$.  For example, the transport term \eqref{eq:individTransportTerm} has the form
\ali{ 
U_{T,I}^l &= e^{i \la \xi_I} u_I^l \\
u_I^l &= (\pr_t + v_\ep^j \pr_j) {\tilde v}_I^l + {\tilde v}_I^j \pr_j v_\ep^l \label{eq:transportTermIind}  
}
Furthermore, we desire a solution $Q_{O, I}$ to \eqref{eq:solveIndividError2} which also has compact support around the support of $U_I$.  Concerning the support of the waves, note that the terms in \eqref{eq:individErrorTermsAgain} have the common feature that they are supported in the cylinder
\[ \supp U_{T, I} \cup \supp U_{H,IJ} \subseteq \supp V_I \subseteq \ECyl_{v_\ep}(\frac{2}{3}\tau, A \rho; t(I), x(I))\]
with $A = e^{\frac{2}{3} \nrm{\nb v_{\eps}}_{0} \tau}$ as discussed in Section \ref{sec:chooseLengthScaleOfConstruction}.  Our solution $Q_{O, I}$ will have support in the same cylinder, from which \eqref{eq:limitedWaves} follows.

Before we can find a compactly supported solution $Q_{O,I}$ to \eqref{eq:solveIndividError2}, it is necessary to check that the terms $U_I^l$ satisfy the orthogonality conditions necessary to solve \eqref{eq:solveIndividError2}.  For the terms $U_{H,IJ}^l$ and the term $\pr_j (v_\ep^j V_I^l + V_I^j v_\ep^l)$ in $U_{T,I}^l$, the orthogonality conditions are obvious as both terms have already been represented as the divergence of a symmetric tensor with compact support.  For the term $\pr_t V_I^l$, the orthogonality conditions follow from our technique of taking $V_I$ of double curl form.  Namely, if $K^l$ is any solution to the equation $\pr_j K_l + \pr_l K_j = 0$ on $\R^n$, then $K^l$ is a linear combination of translational and rotational vector fields, and in particular its second derivative $\nab^2 K$ vanishes.  It follows that
\ali{
\int \pr_t V_I \cdot K dx &= \fr{d}{dt} \int \nab \times \nab \times Y_I \cdot K dx \\
&= \fr{d}{dt} \int Y_I \cdot \nab \times \nab \times K dx = 0
}
Thus, there is no immediate obstruction to obtaining a compactly supported solution $Q_{O,I}$ to \eqref{eq:solveIndividError2}.

\begin{rem}\label{rem:relationToAngMoment}  There is a second approach to designing building blocks $V_I$ satisfying the orthogonality conditions $\int \pr_t V_I \cdot K dx = 0$ which is more closely related to the discussion in Section \ref{sec:theMainIssue}.  Namely, one can also consider building blocks of the form $V_I = \nab \times W_I + \de V_I$ where the ``corrector term'' $\de V_I$ is a small correction which accounts for the possibility that the term $\nab \times W_I$ may fail to conserve angular momentum.  A simple way to design suitable corrector terms $\de V_I$ is to use approximate rotation fields 
\[ \de V_I = \sum_{1 \leq i < j \leq 3} c_{I, (ij)}(t) \phi\Big(\fr{|x - \bar{x}_I(t)|}{\rho}\Big) \Om_{(ij)}(x - \bar{x}_I(t)) \]
such as those introduced in line \eqref{eq:approxAngularFields} of the proof of Proposition \ref{prop:conserveLaw}, where $\bar{x}_I(t)$ is a point which moves along the coarse scale flow.  This approach could be considered more natural from a geometric point of view and leads to an interesting alternative proof of the Main Lemma, but it requires further estimates to implement so we have taken the simpler route of taking $V_I$ to have double curl form.
\end{rem}

It now remains to construct a solution to equation \eqref{eq:solveIndividError2} and to establish the oscillatory estimate $\co{ Q_{O,I} } \leq \la^{-1} \co{U_I}$ up to lower order terms.   These tasks are taken up in Sections \ref{sec:applyParametrix} and \ref{sec:solvingSymmDiv} below.

%% file: applyParametrix.tex
Here we consider the general problem of finding compactly supported solutions to the symmetric divergence equation
\ali{
\pr_j Q^{jl} &= e^{i \la \xi} u^l  \label{eq:symmDivWithOsc}
}
where the right hand side is supported on a cylinder $\ECyl_{v_\ep}(\frac{2}{3} \tau, A \rho; t(I), x(I))$ and satisfies the necessary orthogonality conditions for a solution to exist.  In our applications, the phase function $\xi$ is either one of the phase functions $\xi_I$ or the sum $\xi_I + \xi_J$ of two interacting, nonconjugate phase functions.  In every case, the amplitude $u^l$ turns out to satisfy the estimates
\ali{
\co{ \nab^k u^l } + \tau \co{ \nab^k \Ddt u^l } &\leq C_k B_\la^{-1/2} \la (N^{1/2} \Xi )^k e_R', \qquad k \geq 0 \label{eq:theAmplitudeSize} 
}
where $e_R' = \fr{e_v^{1/4} e_R^{3/4}}{N^{1/2}}$ is the target size of the new stress $R_1$ expressed in \eqref{eq:newFreqEnLevels1}.  The amplitudes $u^l$ are also supported in a cylinder
\[ \supp u^l \subseteq \ECyl_{v_\ep}(\fr{2 \tau}{3}, A \rho; t(I), x(I) ) \]
of size $\rho \sim c_\rho \Xi^{-1}$, where $c_\rho$ is the constant chosen in Section \ref{sec:chooseLengthScaleOfConstruction}.  See Section 26 of \cite{isett} for details, particularly Section 26.2.  Here the factors of $N^{1/2}$ come from the factors of $N^{1/2}$ in the estimates of Section \ref{sec:correctionEstimates}.

In solving the first order, elliptic equation \eqref{eq:symmDivWithOsc}, we expect the solution $Q$ to gain a factor $\la^{-1}$ in the estimate $\co{Q} \leq \la^{-1} \co{ u }$ modulo lower order terms.  In \cite{deLSzeCts}, \cite{deLSzeHoldCts}, De Lellis and Sz\'{e}kelyhidi gave an approach to obtaining this cancellation based on the method of nonstationary phase.  The approach we take here follows the approach in \cite{isett}, which is a slight adaptation of the method in \cite{deLSzeCts, deLSzeHoldCts} generalized to nonlinear phase functions.  The main distinction is that the approach we take does not involve proving that the operators $R^{jl}[U]$ we construct for solving the equation \eqref{eq:symmDivWithOsc} exhibit cancellation when the input $U$ has the form $U^l = e^{i \la \xi} u^l$.  Instead, we obtain the necessary cancellation through a parametrix expansion of the solution.  We also avoid the use of Schauder estimates, which would impose a super-exponential growth of frequencies in the iteration by requiring $C^\a$ rather than $C^0$ control of the data.

To begin, we write down a first order approximate solution to \eqref{eq:symmDivWithOsc} of the form
\ali{
 Q_{(1)}^{jl} &= \fr{1}{\la} e^{i \la \xi} q_{(1)}^{jl} \label{eq:firstOrderParam}
}
where the amplitude $q_{(1)}^{jl}$ is a symmetric tensor solving the underdetermined linear equation
\ali{
i \pr_j \xi q_{(1)}^{jl} &= u^l \label{eq:undetLinEqn}
}
pointwise. Following \cite{isett}, we begin constructing a solution to \eqref{eq:undetLinEqn} by first decomposing $u^l$ into
\[ u^l = u_\perp^l + \fr{(u \cdot \nab \xi)}{|\nab \xi|^2} \pr^l \xi = u_\perp^l + u_\parallel^l, \]
so that $ u_\perp \in \langle \nab \xi \rangle^\perp$ and $u_\parallel \in \langle \nab \xi \rangle$ pointwise.  We then define
\ali{
 q_{(1)}^{jl} &= -i ( q_\perp^{jl} + q_\parallel^{jl} ) = q^{jl}(\nab \xi)[u], \label{eq:defnQjlMap}
}
where the tensors
\[ 
 q_\perp^{jl} = \fr{1}{|\nab \xi|^2}(\pr^j \xi u_\perp^l + \pr^l \xi u_\perp^j), \quad q_\parallel^{jl} = \fr{(u \cdot \nab \xi)}{|\nab \xi|^2} \de^{jl} \]
solve $\pr_j \xi q_\perp^{jl} = u_\perp^l$ and $\pr_j \xi q_\parallel^{jl} = u_\parallel^l$ pointwise.

The important properties of the map defined in\footnote{Another example of a satisfactory map $q^{jl}$ can be read off from the symbol of the operator in Definition 4.2 of  \cite{deLSzeCts}.  Our construction of \eqref{eq:defnQjlMap} can likewise be regarded as giving the symbol of an order $-1$ operator which solves the symmetric divergence equation.} \eqref{eq:defnQjlMap} are that $q^{jl}(\nab \xi)[u]$ is linear in $u$, homogeneous of degree $-1$ in $\nab \xi$, and smooth outside of $\nab \xi = 0$.  Thus, the main term $Q_{(0)}$ in \eqref{eq:firstOrderParam} obeys the desired estimate $\co{ Q_{(1)} } \leq C \la^{-1} \co{u}$, using the uniform bounds on $\co{ ~|\nab \xi|^{-1}~ }$ which are satisfied by all the phase functions involved in the construction.  We can then construct an exact solution to \eqref{eq:symmDivWithOsc} of the form $Q^{jl} = Q_{(1)} + \widetilde{Q}_{(1)}^{jl}$ by letting $\widetilde{Q}_{(1)}^{jl}$ solve the equation $\pr_j \widetilde{Q}_{(1)}^{jl} = - \fr{1}{\la} e^{i \la \xi} \pr_j q_{(1)}^{jl}$, noting that the right hand side now has a smaller amplitude than before thanks to the factor of $1/\la$.  

To improve on the first order expansion \eqref{eq:firstOrderParam}, we build the solution to \eqref{eq:symmDivWithOsc} as an approximate solution plus an error
\ali{
Q^{jl} &= Q_{(D)}^{jl} + \widetilde{Q}_{(D)}^{jl}  \label{eq:paramPlusRemainder} \\
Q_{(D)}^{jl} &= \sum_{k = 1}^{D} \fr{1}{\la} e^{i \la \xi} q_{(k)}^{jl} \label{eq:parametrixTerm}
}
The amplitude $q_{(k)}^{jl}$ of the $k$'th term is obtained by solving the linear equation
\ali{
i \pr_j \xi q_{(k)}^{jl} = u_{(k)}^l, \quad u_{(1)}^l = u^l, \qquad u_{(k+1)}^l = \fr{-1}{\la} \pr_j q_{(k)}^{jl}
}
using the function $q_{(k)}^{jl} = q^{jl}(\nab \xi)[u_{(k)}]$ defined in \eqref{eq:defnQjlMap}.  For $Q^{jl}$ to be a solution of \eqref{eq:symmDivWithOsc}, the remainder term in \eqref{eq:paramPlusRemainder} must be chosen to solve the equation
\ali{
\pr_j \widetilde{Q}_{(D)}^{jl} &= e^{i \la \xi} u_{(D+1)}^l \label{eq:getRidOfError}
}

Thanks to the estimate \eqref{eq:theAmplitudeSize}, the bounds on the amplitude $u_{(k)}$ become smaller with each iteration of the parametrix by a factor of
\ali{
\fr{|\nab|}{\la} &\leq C B_\la^{-1} \fr{N^{1/2} \Xi}{N \Xi} \leq C B_\la^{-1} N^{-1/2}
}
After taking $D$ terms in the expansion, the bounds for $u_{(D)}^l$ have the form
\ali{
\co{ \nab^k u_{(D)}^l } + \tau \co{ \nab^k \Ddt u_{(D)}^l } &\leq C_k B_\la^{-D} N^{-D/2} B_\la^{-1/2} \la (N^{1/2} \Xi )^k e_R'
}
In particular, since $\la = B_\la N \Xi$, for $D \geq 2$ we have 
\ali{
\co{ u_{(D)}^l } + \tau \co{ \Ddt u_{(D)} } &\leq C_k B_\la^{-1}  N^{-D/2 + 1} \Xi e_R' \label{eq:afterDIterations}
}
Our goal is to make sure the solution $\widetilde{Q}_{(D)}^{jl}$ to \eqref{eq:getRidOfError} has $C^0$ norm bounded by a multiple of $e_R'$.  In previous constructions of H\"{o}lder continuous solutions on the torus, it has been necessary to assume a super-exponential growth of frequencies (i.e. $N \geq \Xi^\eta$ for some $\eta > 0$), so that the estimate \eqref{eq:afterDIterations} gains a power of $N^{-D/2 + 1} \leq N^{-1} \Xi^{-1} \approx \la^{-1}$ once $D$ is chosen sufficiently large.  
In our case, however, we will gain a smallness factor of $\rho \sim \Xi^{-1}$ from our new method of solving Equation \eqref{eq:getRidOfError}, thus eliminating the apparent need for super-exponential growth of frequencies.  

We take $D=3$, which leaves us with the following estimate for the amplitude in \eqref{eq:getRidOfError}
\ali{
\co{ u_{(D+1)}^l } + \tau \co{ \Ddt u_{(D+1)} } &\leq C B_\la^{-1} \Xi e_R' \label{eq:coBounduD}
}
This choice of $D$ leads also to the estimates
\ali{
\co{ \nab^k u_{(D+1)}^l } + \tau \co{ \nab^k \Ddt u_{(D+1)}^l } &\leq C_k B_\la^{-1} (N^{1/2} \Xi )^k \Xi e_R'
}
The data $U_{(D+1)}^l = e^{i \la \xi} u_{(D+1)}^l$ on the right hand side of \eqref{eq:getRidOfError} now obeys the estimates
\ali{
\co{ \nab^k U_{(D+1)}^l } + \tau \co{ \nab^k \Ddt U_{(D+1)}^l } &\leq C_k B_\la^{-1} (B_\la N \Xi)^k \Xi e_R'
}
According to Theorem \ref{thm:inv4divEq}, there is a solution $\widetilde{Q}_{(D)}^{jl}$ to the equation \eqref{eq:getRidOfError} with support in the same Eulerian cylinder
\[ \supp \widetilde{Q}_{(D)}^{jl} \subseteq \ECyl_{v_\ep}(\fr{2 \tau}{3}, A \rho; t(I), x(I) ) \]
such that $\widetilde{Q}_{(D)}^{jl}$ obeys the estimates
\ali{
\co{ \nab^k \widetilde{Q}_{(D)}^{jl} } + \tau \co{ \nab^k \Ddt \widetilde{Q}_{(D)}^{jl} } &\leq C_k B_\la^{-1} (B_\la N \Xi)^k e_R' \label{eq:boundForStressSoln}
}
We emphasize in particular that the estimate for the solution of Theorem \ref{thm:inv4divEq} gains a factor of $A \rho \sim \Xi^{-1}$, which is consistent with dimensional analysis of the equation.

If $B_\la$ is sufficiently large, then we can guarantee that each term $\widetilde{Q}^{jl}_{(D)}$ has size bounded by $\co{ \widetilde{Q}^{jl}_{(D)} } \leq \fr{1}{B} e_R'$ where $B$ can be any large constant.  In particular, by Proposition \ref{prop:limitedInteractions} on limited interactions, we can guarantee that the sum of all stress terms $\widetilde{Q}^{jl}_{(D)}$ obtained by this procedure is bounded uniformly by $\fr{1}{500} e_R'$, and that the bound \eqref{eq:boundForStressSoln} is also satisfied for the sum of these terms.

On the other hand, the parametrix term \eqref{eq:parametrixTerm} also satisfies the same estimate \eqref{eq:boundForStressSoln}, with the main term contribution coming from the first term $Q_{(1)}^{jl}$, and the number of such $Q_{(D)}^{jl}$ which are nonzero at any given point is likewise bounded.  Choosing $B_\la$ sufficiently large, we can therefore guarantee that the entire contribution $Q_O^{jl} = \sum_I Q_{O,I}^{jl}$ to the stress $R_1$ obeys the estimates
\ali{
\co{ Q_O } &\leq \fr{1}{40} e_R'  \label{eq:boundForQO}
}
At last, we choose $B_\la$ so that the bound \eqref{eq:boundForQO} is satisfied, which implies the desired bound for
\ali{
 \co{ R_1 } &\leq \co{ Q_M } + \co{Q_S} + \co{ Q_O } \leq e_R'
}
With the construction fully determined, it now remains to check that the spatial and advective derivatives of $R_1$ obey the bounds demanded by the Main Lemma (Lemma \ref{lem:mainLemma}).

With the above choice of $B_\la$, we obtain
\ali{
\co{ \nab^k Q_{O}^{jl} } + \tau \co{ \nab^k \Ddt Q_{O}^{jl} } &\leq C_k (N \Xi)^k e_R' \label{eq:boundForStressSolnO}
}
The bound \eqref{eq:boundForStressSolnO} is clearly enough to conclude that the new frequency-energy levels are satisfied for the spatial derivatives of $Q_O$, as the cost of a spatial derivative is at most $|\nab| \leq C_{0} N \Xi$.  Also, the cost of taking an advective derivative is bounded by
\ali{
 \left|\Ddt\right| &\leq \tau^{-1} = C \left( \fr{e_v^{1/2}}{e_R^{1/2} N} \right)^{-1/2} \Xi e_v^{1/2}
}
which is no larger than the required estimate
\ali{
 \left|\Ddt\right| &\unlhd \Xi' (e_v')^{1/2} = C_{0} N \Xi e_R^{1/2} \label{eq:advecDerivBoundWeNeed}
}
thanks to the condition $N \geq \left( \fr{e_v^{1/2}}{e_R^{1/2}} \right)$.  From \eqref{eq:advecDerivBoundWeNeed}, it is straightforward to conclude the necessary bounds on 
\[ (\pr_t + v_1 \cdot \nab) Q_{O} = \Ddt Q_O + (v - v_\ep) \cdot \nab Q_O + V \cdot \nab Q_O \]
using the spatial derivative bounds \eqref{eq:boundForStressSolnO}.  Namely, the derivative $(\pr_t + v_1 \cdot \nab)$ costs at most
\ali{
| (\pr_t + v_1 \cdot \nab) | &\leq \left|\Ddt\right| + | (v - v_\ep) \cdot \nab| + | V \cdot \nab| \\
&\leq \left|\Ddt\right| + C \fr{e_v^{1/2}}{N} (N \Xi)  +C e_R^{1/2} (N \Xi ) \\
&\leq C N \Xi e_R^{1/2}
}
\noindent as desired. One can then take spatial derivatives up to order $L - 1$ for each term at a cost of at most $|\nab| \leq C_0 N \Xi$ per derivative as desired, which is carried out in detail in Sections 24 - 26 of \cite{isett}.  Combining the above estimates with the bounds for the terms $Q_M$ and $Q_S$ already estimated in \cite{isett}, we conclude our proof of the Main Lemma (Lemma \ref{lem:mainLemma}). \qedhere

%% file: solveSymmDiv4.tex
We now present our method of solving the underdetermined elliptic equation from which we recover the new stress in the construction.  
The analysis in this Section is independent of the earlier part of this paper, and in particular holds on $\bbR \times \bbR^{d}$ for any $d \geq 2$.

For a symmetric tensor $R^{j \ell} = R^{\ell j}$ and vector field $U^{\ell}$ on $\bbR^{d}$, consider the divergence equation
\begin{equation} \label{eq:symmDivEq}
	\rd_{j} R^{j \ell} = U^{\ell}.
\end{equation}
In what follows, \eqref{eq:symmDivEq} will be referred to as the \emph{symmetric divergence equation}. 

\subsection{Main result for the symmetric divergence equation} 

The following is our main result regarding compactly supported solutions to the symmetric divergence equation \eqref{eq:symmDivEq}.

\begin{thm}[Compactly supported solutions to the symmetric divergence equation] \label{thm:inv4divEq}
Let $A, N, \Xi, e_{v}$ be positive numbers, $L \geq 1$ be a positive integer and $v_{\eps} = (v_{\eps}^{1}, \ldots, v_{\eps}^{d})$ be a vector field on $\bbR \times \bbR^{d}$ such that 
\begin{equation} \label{eq:inv4divEq:est4nbv}
	\nrm{\nb^{\bt} v_{\eps}}_{C^{0}_{t,x}} \leq A \Xi^{\abs{\bt}} e_{v}^{1/2}, \qquad  1 \leq \abs{\bt} \leq L.
\end{equation}

Furthermore, let $U^{\ell}$ be a vector field with zero linear and angular momenta, i.e.,
\begin{align} 
	\int U^{\ell}(t,x) \, \ud x =&0, \label{eq:inv4divEq:hyp:1} \\
	\int (x^{j} U^{\ell} - x^{\ell} U^{j})(t,x) \, \ud x =& 0 \label{eq:inv4divEq:hyp:2}
\end{align}
for all $t$, and such that
\begin{equation} \label{eq:inv4divEq:supp4U}
	\supp U \subseteq \ECyl_{v_{\eps}}(\bar{\tau}, \bar{\rho}; t(I), x(I)),
\end{equation}
for some $(t(I), x(I)) \in \bbR \times \bbR^{d}$ and $0 < \bar{\tau} \leq \Xi^{-1} e_{v}^{-1/2}$. 
Assume also that for
\begin{equation*}
	\Lmb > 0, \qquad
	0 < \hat{\tau} \leq \Xi^{-1} e_{v}^{-1/2},
\end{equation*}
the vector field $U$ obeys the estimates
\begin{equation} \label{eq:inv4divEq:hyp:3} 
\begin{aligned}
	\nrm{\nb^{\bt} U}_{C_{t,x}^{0}} \leq& A \Lmb^{\abs{\bt}}	\qquad \abs{\bt} = 0, \ldots, L, \\
	\nrm{\nb^{\bt} (\rd_{t} + v_{\eps} \cdot \nb) U}_{C_{t,x}^{0}} \leq& A \hat{\tau}^{-1} \Lmb^{\abs{\bt}}	\qquad \abs{\bt} = 0, \ldots, L-1.
\end{aligned}
\end{equation}

Then there exists a solution $R^{j \ell}[U]$ to the symmetric divergence equation \eqref{eq:symmDivEq}, depending linearly on $U^{\ell}$, with the following properties:
\begin{enumerate}
\item The support of $R^{j \ell}[U]$ stays in the cylinder $\ECyl_{v_{\eps}}(\bar{\tau}, \bar{\rho}; t(I), x(I))$, i.e.,
\begin{equation} \label{eq:inv4divEq:supp4Rjl}
	\supp R^{j \ell}[U] \subseteq \ECyl_{v_{\eps}}(\bar{\tau}, \bar{\rho}; t(I), x(I)).
\end{equation}

\item There exists $C > 0$ such that for $\abs{\bt} = 0, \ldots, L$,
\begin{equation} \label{eq:inv4divEq:1}
	\nrm{\nb^{\bt} R^{j \ell}[U]}_{C^{0}_{t,x}} \leq C A \bar{\rho} \sum_{m=0}^{\abs{\bt}}  \bar{\rho}^{-(\abs{\bt}-m)} \Lmb^{m}
\end{equation}

\item There exists $C > 0$ such that for $\abs{\bt} = 0, \ldots, L-1$,
\begin{equation} \label{eq:inv4divEq:2}
	\nrm{\nb^{\bt} (\rd_{t} + v_{\eps} \cdot \nb) R^{j \ell}[U]}_{C^{0}_{t,x}} \leq C A \hat{\tau}^{-1} \bar{\rho} \sum_{m_{0}+m_{1}+m_{2} = \abs{\bt}} \Xi^{m_{0}} \bar{\rho}^{-m_{1}} \Lmb^{m_{2}}
\end{equation}
where the sum is over all triplets of non-negative integers $(m_{0}, m_{1}, m_{2})$ such that $m_{0} + m_{1} + m_{2} = \abs{\bt}$.
\end{enumerate}
\end{thm}

\begin{rem} 
This theorem should be compared with Theorem 27.1 in \cite{isett}, which was proved by solving a transport equation\footnote{We remark that the method used in \cite{isett} seems to be very special to the torus.  A key ingredient in this approach is that the transport by a divergence free vector field preserves the integral zero property, which is the only necessary condition to solve the symmetric divergence equation on $\T^3$.  On the other hand, the orthogonality conditions from angular momentum conservation seem to prevent such an approach from applying to the whole space.} using a Helmholtz-type solution operator.  The key difference is, of course, that the present theorem preserves the support property \eqref{eq:inv4divEq:supp4U} whereas Theorem 27.1 in \cite{isett} does not. We would like to point out two more differences:
\begin{enumerate}
\item Theorem 27.1 in \cite{isett} gives estimates in $L^{p}_{x}$ with $1 < p < \infty$ (more specifically, $p = 4$), whereas the present theorem operates directly in $C^{0}_{t,x}$.
\item Whereas Theorem 27.1 in \cite{isett} gains a derivative (by Calder\`on-Zygmund theory), we do not prove gain of derivative in Theorem \ref{thm:inv4divEq}. On the other hand, we \emph{do} gain a factor of $\bar{\rho}$ (i.e., the spatial scale of $U^{\ell}$), which is crucial for removing the super-exponential growth assumption $N \geq \Xi^{\eta}$ in the Main Lemma.
\end{enumerate}

\end{rem}

\subsection{Derivation of the solution operator}
The purpose of this subsection is to give a derivation of the solution operator $R^{j\ell}[U]$ for \eqref{eq:symmDivEq} in Theorem \ref{thm:inv4divEq}. 

For the moment, we shall omit the time variable and work entirely on $\bbR^{d}$. Let $U^{\ell}= U^{\ell}(x)$ be a vector field supported on some ball $B(\bar{\rho}; x_{0})$. For simplicity, we will furthermore assume that $U$ is smooth and $x_{0} = 0$. 

Our first idea is to use the Fourier transform and Taylor expand $\widehat{U}^{\ell}(\xi)$ about the frequency origin $\xi = 0$ in the Fourier space.  We will then try to write the terms of the Taylor expansion as the divergence of a symmetric tensor, up to some terms evaluated at $\xi = 0$. Translating the resulting formula to the physical space, we shall obtain a solution operator which possess the desired (physical space) support property, albeit with a mild singularity at $0$.

Indeed, we first compute
\begin{align*}
\widehat{U}^{\ell}(\xi) 
= & 	\widehat{U}^{\ell}(0) + \bb( \int_{0}^{1} \rd^{k} \widehat{U}^{\ell} (\sgm \xi) \, \ud \sgm \bb) \xi_{k} \\
= & 	\widehat{U}^{\ell}(0) + \frac{1}{2}\bb( \int_{0}^{1} (\rd^{k} \widehat{U}^{\ell} + \rd^{\ell} \widehat{U}^{k}) (\sgm \xi) \, \ud \sgm \bb) \xi_{k} 
	+ \frac{1}{2}\bb( \int_{0}^{1} (\rd^{k} \widehat{U}^{\ell} - \rd^{\ell} \widehat{U}^{k}) (\sgm \xi) \, \ud \sgm \bb) \xi_{k} \\
= & 	\widehat{U}^{\ell}(0) + \frac{1}{2} (\rd^{k} \widehat{U}^{\ell} - \rd^{\ell} \widehat{U}^{k})(0) \xi_{k} 
	+ \frac{1}{2}\bb( \int_{0}^{1} (\rd^{k} \widehat{U}^{\ell} + \rd^{\ell} \widehat{U}^{k}) (\sgm \xi) \, \ud \sgm \bb) \xi_{k} \\
	&+ \frac{1}{2} \bb( \int_{0}^{1} (1-\sgm) (\rd^{j} \rd^{k} \widehat{U}^{\ell} - \rd^{j} \rd^{\ell} \widehat{U}^{k})(\sgm \xi) \, \ud \sgm \bb) \xi_{j} \xi_{k}.
\end{align*}

Note that the third term on the right-hand side is (formally) the Fourier transform of a divergence of a symmetric tensor. The last term may also be written as a sum of two terms of such type by the following calculation, which is inspired by the closely related physical space identity \eqref{eq:usingSymmDerivFormula}:
\begin{align*}
\frac{1}{2} \bb( \int_{0}^{1} (1-\sgm) (\rd^{j} \rd^{k} \widehat{U}^{\ell} - \rd^{j} \rd^{\ell} \widehat{U}^{k})(\sgm \xi) \, \ud \sgm \bb) \xi_{j} \xi_{k}
= & 	\frac{1}{2} \bb( \int_{0}^{1} (1-\sgm) \xi_{k} (\rd^{j} \rd^{k} \widehat{U}^{\ell} + \rd^{\ell} \rd^{k} \widehat{U}^{j})(\sgm \xi) \, \ud \sgm \bb) \xi_{j} \\
& - \bb( \int_{0}^{1} (1-\sgm) \xi_{k} (\rd^{\ell} \rd^{j} \widehat{U}^{k})(\sgm \xi) \, \ud \sgm \bb) \xi_{j} \\
\end{align*}

Therefore, assuming
\begin{equation} \label{eq:zeroMomentum}
\begin{gathered}
	\widehat{U}^{\ell} (0) = \int U^{\ell}(x) \, \ud x= 0, \\
	 \bb( \frac{1}{i} \rd^{k} \widehat{U}^{\ell} - \frac{1}{ i} \rd^{\ell} \widehat{U}^{k} \bb) (0) = \int (x^{\ell} U^{k} - x^{k} U^{\ell})(x) \, \ud x = 0,
\end{gathered}
\end{equation}
which is equivalent to the assumptions \eqref{eq:inv4divEq:hyp:1}, \eqref{eq:inv4divEq:hyp:2}  on $U$, the following formula for $\widehat{U}^{\ell}(\xi)$ holds:
\begin{align*}
	\widehat{U}^{\ell} (\xi)
	=&\frac{1}{2}\bb( \int_{0}^{1} (\rd^{j} \widehat{U}^{\ell} + \rd^{\ell} \widehat{U}^{j}) (\sgm \xi) \, \ud \sgm \bb) \xi_{j} \\
	&	+ \frac{1}{2} \bb( \int_{0}^{1} (1-\sgm) \xi_{k} (\rd^{j} \rd^{k} \widehat{U}^{\ell} + \rd^{\ell} \rd^{k} \widehat{U}^{j})(\sgm \xi) \, \ud \sgm \bb) \xi_{j} \\
	&	 - \bb( \int_{0}^{1} (1-\sgm) \xi_{k} (\rd^{\ell} \rd^{j} \widehat{U}^{k})(\sgm \xi) \, \ud \sgm \bb) \xi_{j} \\
\end{align*}

Let us define $r^{j \ell}[U] := r^{j\ell}_{0}[U] + r^{j \ell}_{1}[U] + r^{j \ell}_{2}[U]$, where
\begin{align}
	r^{j\ell}_{0}[U]
	= & \calF^{-1} [\frac{1}{2}\bb( \int_{0}^{1} \bb( \frac{1}{i} \rd^{j} \widehat{U}^{\ell} + \frac{1}{i} \rd^{\ell}\widehat{U}^{j} \bb) (\sgm \xi) \, \ud \sgm \bb)] \\
	r^{j \ell}_{1}[U] 
	=& \calF^{-1} [ \frac{1}{2} \bb( \int_{0}^{1} (1-\sgm) (i \xi_{k}) (- \rd^{j} \rd^{k} \widehat{U}^{\ell} - \rd^{\ell} \rd^{k} \widehat{U}^{j} )(\sgm \xi) \, \ud \sgm \bb)], \\
	r^{j \ell}_{2}[U]
	= & \calF^{-1}[ - \bb( \int_{0}^{1} (1-\sgm) (i \xi_{k}) ( - \rd^{j} \rd^{\ell} \widehat{U}^{k} )(\sgm \xi) \, \ud \sgm \bb) ]. 
\end{align}

Computing the inverse Fourier transform, we arrive at the formal formulae
\begin{align}
	r^{j \ell}_{0}[U]
	=& - \frac{1}{2} \int_{0}^{1} \bb( \frac{x^{j}}{\sgm} U^{\ell}(\frac{x}{\sgm}) + \frac{x^{\ell}}{\sgm} U^{j}(\frac{x}{\sgm}) \bb) \, \frac{\ud \sgm}{\sgm^{d}}, \label{eq:rjl0:formal}\\
	r^{j \ell}_{1}[U]
	=& \frac{1}{2}  \frac{\rd}{\rd x^{k}} \int_{0}^{1} (1-\sgm)  \bb( \frac{x^{j} x^{k}}{\sgm^{2}} U^{\ell}(\frac{x}{\sgm}) + \frac{x^{\ell} x^{k}}{\sgm^{2}} U^{j}(\frac{x}{\sgm}) \bb) \, \frac{\ud \sgm}{\sgm^{d}}, \label{eq:rjl1:formal}\\
	r^{j \ell}_{2}[U]
	= & -  \frac{\rd}{\rd x^{k}} \int_{0}^{1} (1-\sgm) \bb(\frac{x^{\ell} x^{j}}{\sgm^{2}} U^{k}(\frac{x}{\sgm}) \bb) \, \frac{\ud \sgm}{\sgm^{d}} .\label{eq:rjl2:formal}
\end{align}
Thus, for $a = 0,1,2$, the values of $r^{jl}_a[U]$ at a point $x \in \R^3$ are given formally as weighted integrals of $U$ and $\nab U$ along the ray emanating from $x$ away from the origin.

In fact, when interpreted correctly, these expressions already give us a \emph{distributional} solution to \eqref{eq:symmDivEq} with the desired support property
\begin{equation} \label{eq:supp4rjl}
	\supp r^{j \ell} \subseteq B(\bar{\rho}; 0),
\end{equation}
but with a singularity at $x=0$. Indeed, given a test function $\varphi \in C^{\infty}_{c}$, we will define
\begin{align}
	\brk{r^{j \ell}_{0}[U], \varphi} 
	:=& - \lim_{\dlt \to 0+} \frac{1}{2} \int_{\dlt}^{1} \int \bb( \frac{x^{j}}{\sgm} U^{\ell}(\frac{x}{\sgm}) + \frac{x^{\ell}}{\sgm} U^{j}(\frac{x}{\sgm}) \bb) \varphi(x) \, \ud x \frac{\ud \sgm}{\sgm^{d}}, \label{eq:rjl0} \\
	\brk{r^{j \ell}_{1}[U], \varphi}
	:=& - \frac{1}{2} \lim_{\dlt \to 0+} \int_{\dlt}^{1} (1-\sgm)  \int \bb( \frac{x^{j} x^{k}}{\sgm^{2}} U^{\ell}(\frac{x}{\sgm}) + \frac{x^{\ell} x^{k}}{\sgm^{2}} U^{j}(\frac{x}{\sgm}) \bb) \rd_{k} \varphi(x) \, \ud x \frac{\ud \sgm}{\sgm^{d}}, \label{eq:rjl1}\\
	\brk{r^{j \ell}_{2}[U], \varphi}
	:= & \lim_{\dlt \to 0+} \int_{\dlt}^{1} (1-\sgm) \int \bb(\frac{x^{\ell} x^{j}}{\sgm^{2}} U^{k}(\frac{x}{\sgm}) \bb) \rd_{k} \varphi(x) \, \ud x \frac{\ud \sgm}{\sgm^{d}} .\label{eq:rjl2}
\end{align}

These are well-defined (tempered) distributions on $\bbR^{d}$. Indeed, by a simple change of variables, we see that
\begin{equation*}
	\abs{\brk{r^{j \ell}_{0}[U], \varphi}} \leq C_{U, \bar{\rho}} \nrm{\varphi}_{C^{0}_{x}}, \quad
	\abs{\brk{r^{j \ell}_{1}[U], \varphi}} + \abs{\brk{r^{j \ell}_{2}[U], \varphi}} \leq C_{U, \bar{\rho}} \nrm{\nb \varphi}_{C^{0}_{x}}.
\end{equation*}

The support property \eqref{eq:supp4rjl} and smoothness outside $\set{x=0}$ follow immediately from the definition. Moreover, a straightforward computation with distributions shows that 
\begin{equation}
	\rd_{j} r^{j \ell}[U] = U^{\ell} - \bb( \int U^{\ell}(x) \, \ud x \bb) \dlt_{0} - \frac{1}{2} \bb( \int (x^{\ell} U^{j} - x^{j} U^{\ell})(x) \, \ud x \bb) \rd_{j} \dlt_{0}.
\end{equation}

Thus, under the assumption \eqref{eq:zeroMomentum}, we see that $r^{j\ell}[U]$ is a distributional solution to \eqref{eq:symmDivEq}.

Unfortunately, $r^{j\ell}[U]$ as defined above apparently has a singularity at $x=0$. We will overcome this difficulty by exploiting translation invariance of \eqref{eq:symmDivEq}; more precisely, we will conjugate $r^{j \ell}[U]$ by translations and take a smooth average of the resulting formulae, ultimately `smearing out' the singularity.

Given $y \in \bbR^{d}$, let us conjugate the operators $r^{j\ell}_{0}$, $r^{j\ell}_{1}$ and $r^{j\ell}_{2}$ by translation by $y$. Then we are led to the conjugated operator ${}^{(y)} r^{j \ell} = {}^{(y)} r^{j \ell}_{0} + {}^{(y)} r^{j \ell}_{1} + {}^{(y)} r^{j \ell}_{2}$, which is formally defined by
\begin{align}
	{}^{(y)} r^{j \ell}_{0}[U]
	=& - \frac{1}{2} \int_{0}^{1}  \frac{(x-y)^{j}}{\sgm} U^{\ell}(\frac{x-y}{\sgm}+y) + \frac{(x-y)^{\ell}}{\sgm} U^{j}(\frac{x-y}{\sgm}+y)  \, \frac{\ud \sgm}{\sgm^{d}}, \label{eq:yrjl0}\\
	{}^{(y)} r^{j \ell}_{1}[U]
	=& \frac{1}{2} \frac{\rd}{\rd x^{k}} \int_{0}^{1} (1-\sgm)   \frac{(x-y)^{j} (x-y)^{k}}{\sgm^{2}} U^{\ell}(\frac{x-y}{\sgm}+y)  \, \frac{\ud \sgm}{\sgm^{d}} \label{eq:yrjl1}\\
	& + \frac{1}{2}  \frac{\rd}{\rd x^{k}} \int_{0}^{1} (1-\sgm)   \frac{(x-y)^{\ell} (x-y)^{k}}{\sgm^{2}} U^{j}(\frac{x-y}{\sgm}+y)  \, \frac{\ud \sgm}{\sgm^{d}}, \notag \\
	{}^{(y) }r^{j \ell}_{2}[U]
	= & - \frac{\rd}{\rd x^{k}} \int_{0}^{1} (1-\sgm)  \frac{(x-y)^{\ell} (x-y)^{j}}{\sgm^{2}} U^{k}(\frac{x-y}{\sgm}+y)  \, \frac{\ud \sgm}{\sgm^{d}}. \label{eq:yrjl2}
\end{align}

These are to be interpreted as in \eqref{eq:rjl0}--\eqref{eq:rjl2} as distributions. Note that as long as $y \in B(\bar{\rho}; 0)$, the distribution ${}^{(y)} r^{j\ell}[U]$ satisfies the desired support property \eqref{eq:supp4rjl}. Motivated by this consideration, let us take a smooth function $\zeta(y)$ which is supported in $B(\bar{\rho}; 0)$ and satisfies 
\ali{
\int \zeta(y) dy &= 1 
}
We now define the solution operator ${}^{(\zeta)} \widetilde{R}^{j \ell}[U]$ by averaging ${}^{(y)} r^{j\ell}[U]$ against $\zeta$, i.e.,
\begin{equation} 
	{}^{(\zeta)} \widetilde{R}^{j \ell}[U](x) = \int  {}^{(y)} r^{j \ell}(x) \zeta(y) \, \ud y.
\end{equation}
We will finally obtain the solution operator $R^{j \ell}[U]$ of Theorem \ref{thm:inv4divEq} by making an appropriate choice of $\zeta$ depending on time.



From the discussion above, we see that ${}^{(\zeta)} \widetilde{R}^{j \ell}[U]$ inherits the desirable properties of $r^{j \ell}[U]$. Indeed, assuming \eqref{eq:zeroMomentum}, ${}^{(\zeta)} \widetilde{R}^{j \ell}[U]$ is a (distributional) solution to \eqref{eq:symmDivEq} satisfying the support property
\begin{equation} \label{eq:supp4RjlTilde}
	\supp {}^{(\zeta)} \widetilde{R}^{j \ell}[U] \subseteq B(\bar{\rho}; 0).
\end{equation}

As we shall see below, thanks to averaging with respect to $\zeta$, ${}^{(\zeta)} \widetilde{R}^{j \ell}[U]$ will moreover turn out to be smooth in the spatial variables provided $U$ is smooth as well (see in particular the calculations \eqref{eq:theZTrickAppears} and \eqref{eq:R2actuallyr2} below).

\subsection{Formula for $R^{j \ell}[U]$ and basic properties}
Let $U^{\ell}$ be a vector field satisfying the hypotheses \eqref{eq:inv4divEq:hyp:1}, \eqref{eq:inv4divEq:hyp:2} and \eqref{eq:inv4divEq:supp4U} of Theorem \ref{thm:inv4divEq}. Denote by $\overline{v}_{\eps}(t)$ the value of the coarse scale velocity $v_{\eps}$ at $\Phi_{t-t(I)}(t(I), x(I))$. For $t \in [t(I)-\bar{\tau}, t(I)+\bar{\tau}]$, this point is exactly the center of the cross-section $\ECyl_{v_{\eps}}(\bar{\tau}, \bar{\rho}; t(I), x(I)) \cap \set{t} \times \bbR^{d}$.  

Recall from the previous subsection that we need to choose a (spatially) smooth function $\zeta$ with integral $1$ in order to determine our solution operator for \eqref{eq:symmDivEq}. We shall define a function $\zeta = \zeta(t,x)$ adapted to $\ECyl_{v_{\eps}}(\bar{\tau}, \bar{\rho}; t(I), x(I))$ according to the following procedure: given a smooth function $\widetilde{\zeta} = \widetilde{\zeta}(x)$ with $\supp \widetilde{\zeta} \subseteq B(\bar{\rho}; x(I))$ and $\int \widetilde{\zeta}(x) dx = 1$, let $\zeta$ be the solution to the transport equation
\begin{equation} 
\left\{
\begin{aligned}
	(\rd_{t} + \overline{v}^{j}_{\eps}(t) \rd_{j}) \zeta(t,x) =& 0 \qquad \hbox{ for } t \in [t(I)-\bar{\tau}, t(I)+\bar{\tau}], \\
	\zeta(t(I), x) =& \widetilde{\zeta}(x).
\end{aligned}
\right.
\end{equation}

Note that $\zeta$ satisfies the support property
\begin{equation} \label{eq:supp4zeta}
	\supp \zeta \subseteq \ECyl_{v_{\eps}}(\bar{\tau}, \bar{\rho}; t(I), x(I)).
\end{equation}
and also satisfies $\int_{\R^d} \zeta(t,y) dy = 1$ at all times $t$.

Moreover, choosing $\widetilde{\zeta}$ to be a bump function adapted to $B(\bar{\rho}; x(I))$, the following estimates hold for $\zeta$:
\begin{equation} \label{eq:bnd4zeta}
	\nrm{\nb^{\bt} \zeta}_{C^{0}_{t,x}} \leq C_{\bt} \bar{\rho}^{-d-\abs{\bt}} \qquad \hbox{ for all } \abs{\bt} \geq 0.
\end{equation}

We are now ready to define the solution operator $R^{j \ell}[U]$. Let $R^{j \ell}[U] := R^{j \ell}_{0}[U] + R^{j \ell}_{1}[U] + R^{j \ell}_{2}[U]$, where
\begin{align} 
R^{j \ell}_{0}[U] 
=& 	-\frac{d}{2} \int_{0}^{1} \int \zeta(t, y) \frac{(x-y)^{j}}{\sgm} U^{\ell}(t, \frac{x-y}{\sgm} + y) \, \frac{\ud y}{\sgm^{d}} \ud \sgm \label{eq:Rjl0}\\
&	-\frac{d}{2} \int_{0}^{1} \int \zeta(t, y) \frac{(x-y)^{\ell}}{\sgm} U^{j}(t, \frac{x-y}{\sgm} + y) \, \frac{\ud y}{\sgm^{d}} \ud \sgm, \notag \\
R^{j \ell}_{1}[U]
=&	\frac{1}{2} \int_{0}^{1} \int (\rd_{k} \zeta)(t, y) \frac{(x-y)^{j} (x-y)^{k}}{\sgm^{2}} U^{\ell}(\frac{x-y}{\sgm} + y) \, \frac{\ud y}{\sgm^{d}} \ud \sgm \label{eq:Rjl1} \\
& + \frac{1}{2} \int_{0}^{1} \int (\rd_{k} \zeta)(t, y) \frac{(x-y)^{\ell} (x-y)^{k}}{\sgm^{2}} U^{j}(\frac{x-y}{\sgm} + y) \, \frac{\ud y}{\sgm^{d}} \ud \sgm, \notag \\
R^{j \ell}_{2}[U]
=&	- \int_{0}^{1} \int (\rd_{k} \zeta)(t, y) \frac{(x-y)^{j} (x-y)^{\ell}}{\sgm^{2}} U^{k}(\frac{x-y}{\sgm} + y) \, \frac{\ud y}{\sgm^{d}} \ud \sgm. \label{eq:Rjl2}
\end{align}

\begin{rem}  One can alternatively define the solution operator $R^{j \ell}[U]$ by letting $\zeta$ solve the transport equation
\begin{equation} 
\label{eq:transportMeasure}
\left\{
\begin{aligned}
	\rd_{t} \zeta + \rd_j( v^{j}_{\eps} \zeta ) =& 0 \qquad \hbox{ for } t \in [t(I)-\bar{\tau}, t(I)+\bar{\tau}], \\
	\zeta(t(I), x) =& \widetilde{\zeta}(x).
\end{aligned}
\right.
\end{equation}
This approach specifies that $\zeta$ is transported as a measure, and the property $\int_{\R^d} \zeta(t,y) dy = 1$ is preserved by the transport equation.

This alternative construction is very natural from a Lagrangian point of view, whereas the method we have chosen is more Eulerian in nature.  The resulting operators obey essentially the same estimates for the length and time scales in our application, so there is no clear advantage of one approach over the other.  It is also unclear whether one approach is more natural for our purpose, as the problem of solving the symmetric divergence equation with good transport properties is in some sense both Eulerian and Lagrangian at the same time.  Here we have chosen the more Eulerian approach since it leads to a shorter proof in our context, although the approach based on solving \eqref{eq:transportMeasure} leads to an interesting alternative proof of boundedness for the commutator $[\Ddt, R^{jl}]$.  \qedsymbol
\end{rem}

The following Proposition summarizes the basic properties of the operator $R^{j \ell}[U]$.
\begin{prop} \label{prop:Rjl}
Let $U^{\ell}$ be a vector field on $\bbR \times \bbR^{d}$ satisfying the hypotheses \eqref{eq:inv4divEq:hyp:1}, \eqref{eq:inv4divEq:hyp:2} and \eqref{eq:inv4divEq:supp4U} of Theorem \ref{thm:inv4divEq}. Define $R^{j \ell}[U] := R^{j \ell}_{0}[U] + R^{j \ell}_{1}[U] + R^{j \ell}_{2}[U]$ by \eqref{eq:Rjl0}, \eqref{eq:Rjl1} and \eqref{eq:Rjl2}. Then $R^{j \ell}[U]$ possesses the following properties:
\begin{enumerate}
\item $R^{j \ell}[U]$ is symmetric in $j, \ell$ and depends linearly on $U$.
\item $R^{j \ell}[U]$ solves the symmetric divergence equation, i.e.,
\begin{equation*}
	\rd_{j} R^{j \ell}[U] = U^{\ell}.
\end{equation*}
\item $R^{j \ell}[U]$ has the support property
\begin{equation} \label{eq:supp4Rjl}
	\supp R^{j \ell}[U] \subseteq \ECyl_{v_{\eps}}(\bar{\tau}, \bar{\rho}; t(I), x(I)).
\end{equation}
\item The following differentiation formulae hold for $R^{j\ell}_{a}[U]$ $(a=0,1,2)$:
\begin{align} 
\nb^{\bt} R^{j \ell}_{0}[U] 
=& 	-\frac{d}{2} \sum_{\bt_{1} + \bt_{2} = \bt} \int_{0}^{1} \int (\nb^{\bt_{1}}\zeta)(t, y) \frac{(x-y)^{j}}{\sgm} (\nb^{\bt_{2}} U^{\ell})(t, \frac{x-y}{\sgm} + y) \, \frac{\ud y}{\sgm^{d}} \ud \sgm \label{eq:dRjl0}\\
&	-\frac{d}{2} \sum_{\bt_{1} + \bt_{2} = \bt} \int_{0}^{1} \int (\nb^{\bt_{1}}\zeta)(t, y) \frac{(x-y)^{\ell}}{\sgm} (\nb^{\bt_{2}} U^{j})(t, \frac{x-y}{\sgm} + y) \, \frac{\ud y}{\sgm^{d}} \ud \sgm, \notag \\
\nb^{\bt} R^{j \ell}_{1}[U]
=&	\frac{1}{2} \sum_{\bt_{1} + \bt_{2} = \bt} \int_{0}^{1} \int (\nb^{\bt_{1}} \rd_{k} \zeta)(t, y) \frac{(x-y)^{j} (x-y)^{k}}{\sgm^{2}} (\nb^{\bt_{2}} U^{\ell})(\frac{x-y}{\sgm} + y) \, \frac{\ud y}{\sgm^{d}} \ud \sgm \label{eq:dRjl1}\\
& + \frac{1}{2} \sum_{\bt_{1} + \bt_{2} = \bt} \int_{0}^{1} \int (\nb^{\bt_{1}} \rd_{k} \zeta)(t, y) \frac{(x-y)^{\ell} (x-y)^{k}}{\sgm^{2}} (\nb^{\bt_{2}} U^{j})(\frac{x-y}{\sgm} + y) \, \frac{\ud y}{\sgm^{d}} \ud \sgm, \notag \\
\nb^{\bt} R^{j \ell}_{2}[U]
=&	- \sum_{\bt_{1} + \bt_{2} = \bt} \int_{0}^{1} \int (\nb^{\bt_{1}} \rd_{k} \zeta)(t, y) \frac{(x-y)^{j} (x-y)^{\ell}}{\sgm^{2}} (\nb^{\bt_{2}} U^{k})(\frac{x-y}{\sgm} + y) \, \frac{\ud y}{\sgm^{d}} \ud \sgm. \label{eq:dRjl2}
\end{align}
where the summations are over all pairs of multi-indices $(\bt_{1}, \bt_{2})$ such that $\bt_{1} + \bt_{2} = \bt$.

\item Define the \emph{approximate advective derivative} $\approxMd$ to be $\approxMd := \rd_{t} + \overline{v}_{\eps}(t) \cdot \nb$.  Then $R^{j \ell}$ commutes with $\approxMd$, i.e.,
\begin{equation} \label{eq:approxMdRjl}
	\approxMd R_{a}^{j\ell}[U] = R^{j \ell}_{a}[\approxMd U] \qquad a =0,1,2.
\end{equation} 

\end{enumerate}
\end{prop}

\begin{proof} 
Symmetry in $j, \ell$, linear dependence on $U$ and the support property \eqref{eq:supp4Rjl} may be easily read off from the definition \eqref{eq:Rjl0}--\eqref{eq:Rjl2}. Next, we prove the differentiation formulae \eqref{eq:dRjl0}--\eqref{eq:dRjl2} and \eqref{eq:approxMdRjl}. 

To justify the various calculations to follow (such as differentiating under the integral sign), the following lemma, whose proof will be given in the next subsection, will be useful:
\begin{lem} \label{lem:babyKeyLp}
Let $\widetilde{\zeta}$ be a non-negative smooth function with $\supp \widetilde{\zeta} \subseteq B(\bar{\rho}; x_{0})$ such that 
\begin{equation} \label{eq:babyKeyLp:hyp}
	\nrm{\widetilde{\zeta}}_{C^{0}_{x}} \leq C_{\bt} A \bar{\rho}^{-d}
\end{equation}
for some $A > 0$. Then for any $k \geq 0$ and $f \in L^{\infty}_{x}$ supported in $B(\bar{\rho}; x_{0})$, we have
\begin{equation} \label{eq:babyKeyLp}
	\sup_{x \in \bbR^{d}, \sgm \in [0,1]} \abs{\int \widetilde{\zeta}(y) \bb( \frac{\abs{x-y}}{\sgm} \bb)^{k} f(\frac{x-y}{\sgm} + y)\, \frac{\ud y}{\sgm^{d}}} \leq C_{k} A \bar{\rho}^{k} \nrm{f}_{L^{\infty}_{x}}.
\end{equation}
\end{lem}

In order to establish \eqref{eq:dRjl0}, it suffices to prove the case $\abs{\bt} =1$, i.e.,
\begin{equation} \label{eq:dRjl0:pf:1}
\begin{aligned}
	\rd_{m} R^{j \ell}_{0}[U](t,x) 
	=& - \frac{d}{2} \int_{0}^{1} \int (\rd_{m} \zeta)(t, y) \frac{(x-y)^{j}}{\sgm} U^{\ell}(t, \frac{x-y}{\sgm} + y) \, \frac{\ud y}{\sgm} \ud \sgm \\
	&- \frac{d}{2} \int_{0}^{1} \int \zeta(t, y) \frac{(x-y)^{j}}{\sgm} (\rd_{m} U^{\ell})(t, \frac{x-y}{\sgm} + y) \, \frac{\ud y}{\sgm} \ud \sgm \\
	& + \hbox{(Symmetric terms in $j, \ell$)}.
\end{aligned}\end{equation}

The case of $\abs{\bt} >1$ will follow from an induction argument, using similar ideas. To prove \eqref{eq:dRjl0:pf:1}, we first proceed as follows:
\begin{align}
\label{eq:theZTrickAppears}
\begin{split}
	\frac{\rd}{\rd x^{m}} R^{j \ell}_{0}[U](t,x)
	= & \frac{\rd}{\rd z^{m}} R^{j \ell}_{0}[U](t,x+z) \bb\vert_{z=0} \\
	= & -\frac{d}{2} \frac{\rd}{\rd z^{m}} \bb\vert_{z=0}  \int_{0}^{1} \int \zeta(t, y) \frac{(x+z-y)^{j}}{\sgm} U^{\ell}(t, \frac{x+z-y}{\sgm} + y) \, \frac{\ud y}{\sgm^{d}} \ud \sgm \\
&	-\frac{d}{2} \frac{\rd}{\rd z^{m}} \bb\vert_{z=0} \int_{0}^{1} \int \zeta(t, y) \frac{(x+z-y)^{\ell}}{\sgm} U^{j}(t, \frac{x+z-y}{\sgm} + y) \, \frac{\ud y}{\sgm^{d}} \ud \sgm,  
\end{split}
\end{align}

Let us concentrate on the first term on the right-hand side; the other term is symmetric to the first one in $j, \ell$. Making a change of variable $\widetilde{y} = y-z$, we get
\begin{align*}
-\frac{d}{2} & \frac{\rd}{\rd z^{m}} \bb\vert_{z=0}  \int_{0}^{1} \int \zeta(t, y) \frac{(x+z-y)^{j}}{\sgm} U^{\ell}(t, \frac{x+z-y}{\sgm} + y) \, \frac{\ud y}{\sgm^{d}} \ud \sgm \\
& = -\frac{d}{2} \frac{\rd}{\rd z^{m}} \bb\vert_{z=0}  \int_{0}^{1} \int \zeta(t, \widetilde{y}+z) \frac{(x-\widetilde{y})^{j}}{\sgm} U^{\ell}(t, \frac{x-\widetilde{y}}{\sgm} + \widetilde{y}+z) \, \frac{\ud \widetilde{y}}{\sgm^{d}} \ud \sgm
\end{align*}

Now differentiating under the integral sign, which is justified by \eqref{eq:supp4zeta}, \eqref{eq:bnd4zeta}, Lemma \ref{lem:babyKeyLp} and the smoothness of $U$, we get the desired formula \eqref{eq:dRjl0:pf:1}.

The proofs of \eqref{eq:dRjl1} and \eqref{eq:dRjl2} are similar and thus omitted. The formula \eqref{eq:approxMdRjl} is also proved in a similar manner, starting from 
\begin{equation*}
	\approxMd R^{j\ell}_{a}[U](t,x) = \frac{\ud}{\ud s} R^{j\ell}_{a}[U](t+s, x+\int_{t}^{s} \overline{v}_{\eps}(s') \, \ud s') \bb\vert_{s=0}
\end{equation*}
for $a=0,1,2$. We also use the fact that $\approxMd \nb^{\bt} \zeta = 0$ for any $\abs{\bt} \geq 0$ by construction. We omit the details.

Now, it only remains to prove that $R^{j \ell}[U]$ is a (distributional) solution to \eqref{eq:symmDivEq} under the assumptions \eqref{eq:inv4divEq:hyp:1} and \eqref{eq:inv4divEq:hyp:2}. For this purpose, it suffices to show that
\begin{equation*}
	R^{j \ell}[U](t,x) = {}^{(\zeta(t, \cdot))} \widetilde{R}^{j \ell}[U(t, \cdot)](x) .
\end{equation*}
where ${}^{(\zeta(t, \cdot))} \widetilde{R}^{j \ell}[U(t, \cdot)]$ has been defined in the previous subsection.

To arrive at the formulae \eqref{eq:Rjl0}--\eqref{eq:Rjl2}, we need to integrate by parts the derivative $\rd_{k}$ on the outside of \eqref{eq:yrjl1} and \eqref{eq:yrjl2} after averaging against $\zeta(y)$. More precisely, consider the expression
\begin{equation*}
	\widetilde{R}^{j \ell}_{2}[U](t,x) := \int \zeta(t, y) {}^{(y)} r^{j\ell}_{2}[U(t, \cdot)](x) \, \ud y.
\end{equation*}

Using \eqref{eq:supp4zeta}, \eqref{eq:bnd4zeta}, Lemma \ref{lem:babyKeyLp} and the differentiation formulae that we established, it is not difficult to justify the following chain of identities:
\begin{align}
\label{eq:R2actuallyr2}
\begin{split}
\widetilde{R}^{j \ell}_{2}[U](t,x)
=& -\frac{\rd}{\rd x^{k}} \int_{0}^{1} \int (1-\sgm) \zeta(t,y) \frac{(x-y)^{\ell} (x-y)^{j}}{\sgm^{2}} U^{k}(t, \frac{x-y}{\sgm}+y)  \, \frac{\ud y}{\sgm^{d}} \ud \sgm \\
=& -\int_{0}^{1} \int (1-\sgm) (\rd_{k} \zeta)(t,y) \frac{(x-y)^{\ell} (x-y)^{j}}{\sgm^{2}} U^{k}(t, \frac{x-y}{\sgm}+y)  \, \frac{\ud y}{\sgm^{d}} \ud \sgm \\
& - \int_{0}^{1} \int (1-\sgm) \zeta(t,y) \frac{(x-y)^{\ell} (x-y)^{j}}{\sgm^{2}} (\rd_{k} U^{k})(t, \frac{x-y}{\sgm}+y)  \, \frac{\ud y}{\sgm^{d}} \ud \sgm.
\end{split}
\end{align}

Note that 
\begin{equation*}
(\rd_{k} U^{k})(t, \frac{x-y}{\sgm}+y) = - \frac{\sgm}{1-\sgm} \frac{\rd}{\rd y^{k}} \bb[ U^{k}(t, \frac{x-y}{\sgm} + y) \bb]
\end{equation*}
which may be integrated by parts in $y$. As a result, we arrive at the formula
\begin{align*}
\widetilde{R}^{j \ell}_{2}[U](t,x)
=& - \int_{0}^{1} \int (\rd_{k} \zeta)(t,y) \frac{(x-y)^{\ell} (x-y)^{j}}{\sgm^{2}} U^{k}(t, \frac{x-y}{\sgm}+y)  \, \frac{\ud y}{\sgm^{d}} \ud \sgm \\
& + \int_{0}^{1} \int \zeta(t,y) \frac{(x-y)^{\ell}}{\sgm} U^{j}(t, \frac{x-y}{\sgm}+y)  \, \frac{\ud y}{\sgm^{d}} \ud \sgm \\
&+ \int_{0}^{1} \int \zeta(t,y) \frac{(x-y)^{j}}{\sgm} U^{\ell}(t, \frac{x-y}{\sgm}+y)  \, \frac{\ud y}{\sgm^{d}} \ud \sgm.
\end{align*}

Similarly, we compute
\begin{align*}
\widetilde{R}^{j \ell}_{1}[U](t,x)
:=& 	\int \zeta(t,y) {}^{(y)} r^{j \ell}_{1}[U(t, \cdot)](x) \, \ud y \\
=&	\int_{0}^{1} \int (\rd_{k} \zeta)(t,y)  \frac{(x-y)^{j} (x-y)^{k}}{\sgm^{2}} U^{\ell}(t, \frac{x-y}{\sgm}+y)  \, \frac{\ud y}{\sgm^{d}} \ud \sgm \\
&	- \frac{d+1}{2} \int_{0}^{1} \int \zeta(t,y) \frac{(x-y)^{\ell}}{\sgm} U^{j}(t, \frac{x-y}{\sgm}+y)  \, \frac{\ud y}{\sgm^{d}} \ud \sgm \\
& + \hbox{(Symmetric terms in $j, \ell$)},
\end{align*}
and
\begin{align*}
\widetilde{R}^{j \ell}_{0}[U](t,x)
:=& \int \zeta(t,y) {}^{(y)} r^{j \ell}_{0}[U(t, \cdot)](x) \, \ud y \\
=& - \frac{1}{2}\int_{0}^{1} \int \zeta(t,y) \frac{(x-y)^{\ell}}{\sgm} U^{j}(t, \frac{x-y}{\sgm}+y)  \, \frac{\ud y}{\sgm^{d}} \ud \sgm \\
&- \frac{1}{2}\int_{0}^{1} \int \zeta(t,y) \frac{(x-y)^{j}}{\sgm} U^{\ell}(t, \frac{x-y}{\sgm}+y)  \, \frac{\ud y}{\sgm^{d}} \ud \sgm.
\end{align*}

It therefore follows that 
\begin{equation*}
{}^{(\zeta(t, \cdot))} \widetilde{R}^{j \ell}[U(t, \cdot)](x) 
= \widetilde{R}^{j \ell}_{0}[U](t,x) + \widetilde{R}^{j \ell}_{1}[U](t,x) + \widetilde{R}^{j \ell}_{2}[U](t,x)
= R^{j \ell}[U](t,x),
\end{equation*}
as desired.
\end{proof}

\subsection{Estimates for the solution operator and proof of Theorem \ref{thm:inv4divEq}}
In this subsection, we begin by deriving a key technical lemma (Lemma \ref{lem:keyLp}) which allows us to derive $L^{p}$ estimates for the operator $R^{j\ell}[U]$ (Lemma \ref{lem:LpBnd4Rjl}). Next, we use Proposition \ref{prop:Rjl} and Lemma \ref{lem:keyLp} to establish various commutator estimates. Using the results developed so far, a proof of Theorem \ref{thm:inv4divEq} is given at the end.

\begin{lem} \label{lem:keyLp}
Given $\bar{\rho} > 0$, let $\widetilde{\zeta}$ be a non-negative smooth function with $\supp \widetilde{\zeta} \subseteq B(\bar{\rho}; x_{0})$ such that 
\begin{equation} \label{eq:keyLp:hyp}
	\nrm{\widetilde{\zeta}}_{C^{0}_{x}} \leq A \bar{\rho}^{-d}
\end{equation}
for some $A > 0$. Then the following statements hold:

\begin{enumerate}
\item For any $k \geq 0$ and $f \in L^{\infty}_{x}$ supported in $B(\bar{\rho}; x_{0})$, we have
\begin{equation} \label{eq:keyLp:1}
	\sup_{x \in \bbR^{d}, \sgm \in [0,1]} \abs{\int \widetilde{\zeta}(y) \bb( \frac{\abs{x-y}}{\sgm} \bb)^{k} f(\frac{x-y}{\sgm} + y)\, \frac{\ud y}{\sgm^{d}}} \leq C_{k} A \bar{\rho}^{k} \nrm{f}_{L^{\infty}_{x}}.
\end{equation}

\item Moreover for any $k \geq 0$ and $f \in L^{p}_{x}$ $(1 \leq p \leq \infty)$ supported in $B(\bar{\rho}; x_{0})$, we have
\begin{equation} \label{eq:keyLp:2}
	\nrm{\int_{0}^{1} \int \widetilde{\zeta}(y) \bb( \frac{\abs{x-y}}{\sgm} \bb)^{k} f(\frac{x-y}{\sgm} + y)\, \frac{\ud y}{\sgm^{d}} \ud \sgm}_{L^{q}_{x}}  \leq C_{k, p, q} A \bar{\rho}^{k+\frac{d}{q} - \frac{d}{p}} \nrm{f}_{L^{p}_{x}} 
\end{equation}
provided that $\max \set{0, \frac{d}{p} - 1} \leq \frac{d}{q} \leq \frac{d}{p}$ and $\frac{d}{p} -1 \neq \frac{d}{q}$.
\end{enumerate}
\end{lem}
For the applications to the main theorems of the paper, we will only require the $L^\infty$ type estimate \eqref{eq:keyLp:1}.  We include the bound \eqref{eq:keyLp:2} as well because we believe this estimate to be of independent interest, and also to illustrate the robustness of the method.  We also believe it is an interesting question whether one can establish the scaling critical case $\frac{d}{p} -1 = \frac{d}{q}$ of \eqref{eq:keyLp:2}.
\begin{proof} 
We begin by reducing both \eqref{eq:keyLp:1} and \eqref{eq:keyLp:2} to the case $k = 0$. Indeed, suppose that the $k=0$ case holds for both inequalities. By the triangle inequality, we have
\begin{equation*}
	\frac{\abs{x-y}}{\sgm} \leq \abs{\frac{x-y}{\sgm} + y - x_{0}} + \abs{y - x_{0}}.
\end{equation*}

Note that, within the integral, the first and second terms on the right-hand side are $\leq \bar{\rho}$ by the support properties of $f$ and $\widetilde{\zeta}$, respectively. This implies 
\begin{equation*}
\bb( \frac{\abs{x-y}}{\sgm} \bb)^{k} \leq 2^{k} \bar{\rho}^{k},
\end{equation*}
which implies the $k > 0$ case of \eqref{eq:keyLp:1} and \eqref{eq:keyLp:2}, respectively.

Next, let us prove \eqref{eq:keyLp:1} in the case $k=0$. Let us start with the bound
\begin{equation} \label{eq:keyLp:pf:1:1}
	\sup_{x} \abs{\int \widetilde{\zeta}(y)  f(\frac{x-y}{\sgm} + y)\, \frac{\ud y}{\sgm^{d}}} 
	\leq C\sgm^{-d} \bar{\rho}^{d} \nrm{\widetilde{\zeta}}_{L^{\infty}} \nrm{f} _{L^{\infty}_{x}}
	= C A \sgm^{-d} \nrm{f} _{L^{\infty}_{x}}.
\end{equation}

This estimate degenerates as $\sgm \to 0$. On the other hand, making the change of variables
\begin{equation} \label{eq:keyLp:pf:1:2}
	z = \frac{x-y}{\sgm} + y
\end{equation}
we have
\begin{equation} \label{eq:keyLp:pf:1:3}
\begin{aligned}
\sup_{x} \abs{\int \widetilde{\zeta}(y)  f(\frac{x-y}{\sgm} + y)\, \frac{\ud y}{\sgm^{d}}} 
&= \sup_{x} \abs{\int \widetilde{\zeta} \bb( \frac{1}{1-\sgm} x - \frac{\sgm}{1-\sgm} z \bb)  f(z)  \, \frac{\ud z}{(1-\sgm)^{d}}} \\
&	\leq C (1-\sgm)^{-d} \bar{\rho}^{d} \nrm{\widetilde{\zeta}}_{L^{\infty}} \nrm{f} _{L^{\infty}_{x}}
	= C A (1-\sgm)^{-d} \nrm{f} _{L^{\infty}_{x}}.
\end{aligned}
\end{equation}

Combining \eqref{eq:keyLp:pf:1:1} and \eqref{eq:keyLp:pf:1:3}, we obtain \eqref{eq:keyLp:1}.

We are now left to establish \eqref{eq:keyLp:2} in the case $k=0$. For this purpose, define
\begin{equation*}
L_{\widetilde{\zeta}}[f] := \int_{0}^{1} \int \widetilde{\zeta}(y) f(\frac{x-y}{\sgm} + y)\, \frac{\ud y}{\sgm^{d}} \ud \sgm
\end{equation*}

The idea is to interpolate the following $L^{p}$ estimates. 
 \begin{align}
	\nrm{L_{\widetilde{\zeta}}[f]}_{L^{1}_{x}} & \leq C A \nrm{f}_{L^{1}_{x}}  \label{eq:keyLp:pf:2:1}\\
	\nrm{L_{\widetilde{\zeta}}[f]}_{L^{q_{\dlt}}_{x}} & \leq C_{\dlt} A \bar{\rho}^{-1+\dlt} \nrm{f}_{L^{1}_{x}}, \quad \frac{d}{q_{\dlt}} = d - 1 + \dlt,\label{eq:keyLp:pf:2:2}\\
	\nrm{L_{\widetilde{\zeta}}[f]}_{L^{\infty}_{x}} & \leq C A \nrm{f}_{L^{\infty}_{x}} \label{eq:keyLp:pf:2:3}\\
	\nrm{L_{\widetilde{\zeta}}[f]}_{L^{\infty}_{x}} & \leq C_{\dlt} A \bar{\rho}^{-1+\dlt} \nrm{f}_{L^{p_{\dlt}}_{x}}, \quad \frac{d}{p_{\dlt}} = 1 - \dlt. \label{eq:keyLp:pf:2:4}
\end{align}

We remark that the support assumption on $f$ can be removed by replacing $L_{\widetilde{\zeta}}[f]$ by $\widetilde{L}_{\widetilde{\zeta}}[f] := L_{\widetilde{\zeta}}[1_{B(\bar{\rho};x_{0})} f]$. For convenience, we shall work with $L_{\widetilde{\zeta}}[f]$ below.

The first estimate, namely \eqref{eq:keyLp:pf:2:1}, follows from Fubini. For \eqref{eq:keyLp:pf:2:2}, we estimate using the change of variables \eqref{eq:keyLp:pf:1:2} and Minkowski as follows:
\begin{align*}
	\nrm{L_{\widetilde{\zeta}}[f]}_{L^{q_{\dlt}}_{x}} 
	& \leq  \int \int_{0}^{1} \nrm{\widetilde{\zeta} \bb( \frac{1}{1-\sgm} x - \frac{\sgm}{1-\sgm} z \bb)}_{L^{q_{\dlt}}_{x}}  \abs{f(z)}  \, \frac{\ud z}{(1-\sgm)^{d}} \ud \sgm  \\
	& \leq \int_{0}^{1} (1-\sgm)^{\frac{d}{q_{\dlt}} - d} \, \ud \sgm \nrm{\widetilde{\zeta}}_{L^{q_{\dlt}}_{x}} \nrm{f}_{L^{1}_{z}} \\
	& \leq C_{\dlt} A \bar{\rho}^{-1+\dlt} \nrm{f}_{L^{1}_{x}}.
\end{align*}
where we have used the fact that $d/q_{\dlt} - d = -1 + \dlt > -1$ on the last line.  

Next, \eqref{eq:keyLp:pf:2:3} follows by simply integrating \eqref{eq:keyLp:1} in $\sgm$ from $0$ to $1$. Finally, \eqref{eq:keyLp:pf:2:4} is established as follows, using a similar idea as before:
\begin{align*}
	\nrm{L_{\widetilde{\zeta}}[f]}_{L^{\infty}_{x}} 
	& = \sup_{x} \abs{\int \int_{0}^{1} \widetilde{\zeta} \bb( \frac{1}{1-\sgm} x - \frac{\sgm}{1-\sgm} z \bb)  f(z) \, \frac{\ud z}{(1-\sgm)^{d}} \ud \sgm} \\
	& \leq  \bar{\rho}^{d-1+\dlt} \int_{0}^{1/2} (1-\sgm)^{-d} \, \ud \sgm \nrm{\widetilde{\zeta}}_{L^{\infty}_{z}} \nrm{f}_{L^{p_{\dlt}}_{z}} \\
	& \phantom{\leq} + r \int_{1/2}^{1} (1-\sgm)^{-d} \bb(\frac{1-\sgm}{\sgm}\bb)^{\frac{d}{p'_{\dlt}}}\, \ud \sgm \nrm{\widetilde{\zeta}}_{L^{p'_{\dlt}}_{z}} \nrm{f}_{L^{p_{\dlt}}_{z}} \\
	& \leq C_{\dlt} A \bar{\rho}^{-1+\dlt} \nrm{f}_{L^{p_{\dlt}}_{x}}. \qedhere
\end{align*}
\end{proof}

As a consequence of the previous lemma and the differentiation formulae \eqref{eq:dRjl0}--\eqref{eq:dRjl2} and \eqref{eq:approxMdRjl}, we obtain the following $L^{p}$ estimates for $R^{j \ell}$ and the commutator between $\nb^{\bt}$ and $R^{j \ell}$.
\begin{lem} [$L^{p}$--$L^{q}$ bounds for $R^{j \ell}$] \label{lem:LpBnd4Rjl}
Let $U^{\ell}$ be a smooth vector field on $\bbR \times \bbR^{d}$ satisfying the hypotheses \eqref{eq:inv4divEq:hyp:1}, \eqref{eq:inv4divEq:hyp:2} and \eqref{eq:inv4divEq:supp4U} of Theorem \ref{thm:inv4divEq}. Define $R^{j \ell}[U] := R^{j \ell}_{0}[U] + R^{j \ell}_{1}[U] + R^{j \ell}_{2}[U]$ by \eqref{eq:Rjl0}, \eqref{eq:Rjl1} and \eqref{eq:Rjl2}. Then for every $1 \leq p \leq \infty$, we have
\begin{equation} \label{eq:LpBnd4Rjl}
	\nrm{R^{j \ell}[U]}_{C^{0}_{t}L^{q}_{x}} \leq C_{p, q} \bar{\rho}^{1+\frac{d}{q}-\frac{d}{p}} \nrm{U}_{C^{0}_{t}L^{p}_{x}}
\end{equation}
provided that $\max \set{0, \frac{d}{p} - 1} \leq \frac{d}{q} \leq \frac{d}{p}$ and $\frac{d}{p} -1 \neq \frac{d}{q}$. 
\end{lem}

\begin{lem}[Commutator between $\nb^{\bt}$ and $R^{j \ell}$] \label{lem:CommNbRjl}
Let $U^{\ell}$ and $R^{j \ell}[U]$ be as in the hypotheses of Lemma \ref{lem:LpBnd4Rjl}. 
Then for every multi-index $\bt$ and $1 \leq p \leq \infty$, we have
\begin{equation} \label{eq:CommNbRjl}
	\nrm{[\nb^{\bt}, R^{j \ell}][U]}_{C^{0}_{t}L^{q}_{x}} \leq C_{\bt, p, q} \bar{\rho}^{1+\frac{d}{q}-\frac{d}{p}} \sum_{\bt_{1}+\bt_{2} = \bt : \bt_{2} \neq \bt} \bar{\rho}^{-\abs{\bt_{1}}} \nrm{\nb^{\bt_{2}} U}_{C^{0}_{t}L^{p}_{x}}
\end{equation}
provided that $\max \set{0, \frac{d}{p} - 1} \leq \frac{d}{q} \leq \frac{d}{p}$ and $\frac{d}{p} -1 \neq \frac{d}{q}$. \end{lem}
These lemmas follow immediately by applying Lemma \ref{lem:keyLp} to the differentiation formulae \eqref{eq:dRjl0}--\eqref{eq:dRjl2} on each time slice, keeping in mind the properties \eqref{eq:supp4zeta} and \eqref{eq:bnd4zeta} of $\zeta$. We omit the details.

In preparation for estimating the advective derivative of $R^{j \ell}[U]$, we prove the following general commutator estimate
\begin{lem}[Commuting with vector fields] \label{lem:commuteWithVecFlds} Let  $R^{j \ell}[U]$ be as in Lemma \ref{lem:LpBnd4Rjl}, and let $Z$ and $\widetilde{U}^\ell$ be smooth vector fields on $\bbR^{d}$.  Then
\ali{
\nrm{ [Z \cdot \nab, R^{jl}][\widetilde{U}] }_{L_x^q} &\leq C_{p,q} \bar{\rho}^{1 + \fr{d}{q} - \fr{d}{p}}( \bar{\rho}^{-1} \co{Z} + \co{\nab Z} )  \nrm{\widetilde{U}}_{L^{p}_{x}} \label{eq:genCommuteBound}
}
provided that $\max \set{0, \frac{d}{p} - 1} \leq \frac{d}{q} \leq \frac{d}{p}$ and $\frac{d}{p} -1 \neq \frac{d}{q}$.
\end{lem}
\begin{proof}
We claim that, for $R^{j\ell}_0$ and $R^{j\ell}_a$, $a = 1,2$, the following pointwise estimates hold
\begin{align} 
\label{eq:keyCommdv:pf:worst:1}
\begin{split}
\abs{[Z \cdot \nab, R^{j \ell}_{0} ] [\widetilde{U}](x)}
&\leq C \co{Z} ~\int_{0}^{1} \int \abs{\nb \zt( y)} \bb( \frac{\abs{x-y}}{\sgm} \bb) \abs{\widetilde{U}( \frac{x-y}{\sgm} + y)} \frac{\ud y}{\sgm^{d}} \ud \sgm \\
&+ C \co{\nab Z} \int_{0}^{1} \int \abs{\zt( y)} \bb( \frac{\abs{x-y}}{\sgm} \bb) \abs{\widetilde{U}( \frac{x-y}{\sgm} + y)} \frac{\ud y}{\sgm^{d}} \ud \sgm \\
&+ C \co{\nab Z} \int_{0}^{1} \int \abs{\nb \zt( y)} \bb( \frac{\abs{x-y}}{\sgm} \bb)^2 \abs{\widetilde{U}( \frac{x-y}{\sgm} + y)} \frac{\ud y}{\sgm^{d}} \ud \sgm
\end{split} \\
\label{eq:keyCommdv:pf:worst:2}
\begin{split}
\abs{[Z \cdot \nab, R^{j \ell}_{a} ] [\widetilde{U}](x)}
&\leq C \co{Z} ~\int_{0}^{1} \int \abs{\nb^{(2)} \zt( y)} \bb( \frac{\abs{x-y}}{\sgm} \bb)^2 \abs{\widetilde{U}( \frac{x-y}{\sgm} + y)} \frac{\ud y}{\sgm^{d}} \ud \sgm \\
&+ C \co{\nab Z} \int_{0}^{1} \int \abs{\nb \zt( y)} \bb( \frac{\abs{x-y}}{\sgm} \bb)^2 \abs{\widetilde{U}( \frac{x-y}{\sgm} + y)} \frac{\ud y}{\sgm^{d}} \ud \sgm \\
&+ C \co{\nab Z} \int_{0}^{1} \int \abs{\nb^{(2)} \zt( y)} \bb( \frac{\abs{x-y}}{\sgm} \bb)^3 \abs{\widetilde{U}( \frac{x-y}{\sgm} + y)} \frac{\ud y}{\sgm^{d}} \ud \sgm
\end{split} 
\end{align}
From these claims, the desired estimate \eqref{eq:genCommuteBound} follows by Lemma \ref{lem:keyLp}.  We remark that the variable $t$ plays no role in the proof.

The estimates \eqref{eq:keyCommdv:pf:worst:1} and \eqref{eq:keyCommdv:pf:worst:2} 
are all proved similarly; we give a detailed proof of \eqref{eq:keyCommdv:pf:worst:1}, and omit the details for the latter.  We begin by applying the differentiation formula \eqref{eq:dRjl0} to compute
\ali{
[Z \cdot \nab, R^{j \ell}_{0} ][\widetilde{U}] =& - \fr{d}{2} \int_0^1\int Z^k(x) \pr_k \zeta(y) \frac{(x-y)^{j}}{\sgm} \widetilde{U}^\ell( \frac{x-y}{\sgm} + y ) \frac{\ud y}{\sgm^{d}} \ud \sgm \label{eq:termComesOKZnabComm1} \\
&- \fr{d}{2} \int_0^1\int Z^k(x) \pr_k \zeta(y) \frac{(x-y)^{\ell}}{\sgm} \widetilde{U}^j( \frac{x-y}{\sgm} + y ) \frac{\ud y}{\sgm^{d}} \ud \sgm \label{eq:termComesOKZnabComm2} \\
&- \fr{d}{2} \int_0^1\int \zeta(y) \frac{(x-y)^{j}}{\sgm} (Z^k(x) - Z^k(z)) \pr_k \widetilde{U}^\ell ( \frac{x-y}{\sgm} + y ) \frac{\ud y}{\sgm^{d}} \ud \sgm \label{eq:termNeedToIntByPartsZk1} \\
&- \fr{d}{2} \int_0^1\int \zeta(y) \frac{(x-y)^{\ell}}{\sgm} (Z^k(x) - Z^k(z)) \pr_k \widetilde{U}^j ( \frac{x-y}{\sgm} + y ) \frac{\ud y}{\sgm^{d}} \ud \sgm \label{eq:termNeedToIntByPartsZk2}
}
Here we write $z = \frac{x-y}{\sgm} + y$ for the argument of $U$.  

The terms \eqref{eq:termComesOKZnabComm1}, \eqref{eq:termComesOKZnabComm2} are immediately seen to verify \eqref{eq:keyCommdv:pf:worst:1}, so it only remains to estimate the latter terms.  We will focus on the term \eqref{eq:termNeedToIntByPartsZk1} since the last term is treated identically.  

Starting with the identity
\ALI{
\pr_k U^\ell(z) = \pr_k U\left( \frac{x-y}{\sgm} + y\right) = - \fr{\si}{(1-\si)} \fr{\pr}{\pr y^k} \left[ U\left( \frac{x-y}{\sgm} + y \right)\right],
}
we integrate by parts in $y$ to obtain 
\ali{
\eqref{eq:termNeedToIntByPartsZk1} =& - \fr{d}{2} \int_0^1\int \pr_k\zeta(y) \frac{(x-y)^{j}}{\sgm} \fr{\si (Z^k(x) - Z^k(z))}{(1 - \si)} \widetilde{U}^\ell ( \frac{x-y}{\sgm} + y ) \frac{\ud y}{\sgm^{d}} \ud \sgm \\
&- \fr{d}{2} \int_0^1\int \zeta(y) \de_k^{j} \fr{(Z^k(x) - Z^k(z))}{(1- \si)} \widetilde{U}^\ell ( \frac{x-y}{\sgm} + y ) \frac{\ud y}{\sgm^{d}} \ud \sgm \\
&- \fr{d}{2} \int_0^1\int \zeta(y) \frac{(x-y)^{j}}{\sgm} \pr_k Z^k(z) \widetilde{U}^\ell ( \frac{x-y}{\sgm} + y ) \frac{\ud y}{\sgm^{d}} \ud \sgm
}
The estimates \eqref{eq:keyCommdv:pf:worst:1} now follow from the identity 
\[ \fr{z - x}{1 - \si} = \fr{x - y}{\si} \] 
and the pointwise bound
\[ \left(\fr{\abs{Z^k(x) - Z^k(z)}}{(1- \si)}\right) \leq \co{ \nab Z } \left( \fr{|z - x|}{(1 - \si)} \right) \qedhere \] 

\end{proof}

The key tool in estimating the advective derivative of $R^{j \ell}[U]$ will be the following estimate for the commutator $[(v_{\eps} - \overline{v}_{\eps}) \cdot \nb, R^{j \ell}]$, which we derive from the commutator estimates of Lemmas \ref{lem:CommNbRjl} and \ref{lem:commuteWithVecFlds}.
\begin{lem}[Commutator between $(v_{\eps} - \overline{v}_{\eps}) \cdot \nb$ and $R^{j \ell}$] \label{lem:keyComm}
Let $U^{\ell}$ and $R^{j \ell}[U]$ be as in the hypotheses of Lemma \ref{lem:LpBnd4Rjl}. Then for every multi-index $\bt$ with $|\bt| \leq L - 1$ and $1 \leq p \leq \infty$, we have
\begin{equation} \label{eq:keyComm}
\begin{aligned}
	& \nrm{\nb^{\bt} [(v_{\eps} - \overline{v}_{\eps}) \cdot \nb, R^{j \ell}][U]}_{C^{0}_{t}L^{q}_{x}} \\
		& \qquad \leq C_{\bt, p, q} \Xi e_{v}^{1/2} \bar{\rho}^{1+\frac{d}{q}-\frac{d}{p}} \sum_{J_{0}+J_{1}+J_{2} = \abs{\bt}} 
				\bar{\rho}^{-J_{0}} \Xi^{J_{1}}  \nrm{\nb^{(J_{2})} U}_{C^{0}_{t}L^{p}_{x}},
\end{aligned}
\end{equation}
provided that $\max \set{0, \frac{d}{p} - 1} \leq \frac{d}{q} \leq \frac{d}{p}$ and $\frac{d}{p} -1 \neq \frac{d}{q}$. The summation is over all triplets of non-negative integers $(J_{0}, J_{1}, J_{2})$ such that $J_{0} + J_{1} + J_{2} = \abs{\bt}$.
\end{lem}

\begin{proof} 
In $ \nb^{\bt} [(v_{\eps} - \overline{v}_{\eps}) \cdot \nb, R^{j \ell}][U] $, we will find that the worst case occurs when all derivatives fall on $U$, or when all the derivatives fall on the vector field $Y := (v_{\eps} - \overline{v}_{\eps})$.  
  
	Observe that, for $(t,x) \in \ECyl_{v_{\eps}}(\bar{\tau}, \bar{\rho}; t(I), x(I))$ we have the estimates
\ali{
\co{ \nab^\ga Y } &\leq C\Xi^{|\ga|} e_v^{1/2} \qquad 1 \leq |\ga| \leq L \label{eq:C0nabYestimates} \\
|Y (t,x)|  &\leq C \bar{\rho} \Xi e_v^{1/2} \label{eq:C0Yestimate}
}
Let us now decompose $\nb^{\bt} [(v_{\eps} - \overline{v}_{\eps}) \cdot \nb, R^{j \ell}][U] = \nb^{\bt} [Y \cdot \nb, R^{j \ell}][U]$ as 
\ali{
\nb^{\bt} [Y \cdot \nb, R^{j \ell}][U] &= \nab^\b \left( Y^k \pr_k R^{j\ell}[U] \right) - \nab^\b R^{jl}[Y^k \pr_k U] \\
= & \nab^\b \left( Y^k \pr_k R^{j\ell}[U] \right) - R^{jl}[ \nab^\b(Y^k \pr_k U) ] \label{eq:willExpandCommutator} \\
&- [\nab^\b, R^{jl}][Y^k \pr_k U] \label{eq:nabBcommutatorU}
}
The term \eqref{eq:nabBcommutatorU} can be estimated using Lemma \ref{lem:CommNbRjl} by
\ali{
\label{eq:theFirstNbBCommutator}
\nrm{\eqref{eq:nabBcommutatorU}}_{C^{0}_{t} L^{q}_{x}} &\leq C \bar{\rho}^{1 + \fr{d}{q} - \fr{d}{p}} \sum_{\substack{\b_1 + \b_2 + \b_3 = \b \\ \b_1 \neq 0}} \bar{\rho}^{- \b_1} \co{\nb^{\b_2} Y} \| \nb^{\b_3 + 1} U \|_{C^{0}_{t} L^{p}_{x}} 
}
We separate out the cases $\b_2 = 0$ and $1 \leq |\b_2| \leq L - 1$ according to estimates \eqref{eq:C0nabYestimates}-\eqref{eq:C0Yestimate}.  In every case, we obtain
\ali{
\nrm{\eqref{eq:nabBcommutatorU}}_{C^{0}_{t} L^{q}_{x}} &\leq C \bar{\rho}^{1 + \fr{d}{q} - \fr{d}{p}} \Xi e_v^{1/2} \sum_{J_0 + J_1 + J_2 = \b} \bar{\rho}^{- J_0} \Xi^{J_1} \| \nb^{(J_2)}  U \|_{C^{0}_{t} L^{p}_{x}} 
}
We estimate the term \eqref{eq:willExpandCommutator} by first expanding into terms of the form
\ali{
\label{eq:expandedFirstCommutator}
\begin{split}
\eqref{eq:willExpandCommutator} &= \sum_{\b_1 + \b_2 = \b} \nb^{\b_1} Y^k \pr_k \nb^{\b_2} R^{jl}[U] - R^{jl}[\nab^{\b_1} Y^k \pr_k \nb^{\b_2} U] \\
&= \sum_{\b_1 + \b_2 = \b} E_{\b_1, \b_2}
\end{split}
}
Each term on the right hand side of \eqref{eq:expandedFirstCommutator} can be expanded as follows
\ali{
E_{\b_1, \b_2} =& [\nb^{\b_1} Y \cdot \nab, R^{jl}][\nb^{\b_2} U] \label{eq:newCommutatorForm} \\
& + \nb^{\b_1} Y^k \pr_k [\nb^{\b_2}, R^{jl}][U] \label{eq:prkCommutatorTerm}
}
We now express \eqref{eq:prkCommutatorTerm} as a sum of commutators
\ali{
\eqref{eq:prkCommutatorTerm} &= \nb^{\b_1} Y^k [\pr_k\nb^{\b_2}, R^{jl}][U] - \nb^{\b_1} Y^k [\pr_k, R^{jl}][\nb^{\b_2} U] \label{eq:nowprkIsACommutator}
}
Each term of the form \eqref{eq:prkCommutatorTerm} can now be bounded using Lemma \ref{lem:CommNbRjl} by 
\ALI{
\nrm{\eqref{eq:prkCommutatorTerm}}_{C^{0}_{t} L^{q}_{x}} &\leq C \bar{\rho}^{1 + \fr{d}{q} - \fr{d}{p}} \sum_{\substack{J_0 + J_1 + J_2 = |\b| + 1  \\ J_1 + J_2 \neq |\b| + 1}} \bar{\rho}^{- J_0} \co{\nb^{(J_1)} Y} \| \nb^{(J_2)} U \|_{C^{0}_{t} L^{p}_{x}} 
}
The bound \eqref{eq:keyComm} for this term now follows from \eqref{eq:C0nabYestimates}-\eqref{eq:C0Yestimate}.

The remaining terms from \eqref{eq:newCommutatorForm} all have a commutator form $[Z \cdot \nab, R^{jl}][\widetilde{U}]$ where $Z = \nb^{\b_1} Y$ and $\widetilde{U} = \nb^{\b_2} U$.  Appling Lemma \ref{lem:commuteWithVecFlds}, we have
\ali{
\nrm{\eqref{eq:prkCommutatorTerm} }_{C^{0}_{t} L^{q}_{x}} &\leq  C \bar{\rho}^{1 + \fr{d}{q} - \fr{d}{p}} \sum_{J_1 + J_2 = |\b|} (\bar{\rho}^{-1} \co{\nab^{(J_1)} Y} + \co{\nab^{(J_1 + 1)} Y}) \| \nb^{(J_2)} U \|_{C^{0}_{t} L^{p}_{x}}
}
Note that at most $|\b| + 1 \leq L$ derivatives fall on $Y$.  Recalling once more the estimates \eqref{eq:C0nabYestimates}-\eqref{eq:C0Yestimate}, we obtain Lemma \ref{lem:keyComm}.
\end{proof}

We are now ready to give a proof of Theorem \ref{thm:inv4divEq}.
\begin{proof} [Proof of Theorem \ref{thm:inv4divEq}]
In view of Proposition \ref{prop:Rjl} and Lemmas \ref{lem:LpBnd4Rjl}, \ref{lem:CommNbRjl} with $p=q=\infty$, we are only left to establish the estimate \eqref{eq:inv4divEq:2}. The idea is to write the advective derivative as
\begin{equation*}
	\rd_{t} + v_{\eps} \cdot \nb = \approxMd + (v_{\eps} - \overline{v}_{\eps})^{k} \rd_{k} \, .
\end{equation*}

Then using the fact that $[\approxMd, R^{j \ell}] = 0$, for any multi-index $\bt$ with $0 \leq \abs{\bt} \leq L-1$, we have
\begin{equation*}
	\nb^{\bt} (\rd_{t} + v_{\eps} \cdot \nb) R^{j \ell} [U] = \nb^{\bt} ( R^{j \ell} [(\rd_{t} + v_{\eps} \cdot \nb) U] ) + \nb^{\bt} [(v_{\eps} - \overline{v}_{\eps})^{k} \rd_{k}, R^{j \ell}][U].
\end{equation*}

Applying Lemmas \ref{lem:LpBnd4Rjl}--\ref{lem:keyComm} with $p = q = \infty$ and using \eqref{eq:inv4divEq:hyp:3}, the desired estimate \eqref{eq:inv4divEq:2} follows. \qedhere
\end{proof}

%% file: iterateMainLemImpThm2.tex
\newcommand{\HolderPert}{\dlt} 

In this section, we illustrate how the Main Lemma (Lemma \ref{lem:mainLemma}) can be used to establish Theorem \ref{thm:eulerOnRn:cptPert} on the perturbation of smooth Euler flows. 
The basic strategy is the same as in Section 11 of \cite{isett} and the construction in \cite{deLSzeHoldCts}; namely, we iterate the Main Lemma to produce a sequence of solutions $(v_{(k)}, p_{(k)}, R_{(k)})$ to the Euler-Reynolds equations, which converges to a solution $(v, p)$ to the Euler equations as $k \to \infty$ with the desired properties. However, there are a few notable differences compared to \cite{isett}:
\begin{enumerate}
\item As discussed in Section \ref{sec:theMainLemma}, the condition $N \geq \Xi^{\eta}$ in the Main Lemma of \cite{isett}, which forced the frequency $\Xi_{(k)}$ to grow double-exponentially in $k$, is absent from our Main Lemma.
We are therefore able to choose frequencies which grow only exponentially in $k$; see \eqref{eq:MLMT:parameterEvolution:Xi}. 
Having this property makes our solutions closer to the physical picture of turbulence, as discussed in \S \ref{subsec:intro:turbulence}.  We also remark that the exponential growth of frequency makes our proof of Theorem \ref{thm:eulerOnRn:cptPert} simpler compared to that in \cite{isett}, as the evolution laws for the parameters \eqref{eq:MLMT:parameterEvolution:Xi}, \eqref{eq:MLMT:parameterEvolution:ev} and \eqref{eq:MLMT:parameterEvolution:eR} are more straightforward. In particular, we need not rely on the `matrix technology' of \cite[\S 11.2.4]{isett}.

\item In \cite{isett}, an energy function $e_{(k)}(t)$ had to be constructed at each step in order to apply the Main Lemma. In the present case, we need to construct instead an appropriate energy density function $e_{(k)}(t,x)$ (whose integral in $x$ is the energy function $e_{(k)}(t)$), which in particular satisfies a point-wise upper bound \eqref{ineq:goodEnergy} for its advective derivative. In order to achieve \eqref{ineq:goodEnergy}, we employ the machinery of mollification along the flow of $v_{(k)}$. Note, however, that the only apriori information on $v_{(k)}$ we have is that $\nb^{m} v_{(k)} \in C^{0}_{t,x}$ for $m=1, \ldots, L$ (from its frequency and energy levels), which is far weaker than those on $v_{\eps}$ in the previous applications of mollification along the flow. This information turns out to be just sufficient for our construction; see Subsections \ref{subsec:MLMT:moll} and \ref{subsec:MLMT:construction}.

\item In order to ensure the compact support property, we need to keep track of how the support of $(v_{(k)} - v_{(0)}, p_{(k)} - p_{(0)}, R_{(k)})$ enlarges at each step. The Main Lemma (see \eqref{eq:lowBoundEoftx}, \eqref{eq:goalForR1supp}, \eqref{eq:goalForVPsupp}) states that the support of $(v_{(k+1)} - v_{(0)}, p_{(k+1)} - p_{(0)}, R_{(k+1)})$ is contained in a $v_{(k)}$-adapted cylindrical neighborhood of the support of $(v_{(k)} - v_{(0)}, p_{(k)} - p_{(0)}, R_{(k)})$, with duration $\aeq \tht_{(k)}$ and base radius $\aeq \Xi_{(k)}^{-1}$.  
A technical nuance here is that the cylinders are adapted to a different velocity field $v_{(k)}$ at each step. We get around this issue by choosing $e_{(k)}(t,x)$ carefully so that, roughly speaking, the edge of the support of $(v_{(k+1)} - v_{(0)}, p_{(k+1)} - p_{(0)}, R_{(k+1)})$ lies outside the support of $v_{(k)} - v_{(0)}$.  This strategy allows us to describe the enlargement of support in terms of $v_{(0)}$-adapted cylindrical neighborhoods at all steps. Then, in view of the exponential decay of $\tht_{(k)}, \Xi^{-1}_{(k)}$, the desired compact support property follows. For details, see \S\S \ref{subsubsec:ED}, \ref{subsubsec:supp}.

\item Using the new local estimate \eqref{eq:energyPrescribed} for the energy increment in the Main Lemma, we show that the constructed solution $(v, p)$ cannot be identical to any $W^{1/5, 1}_{x}$ or $C^{1/5}_{x}$ vector field on any spatial ball contained in $\Omg_{(0)}$. This fact implies Statement 3 in the Theorem \ref{thm:eulerOnRn:cptPert}.
\end{enumerate}

In Subsection \ref{subsec:MLMT:moll}, we discuss the procedure of mollification along the flow of a vector field with limited regularity, which is used to construct the energy density function $e_{(k)}^{1/2}$. In Subsection \ref{subsec:MLMT:reduceMT}, we reduce the proof of Theorem \ref{thm:eulerOnRn:cptPert} to constructing a sequence $(v_{(k)}, p_{(k)}, R_{(k)})$ of solutions to the Euler-Reynolds system 
that satisfies certain claims, i.e., Claims \ref{claim:MLMT:zeroR}--\ref{claim:MLMT:nontrivial}. In Subsection \ref{subsec:MLMT:construction}, we present the construction of the sequence $(v_{(k)}, p_{(k)}, R_{(k)})$, and in Subsection \ref{subsec:MLMT:verifyClaims}, we verify the claims made in Subsection \ref{subsec:MLMT:reduceMT} with such sequence, thereby concluding the proof of Theorem \ref{thm:eulerOnRn:cptPert}.

\subsection{Mollification along the flow of a vector field with limited regularity} \label{subsec:MLMT:moll}

Let $L \geq 1$, and $v = (v^{1}, v^{2}, v^{3})$ be a vector field on $\bbR \times \bbR^{3}$ whose frequency and energy levels are below $(\Xi, e_{v})$ to order $L$ in $C^{0}$ ($L \geq 1$), in the sense that the following estimate holds.
\begin{equation} \label{eq:MLMT:freqEnergy4v}
	\nrm{\nb^{m} v}_{C^{0}_{t,x}} \leq \Xi^{k} e_{v}^{1/2} \hskip5em m = 1, \ldots, L.
\end{equation}

Recall that \emph{the flow of $v$} is the map ${}^{(v)} \Phi_{s}(t,x) = (t+s, {}^{(v)} \Phi'_{s}(t,x) )$, where ${}^{(v)} \Phi'_{s}$ is the unique solution to the ODE
\begin{equation} \label{eq:MLMT:ODE4Phi}
	\rd_{s} {}^{(v)} \Phi'_{s}(t,x) = v(t+s, {}^{(v)} \Phi'_{s}(t, x)), \quad 
	{}^{(v)} \Phi'_{0}(t,x) = x.
\end{equation}  

As $\nb v$ is uniformly bounded on $\bbR \times \bbR^{3}$ by \eqref{eq:MLMT:freqEnergy4v}, ${}^{(v)} \Phi'_{s}(t,x)$ extends indefinitely in $s$. By continuous dependence on parameters for ODEs, it follows that ${}^{(v)} \Phi'_{s}(t,x), \rd_{s} {}^{(v)} \Phi'_{s}(t,x)$ are continuous in $(t,x, s) \in \bbR \times \bbR^{3} \times \bbR$. Moreover, by differentiating the ODE \eqref{eq:MLMT:ODE4Phi} in $x$, we have that $\nb^{m} ({}^{(v)} \Phi'_{s})$ is continuous in $(t,x,s)$ for $m = 1, \ldots L$. In fact, the following Lemma can be read off from \cite[Proof of Proposition 18.1]{isett}.

\begin{lem} \label{lem:MLMT:est4DPhi}
Let $v$ be a vector field on $\bbR \times \bbR^{3}$ whose frequency and energy levels are below $(\Xi, e_{v})$ to order $L$ in $C^{0}$ ($L \geq 1$), in the sense that \eqref{eq:MLMT:freqEnergy4v} holds. Then for every $1 \leq m \leq L$, there exist constants $C_{a,1}, C_{a,2} > 0$ such that $\nb^{m}({}^{(v)} \Phi'_{s})$ obeys the estimate
\begin{equation} \label{eq:MLMT:est4DPhi}
	\abs{\nb^{m} ({}^{(v)}\Phi'_{s})(t,x)} \leq C_{m,1} e^{C_{m,2} \Xi e_{v}^{1/2} s} \Xi^{m-1}.
\end{equation}
\end{lem}

It is also true that $\rd_{t} ({}^{(v)} \Phi'_{s})$ is continuous. However, this property does not follow directly by differentiating \eqref{eq:MLMT:ODE4Phi}, as we have not assumed anything about $\rd_{t} v$. Rather, it is a consequence of the following Lemma.

\begin{lem} \label{lem:MLMT:diffAlongFlow:Phi}
Let $v$ be a vector field on $\bbR \times \bbR^{3}$ whose frequency and energy levels are below $(\Xi, e_{v})$ to order $L$ in $C^{0}$ ($L \geq 1$), in the sense that \eqref{eq:MLMT:freqEnergy4v} holds. Then for every $(t,x,s) \in \bbR \times \bbR^{3} \times \bbR$ and $\sgm \in \bbR$, we have
\begin{equation} \label{eq:MLMT:diffAlongFlow:Phi:1}
	{}^{(v)} \Phi_{s}({}^{(v)} \Phi_{\sgm}(t,x)) = {}^{(v)} \Phi_{s+\sgm}(t, x).
\end{equation}

Moreover, $\rd_{t} ({}^{(v)} \Phi'_{s}(t, x))$ is continuous in $(t,x, s) \in \bbR \times \bbR^{3} \times \bbR$, and
\begin{equation} \label{eq:MLMT:diffAlongFlow:Phi:2}
	(\rd_{t} + v(t,x) \cdot \nb) {}^{(v)} \Phi'_{s}(t,x) = \rd_{s} {}^{(v)} \Phi'_{s}(t,x) = v(t, {}^{(v)} \Phi'_{s}(t,x)).
\end{equation}
\end{lem}
\begin{proof} 
Equation \eqref{eq:MLMT:diffAlongFlow:Phi:1} can be proved by differentiating both sides by $s$, and observing that both sides solve the same ODE with the same data at $s=0$. Then the continuity of $\rd_{t} ( {}^{(v)} \Phi'_{s})$ and \eqref{eq:MLMT:diffAlongFlow:Phi:2} follow by differentiating at $\sgm=0$ and using the ODE \eqref{eq:MLMT:ODE4Phi}. \qedhere
\end{proof}

\begin{rem} 
When $v$ is smooth, \eqref{eq:MLMT:diffAlongFlow:Phi:2} is equivalent to the fact that the Lie derivative of the vector field $\rd_{t} + v \cdot \nb$ along itself is zero, which is equivalent also to the fact that the commutator of $\rd_{t} + v \cdot \nb$ with itself is zero. The discussion on the process of mollification along the flow in \cite{isett} depended on these facts. Here, since $v$ only has limited regularity (in particular, $v$ is only continuous in $t$), our approach relies instead on the integrated version \eqref{eq:MLMT:diffAlongFlow:Phi:1}.
\end{rem}

Given a smooth function $F$ on $\bbR \times \bbR^{3}$ with compact support, we define its \emph{mollification $\moll{v}{F}_{\tMP, \spMP}$ in space and along the flow of $v$} by the formula
\begin{equation} \label{eq:MLMT:mollifyDefn}
\moll{v}{F}_{\tMP, \spMP}(t,x) := \iint F( {}^{(v)} \Phi_{s} (t,x) + (0, h)) \, \eta_{\spMP}(h) \eta_{\tMP}(s) \, \ud h \ud s.
\end{equation}
where $\tMP, \spMP$ are mollification parameters, $\eta_{\tMP}(s) = \frac{1}{\tMP} \eta_{0}(s / \tMP)$ and $\eta_{\spMP}(h) = \frac{1}{\spMP^{3}} \eta_{1}(h / \spMP)$. Here, $\eta_{0}, \eta_{1}$ are smooth, compactly supported functions on $\bbR$ and $\bbR^{3}$, respectively, such that $\int_\R \eta_{0} \, \ud t = \int_{\R^3} \eta_{1} \, \ud x = 1$, $\supp \, \eta_{0} \subseteq \set{t : \abs{t} \leq 1}$ and $\supp \, \eta_{1} \subseteq \set{x : \abs{x} \leq 1}$.   

The main result of this subsection is Proposition~\ref{prop:MLMT:est4moll} below regarding the regularity of $\moll{v}{F}_{\tMP, \spMP}$ when $v$ merely satisfies $\nb^{m} v \in C^{0}_{t,x}$ for $m = 1, \ldots L$.  
In this Proposition, we consider not only smooth functions $F$, but also locally integrable functions, as our construction involves applying  formula \eqref{eq:MLMT:mollifyDefn} to a function which belongs to $L_{t,x}^\infty$ (the characteristic function of a measurable subset of $\R \times \R^3$).  Within the proof of Proposition~\ref{prop:MLMT:est4moll} below, we show that the formula \eqref{eq:MLMT:mollifyDefn} gives a well defined, continuous function of $(t,x)$ whenever $F$ is locally integrable, and we establish bounds on the regularity of $\moll{v}{F}_{\tMP, \spMP}$  under the assumption that $F$ belongs to $L^{\infty}_{t,x}$.  In particular, the value of \eqref{eq:MLMT:mollifyDefn} is well-defined at every point $(t,x)$ and is independent of the almost-everywhere equivalence class of $F$.    


\begin{prop} \label{prop:MLMT:est4moll}
Let $v$ be a vector field on $\bbR \times \bbR^{3}$ whose frequency and energy levels are below $(\Xi, e_{v})$ to order $L$ in $C^{0}$ ($L \geq 1$), in the sense that \eqref{eq:MLMT:freqEnergy4v} holds. Then for every locally integrable $F$ on $\bbR \times \bbR^{3}$, the following statements hold.
\begin{enumerate}
\item For $0 \leq k \leq 1$, $0 \leq m+k \leq L$, $\nb^{m} \rd_{t}^{k} (\moll{v}{F}_{\tMP, \spMP})$ is continuous in $(t,x) \in \bbR \times \bbR^{3}$.
\item Suppose furthermore that $F \in L^{\infty} (\bbR \times \bbR^{3})$. Then there exist constants $C_{1}, C_{2} > 0$, which depends only on $L$, such that for every $1 \leq m \leq L$, the following quantitative estimates hold.
\begin{align} 
	\nrm{\nb^{m} (\moll{v}{F}_{\tMP, \spMP})}_{C^{0}_{t,x}} 
	\leq & C_{1} e^{C_{2} \Xi e_{v}^{1/2} \tMP} \big[ (\spMP)^{-m} + \Xi^{m} \big] \nrm{F}_{L^{\infty}_{t,x}} \label{eq:MLMT:est4moll:1} \\
	\nrm{\nb^{m-1} (\rd_{t} + v \cdot \nb) (\moll{v}{F}_{\tMP, \spMP})}_{C^{0}_{t,x}} 
	\leq & C_{1} e^{C_{2} \Xi e_{v}^{1/2} \tMP} \big[ (\spMP)^{-(m-1)} + \Xi^{(m-1)} \big] \tMP ^{-1} \nrm{F}_{L^{\infty}_{t,x}} \label{eq:MLMT:est4moll:2} 
\end{align}
\end{enumerate}
\end{prop}

\begin{proof} [Proof of Proposition \ref{prop:MLMT:est4moll}]
For convenience, we shall omit $(v)$ in ${}^{(v)} \Phi_{s}$. We first observe that the value of $\moll{v}{F}_{\tMP, \spMP}(t,x)$ is well-defined at each point $(t,x) \in \R \times \R^3$ whenever $F$ is a locally integrable function.
Let $(t,x)$ be any point in $\R \times \R^3$.  By making a change of variables $(s, h) \mapsto (\sgm, y) = (t+s, \Phi'_{s}(t,x) + h)$ in the integral~\eqref{eq:MLMT:mollifyDefn}, we may write
 \begin{align}
	\moll{v}{F}_{\tMP, \spMP}(t,x) = \int F(\sgm, y) \psi(\sgm, y ; t,x) \, \ud \sgm \ud y. \label{eq:afterCOVformula}
\end{align}
for all smooth functions $F : \R \times \R^3 \to \R$.  In formula \eqref{eq:afterCOVformula}, the function
\begin{equation*} 
	\psi(\sgm,y; t, x) = \eta_{\tMP}(\sgm-t) \eta_{\spMP}(y - \Phi'_{\sgm-t}(t,x)).
\end{equation*}
is a bounded, measurable function of $(\si, y)$ with compact support.  From this observation, the formula~\eqref{eq:afterCOVformula} is well-defined at each point $(t,x) \in \R \times \R^3$ for any locally integrable function by approximation from the case where $F$ is a smooth function.  Furthermore, the function $\psi(\si, y; t,x)$ is bounded in $L_{\si, y}^\infty$ uniformly in $(t,x)$, the value $\psi(\si, y; t,x)$ at every point $(\si, y)$ depends continuously on $(t,x)$, and the support of $\psi$ remains uniformly bounded whenever $(t,x)$ range over a bounded subset of $\R \times \R^3$.  Combining these observations, an application of the dominated convergence theorem shows that Formula~\eqref{eq:afterCOVformula} defines a continuous function of $(t,x)$ whenever $F$ is locally integrable.  Statement 1 now follows similarly by differentiating under the integral sign, using the dominated convergence theorem to pass to the limit in each sequence of difference quotients. 

Next, we turn to Statement 2. Let $\bt$ be a multi-index with $\abs{\bt} = m$. Then differentiating under the integral sign (which is justified for $F$ smooth) and applying the chain rule, we see that 
\begin{equation*}
	\nb^{\bt} [\moll{v}{F}_{\tMP, \spMP} ] (t,x)=  \int \nb^{\bt} [(\eta_{\spMP} \ast F)(\Phi_{s}(t,x))] \eta_{\tMP}(s) \, \ud s
\end{equation*}
is a linear combination of terms of the form
\begin{equation} \label{eq:MLMT:est4moll:pf:1}
	\int [ \rd_{j_{1}} \cdots \rd_{j_{K}} (\eta_{\spMP} \ast F) ] ( \Phi_{s}(t,x) ) \prod_{i=1}^{K} \nb^{\bt_{i}} \Phi^{j_{i}}_{s} (t,x) \, \eta_{\tMP}(s) \, \ud s
\end{equation}
where $0 \leq K \leq m$ and $\bt_{1}, \ldots, \bt_{K}$ are multi-indices such that $\bt_{1} + \cdots + \bt_{K} = \bt$. Using Lemma \ref{lem:MLMT:est4DPhi}, the standard convolution estimate
\begin{equation*}
\rd_{j_{1}} \cdots \rd_{j_{K}} (\eta_{\spMP} \ast F) \leq C (\spMP)^{-K} \nrm{F}_{L^{\infty}_{t,x}},
\end{equation*}
and the fact that $\int \abs{\eta_{\tMP}} \, \ud s \leq C$ (independent of $\tMP$),
we see that the $C^{0}$ norm of \eqref{eq:MLMT:est4moll:pf:1} is bounded from above by
\begin{align*}
	 \leq C e^{C \Xi e_{v}^{1/2} \tMP} (\spMP)^{-K} \nrm{F}_{L^{\infty}_{t,x}} \prod_{i=1}^{K}\Xi^{\abs{\bt_{i}} - 1} 
	 \leq C e^{C \Xi e_{v}^{1/2} \tMP} (\Xi^{m}+(\spMP)^{-m})  \nrm{F}_{L^{\infty}_{t,x}}
\end{align*}
where the last inequality follows from Young's inequality, using the fact that $0 \leq K \leq m$ and $\abs{\bt_1} + \cdots + \abs{\bt_K} = m$. This estimate proves \eqref{eq:MLMT:est4moll:1}. To prove \eqref{eq:MLMT:est4moll:2}, note that 
\begin{equation} \label{eq:MLMT:diffAlongFlow:moll}
	(\rd_{t} + v(t,x) \cdot \nb) \moll{v}{F}_{\tMP, \spMP}(t,x) = - \iint F( \Phi_{s}(t,x) + (0, h)) \eta_{\spMP}(h) \frac{\ud}{\ud s} \eta_{\tMP}(s) \, \ud h \ud s.
\end{equation}

Indeed, for every $\sgm \in \bbR$, we have
\begin{equation*}
	\moll{v}{F}_{\tMP, \spMP}(\Phi_{\sgm}(t,x)) = \iint F( \Phi_{s+\sgm}(t,x) + (0, h)) \eta_{\spMP}(h) \eta_{\tMP}(s) \, \ud h \ud s,
\end{equation*}
by \eqref{eq:MLMT:diffAlongFlow:Phi:1}. Making a change of variable $s' = s+\sgm$ and differentiating at $\sgm = 0$, we obtain \eqref{eq:MLMT:diffAlongFlow:moll}. Then \eqref{eq:MLMT:est4moll:2} can be proved in a similar manner as before, using the fact that $\int \abs{\frac{\ud}{\ud s} \eta_{\tMP}} \, \ud s \leq C (\tMP)^{-1}$.  The estimates of Proposition~\ref{prop:MLMT:est4moll} for $F \in L^\infty$ now follow by from the case where $F$ is smooth by a straightforward approximation argument, as in the proof of continuity of $\moll{v}{F}_{\tMP, \spMP}$. \qedhere
\end{proof}

Proposition \ref{prop:MLMT:est4moll} will be used later to obtain the desired \emph{upper} bounds for the energy density. In order to obtain the desired lower bound, we need to know about the locality of the mollification $\moll{v}{F}_{\tMP, \spMP}$. This property can be described succinctly by using Eulerian cylinders adapted to $v$ (Definition \ref{def:vCylinder}), as the following lemma shows. 
\begin{lem} [Locality of the mollification] \label{lem:MLMT:locality4Mollification}
Let $v$ be a vector field on $\bbR \times \bbR^{3}$ whose frequency and energy levels are below $(\Xi, e_{v})$ to order $L$ in $C^{0}$ ($L \geq 1$), in the sense that \eqref{eq:MLMT:freqEnergy4v} holds. Also, let $F$ be a locally integrable function on $\bbR \times \bbR^{3}$ and $\tMP, \spMP > 0$. 
Then for every $(t,x) \in \bbR \times \bbR^{3}$, the mollification $\moll{v}{F}_{\tMP, \spMP}(t,x)$ depends only on the values of $v$ and $F$ on $\ECyl_{v}(\tMP, \spMP; t, x)$. Furthermore, the advective derivative $(\rd_{t} + v \cdot \nb) [\moll{v}{F}_{\tMP, \spMP}] (t,x)$ also depends only on the values of $v$ and $F$ on $\ECyl_{v}(\tMP, \spMP; t, x)$ as well.
\end{lem}
\begin{proof} 
This follows from the definition \eqref{eq:MLMT:mollifyDefn}, the identity \eqref{eq:MLMT:diffAlongFlow:moll} and our choice of $\eta_{\tMP}, \eta_{\spMP}$. \qedhere
\end{proof}

\subsection{Reduction of Theorem \ref{thm:eulerOnRn:cptPert}}  \label{subsec:MLMT:reduceMT}
We are now ready to begin the proof of Theorem \ref{thm:eulerOnRn:cptPert}. In this Subsection, we reduce the proof of Theorem \ref{thm:eulerOnRn:cptPert} to constructing a sequence $(v_{(k)}, p_{(k)}, R_{(k)})$ of solutions to the Euler-Reynolds system that satisfies certain claims (Claims \ref{claim:MLMT:zeroR}--\ref{claim:MLMT:nontrivial}).

From the hypotheses of Theorem \ref{thm:eulerOnRn:cptPert}, recall that we are given positive numbers $\eps, \HolderPert > 0$, a smooth solution $(v_{(0)}, p_{(0)})$ to the incompressible Euler equations on $\bbR \times \bbR^{3}$ and pre-compact open sets $\Omg_{(0)}, \calU$ such that $\Omg_{(0)} \neq \emptyset$ and
\begin{equation} \label{}
	\overline{\Omg_{(0)}} \subseteq \calU .
\end{equation}

From these inputs, we shall produce in the following Subsections (Subsections \ref{subsec:MLMT:construction} and \ref{subsec:MLMT:verifyClaims}) a sequence $(v_{(k)}, p_{(k)}, R_{(k)})$ of solutions to the Euler-Reynolds system which satisfies the following Claims: 

\begin{claim}[Vanishing of the Euler-Reynolds stress] \label{claim:MLMT:zeroR}
The Euler-Reynolds stress $R_{(k)}$ converges uniformly to zero, i.e., $\nrm{R_{(k)}}_{C^{0}} \to 0$ as $k \to \infty$.
\end{claim}

\begin{claim}[Compact support in space-time] \label{claim:MLMT:cptSupp}
There exists a pre-compact set $\Omg_{(\infty)} \subseteq \bbR \times \bbR^{3}$ such that $\overline{\Omg_{(\infty)}} \subseteq \calU$ and for every $k \geq 0$,
\begin{equation*}
\supp (v_{(k)} - v_{(0)}, \, p_{(k)} - p_{(0)}) \subseteq \Omg_{(\infty)}.
\end{equation*}
\end{claim}

\begin{claim}[H\"older regularity of the solution] \label{claim:MLMT:HolderReg}
For $\alp = \frac{1}{5} - \eps$, the sequence $(v_{(k)}, p_{(k)})$ is Cauchy in $C^{\alp}_{t,x} \times C^{2 \alp}_{t,x}$ as $k \to \infty$. Moreover, for every $k \geq 0$, we have
\begin{equation} \label{eq:MLMT:smallHolder}
\nrm{v_{(k)} - v_{(0)}}_{C^{\alp}_{t,x}} + \nrm{p_{(k)} - p_{(0)}}_{C^{2\alp}_{t,x}} \leq \frac{\HolderPert}{2}.
\end{equation}
\end{claim}

We state Claims \ref{claim:MLMT:energyIncrease}-\ref{claim:MLMT:nontrivial}  
 using the notation 
\ALI{
I[\Om_{(0)}] &:= \{ t \in \R ~:~ \Om_{(0)} \cap \{ t \} \times \R^3 \neq \emptyset \} \\
S_{t_*}[\Om_{(0)}] &:= \{ x \in \R^3 ~:~ (t_\star, x) \in \Om_{(0)} \}.
}

\begin{claim}[Increase of local energy] \label{claim:MLMT:energyIncrease}
For every $t_{\star} \in I[\Omg_{(0)}]$ and smooth, compactly supported function $\psi$ such that $\psi \equiv 1$ on $S_{t_{\star}} [\calU]$, we have
\begin{equation} \label{eq:energyIncrease}
	\int \psi(x) \frac{\abs{v_{(k+1)}(t_{\star}, x)}^{2}}{2} \, \ud x >  \int \psi(x) \frac{\abs{v_{(k)}(t_{\star}, x)}^{2}}{2} \, \ud x
\end{equation}
for every $k \geq 0$.

\end{claim}

\begin{claim}[Irregularity of the solution] \label{claim:MLMT:nontrivial}
For any $t_{\star} \in I[\Omg_{0}]$ and $B(\rho_{\star}; x_{\star}) \subseteq S_{t_{\star}}[\Omg_{(0)}]$, let $\psi = \psi(x)$ be a smooth function on $\bbR^{3}$ such that $\supp \psi \subseteq B(\rho_{\star}; x_{\star})$, $\psi \geq 0$ and $\int \psi(x) \, \ud x = 1$. Then for every $u \in W^{1/5, 1}_{x}(B(\rho_{\star}; x_{\star}))\cup C^{1/5}_{x}(B(\rho_{\star}; x_{\star}))$, there exists $k_{\star} = k_{\star}(\rho_{\star}, t_{\star}, x_{\star}, v_{(0)}, \psi, u) \geq 0$ such that
\begin{equation} \label{eq:nontrivial}
	\int \psi(x) \frac{\abs{(v_{(k+1)} - u)(t_{\star}, x)}^{2}}{2} \, \ud x >  \int \psi(x) \frac{\abs{(v_{(k)} - u)(t_{\star}, x)}^{2}}{2} \, \ud x
\end{equation}
holds for all $k \geq k_{\star}$.
\end{claim}

Assuming these Claims, Theorem \ref{thm:eulerOnRn:cptPert} follows rather immediately.

\begin{proof} [Proof of Theorem \ref{thm:eulerOnRn:cptPert} assuming Claims \ref{claim:MLMT:zeroR}--\ref{claim:MLMT:nontrivial}]
By Claims \ref{claim:MLMT:zeroR} and \ref{claim:MLMT:HolderReg}, it follows that $(v, p) := \lim_{k \to \infty} (v_{(k)}, p_{(k)})$ exists in $C^{1/5-\eps}_{t,x} \times C^{2 (1/5-\eps)}_{t,x}$, and is a solution to the incompressible Euler equations. Moreover, by \eqref{eq:MLMT:smallHolder}, it follows that
\begin{equation*}
	\nrm{v - v_{(0)}}_{C^{1/5-\eps}_{t,x}} + \nrm{p - p_{(0)}}_{C^{2(1/5-\eps)}_{t,x}} \leq \frac{\HolderPert}{2} < \HolderPert,
\end{equation*}
which proves Statement 2. Statements 1 and 4 of Theorem \ref{thm:eulerOnRn:cptPert} then follow from Claims \ref{claim:MLMT:cptSupp} and \ref{claim:MLMT:energyIncrease}, respectively. 
Finally, for every $t_{\star} \in I[\Omg_{0}]$, $B(\rho_{\star}; x_{\star}) \subseteq S_{t_{\star}}[\Omg_{(0)}]$ and $u \in W^{1/5,1}_{x}(B(\rho_{\star}; x_{\star})) \cup C^{1/5}_{x}(B(\rho_{\star}; x_{\star}))$, Claim \ref{claim:MLMT:nontrivial} shows that there exists a non-negative function $\psi$ supported in $B(\rho_{\star}; x_{\star})$ and $k_{\star} \geq 0$ such that the quantity
\begin{equation*}
\int \psi(x) \frac{\abs{(v_{(k+1)} - u)(t_{\star}, x)}^{2}}{2} \, \ud x
\end{equation*}
is \emph{strictly increasing} for $k \geq k_{\star}$. Since this integral is non-negative for every $k$, it follows that 
\begin{equation*}
	\int \psi(x) \frac{\abs{(v - u)(t_{\star}, x)}^{2} }{2 } \, \ud x > 0.
\end{equation*}

Thus, $v \neq u$ on $B(\rho_{\star}; x_{\star})$. 
Since $u$ can be an arbitrary function in $W^{1/5,1}_{x}(B(\rho_{\star}; x_{\star}))$ or in $C^{1/5}_{x}(B(\rho_{\star}; x_{\star}))$, it follows that $v$ belongs to neither $W^{1/5, 1}_{x}(B(\rho_{\star}; x_{\star}))$ nor $C^{1/5}_{x}(B(\rho_{\star}; x_{\star}))$.  
As $t_{\star} \in I[\Omg_{(0)}]$ and $B(\rho_{\star}; x_{\star}) \subseteq S_{t_{\star}}[\Omg_{(0)}]$ can be arbitrary, Statement 3 follows. \qedhere
\end{proof}

The following Subsections will be devoted to the construction of a sequence $(v_{(k)}, p_{(k)}, R_{(k)})$ which satisfies the above claims. More precisely, the construction process itself will be described in Subsection \ref{subsec:MLMT:construction}, and the Claims \ref{claim:MLMT:zeroR}--\ref{claim:MLMT:nontrivial} will be verified for the constructed sequence $(v_{(k)}, p_{(k)}, R_{(k)})$ in Subsection \ref{subsec:MLMT:verifyClaims}.

\subsection{Construction of $(v_{(k)}, p_{{(k)}}, R_{(k)})$} \label{subsec:MLMT:construction}
In this Subsection, we describe the construction of the sequence $(v_{(k)}, p_{(k)}, R_{(k)})$, which will be shown to satisfy Claims \ref{claim:MLMT:zeroR}--\ref{claim:MLMT:nontrivial} in Subsection \ref{subsec:MLMT:verifyClaims}. 
The basic scheme is as follows: Given $(v_{(k)}, p_{(k)}, R_{(k)})$ with frequency and energy levels below $(\Xi_{(k)}, e_{v, (k)}, e_{R, (k)})$, along with sets $\Omg_{(k)}, \widetilde{\Omg}_{(k)}$ such that
\begin{gather} 
\supp \, (v_{(k)} - v_{(0)}, p_{(k)}- p_{(0)}, R_{(k)}) \subseteq  \Omg_{(k)}, \label{eq:MLMT:Omgs:1} \\
\ECyl_{v_{(0)}} (5\tht_{(k)}, 5000 \Xi_{(k)}^{-1} ; \, \Omg_{(k)})
\subseteq  \widetilde{\Omg}_{(k)}  
\subseteq \overline{\widetilde{\Omg}_{(k)}} 
\subseteq \calU \label{eq:MLMT:Omgs:2}
\end{gather}
(where $\tht_{(k)} = \Xi_{(k)}^{-1} e_{v, (k)}^{-1/2}$) we use the Main Lemma to produce $(v_{(k+1)}, p_{(k+1)}, R_{(k+1)})$ with frequency and energy levels below $(\Xi_{(k+1)}, e_{v, (k+1)}, e_{R, (k+1)})$ satisfying the ansatz
\begin{align} 
		\Xi_{(k+1)} = \, & C_{0} Z^{5/2} \Xi_{(k)}		\label{eq:MLMT:parameterEvolution:Xi} \\
	e_{v, (k+1)} = \, & e_{R, (k)} 					\label{eq:MLMT:parameterEvolution:ev}\\
	e_{R, (k+1)} = \, & \frac{e_{R, (k)}}{Z}		\label{eq:MLMT:parameterEvolution:eR} 
\end{align}
where $C_{0}$ is the constant in the Main Lemma and $Z$ is a parameter to be specified. Note that $\Xi_{(k)}$ grows exponentially, and $e_{v, (k)}$ and $e_{R, (k)}$ decay exponentially.
We also construct $\Omg_{(k+1)}, \widetilde{\Omg}_{(k+1)}$ satisfying \eqref{eq:MLMT:Omgs:1}, \eqref{eq:MLMT:Omgs:2} and furthermore
\begin{equation} \label{eq:MLMT:Omgs:3}
	\widetilde{\Omg}_{(k+1)} \subseteq \widetilde{\Omg}_{(k)}.
\end{equation}

\subsubsection{The base case}
Here, we choose the parameters $\Xi_{(0)}, e_{v, (0)}, e_{R, (0)}$.  We will also choose $\widetilde{\Omg}_{(0)}$ so that \eqref{eq:MLMT:Omgs:2} holds. These choices will serve as the base step for the construction sketched above.

\begin{rem}  In general, one can construct solutions by taking the initial frequency and energy levels $\Xi_{(0)}$ and $e_{v,(0)}$ to be any values for which the bounds \eqref{bound:nabkv}-\eqref{bound:nabkp} hold for the initial velocity and pressure $(v_{(0)}, p_{(0)})$.  With such a choice of parameters, it is natural to regard the pair $(\Xi_{(0)}^{-1}, e_{v,(0)}^{1/2})$ as a characteristic length scale and velocity for the solutions constructed by our procedure.  In our proof below, we will take a more specific choice of $(\Xi_{(0)}, e_{v,(0)})$ that is convenient for proving Claims \ref{claim:MLMT:zeroR}--\ref{claim:MLMT:nontrivial}.
\end{rem}


\paragraph*{Choice of $e_{v, (0)}$ and $e_{R, (0)}$.} 
We choose 
\begin{equation} \label{eq:MLMT:baseCase:0}
	e_{v, (0)} = 1, \quad
	e_{R, (0)} = Z^{-1}
\end{equation}
where $Z > 1$ is a large parameter to be chosen later; in fact, it will be finally fixed in the Proof of Claim \ref{claim:MLMT:energyIncrease} in Subsection \ref{subsec:MLMT:verifyClaims}. We take $\Xi_{(0)} > 1 $ sufficiently large so that 
\begin{equation} \label{eq:MLMT:baseCase:1}
(v_{(0)}, p_{(0)}, R_{(0)}) \hbox{ has frequency and energy levels below } (\Xi_{(0)}, 1, \frac{1}{Z}).
\end{equation}

This choice of $\Xi_{(0)}$ can be made independently of the choice of $Z$, since $R_{(0)} = 0$. 

\paragraph*{Choice of $\widetilde{\Omg}_{(0)}$ and $\Xi_{(0)}$.}
We choose
\begin{equation}
\widetilde{\Omg}_{(0)} := \ECyl_{v_{(0)}} (5 \tht_{(0)}, 5000 \Xi_{(0)}^{-1} ; \, \Omg_{(0)}),
\end{equation}
which makes the first inclusion in \eqref{eq:MLMT:Omgs:2} hold automatically. Since $\Omg_{(0)}$ is pre-compact, we may ensure that the last inclusion in \eqref{eq:MLMT:Omgs:2} holds as well by choosing $\Xi_{(0)} > 1$ larger if necessary. We remark that \eqref{eq:MLMT:Omgs:1} also holds, since the left-hand side is empty for $k=0$. 

\subsubsection{Choosing the parameters for $k \geq 1$}
Here, we describe the choice of parameters needed to apply the Main Lemma in order to construct $(v_{(k+1)}, p_{(k+1)}, R_{(k+1)})$, except for the choice of the energy density $e_{(k)}(t,x)$.

From \eqref{eq:theNewEnergyLevel} of the Main Lemma, the Ansatz \eqref{eq:MLMT:parameterEvolution:eR} and base case \eqref{eq:MLMT:baseCase:0}, we are led to the choices
\begin{equation} \label{eq:MLMT:parameterEvolution:ML}
	\frac{e_{v, (k)}}{e_{R, (k)}} = Z, \qquad
	N_{(k)} = Z^{2} \bb( \frac{e_{v, (k)}}{e_{R, (k)}} \bb)^{1/2} = Z^{5/2}.
\end{equation}
for $k \geq 0$. Note that $Z > 1$ is enough to ensure \eqref{eq:conditionsOnN2}.
Accordingly, we choose $\Xi_{(k+1)}$ to be
\begin{equation*} 
	\Xi_{(k+1)} = C_{0} N_{(k)} \Xi_{(k)} = C_{0} Z^{5/2} \Xi_{(k)},
\end{equation*}
where $C_{0} > 1$ is the constant given by the Main Lemma, which depends only on $M > 0$. The latter constant will be chosen to be $M = C_{1} e^{C_{2}}$, where $C_{1}, C_{2}$ are constants in Proposition \ref{prop:MLMT:est4moll}; see \eqref{eq:MLMT:chooseED:upper}. 

\begin{rem} 
The size of the constant $C_{0}$ in the Main Lemma determines whether the constructed solution $(v, p)$ belongs to $C^{1/5}_{t,x} \times C^{2/5}_{t,x}$ or not. In our proof of the Main Lemma, recall that $C_{0}$ was chosen to be sufficiently large in order to absorb many implicit constants that arose in the proof. In particular, $C_{0} > 1$, and as we shall see below, this inequality forces the constructed solution $(v, p)$ to fail to belong to $C^{1/5}_{t,x} \times C^{2/5}_{t,x}$ locally, as stated in Theorem \ref{thm:eulerOnRn:cptPert} (see also Claim \ref{claim:MLMT:nontrivial}). On the other hand, if we had $C_{0} \leq 1$, then it would follow that $(v, p)$ belongs to $C^{1/5}_{t,x} \times C^{2/5}_{t,x}$, by a slight variant of our proof of Claim \ref{claim:MLMT:HolderReg} below.
\end{rem}

At this point, we take $Z > 1$ to be sufficiently large to make sure that the space- and time-scales $\Xi^{-1}_{(k)}$, $\tht_{(k)}$ decrease sufficiently fast to be used in the construction of $e_{(k)}(t,x)$ below.  In particular, our $Z$ will satisfy the hypothesis of the following lemma.

\begin{lem} \label{lem:MLMT:parameterEvolution}
Let $\Xi_{(k)}$, $e_{v, (k)}$, $e_{R, (k)}$, $N_{(k)}$ and $\tht_{(k)}$ be chosen inductively according to \eqref{eq:MLMT:parameterEvolution:Xi}, \eqref{eq:MLMT:parameterEvolution:ev}, \eqref{eq:MLMT:parameterEvolution:eR} and \eqref{eq:MLMT:parameterEvolution:ML} from the case $k=0$ given above. Then there exists $Z_{0} > 0$ such that if $Z \geq Z_{0}$, then we have
\begin{equation} \label{eq:MLMT:scaleDecrease}
\Xi_{(k+1)}^{-1} \leq \frac{1}{5000} \Xi_{(k)}^{-1}, \quad 
\tht_{(k+1)} \leq \frac{1}{500} \tht_{(k)}.
\end{equation}
\end{lem}

\begin{proof} 
The first inequality follows from \eqref{eq:MLMT:parameterEvolution:Xi}, by taking $C_{0} Z_{0}^{5/2} \geq 5000$. To prove the second inequality, note that
\begin{equation} \label{eq:MLMT:parameterEvolution:tht}
	\tht_{(k+1)} 
	= \Xi_{(k+1)}^{-1} e_{v, (k+1)}^{-1/2}
	= C_{0}^{-1} Z^{-2} \Xi_{(k)}^{-1} e_{v, (k)}^{-1/2}
	= C_{0}^{-1} Z^{-2} \tht_{(k)}.
\end{equation}

Thus, taking $C_{0} Z_{0}^{2} \geq 500$, the second inequality follows. \qedhere
\end{proof}

\subsubsection{Choosing the energy density} \label{subsubsec:ED}
We now describe how to choose the energy density $e_{(k)}(t,x)$, which satisfies the hypotheses \eqref{eq:lowBoundEoftx} and \eqref{ineq:goodEnergy} of the Main Lemma. This choice allows us to invoke the Main Lemma to produce $(v_{(k+1)}, p_{(k+1)}, R_{(k+1)})$ with frequency and energy levels below $(\Xi_{(k+1)}, e_{v, (k+1)}, e_{R, (k+1)})$ satisfying \eqref{eq:MLMT:parameterEvolution:ev} and \eqref{eq:MLMT:parameterEvolution:eR}.

Recall that we are given $\Omg_{(k)}$, $\widetilde{\Omg}_{(k)}$ satisfying \eqref{eq:MLMT:Omgs:1}, \eqref{eq:MLMT:Omgs:2}.
Let $\chi_{(k)}$ be the characteristic function of $\ECyl_{v_{(k)}}(2 \tht_{(k)}, 2 \Xi_{(k)}^{-1}; \, \Omg_{(k)})$.  Note that $\chi_{(k)}$ is a locally integrable function\footnote{Strictly speaking, one must check at this point that the set $\ECyl_{v_{(k)}}(2 \tht_{(k)}, 2 \Xi_{(k)}^{-1}; \, \Omg_{(k)})$ and the function $\chi_{(k)}$ are measurable.  This point can be proven by noting that $\ECyl_{v_{(k)}}(2 \tht_{(k)}, 2 \Xi_{(k)}^{-1}; \, \Omg_{(k)})$ is a countable union of compact subsets of $\R \times \R^3$. 
of $(t,x)$.}  
Define $e^{1/2}_{(k)}$ to be $(K e_{R, (k)})^{1/2}$ times the mollification of $\chi_{(k)}$ in space and along the flow of $v_{(k)}$, with parameters $\tMP = \frac{1}{100} \tht_{(k)}$ and $\spMP = \frac{1}{100} \Xi_{(k)}^{-1}$. More precisely,
\begin{equation} \label{eq:MLMT:chooseED}
	e^{1/2}_{(k)}(t,x) := (K e_{R, (k)})^{1/2} (\moll{v_{(k)}}{[\chi_{(k)}]}_{\frac{1}{100} \tht_{(k)}, \frac{1}{100} \Xi_{(k)}^{-1}})(t,x).
\end{equation}

The desired upper bound \eqref{ineq:goodEnergy} follows from
\begin{equation} \label{eq:MLMT:chooseED:upper}
|| \nab^m (\pr_t + v_{(k)} \cdot \nab)^r e_{(k)}^{1/2} ||_{C^0} \leq M \Xi_{(k)}^m (\Xi_{(k)} e_{v, (k)}^{1/2})^r e_{R, (k)}^{1/2} \hskip3em 0 \leq r \leq 1, 0 \leq m + r \leq L 
\end{equation}
which in turn is an immediate consequence of Proposition \ref{prop:MLMT:est4moll} (with $M = C_{1} e^{C_{2}}$). 

Next, we verify that the desired lower bound holds, i.e.,
\begin{equation} \label{eq:MLMT:chooseED:lower}
	e_{(k)}(t,x) \geq K e_{R, (k)} \quad \hbox{for} \quad (t,x) \in \ECyl_{v_{(k)}} (\tht_{(k)}, \Xi^{-1}_{(k)}; \, \Omg_{(k)}).
\end{equation}

By \eqref{eq:ECylECylinECyl} (Lemma \ref{lem:cylContainmentProps}), we see that 
\begin{equation*}
\ECyl_{v_{(k)}} (\frac{1}{100} \tht_{(k)}, \frac{1}{100} \Xi^{-1}_{(k)}; \, \ECyl_{v_{(k)}} (\tht_{(k)}, \Xi^{-1}_{(k)}; \, \Omg_{(k)})) \subseteq \ECyl_{v_{(k)}} (2 \tht_{(k)}, 2 \Xi^{-1}_{(k)}; \, \Omg_{(k)}).
\end{equation*}

Thus, for every $(t,x) \in \ECyl_{v_{(k)}} (\tht_{(k)}, \Xi^{-1}_{(k)}; \, \Omg_{(k)})$, we have $\chi_{(k)} \equiv 1$ on $\ECyl_{v_{(k)}} (\tht_{(k)}, \Xi^{-1}_{(k)}; t, x)$. Using Lemma \ref{lem:MLMT:locality4Mollification} (Locality of the mollification) to replace $\chi_{(k)}$ by $1$, and noting that the mollification of the latter is trivially $\equiv 1$, we conclude that \eqref{eq:MLMT:chooseED:lower} holds, in fact, with equality.

\subsubsection{Controlling the enlargement of support} \label{subsubsec:supp}
To continue the construction, we need to choose $\Omg_{(k+1)}$ and $\widetilde{\Omg}_{(k+1)}$ so that \eqref{eq:MLMT:Omgs:1}, \eqref{eq:MLMT:Omgs:2}, \eqref{eq:MLMT:Omgs:3} hold. We define $\Omg_{(k+1)}$, $\widetilde{\Omg}_{(k+1)}$ to be appropriate $v_{(0)}$-adapted cylindrical neighborhoods of $\Omg_{(k)}$, i.e.,
\begin{equation}
	\Omg_{(k+1)} := \LCyl_{v_{(0)}} (4 \tht_{(k)}, 2000 \Xi_{(k)}^{-1} ; \, \Omg_{(k)}), \quad
	\widetilde{\Omg}_{(k+1)} := \ECyl_{v_{(0)}} (5 \tht_{(k)}, 5000 \Xi_{(k)}^{-1} ; \, \Omg_{(k)}).
\end{equation}

We first establish \eqref{eq:MLMT:Omgs:1} for $k+1$. By construction, note that
\begin{equation*}
	\supp \, e_{(k)} 
	\subseteq \LCyl_{v_{(k)}} (\frac{1}{100} \tht_{(k)}, \frac{1}{100} \Xi_{(k)}^{-1} ; \, \ECyl_{v_{(k)}}(2 \tht_{(k)}, 2 \Xi_{(k)}^{-1}; \, \Omg_{(k)}))
\end{equation*}

By \eqref{eq:LCylECylinECyl}, \eqref{eq:ECylECylinECyl} of Lemma \ref{lem:cylContainmentProps} and Lemma \ref{lem:cylinderEquiv} (Equivalence of Eulerian and Lagrangian Cylinders), we have
\begin{equation*}
	\ECyl_{v_{(k)}}(\tht_{(k)}, \Xi_{(k)}^{-1}; \, \supp \, e_{(k)}) \subseteq \LCyl_{v_{(k)}} (4 \tht_{(k)}, 2000 \Xi_{(k)}^{-1} ; \, \Omg_{(k)}).
\end{equation*}

Since $\supp \, (v_{(k)} - v_{(0)}) \subseteq \Omg_{(k)}$, Lemma \ref{lem:MLMT:fsp4Flow} applies and it follows that
\begin{equation*}
	\LCyl_{v_{(k)}} (4 \tht_{(k)}, 2000 \Xi_{(k)}^{-1} ; \, \Omg_{(k)})
	= \LCyl_{v_{(0)}} (4 \tht_{(k)}, 2000 \Xi_{(k)}^{-1} ; \, \Omg_{(k)})
	= \Omg_{(k+1)}.
\end{equation*}

As $(V_{(k)}, P_{(k)}, R_{(k))}) = (v_{(k+1)} - v_{(k)}, p_{(k+1)} - p_{(k)}, R_{(k)})$ produced by the Main Lemma is supported in $\ECyl_{v_{(k)}}(\tht_{(k)}, \Xi_{(k)}^{-1}; \, \supp \, e_{(k)})$, we see that \eqref{eq:MLMT:Omgs:1} holds for $k+1$. 

Next, by \eqref{eq:MLMT:Omgs:2} for $k$, we see that \eqref{eq:MLMT:Omgs:3} holds, i.e., $\widetilde{\Omg}_{(k+1)} \subseteq \widetilde{\Omg}_{(k)}$. In particular, note that the last inclusion in \eqref{eq:MLMT:Omgs:2} holds for $\widetilde{\Omg}_{(k+1)}$.

Finally, we need to verify that the first inclusion in \eqref{eq:MLMT:Omgs:2} holds for $k+1$. By \eqref{eq:MLMT:scaleDecrease}, it suffices to show that
\begin{equation} \label{eq:MLMT:Omgs:3'}
	\ECyl_{v_{(0)}} (\tht_{(k)}, \Xi_{(k)}; \, \Omg_{(k+1)} ) \subset \widetilde{\Omg}_{(k+1)}.
\end{equation}

Note that we use $v_{(0)}$ instead of $v_{(k)}$ on the left-hand side.
Applying \eqref{eq:ECylLCylinECyl} in Lemma \ref{lem:cylContainmentProps}, \eqref{eq:MLMT:Omgs:2} for $k$, and using the fact that $e^{5 \tht_{k} \nrm{\nb v_{(0)}}_{C^{0}}} \leq e^{\frac{1}{100}} \leq 2$ by \eqref{eq:MLMT:scaleDecrease}, the desired inclusion \eqref{eq:MLMT:Omgs:3'} follows.

\subsection{Verification of Claims \ref{claim:MLMT:zeroR} -- \ref{claim:MLMT:nontrivial}} \label{subsec:MLMT:verifyClaims}
Here, we complete the proof of Theorem \ref{thm:eulerOnRn:cptPert} by establishing the Claims \ref{claim:MLMT:zeroR}--\ref{claim:MLMT:nontrivial}, which were made in Subsection \ref{subsec:MLMT:reduceMT}. 

\begin{proof} [Proof of Claim \ref{claim:MLMT:zeroR}: Vanishing of the Euler-Reynolds stress] 
This claim is obvious from construction, since
\begin{equation*}
\nrm{R_{(k)}}_{C^{0}} \leq e_{R, (k)} \to 0 \hbox{ as } k \to \infty. \qedhere
\end{equation*}
\end{proof}

\begin{proof} [Proof of Claim \ref{claim:MLMT:cptSupp}: Compact support in space-time] 
Let $\Omg_{(\infty)} := \cup_{k=1}^{\infty} \Omg_{(k)}$. By construction, for every $k \geq 0$ we have
\begin{equation*}
\supp \, (v_{(k)} - v_{(0)}, \, p_{(k)} - p_{(0)}) \subseteq \Omg_{(k)}
\subseteq \Omg_{(\infty)}.
\end{equation*}
Note furthermore that $\overline{\Omg_{(\infty)}} \subseteq \overline{\widetilde{\Omg}_{(0)}} \subseteq \calU$, from which the claim follows. \qedhere
\end{proof}

\begin{proof} [Proof of Claim \ref{claim:MLMT:HolderReg}: H\"older regularity of the solution] 
Note that $v_{(K)} = v_{(0)} + \sum_{k=1}^{K} V_{(k)}$ and $p_{(K)} = p_{(0)} + \sum_{k=1}^{K} P_{(k)}$, where
\begin{align}
	\nrm{V_{(k)}}_{C^{0}} \leq & C e_{R, (k)}^{1/2} = C Z^{-(k+1)/2}, \label{eq:recallC0correcBdV}\\
	\nrm{P_{(k)}}_{C^{0}} \leq & C e_{R, (k)}^{1/2} = C Z^{-(k+1)}, \\
	\nrm{\nb_{t,x} V_{(k)}}_{C^{0}} \leq & C C_{0} N_{(k)} \Xi_{(k)} e_{R, (k)}^{1/2} = C C_{0}^{k+1} Z^{5(k+1)/2} Z^{-(k+1)/2}, \label{ineq:C1correcBdV} \\
	\nrm{\nb_{t,x} P_{(k)}}_{C^{0}} \leq & C C_{0} N_{(k)} \Xi_{(k)} e_{R, (k)} = C C_{0}^{k+1} Z^{5(k+1)/2} Z^{-(k+1)} , \label{ineq:C1correcBdP}
\end{align}
by the Main Lemma and the base case $e_{R, (0)} = Z^{-1}$.   The estimates \eqref{ineq:C1correcBdV}, \eqref{ineq:C1correcBdP} for the time derivative $\rd_{t}$ follow by writing \[ \rd_{t} = (\rd_{t} + v_{(k)} \cdot \nb) - v_{(k)} \cdot \nb = (\rd_{t} + v_{(k)} \cdot \nb) - (v_{(k)} - v_{(0)}) \cdot \nb - v_{(0)} \cdot \nb \]
 and noting that the advective derivative obeys an even more favorable estimate than needed, while the terms $(v_{(k)} - v_{(0)})$ and $v_{(0)}$ are bounded uniformly on $\Om_{(k)}$, independent of $k$.  The uniform boundedness of $(v_{(k)} - v_{(0)}) = \sum_{k' = 0}^{k-1} V_{(k')}$ follows by summing \eqref{eq:recallC0correcBdV} in $k'$.  Also, the $C^0$ norm of $v_{(0)}$ over $\Om_{(k)}$ is also bounded uniformly in $k$, as $v_{(0)}$ is smooth and the sets $\Omg_{(k)}$ are contained in a fixed compact set $\overline{\widetilde{\Omg}_{(0)}}$ by Claim \ref{claim:MLMT:cptSupp}.  Therefore the constants in \eqref{ineq:C1correcBdV}-\eqref{ineq:C1correcBdP} are independent of $k$.

By interpolation of \eqref{eq:recallC0correcBdV}-\eqref{ineq:C1correcBdP}, we obtain the following upper bounds on the $C^{\alp}_{t,x}$ norm of $V_{(k)}$ and $P_{(k)}$.
\begin{align} 
	\nrm{V_{(k)}}_{C^{\alp}_{t,x}} \leq & C C_{0}^{\alp(k+1)} Z^{\frac{5 \alp  - 1}{2} (k+1)} , \label{ineq:ctcxaVk} \\
	\nrm{P_{(k)}}_{C^{2 \alp}_{t,x}} \leq & C C_{0}^{2\alp(k+1)} Z^{(5 \alp  - 1)(k+1)} . \label{ineq:ctcxaPk}
\end{align}

Therefore, for $\alp = 1/5-\eps$, choosing $Z > 1$ sufficiently large so that
\begin{equation} 
	C_{0}^{2\alp} Z^{- 5 \eps} < 1, \label{eq:suffLargeZ}
\end{equation}
we see that the bounds \eqref{ineq:ctcxaVk}-\eqref{ineq:ctcxaPk} for $V_{(k)}$ and $P_{(k)}$ can be summed in a geometric series, and therefore $(v_{(k)}, p_{(k)})$ is Cauchy in $C^{\alp}_{t,x} \times C^{2 \alp}_{t,x}$. Moreover, taking $Z$ even larger, we can ensure that the sum
\begin{equation*}
	\sum_{k \geq 0} \nrm{V_{(k)}}_{C^{\alp}_{t,x}} + \nrm{P_{(k)}}_{C^{\alp}_{t,x}} 
\end{equation*}
is arbitrarily small, which proves \eqref{eq:MLMT:smallHolder}.
\end{proof}

\begin{proof} [Proof of Claim \ref{claim:MLMT:energyIncrease}: Increase of local energy] 
The proof below closely follows the argument of \cite[\S 11.2.7]{isett}. We begin by reducing our consideration to a specific $\psi$ for each $t_{\star} \in I[\Omg_{(0)}]$. Indeed, by Claim \ref{claim:MLMT:cptSupp} which has been already verified, the following statement holds: If $\psi, \psi'$ are two smooth, compactly supported, smooth function on $\bbR^{3}$ such that $\psi \equiv \psi'$ on $S_{t_{\star}}[\calU]$, then for every $k \geq 1$ we have
\begin{equation*}
	\int (\psi' - \psi)(x) \frac{\abs{v_{(k)}(t_{\star}, x)}^{2}}{2} \, \ud x = \int (\psi' - \psi)(x) \frac{\abs{v_{(0)}(t_{\star}, x)}^{2}}{2} \, \ud x.
\end{equation*}

Therefore, it suffices to verify \eqref{eq:energyIncrease} for a \emph{specific} $\psi_{t_{\star}}$ for each $t_{\star} \in I[\Omg_{(0)}]$. 
By the pre-compactness of $\Omg_{(0)}$ and $\calU$, there exists a smooth, compactly supported $\psi_{t_{\star}} = \psi_{t_{\star}}(x)$ for each $t_{\star} \in I[\Omg_{(0)}]$ so that $\psi_{t_{\star}} \equiv 1 \hbox{ on } S_{t_{\star}}[\calU]$ and
\begin{equation} \label{eq:MLMT:energyIncrease:pf:1}
	\sup_{t_{\star} \in I[\Omg_{(0)}]} \bb( \nrm{\psi_{t_{\star}}}_{L^{1}_{x}} + \nrm{\nb \psi_{t_{\star}}}_{L^{1}_{x}} \bb)  \leq C < \infty
\end{equation}
for some $C = C(\Omg_{(0)}, \calU)$.

We are now ready to prove \eqref{eq:energyIncrease}.  Here we will often omit the $x$ variable for functions $f(t_\star, x) = f(t_\star)$ depending on $x$.  Recalling that $v_{(k+1)} = v_{(k)} + V_{(k)}$, we compute
\begin{equation} \label{eq:MLMT:energyIncrease:pf:2}
\begin{aligned}
	& \int \psi_{t_{\star}} \frac{\abs{v_{(k+1)}(t_{\star})}^{2}}{2} \, \ud x - \int \psi_{t_{\star}} \frac{\abs{v_{(k)}(t_{\star})}^{2}}{2} \, \ud x \\
	& \quad = \int \psi_{t_{\star}} e_{(k)}(t_{\star}) \, \ud x + \int \psi_{t_{\star}} \bb( \frac{\abs{V_{(k)}(t_{\star})}^{2}}{2} - e_{(k)}(t_{\star}) \bb) \, \ud x + \int \psi_{t_{\star}} \, v_{(k)} \cdot V_{(k)}(t_{\star}) \, \ud x
\end{aligned}
\end{equation}

Given $t_{\star} \in I[\Omg_{(0)}]$, let $x_{\star}$ be a point in $\bbR^{3}$ such that $(t_{\star}, x_{\star}) \in \Omg_{(0)}$. From the construction, note that $e_{(0)}(t,x) \geq K e_{R, (0)}$ on $\ECyl_{v_{(0)}}(\tht_{(0)}, \Xi_{(0)}^{-1}; \, \Omg_{(0)} )$. In particular, we have
\begin{equation} \label{eq:MLMT:energyIncrease:pf:3}
	e_{(k)}(t_{\star}, x) \geq K e_{R, (k)} \hbox{ on } B(\Xi_{(0)}^{-1}; x_{\star} ),
\end{equation}
for $k = 0$. Next, again by construction in Subsection \ref{subsec:MLMT:construction}, note that $\ECyl_{v_{(0)}}(\tht_{(0)}, \Xi_{(0)}^{-1}; \, \Omg_{(0)} ) \subseteq \Omg_{(1)} \subseteq \Omg_{(k)}$ for every $k \geq 1$; therefore, \eqref{eq:MLMT:energyIncrease:pf:3} holds for $k \geq 1$ as well. Thus, we conclude that for every $k \geq 0$, we have
\begin{equation} \label{}
	\int \psi_{t_{\star}} e_{(k)}(t_{\star}) \, \ud x \geq c e_{R, (k)} .
\end{equation}
for some constant $c > 0$ which depends on $\Xi_{(0)}^{-1}$ and $K$, but does not depend on $k$, $Z$ or $t_{\star}$.

On the other hand, by the Main Lemma, we have the bound
\begin{equation} \label{}
\begin{aligned}
\bb\vert \int \psi_{t_{\star}} \bb(\frac{\abs{V_{(k)}(t_{\star})}^{2}}{2} - e_{(k)}(t_{\star}) \bb) \, \ud x \bb\vert 
\leq & C \frac{e_{v, (k)}^{1/2} e_{R, (k)}^{1/2}}{N_{(k)}} \bb( \nrm{\psi_{t_{\star}}}_{L^{1}_{x}} + \Xi_{(k)}^{-1} \nrm{\nb \psi_{t_{\star}}}_{L^{1}_{x}} \bb) \\
\leq & C Z^{-2} e_{R, (k)},
\end{aligned}
\end{equation}
where we used \eqref{eq:MLMT:energyIncrease:pf:1} and the fact that $\Xi_{(k)}^{-1} < 1$ on the last line. Next, we have
\begin{equation} \label{}
\begin{aligned}
\bb\vert \int \psi_{t_{\star}} \, v_{(k)} \cdot V_{(k)} (t_{\star}) \, \ud x \bb\vert
= & \bb\vert \int \psi_{t_{\star}} \, v_{(k)} \cdot \nb \times W_{(k)}(t_{\star}) \, \ud x \bb\vert \\
\leq & \bb\vert \int \psi_{t_{\star}} \, \nb \times v_{(k)} \cdot W_{(k)} (t_{\star}) \, \ud x \bb\vert 
	+ \bb\vert \int (\nb \psi_{t_{\star}} \times v_{(k)}) \cdot W_{(k)}(t_{\star}) \, \ud x \bb\vert.
\end{aligned}
\end{equation}

In this case, $\nb \psi_{t_{\star}} = 0$ on $\supp \, W_{(k)}(t_{\star})$ by hypothesis, and therefore the second term on the last line vanishes. Therefore, by \eqref{eq:MLMT:energyIncrease:pf:1}, we have
\begin{equation} \label{}
\bb\vert \int \psi_{t_{\star}} \, v_{(k)} \cdot V_{(k)} (t_{\star}) \, \ud x \bb\vert \leq C \frac{e_{v, (k)}^{1/2} e_{R, (k)}^{1/2}}{N_{(k)}} \nrm{\psi_{t_{\star}}}_{L^{1}_{x}}
\leq C Z^{-2} e_{R, (k)}.
\end{equation}

In conclusion, we have
\begin{equation} \label{}
\int \psi_{t_{\star}} \frac{\abs{v_{(k+1)}(t_{\star})}^{2}}{2} \, \ud x - \int \psi_{t_{\star}} \frac{\abs{v_{(k)}(t_{\star})}^{2}}{2} \, \ud x
\geq c e_{R, (k)} + C Z^{-2} e_{R, (k)}.
\end{equation}

Taking $Z$ sufficiently large, we obtain the desired claim. 
\end{proof}

\begin{proof} [Proof of Claim \ref{claim:MLMT:nontrivial}: Irregularity of the solution] 
The idea of the proof below is similar to that of Claim \ref{claim:MLMT:energyIncrease}. An important difference, however, is that not only do we take $Z \geq Z_{\star}$ for some large $Z_{\star} > 1$ (as in Claim \ref{claim:MLMT:energyIncrease}), but we also take $k \geq k_{\star}$ for a sufficiently large $k_{\star} \geq 0$. In this proof, we shall say that a constant is \emph{universal} if it is independent of the given $\rho_{\star}, t_{\star}, x_{\star}, v_{(0)}, \psi$ and $u$ in the hypotheses of Claim \ref{claim:MLMT:nontrivial}. A constant $C > 0$ that occurs below is always universal, unless otherwise stated. It is important to note that $Z_{\star}$ is also universal, whereas $k_{\star}$ is not.

Let $t_{\star}, x_{\star}, \rho_{\star}, \psi$ and $u$ be given as in the hypotheses of Claim \ref{claim:MLMT:nontrivial}. Let us assume that $u \in W^{1/5, 1}_{x}(B(\rho_{\star}; x_{\star}))$ since the proof in the case where $u \in C^{1/5}_{x}(B(\rho_{\star}; x_{\star}))$ is identical. Below, we shall use the shorthand $B := B(\rho_{\star}; x_{\star})$. 


As in the proof of Claim \ref{claim:MLMT:energyIncrease}, we begin by computing
\begin{equation} \label{eq:MLMT:nontrivial:pf:1}
\begin{aligned}
	&\int \psi \frac{\abs{(v_{(k+1)} - u)(t_{\star})}^{2}}{2} \, \ud x - \int \psi \frac{\abs{(v_{(k)} - u)(t_{\star})}^{2}}{2} \, \ud x  \\
	&\quad = \int \psi e_{(k)}(t_{\star}) \, \ud x + \int \psi \bb( \frac{\abs{V_{(k)}(t_{\star})}^{2}}{2} - e_{(k)}(t_{\star}) \bb) \, \ud x + \int \psi (v_{(k)} - u) \cdot V_{(k)} (t_{\star}) \, \ud x.
\end{aligned}\end{equation}

Since $\supp \psi \subseteq B \subseteq S_{t_{\star}} [\Omg_{(0)}] \subseteq S_{t_{\star}} [\Omg_{(k)}]$, we have by \eqref{eq:MLMT:chooseED:lower}
\begin{equation} 
	\int \psi e_{(k)} (t_{\star}) \, \ud x \geq K e_{R, (k)}.
\end{equation}

For the second term on the right-hand side of \eqref{eq:MLMT:nontrivial:pf:1}, we have
\begin{equation}  
\begin{aligned}
\bb\vert \int \psi \bb(\frac{\abs{V_{(k)}(t_{\star})}^{2}}{2} - e_{(k)}(t_{\star}) \bb) \, \ud x \bb\vert 
\leq & C \frac{e_{v, (k)}^{1/2} e_{R, (k)}^{1/2}}{N_{(k)}} \bb( \nrm{\psi}_{L^{1}_{x}} + \Xi_{(k)}^{-1} \nrm{\nb \psi}_{L^{1}_{x}} \bb) \\
\leq & C Z_{\star}^{-2} e_{R, (k)} + C \nrm{\nb \psi}_{L^{1}_{x}} Z_{\star}^{-2} \Xi_{(k)}^{-1}  e_{R, (k)}.
\end{aligned}
\end{equation}

To estimate the third term on the right-hand side of \eqref{eq:MLMT:nontrivial:pf:1}, we first write
\begin{equation} \label{eq:MLMT:nontrivial:pf:2}
\int \psi (v_{(k)} - u) \cdot V_{(k)} (t_{\star}) \, \ud x
=  \int \psi (v_{(k)} - u_{\eps}) \cdot V_{(k)} (t_{\star}) \, \ud x + \int \psi (u_{\eps} - u) \cdot V_{(k)} (t_{\star}) \, \ud x,
\end{equation}
where $u_{\eps} = \int u(x-y) \eta_{\eps}(y) \, \ud y$ is a mollification of $u$, $\eta_{\eps}(y) = \eps^{-3} \eta(y/\eps)$ and $\eta$ is a smooth compactly supported function such that $\int \eta = 1$. Here we have assumed that the $\ep$-neighborhood of the support of $\psi$ is contained in $B$, which will be true for sufficiently small $\ep$ chosen in the proof below.  For the last term on the right-hand side of \eqref{eq:MLMT:nontrivial:pf:2}, we estimate
\begin{equation*}
	\bb\vert \int \psi (u_{\eps} - u) \cdot V_{(k)}(t_{\star}) \, \ud x \bb\vert \leq C  \eps^{1/5} e_{R, (k)}^{-1/2} \nrm{\psi}_{C^{0}_{x}} \nrm{u}_{W^{1/5,1}_{x}} \, e_{R, (k)},
\end{equation*}
where we have used the elementary convolution estimate $\nrm{u_{\eps} - u}_{L^{1}_{x}} \leq C \eps^{1/5} \nrm{u}_{W^{1/5, 1}}$. 

Finally, we estimate the first term on the right-hand side of \eqref{eq:MLMT:nontrivial:pf:2}. Integrating by parts and using the triangle inequality, we may write
\begin{equation} 
\begin{aligned}
\bb\vert \int \psi \, (v_{(k)} - u_{\eps}) \cdot V_{(k)} (t_{\star}) \, \ud x \bb\vert 
\leq & \int \abs{\psi \, \nb \times v_{(k)} \cdot W_{(k)} (t_{\star})} \, \ud x 
	+ \int \abs{\psi \, \nb \times u_{\eps} \cdot W_{(k)} (t_{\star})} \, \ud x  \\
	& + \int \abs{(\nb \psi \times v_{(k)}) \cdot W_{(k)}(t_{\star})} \, \ud x 
	+ \int \abs{(\nb \psi \times u_{\eps}) \cdot W_{(k)}(t_{\star})} \, \ud x  \\
	=: & \mathrm{I}_{1} + \mathrm{I}_{2} + \mathrm{I}_{3} + \mathrm{I}_{4}.
\end{aligned}
\end{equation}

For $\mathrm{I}_{1}$, we estimate
\begin{equation}
	\mathrm{I}_{1} \leq C \frac{e_{v, (k)}^{1/2} e_{R, (k)}^{1/2}}{N_{(k)}} \nrm{\psi}_{L^{1}_{x}}
	\leq C Z_{\star}^{-2} e_{R, (k)}.
\end{equation}

We estimate $\mathrm{I}_{2}$ by
\begin{equation}
	\mathrm{I}_{2} 
	\leq C \nrm{\nb u_{\eps}}_{L^{1}_{x}} \frac{e_{R, (k)}^{1/2}}{\Xi_{(k)} N_{(k)}} \nrm{\psi}_{C^{0}_{x}} 
	\leq C \eps^{-4/5} N_{(k)}^{-1} \Xi_{(k)}^{-1} e_{R, (k)}^{-1/2} \nrm{\psi}_{C^{0}_{x}} \nrm{u}_{W^{1/5, 1}_{x}} \, e_{R, (k)},
\end{equation}
where we have used the convolution estimate $\nrm{\nb u_{\eps}}_{L^{1}_{x}} \leq C \eps^{-4/5} \nrm{u}_{W^{1/5, 1}_{x}}$. 

To estimate $\mathrm{I}_{3}$, we begin by noting that
\begin{equation*}
	\nrm{v_{(k)} - v_{(0)}}_{C^{0}(B)} 
	\leq \sum_{j=0}^{k-1} e^{-1/2}_{R, (j)}
	\leq (Z_{\star}^{1/2}-1)^{-1}.
\end{equation*}
Note also that $v_{(0)}$ is bounded on $B$, as it is smooth and $B$ is compact. Therefore, we have
\begin{equation}
\begin{aligned}
	\mathrm{I}_{3}
	\leq & C \nrm{\nb \psi}_{L^{1}_{x}} (\nrm{v_{(k)} - v_{(0)}}_{C^{0}(B)} 
							+ \nrm{v_{(0)}}_{C^{0}(B)}) \frac{e_{R, (k)}^{1/2}}{\Xi_{(k)} N_{(k)}} \\
	\leq & C \nrm{\nb \psi}_{L^{1}_{x}} \bb( (Z_{\star}^{1/2} - 1)^{-1} + \nrm{v_{(0)}}_{C^{0}(B)} \bb) \tht_{(k+1)} e_{R, (k)}.
\end{aligned}
\end{equation}

Finally, for $\mathrm{I}_{4}$, we have
\begin{equation}
	\mathrm{I}_{4}
	\leq C \nrm{\nb \psi}_{C^{0}_{x}} \nrm{u_{\eps}}_{L^{1}_{x}}  \frac{e_{R, (k)}^{1/2}}{\Xi_{(k)} N_{(k)}} 
	\leq C \nrm{\nb \psi}_{C^{0}_{x}} \nrm{u}_{L^{1}_{x}} \tht_{(k+1)} e_{R, (k)}
\end{equation}

Putting everything together, we arrive at
\begin{align*}
\eqref{eq:MLMT:nontrivial:pf:1}
>& K e_{R, (k)}
 - C Z_{\star}^{-2} e_{R, (k)} \\
& - C (\eps^{1/5} e_{R, (k)}^{-1/2} 
	+ \eps^{-4/5} N_{(k)}^{-1} \Xi_{(k)}^{-1} e_{R, (k)}^{-1/2})
		\nrm{\psi}_{C^{0}_{x}} \nrm{u}_{W^{1/5, 1}_{x}} \, e_{R, (k)}  \\
& - C_{\star} ( \Xi_{(k)}^{-1} + \tht_{(k+1)} ) e_{R, (k)} \\
=: & K e_{R, (k)} - E_{1} - E_{2} - E_{3}.
\end{align*}
where $C > 0$ is a universal constant and $C_{\star} > 0$ can depend on $\rho_{\star}$, $Z_{\star}$, $\Xi_{(0)}$, $\nrm{v_{(0)}}_{C^{0}(B)}$, $\nrm{\nb \psi}_{L^{1}_{x}}$, $\nrm{\nb \psi}_{C^{0}_{x}}$ and $\nrm{u}_{L^{1}_{x}}$. Taking $Z_{\star} \geq 2 (C/K)^{1/2}$, we have
\begin{equation*}
-E_{1} \geq - \frac{1}{4} K e_{R, (k)}. 
\end{equation*}

Next, choosing $\eps = N_{(k)}^{-1} \Xi_{(k)}^{-1}$ and recalling the evolution laws for parameters \eqref{eq:MLMT:parameterEvolution:Xi}--\eqref{eq:MLMT:parameterEvolution:eR} and \eqref{eq:MLMT:parameterEvolution:ML}, we see that
\begin{equation*}
-E_{2} \geq - C C_{0}^{-k/5} \Xi_{(0)}^{-1/5} \nrm{\psi}_{C^{0}_{x}} \nrm{u}_{W^{1/5, 1}_{x}} e_{R, (k)}.
\end{equation*}

At this point, observe that $C_{0}^{-k/5} \to 0$ as $k \to \infty$ (since $C_{0} > 1$), and also that $\Xi_{(k)}^{-1}$, $\tht_{(k+1)} \to 0$. Therefore, choosing $k \geq k_{\star}$ sufficiently large (but non-universal), we have
\begin{equation*}
-E_{2} - E_{3} \geq - \frac{1}{4} K e_{R, (k)}.
\end{equation*}

This bound concludes the proof. \qedhere
\end{proof}



%

%% file: prescribingEnergy3.tex
In this Section, we show how our method can be applied to produce solutions with a prescribed energy profile, and we present a proof of Theorem~\ref{thm:eulerOnRn:presEnergy}.  Here we outline our presentation of the proof.

Our construction of solutions in Theorem~\ref{thm:eulerOnRn:presEnergy} will require a modification of Lemma~\ref{lem:mainLemma} which can allow for an energy increment with slightly worse bounds on its advective derivative.  We state this modified Lemma in Section~\ref{sec:modMainLem} below, where we also indicate how the proof of Lemma~\ref{lem:mainLemma} can be adjusted to prove the modified Main Lemma.

To simplify our exposition, we start by proving a variant of Theorem~\ref{thm:eulerOnRn:presEnergy} in the periodic setting which illustrates the main ideas of our algorithm in the simplest case.  This proof is carried out in Section~\ref{sec:prescribePeriodicEnergy}.  In Section~\ref{sec:prescribe:Nonperiodic:Energy}, we then explain how the construction can be modified to handle the nonperiodic setting.

\subsection{A modified Main Lemma}\label{sec:modMainLem}
\input{modifiedMainLem}

\subsection{Prescribing the energy profile: the periodic setting }\label{sec:prescribePeriodicEnergy}
\input{prescribePeriodicEnergy3}

\subsection{Prescribing the energy profile: the nonperiodic setting }\label{sec:prescribe:Nonperiodic:Energy}
\input{prescribeNonperiodicEnergy3}

%% file: modifiedMainLem.tex
The proof of Theorem~\ref{thm:eulerOnRn:presEnergy} will rely on a modification of Lemma~\ref{lem:mainLemma}, which we present here.  The main difference in this modified lemma is that we allow for a worse bound on the advective derivative of the energy increment, but one which is still compatible with the time scale of the construction.  The price we pay is a worse bound for the accuracy with which the energy increment is prescribed.

\begin{lem}[The Modified Main Lemma] \label{lem:mod:mainLemma}
Suppose that $L \geq 2$.  Let $K$ be the constant in Section 7.3 of \cite{isett}, and let $M \geq 1$ be a constant.  
There exist constants $C_{0}, C > 1$, which depend only on $M$ and $L$, such that that following holds:

Let $(v,p,R)$ be any solution of the Euler-Reynolds system whose frequency and energy levels are below $(\Xi, e_v, e_R)$
 to order $L$ in $C^0$.

Define the time-scale $\th = \Xi^{-1} e_v^{-1/2}$, let $N$ be any positive number obeying the bound
$N \geq \left(\fr{e_v}{e_R} \right)^{3/2}$ and define the dimensionless parameter ${\bf b} = \left(\fr{e_v^{1/2}}{e_R^{1/2}N} \right)^{1/2}$.

Let $ e(t, x) : \R\times \R^3 \to \R_{\geq 0} $ be any non-negative function which satisfies the lower bound
\begin{align}
 e(t,x) \geq K e_R \quad \quad \mbox{ for all } (t,x) \in {\hat C}_{v}(\th, \Xi^{-1}; \supp R) \label{eq:lowBoundEoftxM}
\end{align}
(using the notation of Definition \ref{def:vCylinder}) and whose square root satisfies the estimates
\begin{align}
|| \nab^k (\pr_t + v \cdot \nab)^r e^{1/2} ||_{C^0} &\leq M \Xi^k ({\bf b}^{-1} \Xi e_v^{1/2})^r e_R^{1/2} & 0 \leq r \leq 1, \quad 0 \leq k + r \leq L \label{ineq:goodEnergyM}
\end{align}


Then there exists a solution $(v_1, p_1, R_1)$ of the Euler-Reynolds system of the form $v_1 = v + V$, $p_1 = p + P$ such that the frequency and energy levels of $(v_1, p_1)$ are below
\begin{align}
 (\Xi', e_{v}', e_{R}') 
= (C_{0} N \Xi, e_R, {\bf b}^{-1} \fr{e_v^{1/2} e_R^{1/2}}{N}) \label{eq:theNewEnergyLevelM}
\end{align}
to order $L$ in $C^0$ and such that the following are satisfied:
\begin{enumerate}
\item The estimates \eqref{eq:Vco}-\eqref{eq:matWco} for the correction $V = \nab \times W$ hold as stated in Lemma~\ref{lem:mainLemma}.
\item The estimates \eqref{eq:Pco}-\eqref{eq:matPco} for the pressure correction $P$ hold as in Lemma~\ref{lem:mainLemma}.
\item The estimates \eqref{eq:Wco} and \eqref{eq:energyPrescribed} hold for $\co{W}$ and the error in prescribing the energy increment.  Moreover, the constants in these estimates can be made arbitrarily small at the price of increasing the constants $C, C_0$ in the other bounds.  In particular, the constants in \eqref{eq:Wco} and \eqref{eq:energyPrescribed} may be set equal to $1$.
\item The bounds \eqref{eq:goalForR1supp} and \eqref{eq:goalForVPsupp} on the supports of $R_1, V$ and $P$ hold as in Lemma~\ref{lem:mainLemma}.  Moreover, the enlargement of the support in time is slightly smaller than stated there, and we have 
\begin{align}
 \supp R_1 \cup \supp V \cup \supp P \subseteq \ECyl_{v}({\bf b} \th, \Xi^{-1}; \supp e) \label{eq:goalForAllSuppM}
\end{align}
\item The bound \eqref{eq:energyPrescribed} for prescribing the energy increment in Lemma~\ref{lem:mainLemma} is replaced by the bound
\begin{align}
\left| \int_{\R^3} |V|^2(t,x) \psi(x) dx - \int_{\R^3} e(t,x) \psi(x) dx \right| &\leq {\bf b}^{-1} \fr{e_v^{1/2} e_R^{1/2}}{N} \left( \| \psi \|_{L^1} + \Xi^{-1} \| \nab \psi \|_{L^1} \right) \label{eq:energyPrescribedM}
\end{align}
which holds uniformly in $t$ for all $\psi(x) \in C_c^\infty(\R^3)$.
\end{enumerate}
\end{lem}

The proof of Lemma~\ref{lem:mod:mainLemma} is identical to the proof of Lemma~\ref{lem:mainLemma} in the sense that every choice of parameter in the argument is left unchanged.  The only differences in the proof are due to the inferior bound \eqref{ineq:goodEnergyM} on the advective derivative of the energy increment, which leads to worse estimates for a few terms in the argument that we will list here.  Ultimately, the reason we are allowed to relax the bound on the advective derivative is that the cost of the advective derivative in \eqref{ineq:goodEnergyM} coincides with the inverse of the time scale in the construction:
\ali{
 \tau^{-1} &\sim {\bf b}^{-1} \Xi e_v^{1/2} 
}
In particular, there is no room here to allow for a bound which is any worse than \eqref{ineq:goodEnergyM} without losing regularity.  It is therefore necessary to check a few of the estimates to make sure that the proof goes through with straightforward modifications.  Here we list the necessary modifications in the proof.
\begin{itemize}
\item All choices of parameters in the construction ($\ep_v, \ep_x, \ep_t, \tau, \rho, \la$) are exactly the same.  In particular, we have $\tau = a {\bf b} \Xi^{-1} e_v^{-1/2}$ for some constant $a$ chosen sufficiently small.
\item The fact that the bound \eqref{eq:goalForAllSuppM} on the support of the iteration gains a factor of ${\bf b}$ can be observed from inspecting the bound \eqref{eq:bigger:cyl:supp} on the support of the stress.  The time scale in this estimate is bounded by, say, $3 \tau$, which is smaller than ${\bf b} \Xi^{-1} e_v^{-1/2}$ when the small constant $a$ in the definition of $\tau$ is chosen appropriately.
\item The constants in the estimates \eqref{eq:Wco} and \eqref{eq:energyPrescribed} can be made arbitrarily small by taking the constant $B_\la$ in the construction to be sufficiently large when these terms are estimated\footnote{In fact, this improvement also follows already from the original statement of Lemma~\ref{lem:mainLemma}.  One can simply apply the same Lemma with $N'$ equal to a large multiple of $N$.  The estimates \eqref{eq:Wco} and \eqref{eq:energyPrescribed} then improve, as they are homogeneous of degree $-1$ in $N$, while the constants $C_0$ and $C$ may increase by a factor of a constant.}.
\item The choice of $\ep_t = c N^{-1} \Xi^{-1} e_R^{-1/2}$ for the time scale for the mollification along the flow made in \eqref{eq:RmollParamChoice} leads to a worse estimate on the error $\| e^{1/2} - \tilde{e}^{1/2} \|$ made in mollifying the energy increment.  Namely, the bound \eqref{eq:enIntError} loses a factor of ${\bf b}^{-1}$, and is replaced instead by
\ali{
\co{e^{1/2} - \tilde{e}^{1/2} } &\leq {\bf b}^{-1} \fr{e_v^{1/2}}{100 N} \label{eq:enIntErrorM}
}
\item The loss of ${\bf b}^{-1}$ in \eqref{eq:enIntErrorM} ultimately leads to the loss of ${\bf b}^{-1}$ in \eqref{eq:energyPrescribedM} when bounding the error for prescribing the energy increment.  Namely, this estimate introduces a ${\bf b}^{-1}$ in the estimate of line \eqref{eq:prescribeApproxEnInc}.
\item The bound on the first advective derivative of $\tilde{e}^{1/2}$ also worsens by a factor of ${\bf b}^{-1}$.  As a result, the estimates in Proposition~\ref{prop:eRepEstimates} incur a loss of ${\bf b}^{-1}$, and the estimate \eqref{eq:stressEnergyBound} must be replaced by
\ali{
e_R^{1/2} \co{ D^{(a,r)} \tilde{e}^{1/2} } + \co{ D^{(a,r)} R_\ep } &\leq C_a \Xi^a e_R ({\bf b}^{-1} \Xi e_v^{1/2})^{(r \geq 1)} ( N \Xi e_R^{1/2} )^{(r \geq 2)} N^{(a + 1 - L)_+/L}  \label{eq:stressEnergyBoundM}
}
Here again we use the notation introduced in Proposition~\ref{prop:eRepEstimates} for the indicator functions $(r \geq 1)$ and $(r \geq 2)$.  Taking the second advective derivative incurs the same cost of $\ep_t^{-1} = N \Xi e_R^{1/2}$ as in \eqref{eq:stressEnergyBound}, since this estimate arises from differentiating the kernel used to mollify in time along the flow.
\item The inferior bound in \eqref{eq:stressEnergyBoundM} affects the bounds for the advective derivatives of the amplitudes $v_I$ stated in Proposition~\ref{prop:ampEstimates}.  These bounds take on the same pattern as the estimate \eqref{eq:stressEnergyBoundM}, as now all of the worst terms occur when the advective derivatives fall on the factor of $\tilde{e}^{1/2}$.  In particular, the first advective derivative incurs the same cost of $\tau^{-1} \sim {\bf b}^{-1} \Xi e_v^{1/2}$, but now the second advective derivative gives a greater cost of $(N \Xi e_R^{1/2})$.  The estimates which replace \eqref{eq:amplitudeEstimates}-\eqref{eq:tinyCorrectionEstimates} are:
\ali{
 \co{D^{(a,r)} v_I } &\leq C_a \Xi^a e_R^{1/2} \tau^{-(r\geq 1)} (N \Xi e_R^{1/2})^{(r \geq 2)} N^{(a + 1 - L)_+/L}  \label{eq:amplitudeEstimatesM} \\
\co{D^{(a,r)} \de v_I } &\leq C_a B_\la^{-1} N^{-1} \Xi^a e_R^{1/2} \tau^{-(r\geq 1)} (N \Xi e_R^{1/2})^{(r \geq 2)} N^{(a + 2 - L)_+/L} \label{eq:tinyCorrectionEstimatesM}
}
Note that the only difference compared to \eqref{eq:amplitudeEstimates}-\eqref{eq:tinyCorrectionEstimates} lies in the bound on the second advective derivative.
\item The only point in the argument at which the second advective derivative estimate is used comes in estimating the advective derivative of transport term $Q_T^{jl}$.  The main term in $Q_T^{jl}$ is given by the parametrix in Section~\ref{sec:applyParametrix}
\ALI{
Q_T^{jl} &= \sum_I \fr{1}{i \la} e^{i \la \xi_I} q(\nab \xi_I)[(\pr_t + v_\ep^j \pr_j) v_I^l] + \tx{ Lower order terms}
}
As in Section~\ref{sec:stressEstimates}, we must estimate the cost of taking an advective derivative for this term.  According to the estimate \eqref{eq:amplitudeEstimatesM}, the cost of taking a further advective derivative is no longer $\tau^{-1}$ as before, but instead is given by $\ep_t^{-1} = c^{-1} N \Xi e_R^{1/2}$.  This cost is exactly the cost of $\left|\Ddt\right| \unlhd \Xi' (e_v')^{1/2} = C N \Xi e_R^{1/2}$ which we are required to verify for the advective derivative in order to conclude the proof of Lemma~\ref{lem:mod:mainLemma}.
\end{itemize}

Having explained the modifications necessary to prove Lemma~\ref{lem:mod:mainLemma}, we now explain how Lemma~\ref{lem:mod:mainLemma} can be applied to establish Theorem~\ref{thm:eulerOnRn:presEnergy}.

%% file: prescribePeriodicEnergy3.tex
In this Section, we establish a simplified version of Theorem~\ref{thm:eulerOnRn:presEnergy} on prescribing the energy profile of solutions by repeated application of Lemma~\ref{lem:mod:mainLemma}.  For the purpose of exposition, we consider first the problem of prescribing the energy profile for a solution which is periodic in the spatial variables.  By restricting to the periodic setting, we can demonstrate the main ideas in our algorithm for prescribing the energy profile while avoiding some technical details which come into play in the nonperiodic setting.  In Section~\ref{sec:prescribe:Nonperiodic:Energy} below we explain how to modify the argument to prescribe the energy profile in the nonperiodic setting.  

The construction explained in this section involves the introduction of several parameters which must be chosen in a particular logical order.  We provide a summary of this construction and the logical structure of the choice of parameters in Section~\ref{eq:onOrderOfParameters} below.

The theorem we establish in this Section is the following:
\begin{thm}[Periodic Euler flows with prescribed energy profile] \label{thm:eulerOnRn:period:presEnergy}
Let $\a < \a^* \leq 1/5$ and let $I \subseteq \R$ be a bounded open interval.  Let $\bar{e}(t) \geq 0$ be any non-negative function with compact support in $I$ which belongs to the class $\bar{e}(t) \in C_t^{\ga}$ for $\ga = \fr{2 \a^*}{1 - \a^*}$.  Then there exists a weak solution $(v, p)$ to the Euler equations in the class $v \in C_{t,x}^\a(\R \times \T^3)$ with compact support in $I \times \T^3$ such that the energy profile of $v$ is given by
\ali{
\int_{\T^3} |v|^2(t,x) dx &= \bar{e}(t), \qquad t \in \R
}
%
\end{thm}
We will give the proof for the case $\a^* = 1/5$ (in which case $\ga = 1/2$), since this case is most closely related to the proof of Theorem~\ref{thm:eulerOnRn:cptPert}, and the cases $\a^* < 1/5$ can be handled similarly.  We will outline how to handle the more general case in Section~\ref{sec:remarksOnThm} below, where we will also explain how to obtain a one parameter family of solutions tending to $0$ in $C_{t,x}^\a$ as in the statement of Theorem~\ref{thm:eulerOnRn:presEnergy}.

Theorem~\ref{thm:eulerOnRn:period:presEnergy} is proved by iterating Lemma~\ref{lem:mod:mainLemma} in a similar manner to the proof of Theorem~\ref{thm:eulerOnRn:cptPert} in Section~\ref{sec:mainLemImpliesMainThm}.  We remark first of all that Lemma~\ref{lem:mod:mainLemma} holds as stated in the case of $\T^3$ after very slightly modifying the proof to ensure that all the velocity fields involved are periodic in space\footnote{The main modification here is to ensure that the partition of unity defined in \eqref{eq:spacePartition} is a partition of unity of $1$ on the torus.  A partition of unity of this type is already present in \cite{isett}.}.  As in Section~\ref{sec:mainLemImpliesMainThm}, the solution $(v,p)$ stated in Theorem~\ref{thm:eulerOnRn:period:presEnergy} will be obtained as a uniform limit of a sequence of solutions $(v_{(k)}, p_{(k)}, R_{(k)})$ to the Euler-Reynolds equations, beginning with the trivial solution $(0, 0, 0)$.

The sequence of Euler-Reynolds flows $(v_{(k)}, p_{(k)}, R_{(k)})$ will have frequency-energy levels below certain values $(\Xi_{(k)}, e_{v,(k)}, e_{R,(k)})$ which are chosen to satisfy iteration rules of the form
\begin{align} 
		\Xi_{(k+1)} = \, & C_{0} Z^{5/2} \Xi_{(k)}		\label{eq:MLMT:parameterEvolution:Xi:M} \\
	e_{v, (k+1)} = \, & e_{R, (k)} 					\label{eq:MLMT:parameterEvolution:ev:M}\\
	e_{R, (k+1)} = \, & \frac{e_{R, (k)}}{Z}		\label{eq:MLMT:parameterEvolution:eR:M} 
\end{align}
just as in \eqref{eq:MLMT:parameterEvolution:Xi}-\eqref{eq:MLMT:parameterEvolution:eR}. These solutions are obtained by repeatedly applying Lemma~\ref{lem:mod:mainLemma} with a choice of $N = Z^{5/2}$ and a parameter $M$ which will be specified later.  Recall from Section~\ref{subsec:MLMT:verifyClaims} that, for $\a < 1/5$ and $Z$ sufficiently large depending on $\a$ and the constants in the statement of Lemma~\ref{lem:mod:mainLemma}, this choice of parameters leads to convergence of $(v_{(k)}, p_{(k)})$ in the $C_{t,x}^\a \times C_{t,x}^{2 \a}$ norm to a weak solution to the Euler equations.  The choice of a large parameter $Z$ will correspond in the context of the construction to a choice of a rapid frequency and time scale in the first stage of the iteration, and also to a large ratio between consecutive frequencies in the iteration.  We remark that the power $Z^{5/2}$ appearing in the iteration rules \eqref{eq:MLMT:parameterEvolution:Xi:M}-\eqref{eq:MLMT:parameterEvolution:eR:M} corresponds to taking $\alp^{*} = 1/5$.  

In specifying the construction, it will be helpful to introduce the parameters
\ali{
\label{eq:theTimeScales}
\begin{split}
{\bf b} = \left(\fr{e_v^{1/2}}{e_R^{1/2} N} \right)_{(k)}^{1/2} &= Z^{-1} \\ 
\th_{(k)} = \Xi_{(k)}^{-1} e_{v,(k)}^{-1/2}, \qquad &\hat{\tau}_k = {\bf b} \Xi_{(k)}^{-1} e_{v,(k)}^{-1/2} 
\end{split}
}
in order to distinguish the important time scales in the iteration.  The time scale $\th_{(k)}$ corresponds to the natural time scale of motion for the flow of the velocity field $v_{(k)}$, whereas the time scale $\hat{\tau}_k$ corresponds up to a constant to the more rapid time scale employed in the time cutoffs of the construction.

Along the way, we will also keep track of the time supports of the solutions and errors, by defining a sequence of sets $I_{(k)} \subseteq I$ such that the following claims hold

\begin{claim}[Growing Supports] \label{claim:growingSupp} For all $k \geq 0$, we have
\ali{
 \label{eq:supp-in-I}
	\supp v_{(k)} \cup \supp p_{(k)} \cup \supp R_{(k)} &\subseteq I_{(k)} \times \bbT^{3}.  \\
	I_{(k)} &\subseteq I_{(k+1)} \label{subset:growing}
}
\end{claim}


To fully specify the iteration, we must construct the functions $e_{(k)}(t,x)$ which prescribe the energy increment at each stage of the iteration, and specify the parameters of the construction, including the initial frequency energy levels $(\Xi_{(0)}, e_{v,(0)}, e_{R,(0)})$ for the base case of the iteration and the parameter $Z$.  Since we are considering the case of spatially periodic solutions, we may consider energy increments $e_{(k)}(t)$ which depend only on the variable $t$.  

Our goal in choosing the energy increments is to ensure that the solution $(v,p)$ constructed in the limit has energy profile given by $\bar{e}(t)$.  Since the approximate solutions $v_{(k)}(t)$ converge uniformly to the limiting solution, it suffices to show that the energy profiles of the approximate solutions
\ali{
E_k(t) := \int_{\T^3} |v_{(k)}|^2(t,x) dx
}
converge pointwise to the desired energy profile $\bar{e}(t)$ as $k \to \infty$.  This convergence will be obtained by ensuring that the inductive Claims~\ref{claim:room:to:add} and \ref{claim:nearly:saturated} below hold throughout the iteration.  In order to state these claims, we introduce the following notation, which will be used in the remainder of the paper.
\begin{defn}
Given any set $J \subseteq \bbR$ and any $\bar{\tau} \in \R_{\geq 0}$, we define
\begin{equation}
	I(\overline{\tau} ; J) := \set{ t + \Dlt t \in \bbR : t \in J, \ \abs{\Dlt t} \leq \overline{\tau}}.
\end{equation}
\end{defn}

\begin{claim}[There is Always Room to Add More Energy where the Error is Supported] \label{claim:room:to:add}
For every $k \geq 0$ and $t \in \bbR$, we have
\begin{equation} \label{eq:below-bar-e}
	\bar{e}(t) - E_{k}(t) \geq 0.
\end{equation}
Moreover, for all $t \in I(2 \hat{\tau}_{k}; I_{(k)})$, we have
\begin{equation} \label{eq:needed:room}
	\bar{e}(t) - E_{k}(t) \geq 3 K e_{R, (k)}
\end{equation}
where $K$ is the constant in the lower bound \eqref{eq:lowBoundEoftx} of Lemmas~\ref{lem:mainLemma} and \ref{lem:mod:mainLemma}.
%
\end{claim}
\begin{claim}[The Energy Threshold is Nearly Saturated]\label{claim:nearly:saturated}
There is an absolute constant $\overline{M}$ such that the upper bound
\ali{
 \sup_t | \bar{e}(t) - E_k(t) | &\leq \overline{M} e_{R,(k)} \label{eq:nearSaturated}
}
holds uniformly.
\end{claim}
Note that Claim~\ref{claim:nearly:saturated} implies the uniform convergence of $E_k(t) \to \bar{e}(t)$ as $k \to \infty$.  The condition \eqref{eq:needed:room} is required for continuing the iteration; this condition is present to ensure that one can construct an energy increment compatible with the conditions \eqref{eq:lowBoundEoftxM} and \eqref{ineq:goodEnergyM} while condition \eqref{eq:nearSaturated} is maintained.  Thus, the proof of Theorem~\ref{thm:eulerOnRn:period:presEnergy} reduces to specifying an iteration in which Claims~\ref{claim:room:to:add}-\ref{claim:nearly:saturated} remain satisfied\footnote{We point out that the Claims \ref{claim:room:to:add}-\ref{claim:nearly:saturated} are also carried along the iteration in the schemes for prescribing energy in \cite{deLSzeCts,deLSzeHoldCts,deLSzeBuck}.  The difference in this respect is that these papers assume a strictly positive lower bound on $\bar{e}(t)$, and the lower bound \eqref{eq:needed:room} is assumed to hold everywhere.}.  

We now explain our rule for specifying the iteration.  Our construction will involve choosing two large constants ($Y$ and $Z$), which will be chosen in alphabetical order during the course of the proof.  First, we define a sequence of ``gaps'' 
\ali{
\De E_k = Y e_{R, (k+1)} = Y \fr{e_{R,(k)}}{Z} \label{eq:DeEkDef}
}
with $Y$ some constant to be chosen later on.  Our goal is to choose an energy increment $e_{(k)}(t)$ which ensures that the number $\De E_k$ is a lower bound for the gap in the energy profile $(\bar{e}(t) - E_{k+1}(t)) \geq \De E_k$ which remains after stage $k$ of the iteration on the support of the error.  

According to conditions \eqref{eq:lowBoundEoftxM}, \eqref{ineq:goodEnergyM} and \eqref{eq:nearSaturated}, we should choose at each stage an energy increment $e_{(k)}(t)$ which will nearly saturate the energy threshold, but we must leave a gap of size $(\bar{e} - e_{(k)} - E_k) \sim e_{R,(k+1)}$ on the support of the error to ensure we can continue the iteration in the next stage.  A sensible first guess for the energy increment we desire would be the function
\ali{ \label{eq:desired:inc}
\hat{e}_{(k)}(t) &= (\bar{e}(t) - E_k(t) - \De E_k )_+ 
} 
(Recall that we use the notation $y_+ = \max \{y,0\}$.)  The key calculation which motivates this choice is given in Section~\ref{sec:verifyLastClaim} below. 

The problem with this guess is that the function $\hat{e}_{(k)}$  is only Lipschitz, whereas Lemma~\ref{lem:mod:mainLemma} requires control over derivatives of the square root $e_{(k)}^{1/2}(t)$.  We therefore modify the function $\hat{e}_{(k)}^{1/2}$ by prescribing an energy profile of the form
\ali{
e_{(k)}^{1/2}(t) &= (\bar{e} - E_k - \De E_k)_+^{1/2} \ast \eta_{\bMinus \hat{\tau}_k} =  \hat{e}_{(k)}^{1/2} \ast \eta_{\bMinus \hat{\tau}_k}  \label{eq:mollified:increm}
}
The function $\eta_{\bMinus \hat{\tau}_k}$ here denotes a standard, non-negative mollifying kernel in the time variable with support in the interval $\supp \eta_{\bMinus \hat{\tau}_k} \subseteq \{ |t| \leq \bMinus \hat{\tau}_k \}$.  The number $\hat{\tau}_k$ is the timescale of the construction defined in \eqref{eq:theTimeScales}.  For intuition, one can picture the formula~\eqref{eq:mollified:increm} graphically in the case $E_k = 0$ as shifting the graph of $\bar{e}(t)$ downwards by an amount $\De E_k$, taking a square root and then averaging over translates in $t$ by a width less than $\tau_k$. 


We will see during the course of the proof that the $C_t^\ga$ regularity of $\bar{e}$ will be essential for verifying that the assumptions \eqref{eq:lowBoundEoftxM}, \eqref{ineq:goodEnergyM} can be carried on during the iteration.  Without sufficient regularity for the function $\bar{e}$, the regularized function $e_{(k)}$ may be a poor approximation to the desired energy increment \eqref{eq:desired:inc}.  One must also worry that the time mollification in \eqref{eq:mollified:increm} may cause the energy profile of the approximate solutions to exceed the energy threshold if the regularity of $\bar{e}$ is too low. 

\subsection{Prescribing the energy increment: The Base Case}\label{sec:baseEnInc}
\input{prescribeEnBase3}

\subsection{Prescribing the energy increment: Admissibility of the energy function} \label{sec:admissEnergy}
\input{prescribeEnAdmiss3}

\subsection{Verification of Claims~\ref{claim:growingSupp}-\ref{claim:nearly:saturated}} \label{sec:verifyClaims}
\input{prescribeEnClaims3}

\subsection{Proof of Lemma~\ref{lem:commutator:energy:lem}} \label{sec:commutEstimate}
\input{prescribeEnCommutator3}

\subsection{Proof of Theorem~\ref{thm:eulerOnRn:period:presEnergy}: Summary }\label{eq:onOrderOfParameters}
\input{logicExplained}

\subsection{Extending Theorem~\ref{thm:eulerOnRn:period:presEnergy} to Theorem~\ref{thm:eulerOnRn:presEnergy}} \label{sec:remarksOnThm}
\input{prescribeEnRemarks4}

%% file: prescribeEnBase3.tex
We initialize the construction by taking our Euler-Reynolds flow to be $(v_{(0)}, p_{(0)}, R_{(0)}) = (0,0,0)$.  For the initial set of times containing the support of the iteration, we take $I_{(0)} = \emptyset$ to be the empty set.
We must now choose the initial frequency energy levels $(\Xi_{(0)}, e_{v,(0)}, e_{R,(0)})$.  For the initial energy level $e_{R,(0)}$ we take $e_{R,(0)} = \sup_{t \in \R} \bar{e}(t)$. 
This choice and the Ansatz \eqref{eq:MLMT:parameterEvolution:ev:M}-\eqref{eq:MLMT:parameterEvolution:eR:M} dictate our choice of $e_{v,(0)} = Z e_{R,(0)}$.  Observe that these choices guarantee that Claims~\ref{claim:room:to:add}-\ref{claim:nearly:saturated} and the containment \eqref{eq:supp-in-I} hold at the stage $k = 0$.

During the course of the iteration (see Line~\eqref{eq:QkBoundNeed} below), we will have to show that the quotient 
\[ Q_{(k)} = \fr{|\hat{\tau}_{(k)}|^{\ga}}{e_{R,(k+1)}} \]
remains uniformly bounded, independent of the choices of $Y$ and $Z$. We remark that this point is the reason for the numerology $\gmm = \frac{2 \alp^{*}}{1-\alp^{*}}$.  With this motivation, we choose $\Xi_{(0)}$ to achieve the inequality $Q_{(0)} \unlhd 1$.  Recall from \eqref{eq:theTimeScales} that $\hat{\tau}_{(0)} = {\bf b} \Xi_{(0)}^{-1} e_{v,(0)}^{-1/2} = Z^{-3/2} \Xi_{(0)}^{-1} e_{R,(0)}^{-1/2}$ and $e_{R,(1)} = \fr{1}{Z} e_{R,(0)}$.  The goal $Q_{(0)} \leq 1$ is therefore accomplished by choosing a value $\Xi_{(0)}$ such that
\ali{
\Xi_{(0)} &\geq \left(Z^{1 - \fr{3 \ga}{2}} e_{R,(0)}^{-1 - \fr{\ga}{2}} \right)^{1/\ga} \label{eq:choiceOfZeroFreq}
}
We have now specified the initial frequency energy levels (up to the specification of $Z$), but we are not quite ready to proceed with the iteration by applying Lemma~\ref{lem:mod:mainLemma}.  Namely, we want to apply Lemma~\ref{lem:mod:mainLemma} with the choice of energy increment $e_{(k)}(t)$ defined by \eqref{eq:mollified:increm}.  However, in order to apply Lemma~\ref{lem:mod:mainLemma}, we are required to specify the constant $M$ in the upper bounds \eqref{ineq:goodEnergyM} for the energy profile, which turns out to depend on the choice of $Y$.  Once we have determined the value of $Y$ (which is accomplished in Section~\ref{sec:verifyClaims} below), we will be able to apply Lemma~\ref{lem:mod:mainLemma} for a specified value of $M$.  In particular, the constant $C_0$ in the iteration rule \eqref{eq:MLMT:parameterEvolution:Xi:M} comes from Lemma~\ref{lem:mod:mainLemma} and depends on the value of $Y$.


In Section~\ref{sec:admissEnergy} below, we continue the proof of Theorem~\ref{thm:eulerOnRn:period:presEnergy} by verifying that our choice of energy increment $e_{(k)}(t)$ defined by \eqref{eq:mollified:increm} remains for all indices $k$ an admissible choice of energy function in Lemma~\ref{lem:mod:mainLemma} for the sequence of frequency energy levels dictated by \eqref{eq:MLMT:parameterEvolution:Xi:M}-\eqref{eq:MLMT:parameterEvolution:eR:M}.  In the process, we specify the sequence $I_{(k)}$ and verify that Claims~\ref{claim:room:to:add}-\ref{claim:nearly:saturated} hold with this choice of energy increment $e_{(k)}(t)$, provided the function ${\bar e}(t)$ has sufficient regularity.

%% file: prescribeEnAdmiss3.tex
In this Section, we define the sequence $I_{(k)}$, and we verify that the energy function defined in \eqref{eq:mollified:increm} is always an admissible choice of energy function for applying Lemma~\ref{lem:mod:mainLemma}.  

For $k \geq 0$, we define $I_{(k+1)}$ to be\footnote{We remark that in principle the set $I_{(k+1)}$ is allowed to be empty.}
\begin{equation} \label{eq:I-k-def}
	I_{(k+1)} := I(2 \hat{\tau}_{k}; \set{t \in \bbR : \bar{e}(t) - E_{k}(t) \geq \Dlt E_{k}} ).
\end{equation}
In what follows, we will assume that the constant $Y$ has already been chosen so that Claim~\ref{claim:room:to:add} holds, and also that Claim~\ref{claim:nearly:saturated} is satisfied for a specified constant $\overline{M}$.    


With the assumptions that $Y$ has already been chosen and that Claim~\ref{claim:nearly:saturated} is satisfied for a specified constant $\overline{M}$, we obtain the following bounds on the square root of the energy increment:
\ali{
 \| \left(\fr{d}{dt}\right)^r e_{(k)}^{1/2} \|_{C^0_t} &\leq A  (\bMinus {\bf b}^{-1} \Xi e_v^{1/2})^r [ \overline{M} e_{R,(k)} ]^{1/2}, \qquad 0 \leq r \leq 2 \label{eq:ek:has:bound}
}
Indeed, at the level $r = 0$, we have 
\ALI{
\co{ e_{(k)}^{1/2} } &\leq \co{ (\bar{e} - E_k - \De E_k)_+^{1/2} } 
 \leq \co{ \bar{e} - E_k }^{1/2} \leq  \overline{M}^{1/2} e_{R, (k)}^{1/2}
}
using our induction hypothesis Claim~\ref{eq:nearSaturated}.  The estimates for higher derivatives follow by differentiating the mollifier in the definition \eqref{eq:mollified:increm} of $e_{(k)}^{1/2}$.

Here $A$ is some absolute constant, but we have not yet specified $\overline{M}$, which will turn out to depend on our choice of $Y$. We postpone these choices for Section~\ref{sec:verifyClaims}.  

The estimate \eqref{eq:ek:has:bound} specifies the value of $M = A \overline{M}$ with which we may apply Lemma~\ref{lem:mod:mainLemma}.  To conclude that the function $e_{(k)}^{1/2}$ is admissible, we must also verify the lower bound \eqref{eq:lowBoundEoftxM} on the set $t \in I(\th_{(k)}; I_{(k)})$.

For stage $k = 0$, the lower bound \eqref{eq:lowBoundEoftxM} is vacuous.  To see that the lower bound holds at later stages, we first establish a lower bound for the function $(\bar{e} - E_k - \De E_k)_+$ on a slightly larger set of times.  Namely, for any $t' \in I(2 \th_{(k)}; I_{(k)})$, we have a lower bound
\[ \bar{e}(t') - E_k(t') \geq 3 K e_{R,(k)}\] 
by Claim~\ref{claim:room:to:add}.  We now impose the requirement 
\ali{
 Z &\geq 2 Y  \label{ineq:requireZY1}
}
to ensure the bound  $\De E_k = \fr{Y}{Z} e_{R,(k)} \leq \fr{1}{2} K e_{R,(k)}$ for $\De E_k$ defined in \eqref{eq:DeEkDef}.  We now have that 
\ali{
(e(t') - E_k(t') - \De E_k)_+ = (e(t') - E_k(t') - \De E_k) \geq 2 K e_{R,(k)}
}
for all $t' \in I_{(k)} \pm 2\th_{(k)} $.  As a consequence, we have for $t' \in I(2 \th_{(k)}; I_{(k)})$
\ali{
(e(t') - E_k(t') - \De E_k)_+^{1/2} \geq (2 K e_{R,(k)})^{1/2}.
}
From this lower bound, we obtain the desired lower bound 
\begin{equation} \label{eq:ek:has:lb}
e_{(k)}^{1/2}(t) \geq (2 K e_{R,(k)})^{1/2}, \qquad 
 t \in I(\th_{(k)}; I_{(k)}) 
\end{equation}
for the function $e_{(k)}^{1/2} = (e - E_k - \De E_k)_+^{1/2} \ast \eta_{\bMinus \hat{\tau}_k}$ because the time scale $\bMinus \hat{\tau}_k$ in the mollification is smaller than $\th_{(k)}$, and because the kernel in the mollification is non-negative with integral equal to $1$.  Since $\supp R_{(k)} \subseteq I_{(k)} \times \bbT^{3}$ by Claim~\ref{claim:room:to:add}, it follows that our choice of $e_{(k)}(t)$ is admissible for Lemma~\ref{lem:mod:mainLemma}.

To conclude the proof, we now verify Claims~\ref{claim:growingSupp}, \ref{claim:room:to:add} and \ref{claim:nearly:saturated}.  In the process, we will specify the constants $Y$ and $\overline{M}$.

%% file: prescribeEnClaims3.tex
In the following Section, we verify that Claims~\ref{claim:growingSupp}-\ref{claim:nearly:saturated} hold during the iteration given the choice of energy function in \eqref{eq:mollified:increm} and the choice of $I_{(k)}$ defined in \eqref{eq:I-k-def}.  In the process, we explain the choice of the constants 
$Y$ and $\overline{M}$ (which are required to be independent of the constants in Lemma~\ref{lem:mod:mainLemma}, and also must be independent of the parameter $Z$).  In this proof, we will therefore say that a constant $C$ is {\it universal} if it is independent of $Y$, $\overline{M}$ and $Z$, and we will use the letter $\hat{C}$ to denote constants which are universal.  Some of the constants here will depend (in an increasing manner) on the homogeneous H\"{o}lder seminorm of $\bar{e}$, which we denote by
\[ \| \bar{e} \|_{\dot{C}_t^\ga} = \sup_t \sup_{\De t \neq 0 } \fr{|\bar{e}(t + \De t) - \bar{e}(t)|}{|\De t|^\ga} \]
Our starting point is to prove an estimate on the control of the energy profile that will be used at several points in the verification of Claims~\ref{claim:growingSupp}-\ref{claim:nearly:saturated}.

\paragraph{The main estimate on the energy gap}
Claims~\ref{claim:room:to:add}-\ref{claim:nearly:saturated} require us to control the difference $\bar{e}(t) - E_{k+1}(t)$.  Our main tool for estimating this difference is the following Lemma.
\begin{lem}\label{lem:en:int:control} We have an approximation
\ali{
\bar{e}(t) - E_{k+1}(t) &= \bar{e}(t) - E_k(t) - (\bar{e}(t) - E_k(t) - \De E_k)_+ + O((1 + \| \bar{e} \|_{\dot{C}_t^\ga}) e_{R,(k+1)}) 
}
where the constant in the $O(~)$ is universal.
\end{lem}
\begin{proof}
The starting point for establishing this control is the following calculation:
\ali{
E_{k+1}(t) &= E_k(t) + \int_{\T^3} ( |v_{(k)} + V_{(k)}|^2(t,x) - |v_{(k)}|^2(t,x) ) dx \\
&= E_k(t) + \int_{\T^3} |V_{(k)}|^2(t,x) dx  + 2 \int_{\T^3} v_{(k)} \cdot V_{(k)} dx \label{eq:eKplus1:terms}
}
For the last term, Lemma~\ref{lem:mod:mainLemma} gives an estimate
\ali{
\left| \int_{\T^3} v_{(k)} \cdot V_{(k)} dx \right| &\leq \left| \int_{\T^3} \nab \times v_{(k)} \cdot W_{(k)} dx \right| \\
&\leq \hat{C} ( \Xi_{(k)} e_{v,(k)}^{1/2} ) (N_{(k)}^{-1} \Xi_{(k)}^{-1} e_{R,(k)}^{1/2} ) \\
&\leq \hat{C} {\bf b}^2 e_{R,(k)} = \hat{C} {\bf b} e_{R,(k+1)} = \fr{\hat{C}}{Z} e_{R,(k+1)}
}
Note that the constant here can be made smaller than $1$ if $Z$ is larger than some universal constant.

For the second term, Lemma~\ref{lem:mod:mainLemma} gives a bound 
\ali{
\left| \int_{\T^3} |V_{(k)}|^2(t,x) dx - \int_{\T^3} e_{(k)}(t) dx \right| &\leq {\bf b}^{-1} \fr{e_{v,(k)}^{1/2} e_{R,(k)}^{1/2}}{N_{(k)}} \\
&\leq e_{R,(k+1)}
}
We assume here for simplicity that our torus $\T^3$ has unit volume.  Note that both the estimates above use the remark in Lemma~\ref{lem:mod:mainLemma} on the universality of the constants in  \eqref{eq:Wco} and \eqref{eq:energyPrescribed}.

Combining these estimates, \eqref{eq:eKplus1:terms} gives
\ali{
E_{k+1} &= E_k + e_{(k)}(t) + O( e_{R,(k+1)} )
}
where the constant in the $O(\cdot)$ notation is universal.

The proof of Lemma~\ref{lem:en:int:control} concludes by applying the following Lemma, which gives an estimate  for how well the smoothed out energy increment
\[ e_{(k)}(t) = [(\bar{e} - E_k - \De E_k)_+^{1/2} \ast \eta_{\bMinus \hat{\tau}_k} ]^2 \]
 approximates the desired energy increment $\tilde{e}_{(k)}(t)$.
\begin{lem} \label{lem:commutator:energy:lem}
There is a universal constant $\hat{C}$ such that 
\ali{
\| e_{(k)}(t) - (\bar{e} - E_k - \De E_k)_+ \|_{C_t^0} &\leq \hat{C} \left( \| \bar{e} \|_{\dot{C}_t^\ga} |\hat{\tau}_k|^\ga + \bMinus \Xi_{(k)} e_{v,(k)}^{1/2} e_{R,(k)} |\hat{\tau}_k| \right) \label{eq:smoothed:out:en:est}
}
\end{lem}
For now we postpone the proof of Lemma~\ref{lem:commutator:energy:lem}, which is based on the commutator estimate of \cite{CET}, and the following bound on the derivative of the energy profile of the approximate solution
\ali{
\| \fr{d}{dt} E_k \|_{C^0_t} &\leq \hat{C} \Xi_{(k)} e_{v,(k)}^{1/2} e_{R,(k)} \label{eq:energyReg}
}
We will return to the proof of Lemma~\ref{lem:commutator:energy:lem} in Section~\ref{sec:commutEstimate}.

Lemma~\ref{lem:en:int:control} now follows from Lemma~\ref{lem:commutator:energy:lem} if we can estimate the right hand side of \eqref{eq:smoothed:out:en:est} by $O(e_{R,(k+1)})$.  For the second term in \eqref{eq:smoothed:out:en:est}, recalling $|\hat{\tau}_k| = {\bf b} \Xi_{(k)}^{-1} e_{v,(k)}^{-1/2}$ and $e_{R,(k+1)} = Z^{-1} e_{R,(k)} = {\bf b} e_{R,(k)}$ gives
\[ \bMinus \Xi_{(k)} e_{v,(k)}^{1/2} e_{R,(k)} |\hat{\tau}_k| = \bMinus e_{R,(k+1)}\] 
For the first term in \eqref{eq:smoothed:out:en:est}, we want to estimate
\ali{
|\hat{\tau}_k|^{\ga} &= Q_{(k)} e_{R,(k+1)} \notag \\
Q_{(k)} &= \fr{|\hat{\tau}_k|^\ga}{e_{R,(k+1)}} \label{eq:QkBoundNeed}
}
For $k = 0$, we established the inequality $Q_{(0)} \leq 1$ in line \eqref{eq:choiceOfZeroFreq}.  For larger values of $k$ we can decide whether $Q_{(k)}$ increases in size by calculating
\ali{
Q_{(k+1)} &= (C_0^{-\ga} Z^{- 2 \ga} Z ) Q_{(k)} \label{eq:energyStaysOk}
}
Here $C_0 > 1$ is the large constant in Lemma~\ref{lem:en:int:control} for the gain in the frequency level.  Recalling now that $\ga = \fr{1}{2}$, we see that $Q_{(k+1)} \leq Q_{(k)} \leq 1$ for all $k$, which concludes the proof of Lemma~\ref{lem:en:int:control}.
\end{proof}

Given Lemma~\ref{lem:en:int:control}, we are now in a position to verify Claims~\ref{claim:growingSupp}, \ref{claim:room:to:add} and \ref{claim:nearly:saturated}.  We start by verifying Claims~\ref{claim:growingSupp} and \ref{claim:room:to:add}, which will require us to fix our choice of $Y$.  

\subsubsection{Verifying Claims~\ref{claim:growingSupp} and \ref{claim:room:to:add}} \label{subsub:verifyLowerBound}

We now check that Claims~\ref{claim:growingSupp} and \ref{claim:room:to:add} hold provided the constants $Y$ and $Z$ are chosen appropriately.  Our proof of Claims~\ref{claim:growingSupp} and \ref{claim:room:to:add} will proceed by induction, and it will be necessary to couple these claims together in the argument in order to close the induction.  For simplicity, we will suppress the dependence of constants on the norm $\nrm{\bar{e}}_{\dot{C}^{\gmm}_{t}}$ in what follows. 

Let $k \geq 0$.  To confirm Claims~\ref{claim:growingSupp} and \ref{claim:room:to:add}, we must verify that
\begin{align} 
	&\bar{e}(t) - E_{k+1}(t) \geq 0 \qquad \hbox{ for } t \in \bbR, \label{eq:below-bar-e:need} \\
	&\bar{e}(t) - E_{k+1}(t) \geq 3 K e_{R, (k+1)} \quad \hbox{ for } t \in I(2 \tht_{(k+1)}; I_{(k+1)}), \label{eq:saturated:need} \\
	&\supp v_{(k+1)} \cup  \supp p_{(k+1)} \cup \supp R_{(k+1)} \subseteq I_{(k+1)} \times \bbT^{3}, 	\label{eq:supp-in-I:need} \\
	&I_{(k)} \subseteq I_{(k+1)}  \label{subset:inductContainment}
\end{align}
hold under the induction hypothesis that \eqref{eq:below-bar-e:need}--\eqref{subset:inductContainment} hold for $k$ replacing $k+1$.  In the base case of the iteration (i.e. $k + 1 = 0$), the requirements \eqref{eq:below-bar-e:need}-\eqref{eq:supp-in-I:need} are satisfied trivially, as $I_{(0)} = \emptyset$ and $E_{0}(t) = 0$, while the containment \eqref{subset:inductContainment} imposes no restriction.  We now consider the case $k + 1 > 0$.

Recall that in \eqref{eq:I-k-def}, we defined 
\[ I_{(k+1)} = I(2 \hat{\tau}_{k} ; \set{t \in \bbR : \bar{e}(t) - E_{k}(t) \geq \Dlt E_{k}}).\]

We begin by checking the containment \eqref{subset:inductContainment}.  Let $t_0$ be an element of $I_{(k)}$.  By inequality \eqref{eq:saturated:need} for $k$, we have that $\bar{e}(t_0) - E_{k}(t_0) \geq 3 K e_{R, (k)}$.  Recalling that $\De E_{k} = Y e_{R,(k+1)} = Y \fr{e_{R,(k)}}{Z}$, we have that $t_0 \in I_{(k+1)}$ as long as we impose the condition
\ali{
 Z &\geq (3K)^{-1} Y  \label{eq:imposeZ2}
}
on our choice of $Z$.  Assuming this restriction, we have that $I_{(k)} \subseteq I_{(k+1)}$.  

We next verify the containment \eqref{eq:supp-in-I:need} for $k+1$.  Observe that the definition of $I_{(k+1)}$ implies
\begin{equation} \label{eq:supp-e-in-I}
	 I(\hat{\tau}_{k}; \supp e_{(k)}) \times \bbT^{3} \subseteq I_{(k+1)} \times \bbT^{3},
\end{equation}
where $e_{(k)}$ is defined by $e_{(k)}^{1/2} (t) := (\bar{e}(t) - E_{k}(t) - \Dlt E_{k})^{1/2} \ast \eta_{\hat{\tau}_{k}}$. By the containment \eqref{eq:goalForAllSuppM} of Lemma~\ref{lem:mod:mainLemma}, the containment \eqref{eq:supp-e-in-I}, and recalling our choice of $\hat{\tau}_k$, we obtain
\begin{equation} \label{eq:supp-VPR-in-I}
	 \supp R_{(k)} \subseteq I_{(k+1)} \times \bbT^{3}, \quad \supp V_{(k)} \cup \supp P_{(k)} \subseteq I_{(k+1)} \times \bbT^{3}.
\end{equation}
From the definition of $(v_{(k+1)}, p_{(k+1)}) = (v_{(k)} + V_{(k)}, p_{(k)} + P_{(k)})$, it follows that \eqref{eq:supp-in-I:need} holds for $k+1$.

We now prove that \eqref{eq:below-bar-e:need} holds for $k+1$ under the assumption that \eqref{eq:saturated:need} holds for $k+1$.  From the definition of $v_{(k+1)} = v_{(k)} + V_{(k)}$ and the containment \eqref{eq:supp-VPR-in-I}, it follows that
\begin{equation} \label{eq:outside-supp-V}
\bar{e}(t) - 	E_{k+1}(t) = \bar{e}(t) - E_{k}(t) \geq 0 \quad \hbox{ for } t \not \in I_{(k+1)},
\end{equation}
since $t \not \in I_{(k+1)}$ implies that $E_{k+1}(t) = \int \abs{v_{(k)} + V_{(k)}}^{2}(t, x) \, dx = \int \abs{v_{(k)}}^{2}(t,x) d x = E_{k}(t)$.  Under the assumption that \eqref{eq:saturated:need} holds for $k+1$, we also have the inequality \eqref{eq:below-bar-e:need} for $t \in I_{(k+1)}$, which confirms \eqref{eq:below-bar-e:need} for $k+1$.

To complete proof of Claims~\ref{claim:growingSupp} and \ref{claim:room:to:add}, it only remains to establish \eqref{eq:saturated:need} for $k+1$, for an appropriate choice of $Y$ (independent of $k$, of course). First observe that for $t \in I_{(k+1)}$, we have by Lemma~\ref{lem:en:int:control}
\begin{equation}  \label{eq:saturated:t-in-I-k+1}
	\bar{e}(t) - E_{k+1}(t) 
	= \Dlt E_{k} + O(e_{R, (k+1)}) \geq (Y - C) e_{R, (k+1)},
\end{equation}
where $C \geq 0$ is a constant independent of $k$. Now let $t \in I(2 \tht_{(k+1)}; I_{(k+1)})$.  Writing $t$ in the form $t = t' + \Dlt t$ where $t' \in  I_{(k+1)}$ and $\abs{\Dlt t} \leq 2 \tht_{(k+1)}$, we have
\begin{align*}
	\bar{e}(t) - E_{k+1}(t)
	= & \bar{e}(t') - E_{k+1}(t') + (\bar{e}(t) - \bar{e}(t')) - (E_{k+1}(t) -E_{k+1}(t')) \\
	\geq & \bb( Y - C - 2^{\gmm} \nrm{\bar{e}}_{\dot{C}^{\gmm}_{t}} \frac{\tht_{(k+1)}^{\gmm}}{e_{R, (k+1)}} - \hat{C} \Xi_{(k+1)} e_{v, (k+1)}^{1/2} 2 \tht_{(k+1)} \bb) e_{R, (k+1)}
\end{align*}
where we used \eqref{eq:saturated:t-in-I-k+1}, H\"older continuity of $\bar{e}$ and \eqref{eq:energyReg} on the last line.  We bound the last term from below by $-2 \hat{C} e_{R,(k+1)}$, recalling the identity $\tht_{(k+1)} = \Xi_{(k+1)}^{-1} e_{v, (k+1)}^{-1/2}$.  The third term is bounded by
\[ \tht^{\gmm}_{(k+1)} / e_{R, (k+1)} = Z^{\ga-1} Q_{(k+1)} \leq 1, \] 
which follows from the iteration rules, the definitions \eqref{eq:theTimeScales} and \eqref{eq:QkBoundNeed} of $\th_{(k)}$ and $Q_{(k)}$, the choices of $Z, Q_{(0)} \leq 1$, and the fact that $\ga - 1  = -\fr{1}{2} < 0$.
From these estimates we arrive at

\begin{equation} \label{eq:saturated:t-in-I-I-k+1}
	\bar{e}(t) - E_{k+1}(t) \geq (Y - C - 2^{\gmm} \nrm{\bar{e}}_{\dot{C}^{\gmm}_{t}} - 2 \hat{C}) e_{R, (k+1)} \quad \hbox{ for } t \in I(2 \tht_{(k+1)} ; I_{(k+1)})
\end{equation}
Choosing $Y \geq C + 2^{\gmm} \nrm{\bar{e}}_{\dot{C}^{\gmm}_{t}} + 2 \hat{C} + 3 K$, the desired statement \eqref{eq:saturated:need} follows.

\subsubsection{Verifying the upper bound Claim~\ref{claim:nearly:saturated}} \label{sec:verifyLastClaim}
We now verify the upper bound in Claim~\ref{claim:nearly:saturated} for stage $k+1$, and in the process we specify the constant $\overline{M}$ for this upper bound.  This estimate follows quickly from Lemma~\ref{lem:en:int:control} now that the constant $Y$ has already been chosen and we have the lower bound $\bar{e}(t) - E_{k+1}(t) \geq 0$ from Claim~\ref{claim:room:to:add}.  The proof splits into two cases.  In the first case, we consider $t$ for which $\bar{e}(t) - E_k - \De E_k \geq 0$.  In this case, Lemma~\ref{lem:en:int:control} gives
\ali{
\bar{e}(t) - E_{k+1}(t)  &= \bar{e} - E_k - (\bar{e} - E_k - \De E_k)_+ + O(e_{R,(k+1)}) \\
&= \De E_k + O(e_{R,(k+1)}) = Y e_{R,(k+1)} + O(e_{R,(k+1)})  \label{eq:energy:gap:inside}
}
where the constant in the $O(~)$ notation is universal.  In particular, we have \eqref{eq:nearSaturated} on this set.

For other values of $t$, we have an upper bound $\bar{e}(t) - E_k  < \De E_k$.  In this case, Lemma~\ref{lem:en:int:control} gives the same upper bound
\ali{
\bar{e}(t) - E_{k+1}(t)  &= \bar{e} - E_k - (\bar{e} - E_k - \De E_k)_+ + O(e_{R,(k+1)}) \\
&\leq \De E_k + O(e_{R,(k+1)}) = Y e_{R,(k+1)} + O(e_{R,(k+1)}) \label{eq:energy:gap:edge}
}
Recalling that the constant in the $O( )$ notation is universal, the above bound depends only on $Y$.

We now choose the constant $\overline{M}$ in Claim~\ref{claim:nearly:saturated} depending on $Y$ such that the estimates \eqref{eq:energy:gap:inside}-\eqref{eq:energy:gap:edge} hold.  This choice of $\overline{M}$ together with the estimate \eqref{eq:ek:has:bound} now determine the constant $M$ in our applications of Lemma~\ref{lem:mod:mainLemma}.  

With this bound, we have established Claim~\ref{claim:nearly:saturated}, which concludes our proof of Theorem~\ref{thm:eulerOnRn:period:presEnergy}.  The last remaining detail is to establish the Lemma~\ref{lem:commutator:energy:lem}, which had been used in the proof of Lemma~\ref{lem:en:int:control}.  We accomplish this step in Section~\ref{sec:commutEstimate} below.

%% file: prescribeEnCommutator3.tex
In this Section, we complete the proof of Theorem~\ref{thm:eulerOnRn:period:presEnergy} by establishing Lemma~\ref{lem:commutator:energy:lem}.

The main idea is to decompose the difference into two terms
\ali{ 
e_{(k)}(t) - (\bar{e} - E_k - \De E_k)_+ &= T_{I} + T_{II} \\
T_I &= [(\bar{e} - E_k - \De E_k)_+^{1/2} \ast \eta_{\bMinus \hat{\tau}_k} ]^2 - (\bar{e} - E_k - \De E_k)_+ \ast \eta_{\bMinus \hat{\tau}_k} \\
T_{II} &= (\bar{e} - E_k - \De E_k)_+ \ast \eta_{\bMinus \hat{\tau}_k} - (\bar{e} - E_k - \De E_k)_+
}
We can then establish the estimate of Lemma~\ref{lem:commutator:energy:lem} for each term individually using the estimates
\ali{
\begin{split}
|\bar{e}(t + \De t) - \bar{e}(t)| &\leq \| \bar{e} \|_{\dot{C}_t^\ga} |\De t|^\ga \\
|E_k(t + \De t) - E_k(t)| &\leq \hat{C} \Xi_{(k)} e_{v,(k)}^{1/2} e_{R,(k)} |\De t| 
\end{split} \label{eq:energy:reg:vk}
}
To obtain the second estimate involving $E_k$ in \eqref{eq:energy:reg:vk}, we apply an observation in \cite{deLSzeBuck}, which is that this bound can be obtained from the Euler-Reynolds equations in the same way that one usually proves conservation of energy for Euler. 
\ali{
\fr{d}{dt} E_k = \fr{d}{dt} \int_{\T^3} |v_{(k)}|^2(t,x) dx &= \int_{\T^3} v_{(k),l} \pr_j R_{(k)}^{jl} dx \\
&= - \int_{\T^3} \pr_j v_{(k),l} R_{(k)}^{jl} dx
}
The desired bound for the term $T_{II}$
\ali{
 \| T_{II} \|_{C_t^0} &\leq \hat{C} \left( \| \bar{e} \|_{\dot{C}_t^\ga} |\hat{\tau}_k|^\ga + \bMinus \Xi_{(k)} e_{v,(k)}^{1/2} e_{R,(k)} |\hat{\tau}_k| \right) \label{eq:boundForTII}
}
now follows easily from \eqref{eq:energy:reg:vk}, where $\hat{C}$ is some constant depending on the volume of the torus $\T^3$.  

We now show that the bounds in \eqref{eq:energy:reg:vk} also imply the same estimate for the term $T_{I}$.  The main idea is to view the difference $T_I$ as a quadratic commutator term as in the well-known commutator estimate of \cite{CET} (i.e. the term can be written in the form $f \ast \eta_\ep g \ast \eta_\ep - (f g)\ast \eta_\ep$ for the appropriate functions $f$ and $g$ and the appropriate mollifying kernel $\eta_\ep$).  Setting $\hat{e}_{(k)}^{1/2}(t) = (\bar{e}(t) - E_k(t) - \De E_k)_+^{1/2}$, this commutator structure allows us to write the term $T_I$ as
\ali{
T_I &= \int_\R \left( \hat{e}_{(k)}^{1/2}(t + \tau) - \eta_{\bMinus \hat{\tau}_k} \ast \hat{e}_{(k)}^{1/2}(t) \right)^2 \eta_{\bMinus \hat{\tau}_k}(\tau) d\tau \label{eq:goodCommutatorFormula}
}
We first estimate the integrand of \eqref{eq:goodCommutatorFormula} pointwise at each fixed value of $\tau \in \R$. We begin with the elementary inequality 
\[ |(y + \De y)_+^{1/2} - (y)_+^{1/2}| \leq |\De y|^{1/2} \qquad \tx{ for all } y, \De y \in \R. \]
Taking $y = \bar{e}(t) + E_k(t) - \De E_k$ and $y + \De y = \bar{e}(t + \tau) + E_k(t + \tau)- \De E_k$ in the above inequality, we apply the bounds in \eqref{eq:energy:reg:vk} to obtain the estimate
\[ |\hat{e}_{(k)}^{1/2}(t + \tau) - \hat{e}_{(k)}^{1/2}(t) | \leq \hat{C}^{1/2} \left( \| \bar{e} \|_{\dot{C}_t^\ga} |\tau|^{\ga} + \Xi_{(k)} e_{v,(k)}^{1/2} e_{R,(k)} |\tau| \right)^{1/2}  \] 
for all $\tau \in \R$.  
From the above estimate on the modulus of continuity of $\hat{e}_{(k)}^{1/2}$ and the containment $\supp \eta_{\hat{\tau}_k} \subseteq \{ \tau \in \R ~:~ |\tau| \leq \hat{\tau}_k \}$, it is now straightforward to estimate the integrand of \eqref{eq:goodCommutatorFormula} pointwise, and to obtain the desired bound \eqref{eq:boundForTII} for the term $T_I$.  This bound concludes the proof of Lemma~\ref{lem:commutator:energy:lem}.

%% file: logicExplained.tex
In this Section, we summarize the proof of Theorem~\ref{thm:eulerOnRn:period:presEnergy}, and clarify the logical order in which the parameters involved in proving these claims are chosen.  What we have shown in Sections~\ref{sec:baseEnInc}-\ref{sec:commutEstimate} above is the following statement:
\begin{prop}[Summary of the Iteration] \label{prop:itSummary}
Given a positive number $\a < \a^* = 1/5$, an open interval $I \subseteq \R$, a compactly supported, non-negative function $\bar{e}(t) \in C_t^\ga(I)$, $\ga = \fr{2 \a^*}{1-\a^*} = \fr{1}{2}$, and a real number $B > 0$ which bounds the H\"{o}lder semi-norm  $\| \bar{e} \|_{\dot{C}_t^\ga}$ from above, there exist:
\begin{itemize}
\item Non-negative constants $\overline{M}, C_0 $ and $Z$;
\item A sequence of parameters $(\Xi_{(k)}, e_{v,(k)}, e_{R,(k)})$, $\Xi_{(k)} \geq 2$, $e_{v,(k)}, e_{R,(k)} \geq 0$;
\item A sequence of Euler-Reynolds flows $(v_{(k)}, p_{(k)}, R_{(k)})$;
\item A sequence of subsets $I_{(k)} \subseteq I$,
\end{itemize}
such that 
\begin{itemize}
\item The iteration rules \eqref{eq:MLMT:parameterEvolution:Xi:M}-\eqref{eq:MLMT:parameterEvolution:eR:M} relating $(\Xi_{(k)}, e_{v,(k)}, e_{R,(k)})$, $C_0$ and $Z$ hold for all $k \geq 0$.
\item The frequency energy levels of $(v_{(k)}, p_{(k)}, R_{(k)})$ are below $(\Xi_{(k)}, e_{v,(k)}, e_{R,(k)})$ to order $2$ in $C^0$.
\item The sequence $(v_{(k)}, p_{(k)})$ converges in $C_{t,x}^\a \times C_{t,x}^{2 \a}(I \times \T^3)$ to a solution of incompressible Euler.
\item The containment $\supp v_{(k)} \cup \supp p_{(k)} \cup \supp R_{(k)} \subseteq I_{(k)} \times \bbT^{3}$ holds as stated in \eqref{eq:supp-in-I}.
\item The containment $I_{(k)} \subseteq I_{(k+1)}$ holds as stated in \eqref{subset:growing} for all $k \geq 0$. 
\item For $E_k(t) = \fr{1}{2}\int_{\T^3} |v_{(k)}|^2(t,x) dx$, the inequalities \eqref{eq:below-bar-e}, \eqref{eq:needed:room} and \eqref{eq:nearSaturated} which relate the functions $\bar{e}(t)$, $E_k(t)$ to the sets $I_{(k)}$ and the parameters $\overline{M}$, $Z$, and $(\Xi_{(k)}, e_{v,(k)}, e_{R,(k)})$ hold for all $k \geq 0$.
\end{itemize}
\end{prop}
The proof of Proposition~\ref{prop:itSummary} involves the introduction of a parameter $Y$ which is used to define the energy increments $e_{(k)}(t)$ during the iteration.  We also define a parameter ${\bf b} = Z^{-1}$ and time scales $\th_{(k)} = \Xi_{(k)}^{-1} e_{v,(k)}^{-1/2}$ and $\hat{\tau}_k = Z^{-1} \Xi_{(k)}^{-1} e_{v,(k)}^{-1/2}$ for ease of notation.

In the base case of the iteration (Section~\ref{sec:baseEnInc}), we choose the initial Euler-Reynolds flow to be $(v_{(0)}, p_{(0)}, R_{(0)}) = (0, 0, 0)$, and $I_{(0)} = \emptyset$, while the initial the energy level $e_{R,(0)}$ is chosen to depend only on the norm $\|\bar{e}(t)\|_{C_t^0}$.  At this point the parameters $(\Xi_{(0)}, e_{v,(0)})$ remain unspecified as they will depend on the choice of $Z$.  The reason for this dependence is that we desire a sharp time scale in the first stage of the iteration, and this goal is accomplished by taking a large value of $Z$.

The parameter $Y$ appearing in \eqref{eq:DeEkDef} is the next parameter specified.  This parameter is chosen at the end of Section~\ref{subsub:verifyLowerBound}.  The choice of $Y$ depends only on: certain universal constants $C$ and $\hat{C}$ appearing in Lemma~\ref{lem:en:int:control} and Section~\ref{sec:commutEstimate} where Lemma~\ref{lem:commutator:energy:lem} is proven; the universal constant $K$ from the construction; and the upper bound ($B$ above) for the $\dot{C}_t^\ga$ H\"{o}lder semi-norm of $\bar{e}$.  

The constant $\overline{M}$ in Claim~\ref{claim:nearly:saturated:nonperiodic} is the second parameter specified.  This constant depends on the parameter $Y$ and the other absolute constants from the Lemmas in Section~\ref{sec:verifyClaims}.  The choice of $\overline{M}$ is made in Section~\ref{sec:verifyLastClaim}.  With the constant $\overline{M}$ determined, the sequence of upper bounds 
\ali{
 \| (\fr{d}{dt})^r e_{(k)}^{1/2} \|_{C^0_t} &\leq A  (Z \Xi_{(k)} e_{v,(k)}^{1/2})^r [ \overline{M} e_{R,(k)} ]^{1/2}, \qquad 0 \leq r \leq \label{ineq:boundSequence}
}
stated on the right hand side of \eqref{eq:ek:has:bound} are fully determined up to the choice of $Z$ and the determination of $C_0$.  (In this equation, we have substituted $Z$ for ${\bf b}^{-1}$ in order to distinguish the ${\bf b}$ in the definition of $\hat{\tau}_k$ and the parameter $\left( \fr{e_v^{1/2}}{e_R^{1/2}N} \right)$ appearing in Lemma~\ref{lem:mod:mainLemma}.)

We choose $C_0$ to be the constant whose existence is asserted by Lemma~\ref{lem:mod:mainLemma} with $L = 2$ and $M = A \overline{M}$.

With $C_0$ chosen, it is possible to determine the appropriate choice of $Z$ subject to some requirements.  First, $Z$ is sufficiently large depending on $\a$, $C_0$ and other absolute constants to ensure $C_{t,x}^\a \times C_{t,x}^{2\a}$ convergence of the sequence $(v_{(k)}, p_{(k)})$ as in equation \eqref{eq:suffLargeZ}.  More precisely, $Z$ is chosen sufficiently large so that the sequence of bounds on the $C_{t,x}^\a \times C_{t,x}^{2\a}$ norms of the corrections which result from the iteration will be summable (as in \eqref{eq:suffLargeZ}).  Furthermore, $Z$ satisfies the requirements $Z \geq \max\{ 2Y, (3K)^{-1} Y\}$ coming from \eqref{ineq:requireZY1} and \eqref{eq:imposeZ2}.

With the constants $C_0$ and $Z$ determined, the full sequence of parameters $(\Xi_{(k)}, e_{v,(k)}, e_{R,(k)})$ along with the time scales $\th_{(k)}, \hat{\tau}_k$ are determined by induction according to the iteration rules \eqref{eq:MLMT:parameterEvolution:Xi:M}-\eqref{eq:MLMT:parameterEvolution:eR:M}, and the initial choice of $e_{v,(0)} = Z e_{R, (0)}$ and $\Xi_{(0)}$ made in \eqref{eq:choiceOfZeroFreq}.  The energy increment $e_{(0)}^{1/2}(t) = (\bar{e}(t) - Y e_{R,(0)})_+^{1/2} \ast \eta_{\hat{\tau}_0}$ for initializing the iteration has also been determined (it is possible that $e_{(0)}(t) = 0$).  The set $I_{(1)} := I(2 \hat{\tau}_{0}; \set{t \in \bbR : \bar{e}(t) \geq Y e_{R,(0)} } )$ has also been determined according to \eqref{eq:I-k-def} (it is possible that $I_{(1)}$ is empty).

With these parameters, we generate a sequence of Euler-Reynolds flows $(v_{(k)}, p_{(k)}, R_{(k)})$ by repeated application of Lemma~\ref{lem:mod:mainLemma}.  This iteration simultaneously generates a function $e_{(k)}(t)$ and a set $I_{(k)} \subseteq \R$ associated to each Euler Reynolds flow $(v_{(k)}, p_{(k)}, R_{(k)})$ according to the formulas \eqref{eq:mollified:increm} and \eqref{eq:I-k-def}.  The assumption that $\bar{e}$ has compact support in $I$ together with inequality \eqref{eq:needed:room} imply by induction that $I_{(k)} \subseteq I$ for all $k \geq 0$.  Lemma~\ref{lem:mod:mainLemma} is applied in each stage choosing $L = 2$ and the parameter $M$ to be the constant $A \overline{M}$.  In each stage, we take\footnote{We remark that in principle $N$ could be allowed to depend on the stage $k$, as was the case in \cite{isett}. } $N = Z^{5/2}$, and define the energy increment $e_{(k)}(t)$ according to \eqref{eq:mollified:increm}.  

Our choices of parameters have been made such that both $N = Z^{5/2}$ and $e_{(k)}(t)$ defined in \eqref{eq:I-k-def} are admissible according to the requirements $N \geq (e_{v,(k)} / e_{R,(k)})^{3/2}$, \eqref{eq:lowBoundEoftxM} and \eqref{ineq:goodEnergyM}.  The requirement $N \geq (e_{v,(k)} / e_{R,(k)})^{3/2}$ follows by induction from the parameter evolution rules.  To verify the estimates in \eqref{ineq:goodEnergyM}, we must check that the factor of $Z$ appearing in the right hand side of inequality \eqref{ineq:boundSequence} is no larger than the loss of the factor ${\bf b}^{-1} = \left(\fr{e_{v}^{1/2}}{e_{R} N} \right)_{(k)}^{-1/2}$ allowed by the Lemma.  From the parameter evolution rules, it follows by induction that this factor is equal to $Z$ for all $k$.  It follows by induction that the admissibility condition \eqref{ineq:goodEnergyM} is satisfied for all indices $k$ under the assumption inequality \eqref{ineq:boundSequence} holds.  In Section~\ref{sec:admissEnergy}, we verify that the inequality \eqref{ineq:boundSequence} holds for the sequence of functions $e_{(k)}(t)$ using the inductive Claim~\ref{claim:nearly:saturated}.  In Section~\ref{sec:admissEnergy}, we also verify the required lower bound \eqref{eq:lowBoundEoftxM} of Lemma~\ref{lem:mod:mainLemma} using the inductive Claims~\ref{claim:growingSupp}-\ref{claim:room:to:add}.

%% file: prescribeEnRemarks4.tex
Having concluded the proof of Theorem~\ref{thm:eulerOnRn:period:presEnergy}, we outline how our argument above extends to establish Theorem~\ref{thm:eulerOnRn:presEnergy} in the periodic case.  We address the additional technical issues involved in the nonperiodic case in Section~\ref{sec:prescribe:Nonperiodic:Energy}.

We first observe that the ideas of our proof of the case $\ga = \fr{1}{2}$ in Theorem~\ref{thm:eulerOnRn:period:presEnergy} can be extended to give solutions $v \in C_{t,x}^\a$ with prescribed energy profiles $\bar{e} \in C_t^\ga$ having lower regularity $0 < \a < \a^* \leq 1/5$, $\ga = \fr{2 \a^*}{1 - \a^*}$.  For example, one can state a variant of Lemma~\ref{lem:mod:mainLemma} where the stress is reduced at an inferior rate of $e_R' = {\bf b}^{-\b} \fr{e_v^{1/2} e_R^{1/2}}{N}$ for some $\b \geq 1$, and the number ${\bf b}^{-1}$ is replaced by ${\bf b}^{-\b}$ in all of the estimates.  One must also replace the smallness factor ${\bf b}$ in the enlargement of the time support by the smaller factor ${\bf b}^\b$.  In this case, the same argument establishes Theorem~\ref{thm:eulerOnRn:period:presEnergy} with lesser regularity.  The crucial point at which the exponent $\ga = \fr{2 \a^*}{1 - \a^*}$ comes into play is in the estimate \eqref{eq:energyStaysOk}, where we estimate the difference between the smoothed out energy increment and the desired energy increment.  As $\b$ tends to infinity, the threshold $\a^*$ for the H\"{o}lder regularity tends to $0$, as does $\ga = \fr{2 \a^*}{1 - \a^*}$.  In the opposite direction, if one assumes the same lemma but with $0 \leq \b \leq 1$, then the threshold $\a^*$ would tend to $1/3$ while the required regularity $\ga = \fr{2 \a^*}{1 - \a^*}$ would tend to $1$.

Next, we observe that our solutions with prescribed energy profiles in the class $C_t^\ga$ become arbitrarily small in the $C_{t,x}^{\a}$ topology if $\a < \a^*$ and $\ga = \fr{2 \a^*}{1 - \a^*}$ when we consider a one-parameter family of energy profiles tending to $0$ as in the statement of Theorem~\ref{thm:eulerOnRn:presEnergy}.  To check this observation, consider the algorithm in the proof of Theorem~\ref{thm:eulerOnRn:period:presEnergy}, and apply this algorithm to an energy profile 
\ali{
 \bar{e}_A(t) &= A \bar{e}(t) \label{eq:smallerEnProfile}
}
 where $A \leq 1$ is some constant.  Note that the choice of the constants $Y$, $\overline{M}$ and $Z$ in our algorithm do not depend on $A$, but rather depend only on the desired regularity $\a$ and possibly on an upper bound for the $C_t^\ga$ norm of $\bar{e}$ (see Sections~\ref{subsec:MLMT:verifyClaims} and \ref{prop:itSummary}).    
Recall also that the bounds on the $C_{t,x}^\a$ norms of the corrections $V_{(k)}$ decrease exponentially by a certain factor for all indices $k \geq 0$ thanks to the choice of the parameter $Z$.  Thus, to check that the algorithm produces a solution that is arbitrarily small in $C_{t,x}^\a$ as $A$ tends to $0$, the key point is to check that the size of the initial correction $V_{(0)}$ becomes arbitrarily small in $C_{t,x}^\a$ as the parameter $A$ tends to $0$.

To check that this smallness holds, recall from Section~\ref{subsec:MLMT:verifyClaims} that the correction obeys the estimate
\ali{
\| V_{(0)} \|_{C_{t,x}^\a} &\leq C (N_{(0)} \Xi_{(0)})^\a e_{R,(0)}^{1/2} \label{eq:V0Ca:est}
}
The constant $C$ here is universal.  The parameters $N_{(0)}$ and $\Xi_{(0)}$ both depend on $Z$, but we can ignore this dependence since $Z$ is fixed.  What we consider here is the dependence on $A$.  Recall from Section~\ref{sec:baseEnInc} that $e_{R,(0)}$ is proportional to $\| A \bar{e} \|_{C^0}$, and is therefore proportional to $A$.  The initial frequency level is chosen in Section~\ref{sec:baseEnInc} to have size $\Xi_{(0)} \sim e_{R,(0)}^{-\fr{1}{\ga} - \fr{1}{2}}$.  Our estimate \eqref{eq:V0Ca:est} therefore scales as
\[ \| V_{(0)} \|_{C_{t,x}^\a} \leqc A^{\fr{1}{2} - \a( \fr{1}{\ga} + \fr{1}{2}) } \]
with an implied constant depending on the choices of $Y, \overline{M}$ and $Z$.  The above bound tends to $0$ as $A \to 0$ provided $\ga > \fr{2 \a}{1 - \a}$.  Thus, the initial correction $V_{(0)}$, and furthermore the sum of all corrections $\sum_{k=0}^\infty V_{(k)}$, may be made arbitrarily small in $C_{t,x}^\a$ by applying our algorithm to the energy profile $A \bar{e}(t)$ with $A$ small and $\bar{e} \in C_t^\ga$.

From this calculation, we can view our Theorem~\ref{thm:eulerOnRn:presEnergy} as providing some evidence for the conjecture in \cite{isett2} that irregularity of the energy profile is characteristic of generic solutions to Euler with H\"{o}lder regularity strictly below $1/3$.  Here we have shown that, within the range of exponents $\a < 1/5$ and $\ga > \fr{2 \a}{1-\a}$,  Euler flows which are arbitrarily small perturbations of $0$ in $C_{t,x}^\a$ can have energy profiles that fail to have $C^\ga$ regularity in time.  In view of Theorem~\ref{thm:eulerOnRn:cptPert}, it is very likely that our methods show that solutions with such irregular energy profiles can approximate any sufficiently smooth solution to Euler in $C_{t,x}^\a$.  We are optimistic that the method of convex integration can be extended to make statements about generic Euler flows, at least those with regularity below $1/5$.  

%% file: prescribeNonperiodicEnergy3.tex
In this Section, we describe how to modify the proof of Theorem~\ref{thm:eulerOnRn:period:presEnergy} to construct a H\"older continuous weak solution to the Euler equations with a prescribed energy profile in the non-periodic setting.
\begin{thm}[Nonperiodic Euler flows with prescribed energy profile] \label{thm:eulerOnRn:nonperiod:presEnergy}
Let $\a < \a^* \leq 1/5$ and let $I \subseteq \R$ be a bounded open interval.
Let $\bar{e}(t) \geq 0$ be any non-negative function with compact support in $I$ which belongs to the class $\bar{e}(t) \in C_t^{\ga}$ for $\ga = \fr{2 \a^*}{1 - \a^*}$.  Then there exists a weak solution $(v, p)$ to the Euler equations in the class $v \in C_{t,x}^\a(\R \times \R^3)$ with compact support in space-time such that the energy profile of $v$ is given by
\begin{equation} \label{eq:nonperiod:presEnergy}
\int_{\R^3} |v|^2(t,x) dx = \bar{e}(t), \qquad t \in \R
\end{equation}
\end{thm}
In fact, our construction below will produce a weak solution $(v, p)$ to the Euler equations in $C_{t,x}^{\alp}(\bbR \times \bbR^{3})$ supported in the space-time cylinder
\begin{equation} \label{eq:prescribeNonperiodicEnergy:supp-vp}
	\supp (v, p) \subseteq \supp \bar{e} \times B(2;0),
\end{equation}
where $B(2; 0) \subseteq \bbR^{3}$ is simply the closed ball of radius $2$ centered at the origin.  By making straightforward modifications to the argument below, one can also arrange for the solutions of Theorem~\ref{thm:eulerOnRn:nonperiod:presEnergy} $(v, p)$ to have spatial supports contained in an arbitrarily small open subset $\calU \subseteq \R^3$, as asserted in Theorem~\ref{thm:eulerOnRn:presEnergy}.

\begin{proof} 
Unless otherwise stated, we employ the same notation as in the proof of Theorem~\ref{thm:eulerOnRn:period:presEnergy}. As before, we only consider the case $\alp^{*} = 1/5$ and $\gmm = 1/2$. Beginning with the trivial solution $(v_{(0)}, p_{(0)}, R_{(0)}) = (0, 0, 0)$, we will construct a sequence of solutions $(v_{(k)}, p_{(k)}, R_{(k)})$ to the Euler-Reynolds equations that obeys the following properties:
\begin{enumerate}
\item Each solution $(v_{(k)}, p_{(k)}, R_{(k)})$ has frequency-energy levels below $(\Xi_{(k)}, e_{v, (k)}, e_{R, (k)})$, which evolves under the iteration rules \eqref{eq:MLMT:parameterEvolution:Xi:M}--\eqref{eq:MLMT:parameterEvolution:eR:M} (or \eqref{eq:MLMT:parameterEvolution:Xi}--\eqref{eq:MLMT:parameterEvolution:eR}); recall that these rules ensure that $(v_{(k)}, p_{(k)})$ converges to a weak solution to the Euler equations in $C^{\alp}_{t,x} \times C^{2 \alp}_{t,x}$ for $0 < \alp < \alp^{*} =1/5$. 
\item In addition to keeping track of the time support of the Euler-Reynolds flows $(v_{(k)}, p_{(k)}, R_{(k)})$, we now need to take into account their supports in space. We will construct sets $I_{(k)} \subseteq \bbR$, $B_{(k)} \subseteq \bbR^{3}$ so that for each $k$, we have the space-time support condition
\begin{equation}  \label{eq:prescribeNonperiodicEnergy:supp-vpR}
	\supp v_{(k)} \cup \supp p_{(k)} \cup \supp R_{(k)} \subseteq I_{(k)} \times B_{(k)}
\end{equation} 
and $B_{(k)} \subseteq \bbR^{3}$ is an open ball satisfying
\begin{equation} \label{eq:prescribeNonperiodicEnergy:Bk}
	B_{(k)} \subseteq B_{(k+1)} \subseteq B(2; 0).
\end{equation}
The last property ensures that the limiting Euler flow $(v, p)$ obeys \eqref{eq:prescribeNonperiodicEnergy:supp-vp}.
\item Finally, the energy profiles $E_{k}(t)$ converge pointwise to $\bar{e}(t)$ as $k \to \infty$, i.e.,
\begin{equation} \label{eq:prescribeNonperiodicEnergy:E-k-bar-e}
	E_{k}(t) := \int_{\bbR^{3}} \abs{v_{(k)}}^{2}(t,x) \, \ud x \to \bar{e}(t) \quad \hbox{ as } k \to \infty.
\end{equation}
This property achieves the desired energy prescription \eqref{eq:nonperiod:presEnergy}.
\end{enumerate}

As before, we construct the sequence $(v_{(k)}, p_{(k)}, R_{(k)})$ via iteration of Lemma~\ref{lem:mod:mainLemma}. To this end, we need to choose: 
\begin{itemize}
\item The initial space-time set $I_{(0)} \times B_{(0)}$,
\item The initial frequency-energy levels $(\Xi_{(0)}, e_{v, (0)}, e_{R, (0)})$, 
\item The iteration factor $Z$, 
\item The energy increment $\widetilde{e}_{(k)}(t,x)$ and the space-time set $I_{(k)} \times B_{(k)}$ for each $k \geq 0$.
\end{itemize}

We set
\begin{equation}
	I_{(0)} = \emptyset, \quad B_{(0)} = B(1; 0)
\end{equation}
Similarly as before, we then take 
%
%
\begin{equation}  \label{eq:prescribeNonperiodicEnergy:e-vR-0}
	e_{R, (0)} = \frac{A'}{ \abs{B(1;0)}} \max_{t \in I} \bar{e}(t), \quad e_{v, (0)} = Z e_{R, (0)}.
\end{equation}
where $A' > 0$ is an absolute constant that will be specified in \eqref{eq:prescribeNonperiodicEnergy:chi-k-size} below. The iteration factor $Z$ will be chosen so that  

\begin{equation} \label{eq:prescribeNonperiodicEnergy:Z:1}
	Z \geq C_{0}^{\frac{2 \alp}{5 \eps}}, \quad 
	Z \geq (2 C_{0}^{-1})^{2/5},
\end{equation}
where $C_{0} > 1$ is the constant arising in the iteration rule \eqref{eq:MLMT:parameterEvolution:Xi:M}. The first condition, which coincides with \eqref{eq:suffLargeZ}, ensures that the resulting Euler flow $(v, p)$ belongs to $C^{1/5-\eps}_{t,x} \times C^{2/5-2\eps}_{t,x}$. The second condition will be used to control the growth of $B_{(k)}$ defined in \eqref{eq:prescribeNonperiodicEnergy:Bk-def}. We emphasize, however, that the value of $Z$ will be fixed only later. 
Similarly, for $\Xi_{(0)}$, we require
\begin{equation} \label{eq:prescribeNonperiodicEnergy:Xi-0}
	\Xi_{(0)} \geq \max \set{100, \  (Z^{1-\frac{3 \gmm}{2}} e^{-1-\frac{\gmm}{2}}_{R, (0)})^{1/\gmm}},
\end{equation}
but its actual value will be fixed after $Z$ has been chosen. The frequency-energy levels $(\Xi_{(k)}, e_{v, (k)}, e_{R, (k)})$ for $k \geq 1$ are determined by the iteration rules \eqref{eq:MLMT:parameterEvolution:Xi:M}--\eqref{eq:MLMT:parameterEvolution:eR:M}. Note that by \eqref{eq:prescribeNonperiodicEnergy:Z:1}, \eqref{eq:prescribeNonperiodicEnergy:Xi-0} and the iteration rules (see also the proof of \eqref{eq:energyStaysOk} before), we have
\begin{equation} \label{eq:prescribeNonperiodicEnergy:Q-k}
	Q_{(k)} := \frac{\abs{\hat{\tau}_{k}}^{\gmm}}{e_{R, (k+1)}} \leq Q_{(0)} \leq 1.
\end{equation}


Now it only remains to specify the space-time sets $I_{(k)} \times B_{(k)}$ and the energy density $\widetilde{e}_{(k)}(t,x)$. We need $\widetilde{e}_{(k)}(t,x)$ to be admissible in the sense that \eqref{eq:lowBoundEoftxM}, \eqref{ineq:goodEnergyM} are satisfied with $(\Xi_{(k)}, e_{v, (k)}, e_{R, (k)})$ as specified above. On the other hand, we need to ensure that the desired properties \eqref{eq:prescribeNonperiodicEnergy:supp-vpR}, \eqref{eq:prescribeNonperiodicEnergy:Bk} and \eqref{eq:prescribeNonperiodicEnergy:E-k-bar-e} hold with our choice of $I_{(k)} \times B_{(k)}$ and $\widetilde{e}_{(k)}(t,x)$. 

Here, our strategy is to essentially reduce the proof to that of Theorem~\ref{thm:eulerOnRn:period:presEnergy}. We proceed recursively: Under the assumption that $I_{(k)} \times B_{(k)}, v_{(k)}, p_{(k)}, R_{(k)}$ have been constructed so that $(v_{(k)}, p_{(k)}, R_{(k)})$ has frequency-energy level below $(\Xi_{(k)}, e_{v, (k)}, e_{R, (k)})$ and \eqref{eq:prescribeNonperiodicEnergy:supp-vpR} is satisfied, we construct appropriate $I_{(k+1)} \times B_{(k+1)}$ and $\widetilde{e}_{(k)}$. First, we define $B_{(k+1)}$ as
\begin{equation} \label{eq:prescribeNonperiodicEnergy:Bk-def}
	B_{(k+1)} := B(10 \Xi_{(k)}^{-1} ; B_{(k)}) \quad \hbox{ for } k \geq 0
\end{equation}
where $B(\overline{\rho}; S) := \set{x + \Dlt x : x \in S, \ \abs{\Dlt x} \leq \overline{\rho}}$. Since $B_{(0)} = B(1; 0)$, note that each $B_{(j)}$ is a closed ball centered at the origin as well; we will denote the radius of $B_{(j)}$ by $r_{(j)}$. By the second condition in \eqref{eq:prescribeNonperiodicEnergy:Z:1}, \eqref{eq:prescribeNonperiodicEnergy:Xi-0} and the iteration rules, \eqref{eq:prescribeNonperiodicEnergy:Bk} follows. Next, as in the proof of Theorem~\ref{thm:eulerOnRn:period:presEnergy}, we set
\begin{equation} \label{eq:prescribeNonperiodicEnergy:e-prof}
	e^{1/2}_{(k)}(t) = (\bar{e}(t) - E_{k}(t) - \Dlt E_{k})_{+}^{1/2} \ast \eta_{\hat{\tau}_{k}}
\end{equation}
where the gap $\Dlt E_{k}$ now takes the form\footnote{This definition is exactly analogous to \eqref{eq:DeEkDef}, as we assumed $\abs{\bbT^{3}} = 1$ before.}
\begin{equation} \label{eq:prescribeNonperiodicEnergy:e-gap}
	\Dlt E_{k} = Y e_{R, (k+1)} \abs{B_{(k)}},
\end{equation}
and $Y$ is a constant to be chosen later. We then take the energy density $\widetilde{e}^{1/2}_{(k)}(t,x)$ to be of the form
\begin{equation} \label{eq:prescribeNonperiodicEnergy:e-d}
	\widetilde{e}^{1/2}_{(k)}(t,x) = e^{1/2}_{(k)}(t) \chi_{(k)}(x)
\end{equation}
where $\chi_{(k)}(x)$ is a smooth non-negative function on $\bbR^{3}$ that obeys
\begin{align} 
	& \supp  \chi_{(k)} \subseteq B(5 \Xi_{(k)}^{-1}; B_{(k)}), \label{eq:prescribeNonperiodicEnergy:chi-k-supp} \\
	& B(2 \Xi_{(k)}^{-1}; B_{(k)}) \subseteq \set{x \in \bbR^{3} : \chi_{(k)}(x) = \chi_{(k)}(0)},\label{eq:prescribeNonperiodicEnergy:chi-k-const} \\
	& \int \chi_{(k)}(x) = 1, \quad 
	\abs{\nb^{m} \chi_{(k)}} \leq A' \frac{\Xi_{(k)}^{m}}{\abs{B_{(k)}}} \quad \hbox{ for } 0 \leq m \leq 2, \label{eq:prescribeNonperiodicEnergy:chi-k-size}
\end{align}
for some absolute constant $A' > 0$.
To construct such a function $\chi_{(k)}(x)$, consider a smooth radial function $\eta(r)$ which equals 1 on $\set{r \leq r_{(k)} + 2 \Xi_{(k)}^{-1}}$ and vanishes on $\set{r \geq r_{(k)} + 5 \Xi_{(k)}^{-1}}$, and then normalize $\chi_{(k)}(x) = c \eta(\abs{x})$ so that $\int \chi_{(k)} = 1$. Finally, the set $I_{(k+1)}$ is defined as
\begin{equation} \label{eq:prescribeNonperiodicEnergy:I-k-def}
	I_{(k+1)} = I(2 \hat{\tau}_{k}; \set{t \in \bbR : \bar{e}(t) - E_{k}(t) \geq \Dlt E_{k} }) \quad \hbox{ for } k \geq 0,
\end{equation}
where we recall the notation $I(\overline{\tau}; J) = \set{t + \Dlt t : t \in J, \ \abs{\Dlt t} \leq \overline{\tau}}$.

The Ansatz \eqref{eq:prescribeNonperiodicEnergy:e-d} reduces the question of admissibility of the energy density $\widetilde{e}_{(k)}(t,x)$ to that of the energy profile $e_{(k)}(t,x)$, which we have dealt with in the proof of Theorem~\ref{thm:eulerOnRn:period:presEnergy}. Indeed, note that

\begin{equation*}
 \hat{C}_{v_{(k)}}(\tht_{(k)}, \Xi_{(k)}^{-1} ; \supp R_{(k)}) \subseteq I(\tht_{(k)}; I_{(k)}) \times B(\Xi_{(k)}^{-1}; B_{(k)})
\end{equation*}
which follows from $\supp v_{(k)} \cup \supp R_{(k)} \subseteq I_{(k)} \times B_{(k)}$ (i.e., \eqref{eq:prescribeNonperiodicEnergy:supp-vpR} for $k$) and the duality \eqref{iff:EuLagDuality}. Hence, the desired lower bound \eqref{eq:lowBoundEoftxM} follows (using \eqref{eq:prescribeNonperiodicEnergy:chi-k-const}) once we prove the bound $e_{(k)}(t) \chi_{(k)}(0) \geq 2 K e_{R, (k)}$ for $t \in I(\tht_{(k)}; I_{(k)})$. Using \eqref{eq:prescribeNonperiodicEnergy:chi-k-size}, we can further reduce \eqref{eq:lowBoundEoftxM} to the following lower bound on $e_{(k)}(t)$:
\begin{equation} \label{eq:prescribeNonperiodicEnergy:e-prof-lb}
	e_{(k)}(t) \geq \frac{2 K \abs{B_{(k)}}}{A'} e_{R, (k)}, \quad t \in I(\tht_{(k)}; I_{(k)}).
\end{equation}
Next, by \eqref{eq:prescribeNonperiodicEnergy:chi-k-supp} and \eqref{eq:prescribeNonperiodicEnergy:chi-k-const}, we have
\begin{equation} \label{eq:prescribeNonperiodicEnergy:chi-k-supp:2}
	\supp v_{(k)} \cap \supp \nb \chi_{(k)} = \emptyset, 
\end{equation}
which implies
\begin{equation}
	\nb^{m} (\rd_{t} + v_{(k)} \cdot \nb)^{r} \, \widetilde{e}^{1/2}_{(k)}(t,x)
	= (\frac{d}{d t})^{r} e^{1/2}_{(k)}(t) \nb^{m} \chi_{(k)}(x).
\end{equation}
Hence, by \eqref{eq:prescribeNonperiodicEnergy:chi-k-size}, the desired upper bound \eqref{ineq:goodEnergyM} for $\widetilde{e}^{1/2}_{(k)}(t,x)$ with $L = 2$ follows once $M$ is chosen to be such that
\begin{equation} \label{eq:prescribeNonperiodicEnergy:e-prof-ub}
	\nrm{(\frac{d}{d t})^{r} e_{(k)}^{1/2}}_{C^{0}_{t}} \leq \frac{M}{A'} (\bfb^{-1} \Xi_{(k)} e_{v, (k)})^{r} e_{R, (k)}^{1/2}, \quad 0 \leq r \leq 1.
\end{equation}

Repeating the arguments in the proof of Theorem~\ref{thm:eulerOnRn:period:presEnergy}, the following analogues of Claims~\ref{claim:growingSupp}, \ref{claim:room:to:add} and \ref{claim:nearly:saturated} can be established using induction (note that \eqref{eq:prescribeNonperiodicEnergy:e-vR-0} ensures that these claims hold for $k = 0$): 
\begin{claim}[Growing Supports] \label{claim:growingSuppNonperiodic} For all $k \geq 0$, we have
\ali{
 \label{eq:supp-in-I:nonperiodic}
	\supp v_{(k)} \cup \supp p_{(k)} \cup \supp R_{(k)} &\subseteq I_{(k)} \times B_{(k)} \\
	I_{(k)} &\subseteq I_{(k+1)} \label{subset:growintTimeRn} \\ 
B(1,0) \subseteq	B_{(k)} &\subseteq B_{(k+1)} \subseteq B(2,0) \label{subset:growingBallRn}
} 
\end{claim}

\begin{claim}[There is Always Room to Add More Energy where the Error is Supported] \label{claim:room:to:add:nonperiodic}
For every $t \in \bbR$, we have
\begin{equation} \label{eq:below-bar-e:nonperiodic}
	\bar{e}(t) - E_{k}(t) \geq 0.
\end{equation}
Moreover, for $t \in I(2 \hat{\tau}_{k}; I_{(k)})$, we have
\begin{equation} \label{eq:needed:room:nonperiodic}
	\bar{e}(t) - E_{k}(t) \geq \frac{3 K}{A'} e_{R, (k)} \abs{B_{(k)}},
\end{equation}
where $K$ is the constant in the lower bound \eqref{eq:lowBoundEoftx} of Lemmas~\ref{lem:mainLemma} and \ref{lem:mod:mainLemma}, $A'$ is the constant in \eqref{eq:prescribeNonperiodicEnergy:chi-k-size}. 
\end{claim}

\begin{claim}[The Energy Threshold is Nearly Saturated]\label{claim:nearly:saturated:nonperiodic}
There is an absolute constant $\overline{M}$ such that the upper bound
\ali{
 \sup_t | \bar{e}(t) - E_k(t) | &\leq \overline{M} e_{R, (k)} \abs{B_{(k)}} \label{eq:nearSaturated:nonperiodic}
}
holds uniformly.
\end{claim}
We remark that the factors of $\abs{B_{(k)}}$ on the right-hand sides of \eqref{eq:needed:room:nonperiodic} and \eqref{eq:nearSaturated:nonperiodic} ensure that these estimates are dimensionally correct; note that in \eqref{eq:needed:room} and \eqref{eq:nearSaturated}, we had $\abs{\bbT^{3}} = 1$. The presence of the factor $\abs{B_{(k)}}$ does not cause any significant modification of the proof, as $\abs{B_{(k)}}$ is bounded from below and above by absolute constants by construction, i.e., $\abs{B(1;0)} \leq \abs{B_{(k)}} \leq \abs{B(2;0)}$. The absolute constant $A' > 0$ in \eqref{eq:needed:room:nonperiodic} does not introduce any difficulty as well. It is in the proof of these claims that the constant $Y > 0$ in \eqref{eq:prescribeNonperiodicEnergy:e-gap} is fixed, and the iteration constant $Z$ is required to be even larger depending on $Y$. We omit the routine modifications.

We are now ready to conclude the proof of Theorem~\ref{thm:eulerOnRn:nonperiod:presEnergy}. Arguing as in the proofs of \eqref{eq:ek:has:lb} and \eqref{eq:ek:has:bound}, the desired estimates \eqref{eq:prescribeNonperiodicEnergy:e-prof-lb} and \eqref{eq:prescribeNonperiodicEnergy:e-prof-ub} (with a constant $M > 0$ independent of $k$) for $e_{(k)}(t)$ follow from Claims~\ref{claim:growingSuppNonperiodic}, \ref{claim:room:to:add:nonperiodic} and \ref{claim:nearly:saturated:nonperiodic} once $Z$ is chosen to be sufficiently large. Hence Lemma~\ref{lem:mod:mainLemma} (with $L = 2$) can be applied to $(v_{(k)}, p_{(k)}, R_{(k)})$ to produce $(v_{(k+1)}, p_{(k+1)}, R_{(k+1)})$ with frequency-energy levels below $(\Xi_{(k+1)}, e_{v, (k+1)}, e_{R, (k+1)})$.  The support property \eqref{eq:prescribeNonperiodicEnergy:supp-vpR} for  $(v_{(k+1)}, p_{(k+1)}, R_{(k+1)})$ follows from \eqref{eq:goalForAllSuppM}, and \eqref{eq:prescribeNonperiodicEnergy:E-k-bar-e} is a quick consequence of \eqref{eq:nearSaturated:nonperiodic}. \qedhere
\end{proof}

%% file: hPrincipleRemark.tex
In this Appendix, we observe that our construction leads to a result of ``$h$-principle'' type given in Theorem~\ref{thm:hPrinciple} below.  
To motivate this theorem, recall Proposition~\ref{prop:conserveLaw}, according to which every finite energy weak solution to Euler with appropriate integrability conserves linear and angular momentum. Furthermore, note that if $v_{n}$ is a sequence of finite energy solutions to Euler with appropriate uniform integrability (say, the family $\set{(1+\abs{x}) v_{n}(t)}_{n, t}$ is uniformly integrable in $x$), then the weak limit $v_{n} \rightharpoonup v$, provided that it exists, also conserves linear and angular momentum. Theorem~\ref{thm:hPrinciple} essentially says that there are no other conservation laws closed under taking weak limits. More precisely, this theorem
 shows that every smooth, divergence free vector field on $\R \times \R^3$ which conserves both linear and angular momentum can be realized as a weak limit of a sequence of $C_{t,x}^{1/5 - \ep}$ Euler flows in the $L_{t,x}^\infty$ weak-* topology.  We note that the space $L_{t,x}^\infty$ cannot be improved for this type of result in terms of regularity, and the result below implies weak convergence in $L^p$ spaces for $1 < p < \infty$ as well.

\begin{thm}\label{thm:hPrinciple} Let $\ep > 0$ and let $\calU$ be a bounded, convex, open subset of $\R \times \R^3$.  Let $v^l \in \CC_c^\infty(\R \times \R^3)$ be a smooth vector field with compact support in $\calU$ such that for all $t \in \R$ we have
\ALI{
\pr_l v^l(t,x) &= 0 \qquad \forall~x \in \R^3  \\
\fr{d}{dt} \int_{\R^3} v^l(t,x)~dx &= 0 \qquad \forall~l = 1, 2, 3\\
\fr{d}{dt} \int_{\R^3} (x^k v^l(t,x) - x^l v^k(t,x) )~dx &= 0 \qquad \forall~1 \leq k < l \leq 3
}
Then there exists a sequence  of solutions to incompressible Euler in the class $(v_{(k)}, p_{(k)}) \in C_{t,x}^{1/5 - \ep} \times C_{t,x}^{2(1/5-\ep)}(\R \times \R^3)$ such that $\supp v_{(k)} \cup \supp p_{(k)} \subseteq \calU  \tx{ for all } k \in \N$ and $v_{(k)} \rightharpoonup v \tx{ in } L_{t,x}^\infty \tx{ weak-}^\ast$.
\end{thm}
Theorem~\ref{thm:hPrinciple} contributes to the growing literature on $h$-principle type results in fluid equations, for which we refer the reader to \cite{deLSzeHFluid, choff, choffSzeStationary, isettVicol} for further discussion. The result helps to express the point that the only results that appear to be closed under weak limits for low regularity solutions to these equations can be viewed as conservation laws or as time regularity statements.
Here we will outline the main ideas of the proof of Theorem~\ref{thm:hPrinciple}, and we will refer the reader to \cite{isettVicol} for a detailed proof of an analogous result for active scalar equations.  

\subsection{Sketch of proof of Theorem~\ref{thm:hPrinciple}}

Let $\ep > 0$ and let $\calU$ be a bounded, convex, open subset of $\R \times \R^3$.  Let $v^l \in \CC_c^\infty(\calU)$ be an incompressible velocity field which conserves both linear and angular momentum, as in the statement of Theorem~\ref{thm:hPrinciple}.  Consider the vector field $U^l = \pr_t v^l + \pr_j(v^j v^l)$.  One can interpret $U^l(t,x)$ as the force per unit volume (or unit mass) acting on a particle at the point $(t,x)$, since $U^l = \pr_t v^l + v^j \pr_j v^l$ by incompressibility.

Choose a smooth, symmetric tensor field $R^{jl} \in C_c^\infty(\R \times \R^3)$ with compact support in $\calU$ such that 
\ali{
 \pr_j R^{jl} &= U^l  \label{eq:symmDivEqn}
}
As we have seen, it is necessary for $U^l(t,\cdot)$ to be $L^2$-orthogonal to both translation and rotation vector fields at all times $t \in \R$ in order for such a tensor field to exist.  For the vector field $U^l$ above, the orthogonality conditions are equivalent to the conservation laws assumed in Theorem~\ref{thm:hPrinciple}, since the term $\pr_j(v^j v^l)$ is already the divergence of a symmetric tensor.  With these conditions satisfied, we can construct the desired $R^{jl}$ using the operators constructed in Proposition~\ref{prop:Rjl} (where we take the ambient velocity field to be $0$ so that the operator is time-independent).

With this choice of $R^{jl}$, we may view $v^l$ as part of a smooth solution $(v_{(0)}, p_{(0)}, R_{(0)})$ to the Euler-Reynolds equations with velocity field $v_{(0)}^l = v^l$, pressure $p_{(0)} = 0$ and stress tensor $R_{(0)}^{jl} = R^{jl}$ as chosen above.  The proof of Theorem~\ref{thm:hPrinciple} now proceeds along the same lines as the proof of Theorem~\ref{thm:eulerOnRn:cptPert} given in Section~\ref{sec:mainLemImpliesMainThm}.  Namely, beginning with $(v_{(0)}, p_{(0)}, R_{(0)})$, one generates a sequence  of Euler Reynolds flows $(v_{(k)}, p_{(k)}, R_{(k)})$ by repeated application of Lemma \ref{lem:mainLemma} such that the sequence $(v_{(k)}, p_{(k)})$ converges in $C_{t,x}^{1/5 - \ep} \times C_{t,x}^{2(1/5-\ep)}$ to a solution $(\hat{v}, \hat{p})$ of incompressible Euler.  This sequence of Euler Reynolds flows $(v_{(k)}, p_{(k)}, R_{(k)})$ is dictated by the choice of the sequence of frequency energy levels $(\Xi_{(k)}, e_{v,(k)} e_{R,(k)})$, which obey the iteration rules \eqref{eq:MLMT:parameterEvolution:Xi:M}-\eqref{eq:MLMT:parameterEvolution:eR:M}.  The solution $(\hat{v}, \hat{p})$ is thus determined completely by the choice of initial frequency energy levels $(\Xi_{(0)}, e_{v,(0)}, e_{R,(0)})$ and the choice of the frequency $\Xi_{(1)}$ applied in the first stage of the iteration.

  The key point in achieving solutions $\hat{v}$ which are close to the given $v = v_{(0)}$ in $L_{t,x}^\infty$ weak-* is that the initial frequency $\Xi_{(1)}$ (and all subsequent frequencies) may be chosen arbitrarily large in the first stage of the iteration while maintaining a uniform bound on $\|\hat{v} - v \|_{L_{t,x}^\infty}$ that is independent of the choice of $\Xi_{(1)}$.  In fact, one can arrange that $\hat{v} - v = \nab \times W$ where $\co{W} \leq \Xi_{(1)}^{-1} e_{R,(0)}^{1/2}$ can be made arbitrarily small, while maintaining a bound of the form $\co{\hat{v} - v} \leq C e_{R,(0)}^{1/2}$ and uniform control over the support of $\hat{v} - v$.  To arrange that the support of the iteration remains inside a precompact subset of $\Om$, one may choose a larger frequency level $\Xi_{(0)}$ if necessary, since the choice of a sufficiently large frequency level at the beginning of the iteration will cause the time and spatial scales of the entire iteration to become arbitrarily small.  Choosing a sequence of $\Xi_{(1)}$ tending to $\infty$, one obtains the desired sequence\footnote{One may also construct continuous families of such $\hat{v}$ if desired using the construction employed here.} of solutions $\hat{v}$.  We refer to \cite[Proof of Theorem 9.1]{isettVicol} for a detailed implementation of this technique.

%% file: eulerOnRn.bbl
\begin{thebibliography}{BDLIS14}

\bibitem[BDLIS14]{buckDeLIsettSze}
T.~Buckmaster, C.~De~Lellis, P.~Isett, and L.~Sz{\' e}kelyhidi, Jr.
\newblock Anomalous dissipation for {$1/5$}-{H}{\" o}lder {E}uler flows. {T}o
  appear in.
\newblock {\em Annals of Mathematics}, 2014.

\bibitem[BDLS13]{deLSzeBuck}
T.~Buckmaster, C.~De~Lellis, and L.~Sz{\' e}kelyhidi, Jr.
\newblock Transporting microstructures and dissipative {E}uler flows,
  {P}reprint.
\newblock 2013.

\bibitem[BDLS14]{buckDeLSzeOnsCrit}
T.~Buckmaster, C.~De~Lellis, and L.~Sz{\' e}kelyhidi, Jr.
\newblock Dissipative {E}uler flows with {O}nsager-critical spatial regularity,
  {P}reprint.
\newblock 2014.

\bibitem[Bog80]{bogovskii1980solution}
ME~Bogovskii.
\newblock Solution for some vector analysis problems connected with operators
  div and grad, theory of cubature formulas and application of functional
  analysis to problems of mathematical physics.
\newblock {\em Trudy Sem. SL Soboley}, (1):5--40, 1980.

\bibitem[Buc15]{Buckmaster}
Tristan Buckmaster.
\newblock Onsager's {C}onjecture {A}lmost {E}verywhere in {T}ime.
\newblock {\em Comm. Math. Phys.}, 333(3):1175--1198, 2015.

\bibitem[CCFS08]{ches}
A.~Cheskidov, P.~Constantin, S.~Friedlander, and R.~Shvydkoy.
\newblock Energy conservation and {O}nsager's conjecture for the {E}uler
  equations.
\newblock {\em Nonlinearity}, 21(6):1233--1252, 2008.

\bibitem[CET94]{CET}
Peter Constantin, Weinan E, and Edriss~S. Titi.
\newblock Onsager's conjecture on the energy conservation for solutions of
  {E}uler's equation.
\newblock {\em Comm. Math. Phys.}, 165(1):207--209, 1994.

\bibitem[Cho13]{choff}
Antoine Choffrut.
\newblock h-{P}rinciples for the {I}ncompressible {E}uler {E}quations.
\newblock {\em Arch. Ration. Mech. Anal.}, 210(1):133--163, 2013.

\bibitem[CS12]{chesShv}
A.~Cheskidov and R.~Shvydkoy.
\newblock Euler equations and turbulence: analytical approach to intermittency.
  {P}reprint.
\newblock 2012.

\bibitem[CS14]{choffSzeStationary}
A.~Choffrut and L.~Sz{\' e}kelyhidi, Jr.
\newblock Weak solutions to the stationary incompressible {E}uler equations.
  {P}reprint.
\newblock 2014.

\bibitem[DLS09]{deLSzeIncl}
Camillo De~Lellis and L{\'a}szl{\'o} Sz{\'e}kelyhidi, Jr.
\newblock The {E}uler equations as a differential inclusion.
\newblock {\em Ann. of Math. (2)}, 170(3):1417--1436, 2009.

\bibitem[DLS10]{deLSzeAdmiss}
Camillo De~Lellis and L{\'a}szl{\'o} Sz{\'e}kelyhidi, Jr.
\newblock On admissibility criteria for weak solutions of the {E}uler
  equations.
\newblock {\em Arch. Ration. Mech. Anal.}, 195(1):225--260, 2010.

\bibitem[DLS12a]{deLSzeHoldCts}
C.~De~Lellis and L.~Sz{\' e}kelyhidi, Jr.
\newblock Dissipative {E}uler flows and {O}nsager's conjecture. {P}reprint.
\newblock 2012.

\bibitem[DLS12b]{deLSzeHFluid}
Camillo De~Lellis and L{\'a}szl{\'o} Sz{\'e}kelyhidi, Jr.
\newblock The {$h$}-principle and the equations of fluid dynamics.
\newblock {\em Bull. Amer. Math. Soc. (N.S.)}, 49(3):347--375, 2012.

\bibitem[DLS13]{deLSzeCts}
Camillo De~Lellis and L{\'a}szl{\'o} Sz{\'e}kelyhidi.
\newblock Dissipative continuous {E}uler flows.
\newblock {\em Invent. Math.}, 193(2):377--407, 2013.

\bibitem[DLS14]{deLSzeCtsSurv}
Camillo De~Lellis and L{\'a}szl{\'o} Sz{\'e}kelyhidi, Jr.
\newblock Dissipative {E}uler flows and {O}nsager's conjecture.
\newblock {\em J. Eur. Math. Soc. (JEMS)}, 16(7):1467--1505, 2014.

\bibitem[DR00]{duchonRobert}
Jean Duchon and Raoul Robert.
\newblock Inertial energy dissipation for weak solutions of incompressible
  {E}uler and {N}avier-{S}tokes equations.
\newblock {\em Nonlinearity}, 13(1):249--255, 2000.

\bibitem[ES06]{eyinkSreen}
{G. L.} Eyink and {K. R.} Sreenivasan.
\newblock Onsager and the theory of hydrodynamic turbulence.
\newblock {\em Reviews of Modern Physics}, 78, 2006.

\bibitem[Eyi94]{eyink}
Gregory~L. Eyink.
\newblock Energy dissipation without viscosity in ideal hydrodynamics. {I}.
  {F}ourier analysis and local energy transfer.
\newblock {\em Phys. D}, 78(3-4):222--240, 1994.

\bibitem[Eyi08]{eyinkDissip}
Gregory~L. Eyink.
\newblock Dissipative anomalies in singular {E}uler flows.
\newblock {\em Phys. D}, 237(14-17):1956--1968, 2008.

\bibitem[FS78]{frisch1978simple}
Uriel Frisch and Pierre-Louis Sulem.
\newblock A simple dynamical model of intermittent fully developed turbulence.
\newblock {\em Journal of Fluid Mechanics}, 87(4):719--736, 1978.

\bibitem[IO13]{isettOh}
P.~Isett and S.-J. Oh.
\newblock A heat flow approach to {O}nsager's conjecture for the {E}uler
  equations on manifolds. {T}o appear in.
\newblock {\em Transactions of the AMS}, 2013.

\bibitem[Ise12]{isett}
P.~Isett.
\newblock H{\"{o}}lder continuous {E}uler flows in three dimensions with
  compact support in time. {P}reprint.
\newblock 2012.

\bibitem[Ise13]{isett2}
P.~Isett.
\newblock Regularity in time along the coarse scale flow for the {E}uler
  equations. {P}reprint.
\newblock 2013.

\bibitem[IV14]{isettVicol}
P.~Isett and V.~Vicol.
\newblock Holder continuous solutions of active scalar equations. {P}reprint.
\newblock 2014.

\bibitem[Kol41]{K41}
A.~N. Kolmogorov.
\newblock The local structure of turbulence in an incompressible viscous fluid.
\newblock {\em C. R. (Doklady) Acad. Sci. URSS (N.S.)}, 30:301--305, 1941.

\bibitem[Ons49]{onsag}
L.~Onsager.
\newblock Statistical hydrodynamics.
\newblock {\em Nuovo Cimento (9)}, 6 Supplemento(2 (Convegno Internazionale di
  Meccanica Statistica)):279--287, 1949.

\bibitem[Sch93]{scheff}
Vladimir Scheffer.
\newblock An inviscid flow with compact support in space-time.
\newblock {\em J. Geom. Anal.}, 3(4):343--401, 1993.

\bibitem[Shv10]{shvOns}
R.~Shvydkoy.
\newblock Lectures on the {O}nsager conjecture.
\newblock {\em Discrete Contin. Dyn. Syst. Ser. S 3}, 3:473--496, 2010.

\bibitem[Tay11]{taylorBook}
Michael~E. Taylor.
\newblock {\em Partial differential equations {I}. {B}asic theory}, volume 115
  of {\em Applied Mathematical Sciences}.
\newblock Springer, New York, second edition, 2011.

\end{thebibliography}
